\documentclass{amsart}
\usepackage{amssymb}
\usepackage{amsthm}
\usepackage{mathrsfs}
\usepackage{stmaryrd}
\usepackage[french,english]{babel}
\usepackage[utf8]{inputenc}
\usepackage[T1]{fontenc}
\usepackage{lmodern}
\usepackage[all]{xy}
\usepackage{hyperref}
\usepackage{upref}
\usepackage{array}
\usepackage{color}
\usepackage{chngcntr}
\usepackage[titletoc]{appendix}

\def\cD{\mathscr{D}}
\def\PP{\mathscr{P}}
\def\XX{\mathfrak{X}}
\def\UU{\mathfrak{U}}
\def\YY{\mathfrak{Y}}

\def\SS{\mathfrak{S}}
\def\bR{\mathbb{R}}

\def\Rep{\mathbf{Rep}}

\def\FHom{\mathscr{H}om}
\def\Ab{\mathbf{Ab}}

\def\Dbhol{\rD_{\hol}^{\rb}}
\def\Gm{\mathbb{G}_{m}}
\def\Ga{\mathbb{G}_{a}}
\def\A1{\mathbb{A}^1}
\def\P1{\mathbb{P}^1}
\def\cG{\check{G}}

\def\cB{\check{B}}
\def\gg{\mathfrak{g}}
\def\so{\mathfrak{so}}
\def\Gzt{\mathcal{G}(0,2)}
\def\Got{\mathcal{G}(1,2)}
\def\Gmn{\mathcal{G}(m,n)}
\def\cH{{}^{\textnormal{c}}\mathcal{H}}
\def\hH{\mathcal{H}}
\def\wX{\mathbb{X}^{\bullet}}
\def\cwX{\mathbb{X}_{\bullet}}
\def\cT{\check{T}}
\def\Pro{\mathbf{Pro}}
\def\Fr{F\textnormal{-}}
\def\Hyp{\textnormal{Hyp}}
\def\DHyp{\mathscr{H}yp}
\def\THyp{\widetilde{\mathscr{H}}yp}

\def\EE{\mathsf{E}}

\DeclareMathOperator{\Ad}{Ad}
\DeclareMathOperator{\Hke}{Hecke}
\DeclareMathOperator{\Hk}{Hk}

\DeclareMathOperator{\gr}{gr}
\DeclareMathOperator{\et}{\textnormal{\'{e}t}}

\DeclareMathOperator{\coh}{coh}
\DeclareMathOperator{\hol}{hol}
\DeclareMathOperator{\Hol}{Hol}

\DeclareMathOperator{\ad}{ad}
\DeclareMathOperator{\Gal}{Gal}
\DeclareMathOperator{\Ker}{Ker}
\DeclareMathOperator{\Coker}{Coker}
\DeclareMathOperator{\Hom}{Hom}
\DeclareMathOperator{\Sym}{Sym}
\DeclareMathOperator{\End}{End}
\DeclareMathOperator{\Ext}{Ext}

\DeclareMathOperator{\Spec}{Spec}
\DeclareMathOperator{\Spf}{Spf}

\DeclareMathOperator{\rH}{H}

\DeclareMathOperator{\rb}{b}

\DeclareMathOperator{\rD}{D}

\DeclareMathOperator{\rE}{E}

\DeclareMathOperator{\rL}{L}

\DeclareMathOperator{\rR}{R}
\DeclareMathOperator{\rc}{c}
\DeclareMathOperator{\Gr}{Gr}
\DeclareMathOperator{\DR}{DR}
\DeclareMathOperator{\dR}{dR}

\DeclareMathOperator{\id}{id}
\DeclareMathOperator{\rig}{rig}
\DeclareMathOperator{\alg}{alg}
\DeclareMathOperator{\an}{an}
\DeclareMathOperator{\Iso}{Isoc}

\DeclareMathOperator{\Kl}{Kl}
\DeclareMathOperator{\Be}{Be}
\DeclareMathOperator{\Std}{Std}
\DeclareMathOperator{\PGL}{PGL}
\DeclareMathOperator{\GL}{GL}
\DeclareMathOperator{\SL}{SL}
\DeclareMathOperator{\Sp}{Sp}
\DeclareMathOperator{\SO}{SO}
\DeclareMathOperator{\Spin}{Spin}

\DeclareMathOperator{\SP}{Sp}
\DeclareMathOperator{\Conn}{Conn}
\DeclareMathOperator{\Frob}{Frob}
\DeclareMathOperator{\GR}{GR}
\DeclareMathOperator{\AS}{AS}
\DeclareMathOperator{\IC}{IC}
\DeclareMathOperator{\Bun}{Bun}
\DeclareMathOperator{\Loc}{LocSysm}
\DeclareMathOperator{\Tr}{Tr}
\DeclareMathOperator{\Nm}{Nm}

\DeclareMathOperator{\Frac}{Frac}
\DeclareMathOperator{\Sm}{Sm}

\DeclareMathOperator{\un}{un}

\DeclareMathOperator{\act}{act}
\DeclareMathOperator{\pr}{pr}
\DeclareMathOperator{\Log}{Log}
\DeclareMathOperator{\Sat}{Sat}
\DeclareMathOperator{\Vect}{Vec}
\DeclareMathOperator{\Hen}{Hen}
\DeclareMathOperator{\Aut}{Aut}
\DeclareMathOperator{\Out}{Out}
\DeclareMathOperator{\red}{red}
\DeclareMathOperator{\op}{op}

\DeclareMathOperator{\gal}{gal}

\DeclareMathOperator{\geo}{geo}
\DeclareMathOperator{\arith}{arith}
\DeclareMathOperator{\Irr}{Irr}
\DeclareMathOperator{\rank}{rank}
\DeclareMathOperator{\Char}{char}
\DeclareMathOperator{\SSS}{ss}
\DeclareMathOperator{\Con}{Con}
\DeclareMathOperator{\CT}{CT}

\DeclareMathOperator{\sm}{sm}
\DeclareMathOperator{\sing}{sing}
\DeclareMathOperator{\im}{Im}
\DeclareMathOperator{\uni}{uni}

\DeclareMathOperator{\MC}{MC}
\DeclareMathOperator{\MCF}{MCF}
\DeclareMathOperator{\sep}{sep}
\DeclareMathOperator{\nil}{nil}
\DeclareMathOperator{\Lie}{Lie}
\DeclareMathOperator{\Conj}{Conj}
\DeclareMathOperator{\NP}{NP}
\DeclareMathOperator{\HP}{HP}
\DeclareMathOperator{\Weil}{Weil}

\DeclareMathOperator{\tame}{tame}
\DeclareMathOperator{\Sw}{Sw}
\DeclareMathOperator{\aff}{aff}
\DeclareMathOperator{\Sol}{Sol}
\DeclareMathOperator{\Pic}{Pic}

\newtheorem{theorem}{Theorem}[subsection]
\newtheorem{prop}[theorem]{Proposition}
\newtheorem{lemma}[theorem]{Lemma}

\newtheorem{coro}[theorem]{Corollary}

\theoremstyle{definition}
\newtheorem{rem}[theorem]{Remark}
\newtheorem{definition}[theorem]{Definition}
\newtheorem{secnumber}[theorem]{}

\numberwithin{equation}{section}
\numberwithin{equation}{theorem}

\newcolumntype{L}{>{$}l<{$}} 

\newcommand{\quash}[1]{}

\title{Bessel $F$-isocrystals for reductive groups}
\author{Daxin Xu, Xinwen Zhu}
\date{\today}

\AtEndDocument{\bigskip{\footnotesize%
  \textsc{Daxin Xu, Xinwen Zhu, Department of Mathematics, California Institute of Technology, Pasadena, CA 91125.} \par
  \textit{E-mail address}: \texttt{daxinxu@caltech.edu}, \texttt{xzhu@caltech.edu} \par
}}

\begin{document}
\selectlanguage{english}
\maketitle

\begin{abstract}
	We construct the Frobenius structure on a rigid connection $\Be_{\cG}$ on $\Gm$ for a split reductive group $\cG$ introduced by Frenkel-Gross. 
	These data form a $\cG$-valued overconvergent $F$-isocrystal $\Be_{\cG}^{\dagger}$ on $\mathbb{G}_{m,\mathbb{F}_p}$, which is the $p$-adic companion of the Kloosterman $\cG$-local system $\Kl_{\cG}$ constructed by Heinloth-Ng\^o-Yun. 
	By studying the structure of the underlying differential equation, we calculate the monodromy group of $\Be_{\cG}^{\dagger}$ when $\cG$ is almost simple (which recovers the calculation of monodromy group of $\Kl_{\cG}$ due to Katz and Heinloth-Ng\^o-Yun), and establish functoriality between different Kloosterman $\cG$-local systems as conjectured by Heinloth-Ng\^o-Yun. We show that the Frobenius Newton polygons of $\Kl_{\cG}$ are generically ordinary for every $\cG$ and are everywhere ordinary on $|\mathbb{G}_{m,\mathbb{F}_p}|$ when $\cG$ is classical or $G_2$.  
\end{abstract}
\tableofcontents
\section{Introduction}

\subsection{Bessel equations and Kloosterman sums}

\begin{secnumber} \label{Bessel conn intro}
	The classical Bessel differential equation (of rank $n$) with a parameter $\lambda$
         \begin{equation}\label{Bessel connection intro}
               \biggl(x\frac{d}{dx}\biggr)^n(f) - \lambda^n x\cdot f = 0
         \end{equation}
	 has a unique solution which is holomorphic at $0$ : 
	 \begin{equation}
		 \oint_{(S^1)^{n-1}} \exp \lambda\biggl(z_1+z_2+\cdots+z_{n-1}+ \frac{x}{z_1\cdots z_{n-1}}\biggr) \frac{dz_1\cdots dz_{n-1}}{ (2\pi i)^{n-1}z_1\cdots z_{n-1}} =
		 \sum_{r\ge 0} \frac{1}{(r!)^n} (\lambda^n x)^{r}. 
		 \label{solution Bes eqn}
	 \end{equation}

	One may reinterpret this fact using the language of algebraic $\mathscr{D}$-modules as follows. 
	Let $K$ be a field of characteristic zero. The Bessel equation \eqref{Bessel connection intro} can be converted to a connection $\Be_n$ on the rank $n$ trivial bundle on the multiplicative group $\mathbb{G}_{m,K}$
	 \begin{equation} \label{Bessel intro}
		\Be_n:	\quad\nabla=	d + \begin{pmatrix}
		0                  & 1 & 0 & \dots & 0 \\
		0                  & 0 & 1 & \dots & 0 \\
		\vdots          & \vdots & \ddots & \ddots & \vdots \\
		0                  & 0 & 0 & \dots & 1 \\
		\lambda^nx & 0 & 0 & \dots & 0  \end{pmatrix} \frac{dx}{x},
	\end{equation}
         which we call the \textit{Bessel connection (of rank $n$)}. On the other hand, we consider the following diagram
         \begin{equation}\label{Del diag}
	\xymatrix{
		&\mathbb{G}_m^n \ar[rd]^{\mathrm{add}} \ar[ld]_{\mathrm{mult}}& \\
		\mathbb{G}_m && \mathbb{A}^1, }
        \end{equation}
        where $\mathrm{add}$ (resp. $\mathrm{mult}$) denotes the morphism of taking sum (resp. product) of $n$ coordinates of $\Gm^{n}$,
        and define the \textit{Kloosterman $\mathscr{D}$-module} on $\mathbb{G}_{m,K}$ as
	\begin{equation}\label{Kl Dmod intro}
		\Kl_n^{\dR}: =  \rR^{n-1}\mathrm{mult}_{!}(\mathrm{add}^{*}(\EE_{\lambda})),
	\end{equation}
        where
        \begin{equation}
		\EE_{\lambda}=(\mathscr{O}_{\mathbb{A}^1_K}, \nabla=d- \lambda dx)
		\label{exp D-mod}
	\end{equation}
	is the \textit{exponential $\mathscr{D}$-module} on $\mathbb{A}^1_K$. With these notations, the fact that \eqref{solution Bes eqn} is a solution of \eqref{Bessel connection intro} reflects an isomorphism of algebraic $\mathscr{D}$-modules on $\mathbb{G}_{m,K}$
	\[
		\Be_n\simeq \Kl_n^{\dR}.
	\]

	The connection $\Be_n$ is regular singular at $0$ with a unipotent monodromy with a single Jordan block, and is irregular at $\infty$ with irregularity $=1$. Its differential Galois group was calculated by Katz \cite{Katz87}. 
\end{secnumber}

\begin{secnumber}
	There is a parallel theory in positive characteristic. 
	Let $p$ a prime number. 
	For every finite extension $\mathbb{F}_q/\mathbb{F}_p$ and $a\in \mathbb{F}_q^{\times}$, the \textit{Kloosterman sum $\Kl(n;a)$} in $n$-variables is defined by \footnote{The sum \eqref{Kl sum} is slightly different from the standard definition by a factor $(-\frac{1}{\sqrt{q}})^{n-1}$.}
	\begin{equation} \label{Kl sum}
		\Kl(n;a)=\biggl(-\frac{1}{\sqrt{q}}\biggr)^{n-1}\sum_{z_i\in \mathbb{F}_q^{\times}} \exp\biggl(\frac{2\pi i}{p}\Tr_{\mathbb{F}_q/\mathbb{F}_p}\bigl(z_1+z_2+\cdots+z_{n-1}+\frac{a}{z_1\cdots z_{n-1}}\bigr)\biggr).
	\end{equation}
	It admits a sheaf-theoretic interpretation by Deligne \cite{Del77}. Namely, the analog of the exponential $\mathscr{D}$-module in positive characteristic is the Artin-Schreier sheaf $\AS_{\psi}$ on $\mathbb{A}^1_{\mathbb{F}_p}$ associated to a non-trivial character $\psi:\mathbb{F}_p\to \mathbb{Q}_{\ell}(\mu_p)^{\times}$. 
	Deligne defined the \textit{Kloosterman sheaf} $\Kl_n$ as the following complex on $\mathbb{G}_{m,\mathbb{F}_p}$:
\begin{equation}
	\Kl_n=\rR\mathrm{mul}_!(\mathrm{add}^{*}(\AS_{\psi}))[n-1]\biggl(\frac{n-1}{2}\biggr),
	\label{definition Kl}
\end{equation}
and showed the following properties (\cite{Del77} 7.4, 7.8):
\begin{itemize}
	\item[(i)] Fix an embedding $\iota:\mathbb{Q}_{\ell}(\mu_p)\to \mathbb{C}$ such that $\iota\psi(x)=\exp(2\pi i x/p)$ for $x\in \mathbb{F}_p$. The Frobenius trace of $\Kl_n$ at each closed point $a\in \mathbb{F}_q^{\times}$ is equal to the Kloosterman sum $\Kl(n;a)$. 

	\item[(ii)] The complex $\Kl_n$ is concentrated in degree $0$ and is an irreducible local system of rank $n$ and of weight $0$, which implies the Weil bound of the Kloosterman sum $|\Kl(n;a)|\le n$. 

	\item[(iii)] The sheaf $\Kl_n$ is tamely ramified at $0$, and the monodromy is unipotent with a single Jordan block. 

	\item[(iv)] The sheaf $\Kl_n$ is wildly ramified at $\infty$ with Swan conductor $\Sw_{\infty}(\Kl_n)=1$. 
\end{itemize}
	
In (\cite{Katz88} \S~11), Katz calculated the (global) geometric and arithmetic monodromy group of $\Kl_n$ as follow:
\begin{equation} \label{Katz mono Kln}
		G_{\geo}(\Kl_n)=G_{\arith}(\Kl_n)= 
		\left\{
			\begin{array}{ll}
				\SP_n & n~ even, \\
				\SL_n & pn~ odd, \\
				\SO_n & p=2, n~ odd, n\neq 7, \\
				G_2 & p=2, n=7.
			\end{array} 
			\right.
	\end{equation}
	Surprisingly, the exceptional group $G_2$ appears as the monodromy group. 
\end{secnumber}

\begin{secnumber}
	In 70's \cite{Dw74}, Dwork showed that there exists a Frobenius structure on the Bessel connection \eqref{Bessel intro} whose Frobenius traces give the Kloosterman sum. 
	Here a \textit{Frobenius structure} is a horizontal isomorphism between the Bessel connection and its pullback by the ``Frobenius endomorphism'' $F:\mathbb{G}_{m,K}\to \mathbb{G}_{m,K}$ over $K$ defined by $x\mapsto x^{p}$. 
	Although the Bessel connection is an algebraic connection, such a horizontal isomorphism is not algebraic but of $p$-adic analytic nature. 
	
	To explain Dwork's result, we need to introduce certain ring of $p$-adic analytic functions. 
	We set $K=\mathbb{Q}_p(\mu_p)$, equipped with a $p$-adic valuation $|\textnormal{-}|_p$ normalised by $|p|_p=p^{-1}$, and denote by $A^{\dagger}$ the ring of $p$-adic analytic functions with a convergence radius $>1$:
	\begin{equation}
		A^{\dagger}=\bigg\{ \sum_{n=0}^{+\infty} a_nx^n ~|~ a_n\in K, \exists~ \rho>1, \lim_{n\to +\infty} |a_n|_p\rho^n =0 \bigg\}.  
		\label{def A dagger}
	\end{equation}
	This ring is called \textit{the ring of $p$-adic analytic functions on $\P1$ overconvergent along $\{\infty\}$} by Berthelot \cite{Ber96II}. 

	We take an algebraic closure $\overline{K}$ of $K$ and fix an isomorphism $\iota:\overline{K}\to \mathbb{C}$. 	
	There exists a unique element $\pi$ of $K$ which satisfies $\pi^{p-1}=-p$ and corresponds to the character $\exp2\pi i(\frac{-}{p}):\mathbb{F}_p\to \mathbb{C}^{\times}$ (cf. \ref{Dwork isocrystal}(i)). 
\end{secnumber}

\begin{theorem}[Dwork, Sperber \cite{Dw74,Sp77,Sp80}] \label{Dwork thm intro}
	Let $n$ be an integer prime to $p$ and set $\lambda=-\pi$ as above.
	There exists a unique $\varphi(x)\in \GL_n(A^{\dagger})$ such that

	\textnormal{(i)} 
		The matrix $\varphi$ satisfies the differential equation: 
	\begin{displaymath}
		x \frac{d\varphi}{dx} \varphi^{-1} + \varphi 
		\begin{pmatrix}
		0                  & 1 & 0 & \dots & 0 \\
		0                  & 0 & 1 & \dots & 0 \\
		\vdots          & \vdots & \ddots & \ddots & \vdots \\
		0                  & 0 & 0 & \dots & 1 \\
		\lambda^nx & 0 & 0 & \dots & 0  \end{pmatrix} 
		\varphi^{-1} =
		p \begin{pmatrix}
		0                  & 1 & 0 & \dots & 0 \\
		0                  & 0 & 1 & \dots & 0 \\
		\vdots          & \vdots & \ddots & \ddots & \vdots \\
		0                  & 0 & 0 & \dots & 1 \\
		\lambda^nx^p & 0 & 0 & \dots & 0  \end{pmatrix}.
		\end{displaymath}
		That is, $\varphi$ defines a horizontal isomorphism $F^{*}(\Be_n)\xrightarrow{\sim} \Be_n$.

		\textnormal{(ii)}  For $a\in \mathbb{F}_{q}^{\times}$, we have $\iota\Tr \varphi_a=\Kl(n;a)$, where $\varphi_a=\prod_{i=0}^{\deg(a)-1} \varphi(\widetilde{a}^{p^i})$ and $\widetilde{a}\in \overline{K}$ denotes the Teichmüller lifting of $a$. 

		\textnormal{(iii)} If $\{\alpha_1,\cdots,\alpha_n\}$ denote the eigenvalues of $\varphi_a$, then we have $|\alpha_i|_p=p^{\frac{n+1-2i}{2}\deg(a)}$ after reordering $\alpha_i$. 
\end{theorem}

Overconvergent ($F$-)isocrystals on a variety $X$ over $\mathbb{F}_p$ are $p$-adic analogues of $\ell$-adic (Weil) local systems. 
Roughly speaking, they are vector bundles with an integrable connection and a Frobenius structure on some rigid analytic space associated to certain lifting of $X$ to characteristic zero. 
The data $(\Be_n,\varphi)$ form an overconvergent $F$-isocrystal on $\mathbb{G}_{m,\mathbb{F}_p}$ (relative to $K$), which we call the \textit{Bessel $F$-isocrystal (of rank $n$)} and denote by $\Be_n^{\dagger}$. 
By (ii), $\Be_n^{\dagger}$ is the \textit{$p$-adic companion} of the Kloosterman sheaf $\Kl_n$ in the sense of \cite{WeilII,Abe18}.

\subsection{Generalization for reductive groups}
\begin{secnumber} \label{summary of results}
	Recently, there are two generalizations of above results (corresponding to the $\GL_n$-case) for reductive groups from different perspectives. 
	The first one is due to Frenkel and Gross \cite{FG09} from the viewpoint of the Bessel equations. Namely, for each (split) reductive group $\cG$ over a field $K$ of characteristic zero, Frenkel-Gross wrote down an explicit $\cG$-connection $\Be_{\cG}$ on $\mathbb{G}_m$, which specializes to $\Be_n$ when $\cG=\GL_n$. We will call $\Be_{\cG}$ the \textit{Bessel connection of $\cG$} in this paper.
	Another one, due to Heinloth, Ng\^o and Yun \cite{HNY}, is from the viewpoint of the Kloosterman sums. Namely, the authors explicitly constructed, for each (split) reductive group $G$ over the rational function field $\mathbb{F}_p(t)$, a Hecke eigenform of $G$, and defined $\Kl_{\cG}$ as its Langlands parameter, which is an $\ell$-adic $\cG$-local system on $\mathbb{G}_m$ that specializes to $\Kl_n$ if $\cG=\GL_n$. The authors call $\Kl_{\cG}$ the \textit{Kloosterman sheaf of $\cG$}.

	The main subject of this article is to study the $p$-adic aspects of this theory and to unify the previous two constructions. Our main results can be summarized as follows:
\begin{enumerate}
         \item[(i)] we construct the Frobenius structure on $\Be_{\cG}$ and obtain the Bessel $F$-isocrystal $\Be_{\cG}^{\dagger}$ of $\cG$, which is  the $p$-adic companion of $\Kl_{\cG}$ in appropriate sense;
         \item[(ii)] we calculate its geometric and arithmetic monodromy group;
	 \item[(iii)] we show that the Frobenius Newton polygons of $\Be_{\cG}$ (and therefore $\Kl_{\cG}$) are generically ordinary and when $\cG$ is classical or $G_2$ they are everywhere ordinary on $|\mathbb{G}_{m,\mathbb{F}_p}|$.  	 
\end{enumerate}
	 It turns out that our $p$-adic theory also has applications to the $\ell$-adic theory and the arithmetic property of exponential sums associated to $\Kl_{\cG}$. Namely,
\begin{enumerate}
	\item[(iv)] we obtain a different (and more conceptual) calculation of the monodromy group of $\Kl_{\cG}$ (\eqref{Katz mono Kln} and one of the main results of \cite{HNY}), based on the structure of the connection $\Be_{\cG}$;
	 \item[(v)] we prove a conjecture of Heinloth-Ng\^o-Yun on the functoriality of Kloosterman sheaves (\cite{HNY} conjecture 7.3) and therefore obtain identities between different exponential sums.
\end{enumerate}
	We discuss these results in more details in the sequel.
\end{secnumber}

\begin{secnumber}
	Let $\cG$ be a split almost simple group over a field $K$ of characteristic zero. Fix a Borel subgroup $\cB\subset \cG$, and a principal nilpotent element $N$ in the Lie algebra of $\cB$. Let $E$ denote a basis vector of the lowest root space in $\check{\mathfrak{g}}=\Lie(\cG)$.
	In \cite{FG09}, Frenkel and Gross considered a connection on the trivial $\cG$-bundle over $\Gm$: 
	\begin{equation}
		\nabla=d+N\frac{dx}{x}+\lambda^{h} E dx,
		\label{FG connection intro}
	\end{equation}
	where $x$ is a coordinate of $\Gm$, $\lambda\in K$ is a parameter and $h$ is the Coxeter number of $\cG$. 
	We may regard it as a tensor functor from the category of representations of $\cG$ to the category of connections on the trivial bundles on $\Gm$ 
	\begin{equation}
		\Be_{\cG}: \Rep(\cG)\to \Conn(\Gm), \quad (\rho:\cG\to \GL(V)) \to d+d\rho(N\frac{dx}{x}+ \lambda^h E dx).
		\label{functor Be}
	\end{equation}
	This connection is rigid and has a regular singularity at $0$ and an irregular singularity at $\infty$. 
\end{secnumber}

\begin{secnumber}
         Let $G$ be a split almost simple group over $\mathbb{F}_p(t)$ whose dual group is $\cG$. In \cite{HNY}, Heinloth-Ng\^o-Yun wrote down a cuspidal Hecke eigenform $f$ on $G$, and defined the Kloosterman sheaf $\Kl_{\cG}$ for $\cG$ as the Langlands parameter of $f$. For simplicity, we assume that $G$ is simply-connected. If we fix opposite Iwahori subgroups $I(0)^{\mathrm{op}}\subset G(\mathscr O_0)$ and $I(0)\subset G(\mathscr O_{\infty})$ at $0,\infty$, and a non-degenerate character $\varphi: I(1)/I(2)\to \mathbb{Q}(\mu_p)^\times$, where $I(i)$ denotes the $i$th step in the Moy-Prasad filtration of $I(0)$, then $f$ is the unique (up to scalar) non-zero function on $G(\mathbb{F}_p(t))\backslash G(\mathbb A)$ that is, 
         \begin{itemize}
         \item invariant under $G(\mathscr O_x)$ for every place $x\neq 0,\infty$;
         \item invariant under $I(0)^{\mathrm{op}}$ at $0$; 
         \item $(I(1),\varphi)$-equivariant at $\infty$.
         \end{itemize}
	 Then Heinloth-Ng\^o-Yun  defined $\Kl_{\cG}:  \Rep(\cG)\to \Loc(\mathbb{G}_{m,\mathbb{F}_p})$ as a tensor functor from the category of representations of $\cG$ (over $\overline{\mathbb{Q}}_\ell$) to the category of $\ell$-adic local systems on $\mathbb{G}_{m,\mathbb{F}_p}$, such that
	 for every $V\in \Rep(\cG)$ and every $a\in|\mathbb{G}_{m,\mathbb{F}_p}|$,
         \[
		 T_{V,a}(f) = \Tr(\Frob_a,(\Kl_{\cG,V})_{\overline{a}})f
         \]
         where $T_{V,a}$ is the Hecke operator associated to $(V,a)$.
         The actual construction of $\Kl_{\cG}$ uses the geometric Langlands correspondence (see \ref{HNY Kl}). 
\end{secnumber}

Our first main result is the existence of a Frobenius structure on Bessel connections for reductive groups. 

\begin{theorem}[\ref{def Bessel crystal}, \ref{Frobenius slope}] \label{intro main thm}
	Let $K=\mathbb{Q}_p(\mu_p)$, $\overline{K}$ an algebraic closure of $K$ and set $\lambda=-\pi$ as in \ref{Dwork thm intro}. 

	\textnormal{(i)} There exists a unique $\varphi(x)\in \cG(A^{\dagger})$ satisfying the differential equation
	\begin{displaymath}
		x \frac{d\varphi}{dx} \varphi^{-1} + \Ad_{\varphi} (N+\lambda^{h} x E) = p(N+\lambda^h x^p E)
	\end{displaymath}
	and such that via a (fixed) isomorphism $\overline{K}\simeq \overline{\mathbb{Q}}_{\ell}$, for every $a\in \mathbb{F}_q^{\times}$ and $V\in \Rep(\cG)$ 
	\begin{displaymath}
		\Tr(\varphi_a,V)=\Tr(\Frob_a,(\Kl_{\cG,V})_{\overline{a}}),  
	\end{displaymath}
	where $\varphi_a=\prod_{i=0}^{\deg(a)-1} \varphi(\widetilde{a}^{p^i})$ and $\widetilde{a}\in \overline{K}$ denotes the Teichmüller lifting of $a$. 

	In particular, the analytification of the Bessel connection $\Be_{\cG}$ on $\mathbb{G}_{m,K}^{\an}$ is overconvergent and underlies a tensor functor from $\Rep(\cG)$ to the category of overconvergent $F$-isocrystals on $\mathbb{G}_{m,\mathbb{F}_p}$:
	\begin{equation}
		\Be_{\cG}^{\dagger}:\Rep(\cG)\to \Fr\Iso^{\dagger}(\mathbb{G}_{m,\mathbb{F}_p}/K),
	\end{equation}
which can be regarded as the $p$-adic companion of $\Kl_{\cG}$.

	\textnormal{(ii)} Let $\rho\in \wX(T)=\cwX(\cT)$ be the half sum of positive roots. 
	When $\cG$ is of type $A_n,B_n,C_n,D_n$ or $G_2$, for every $a\in |\mathbb{G}_{m,\mathbb{F}_p}|$, the set of $p$-adic order of eigenvalues of $\varphi_a\in \cG(\overline{K})$ (also known as the Frobenius slopes at $a$) is same as that of $\rho(p^{\deg(a)})\in \cG(\overline{K})$. 
	When $\cG$ is of other exceptional type, the same assertion holds generically on $|\mathbb{G}_{m,\mathbb{F}_p}|$.
\end{theorem}

\begin{rem}
	(i) For a $\cG$-valued overconvergent $F$-isocrystal on a smooth variety $X$ over $\mathbb{F}_p$, we say its Newton polygon is ordinary at $a$ if the Frobenius slopes at $a$ are given by $\rho$ (in the above sense). 
	We expect that the Newton polygons of $\Be_{\cG}^{\dagger}$ are always ordinary at each closed point of $\mathbb{G}_{m,\mathbb{F}_p}$.

	(ii) V. Lafforgue \cite{VL11} showed that $\rho$ is the upper bound for the $p$-adic valuations of Hecke eigenvalues of Hecke eigenforms (cf. \ref{def NT polygons} for a precise statement). 
	Drinfeld and Kedlaya \cite{DK17} proved an analogous result for the Frobenius slopes of an indecomposable convergent $F$-isocrystal on a smooth scheme. 
\end{rem}

\begin{secnumber}
	\textbf{Global monodromy groups.} 
	In (\cite{FG09} Cor. 9,10), Frenkel and Gross calculate the differential Galois group $G_{\alg}$ of $\Be_{\cG}$ over $\overline{K}$, which we list in the following table (up to central isogeny):
	\begin{equation}\label{table monodromy}
	\begin{tabular}{L|L}
		\cG & G_{\gal} \\
		\hline
		A_{2n} & A_{2n} \\
		A_{2n-1}, C_{n} & C_{n} \\
		B_{n},D_{n+1} (n\ge 4) & B_{n} \\
		E_{7} & E_{7} \\
		E_{8} & E_{8} \\
		E_{6}, F_{4} & F_{4} \\
		B_{3}, D_{4}, G_{2} & G_{2}.
	\end{tabular}
	\end{equation}
	If we denote by $G_{\geo}$ the geometric monodromy group of $\Be_{\cG}^{\dagger}$ over $\overline{K}$, there exists a canonical homomorphism
\begin{equation}
	G_{\geo}\to G_{\gal}.
	\label{homomorphism gal monodromy groups}
\end{equation}

\begin{theorem}[\ref{thm geometric mono}] \label{intro monodromy gp}
	 \textnormal{(i)} If either $\cG$ is not of type $A_{2n}$, or $p>2$,  the above morphism is an isomorphism. 
	
	 \textnormal{(ii)} If $p=2$ and $\cG=\SL_{2n+1}$, then $G_{\geo}\simeq \SO_{2n+1}$ if $n\neq 3$ and $G_{\geo}\simeq G_2$ if $n=3$. 

	 \textnormal{(iii)} The arithmetic monodromy group $G_{\arith}$ of $\Be_{\cG}^{\dagger}$ is isomorphic to $G_{\geo}$. 
\end{theorem}
In fact, the second part of the theorem follows from the first part and theorem \ref{intro functoriality}(ii) below. 
By companion, this theorem allows us to recover Katz's result on the monodromy group of $\Kl_n$ \eqref{Katz mono Kln} and Heinloth-Ng\^o-Yun's result on the geometric monodromy group of $\Kl_{\cG}$ \cite{HNY} in a different way. 
For instance, the $G_2$-symmetry on $\Be_{7}^{\dagger}$ when $p=2$ \eqref{Katz mono Kln} appears naturally in our approach, compared with Katz' original approach via point counting.
In addition, we also avoid some difficult geometry related to quasi-minuscule and adjoint Schubert varieties, as analyzed in \cite{HNY}. 
\end{secnumber}

Inspired by the rigidity properties of hypergeometric sheaves proved by Katz \cite{Katz90}, Heinloth, Ng\^o and Yun conjectured certain functoriality between Kloosterman sheaves for different groups (\cite{HNY} conjecture 7.3). 
As an application of our $p$-adic theory, we prove this conjecture. 

\begin{theorem}[\ref{thm functoriality}, \ref{functoriality p=2}] \label{intro functoriality}
        \textnormal{(i)}
         For $\cG'\subset \cG$ appearing in the same line in the left column of the above diagram, $\Kl_{\cG}$ is isomorphic to the push-out of $\Kl_{\cG'}$ along $\cG'\to\cG$.
         
	 \textnormal{(ii)} If $p=2$, $\Kl_{\SL_{2n+1}}$ is the push-out of $\Kl_{\SO_{2n+1}}$ along $\SO_{2n+1}\to \SL_{2n+1}$.
\end{theorem}

\begin{secnumber} \label{identity exp sums intro}
	The above theorem allows us to identify various exponential sums associated to Kloosterman sheaves defined by different groups. 	Here are some examples (cf. corollary \ref{identities exp sums}):

	(i) When $\cG=\SO_3\simeq \PGL_2$, we have the following identity for $a\in \mathbb{F}_q^{\times}$:
	\begin{eqnarray} \label{SO3 hyp intro}
		&& \biggl(\sum_{x\in \mathbb{F}_q^{\times}}\psi(\Tr_{\mathbb{F}_q/\mathbb{F}_p}(x+\frac{a}{x}))\biggr)^2-q \\
		&=& \left\{
			\begin{array}{ll}
			        \displaystyle -\frac{1}{\sqrt{q}}\sum_{x_1,x_2\in \mathbb{F}_{q}^{\times}}\psi\biggl(\Tr_{\mathbb{F}_q/\mathbb{F}_p}(x_1+x_2+\frac{a}{x_1x_2})\biggr), & p=2, \\
				\displaystyle \frac{1}{G(\psi^{-1},\rho^{-1})} \sum_{x_1 x_2 x_3=4ay, x_i\in \mathbb{F}_q^{\times}} \psi\biggl(\Tr_{\mathbb{F}_q/\mathbb{F}_p}(x_1+x_2+x_3-y)\biggr)\rho^{-1}(y), & p>2 \\
				
			\end{array} 
			\right. \nonumber
	\end{eqnarray}
	where $\psi(-)=\exp\frac{2\pi i}{p}(-)$, $\rho$ denotes the quadratic character of $\mathbb{F}_q^{\times}$ and $G(\psi^{-1},\rho^{-1})$ the associated Gauss sum. The identity is due to Carlitz \cite{Car69} when $p=2$ and Katz (\cite{Katz09} \S~3) when $p>2$. Our method is completely different from these works.  

	(ii) For $n\ge 2$, via the inclusion $\SO_{2n+1}\to \SO_{2n+2}$, for $a\in \mathbb{F}_q^{\times}$, we have
	\begin{eqnarray} \label{SO2n+1 SO2n+2}
		&& \sum_{uv=a, u,v\in \mathbb{F}_q^{\times}} \biggl(\bigl(\sum_{x\in \mathbb{F}_q^{\times}}
		\psi(\Tr_{\mathbb{F}_q/\mathbb{F}_p}(x+\frac{u}{x}))\bigr)^2-q\biggr)
		\biggl(\sum_{x_i\in \mathbb{F}_q^{\times}}
		\psi(\Tr_{\mathbb{F}_q/\mathbb{F}_p}(x_1+\cdots+x_{2n-3}+\frac{v}{x_1\cdots x_{2n-3}})) \biggr) \\
		&&= -\frac{1}{\sqrt{q}} \biggl(\sum_{x_i\in \mathbb{F}_q^{\times}}
		\psi\biggl(\Tr_{\mathbb{F}_q/\mathbb{F}_p}(x_1+x_2+\cdots +x_{2n}+a\frac{x_1+x_2}{x_1x_2\cdots x_{2n}})\biggr) - q^{n-1}\biggr). \nonumber
	\end{eqnarray}
	One can obtain other identities between different exponential sums, whose sheaf-theoretic incarnations were obtained by Katz \cite{Katz09}. 
\end{secnumber}

\begin{secnumber} We have partial results about the local monodromy of $\Be_{\cG}^\dagger$ (and $\Kl_{\cG}$) at $\infty$. Namely, we show that the nilpotent monodromy operator is trivial and the local Galois representation $\varphi_{\infty}:I_{\infty}\to \cG$ is a simple wild parameter in the sense of Gross-Reeder \cite{GR10} \S~6 (see corollary \ref{local monodromy} (ii)). If $p\nmid n$, one can even show that the local monodromy of $\Be_{n}^\dagger$ at $\infty$ coincides with the Galois representations constructed in \cite{GR10} \S~6.2, by studying the solutions of Bessel differential equation \eqref{Bessel connection intro} at $\infty$. 
(The corresponding $\ell$-adic statement for $\Kl_n$ was proved by Fu and Wan \cite{FW05} thm. 1.1.) 
    
       The above result, together with theorem \ref{intro functoriality}(ii), implies that when $p=2$ and $n$ is an odd integer, the associated local Galois representation for $\Be_{\SO_n}^{\dagger}$ (and $\Kl_{\SO_{n}}$) at $\infty$ coincides with the simple wild parameter constructed in \cite{GR10} \S~6.3. For example, the image of the inertia group $I_{\infty}$ in the case $\cG=\SO_3$ is isomorphic to $A_4$. 	Together with $\Be_{\SO_3,\Std}^{\dagger} \simeq \Be_{\SL_2,\Sym^2}^{\dagger}$ \eqref{SL2 vs SO3}, this allows us to recover Andr\'e's result on the local monodromy group of $\Be_{2}^{\dagger}$ at $\infty$ when $p=2$ (\cite{And02} sections 7, 8). 

\end{secnumber}

\subsection{Strategy of the proof and the organization of the article}  
\begin{secnumber} \label{intro main thm proof}
	Now we outline the strategy to prove the above results. Part (i) of theorem \ref{intro main thm} follows by combining following three ingredients: 
	\begin{itemize}
		\item[(i)] We first mimic Heinloth-Ng\^o-Yun's construction to produce a $\cG$-valued overconvergent $F$-isocrystal $\Kl_{\cG}^{\rig}$ on $\mathbb{G}_{m,\mathbb{F}_p}$ and a $\cG$-bundle with connection $\Kl_{\cG}^{\dR}$ on $\mathbb{G}_{m,K}$ (section \ref{construction HNY}). 
	A key step is to develop the geometric Satake equivalence for arithmetic $\mathscr{D}$-modules, which we will discuss latter \eqref{intro geo Satake}. 

\item[(ii)] Then we show that the overconvergent isocrystal $\Kl_{\cG}^{\rig}$ is isomorphic to the analytification of the $\cG$-connection $\Kl_{\cG}^{\dR}$ (section \ref{Frob str Bessel}) by comparing certain relative de Rham cohomologies and relative rigid cohomologies. 

\item[(iii)] We strengthen a result of the second author \cite{Zhu17} to identify $\Kl_{\cG}^{\dR}$ with $\Be_{\cG}$ (section \ref{compare Kl Be}). 
	\end{itemize}
\end{secnumber}
\begin{secnumber} \label{intro monodromy proof}
	The local monodromy of $\Be_{\cG}^{\dagger}$ at $0$ is principal unipotent, which implies that 
	its geometric monodromy $G_{\geo}$ contains a principal $\SL_2$. This puts strong restrictions on the possible Dynkin diagrams of $G_{\geo}$ (cf. \ref{list lie algebras} for a possible list). 
	A result of Baldassarri \cite{Ba82} (cf. \cite{And02} 3.2), which implies that the $p$-adic slope of $\Be_{\cG}^{\dagger}$ at $\infty$ is less or equal to the formal slope of $\Be_{\cG}$ at $\infty$, allows us to exclude the case $G_{\geo}=\PGL_2$ (or $\SL_2$) in most cases. Together with certain symmetry on $\Be_{\cG}^{\dagger}$, this implies theorem \ref{intro monodromy gp}(i).
\end{secnumber}
\begin{secnumber} \label{functoriality proof}
	The analogous functoriality for Bessel connections $\Be_{\cG}$ (theorem \ref{intro functoriality}(i))  follows from their definition. 
	Then we deduce the functoriality between $\Be_{\cG}^{\dagger}$'s by theorems \ref{intro main thm}(i) and \ref{intro monodromy gp}(i). 
	For theorem \ref{intro functoriality}(ii) (and therefore theorem \ref{intro monodromy gp}(ii)), we construct an isomorphism between the maximal slope quotients of  $\Be_{2n+1}^{\dagger}$ and $\Be_{\SO_{2n+1},\Std}^{\dagger}$ using a refinement of Dwork's congruences \cite{Dw68} in the $2$-adic case. Then we conclude that $\Be_{2n+1}^{\dagger}\simeq \Be_{\SO_{2n+1},\Std}^{\dagger}$ by a recent theorem of Tsuzuki \cite{Tsu19} (cf. appendix \ref{appendix}). 
	Since $\Be_{\cG}^{\dagger}$ is the $p$-adic companion of $\Kl_{\cG}$, theorem \ref{intro functoriality} follows.
\end{secnumber}

\begin{secnumber}
By functoriality, we reduce theorem \ref{intro main thm}(ii) to the corresponding assertion for (Frobenius) Newton polygon of $\Be_{\SL_n,\Std}^{\dagger}$ and of $\Be_{\SO_{2n+1},\Std}^{\dagger}$, which are isomorphic to $\Be_n^{\dagger}$ and a hypergeometric overconvergent $F$-isocrystal \cite{Miy} respectively. 
In these cases, the assertion follows from the results of Dwork, Sperber and Wan \cite{Dw74,Sp80,Wan93}. 
\end{secnumber}

\begin{secnumber} \label{intro geo Satake}
	As mentioned above, in order to carry through the first step of \ref{intro main thm proof}, we need to establish a version of the geometric Satake equivalence for arithmetic $\mathscr{D}$-modules. 
	This is based on the recent development of the six functors formalism, weight theory and nearby/vanishing cycle functors for arithmetic $\mathscr{D}$-modules \cite{Car09,CT12,AC17,AC18,Abe18II}. 
	We will review these theories in subsections \ref{review isocrystals}--\ref{nearby vanishing cycle}.

	To state our result, we first introduce some notations. 
	Let $k$ be a finite field with $q=p^s$ elements and $K$ a finite extension of $\mathbb{Q}_q$. Suppose that there exists a lift $\sigma:K\to K$ of the $t$-th Frobenius automorphism of $k$ for some integer $t$. 
	Let $G$ be a split reductive group over $k$, $\cG$ its Langlands dual group over $K$, $\Gr_G$ the affine Grassmannian of $G$, and $L^+G$ the positive loop group of $G$. 

	Given a $k$-scheme $X$, one may consider the category $\Hol(X/K)$ of holonomic arithmetic $\mathscr{D}$-modules on $X$ and the category $\Hol(X/K_F)$ of objects of $\Hol(X/K)$ with a Frobenius structure, which are the analogues of the category of $\ell$-adic sheaves on $X_{\overline{k}}$ and the category of Weil sheaves on $X$ respectively. 
	We denote by $\Hol_{L^+G}(\Gr_G/K)$ (resp. $\Hol_{L^+G}(\Gr_G/K_F)$) the category of $L^+G$-equivariant objects in $\Hol(\Gr_G/K)$ (resp. $\Hol(\Gr_G/K_F)$). 
       
	The geometric Satake equivalence (for geometric coefficients) states that the category $\Hol_{L^+G}(\Gr_G/K)$ is a neutral Tannakian category over $K$ whose Tannakian group is $\cG$ \eqref{thm Tannakian cat}. 	
	The Tannakian structure and the Frobenius structure on $\Hol_{L^+G}(\Gr_G/K_F)$ allows us to define a homomorphism $\iota: \mathbb{Z}\to \Aut(\cG(K))$ \eqref{general Tannakian Z} and hence a semi-direct product $\cG(K)\rtimes \mathbb{Z}$.

	\begin{theorem} \label{geo satake intro}
		\textnormal{(i) (Geometric coefficients \ref{thm Tannakian cat})} There exists a natural equivalence of monoidal categories between $\Hol_{L^+G}(\Gr_{G}/K)$ and $\Rep(\cG)$. 

		\textnormal{(ii) (Arithmetic coefficients \ref{Satake arith})} 
		There exists an equivalence of monoidal categories between $\Hol_{L^+G}(\Gr_{G}/K_F)$ and the category $\Rep^{\circ}_{K,\sigma}(\cG(K)\rtimes \mathbb{Z})$ of certain $\sigma$-semi-linear representations of $\cG(K)\rtimes \mathbb{Z}$ (cf. \ref{notation semilinear L}).
	\end{theorem} 
       
       Although the strategy of the proof of this theorem is same as the $\ell$-adic case, we need to establish some foundational results in the setting of arithmetic $\mathscr{D}$-modules. 
       We introduce a notion of \emph{universal local acyclicity} (ULA) for arithmetic $\mathscr{D}$-modules and discuss its relation with the nearby/vanishing cycle functors introduced by Abe-Caro and Abe \cite{AC17,Abe18II} in subsection \ref{subsection ULA}. 
       We also prove a version of Braden's hyperbolic localization theorem is this setting in subsection \ref{Braden thm sec}. 
       
       Recall that there are motivic versions of the geometric Satake equivalence \cite{Zhu18, RS19}. The above theorem can be regarded as their $p$-adic realization. (But as far as we know, there is no general construction of the realization functor as we need so the above theorem is not a formal consequence of \emph{loc. cit.}.)
       On the other hand, there is a very recent work of R. Cass \cite{Ca19} on the geometric Satake equivalence for perverse $\mathbb{F}_p$-sheaves. 
       It would be very interesting to see whether there is a version of the geometric Satake equivalence for some $\mathbb Z_p$-coefficient sheaf theory, which after inverting $p$ and mod $p$ specializes to our version and Cass' version respectively. 
        
        We hope our article will lead further investigation of the $p$-adic aspect of the geometric Langlands program in the future.
\end{secnumber}

\begin{secnumber}
	We briefly go over the organization of this article. Section \ref{Dmods} contains a review of and some complements on the theory of arithmetic $\mathscr{D}$-modules and overconvergent ($F$-)isocrystals. 
	In section \ref{sec Sat}, we establish the geometric Satake equivalence for arithmetic $\mathscr{D}$-modules \eqref{geo satake intro}. 
	Subsections \ref{construction HNY}-\ref{Generalised Bessel} are devoted to the proof of theorem \ref{intro main thm}(i) (cf. \ref{intro main thm proof}). 
	We calculate the monodromy group of $\Be_{\cG}^{\dagger}$ in subsection \ref{sec monodromy} (cf. theorem \ref{intro monodromy gp} and \ref{intro monodromy proof}). 
	In subsection \ref{sec functoriality}, we prove the functoriality of Bessel $F$-isocrystals and of Kloosterman sheaves (cf. theorem \ref{intro functoriality}(i) and \ref{functoriality proof}). 
	In subsections \ref{sec hyper} and \ref{Kl classical}, we identify the Bessel $F$-isocrystals for classical groups with certain hypergeometric differential equations studied by Katz and Miyatani \cite{Katz90,Miy}. 
	In particular, we obtain identities in \ref{identity exp sums intro}. 
	In the last subsection \eqref{NT polygons}, we study the Frobenius Newton polygon of $\Be_{\cG}^{\dagger}$ and prove theorem \ref{intro main thm}(ii).
	Appendix \ref{appendix} is devoted to a proof of theorem \ref{intro functoriality}(ii) from the perspective of $p$-adic differential equations. 
\end{secnumber}

\begin{secnumber} \label{basic notation}
	In this article, we fix a prime number $p$. 
	Let $s$ be a positive integer and set $q=p^s$. 
	Let $k$ be a perfect field of characteristic $p$, $\overline{k}$ an algebraic closure of $k$ and $R$ a complete discrete valuation ring with residue field $k$. We set $K=\Frac(R)$. 
	We fix an algebraic closure $\overline{K}$ of $K$. 
	We assume moreover that the $s$-th Frobenius endomorphism $k\xrightarrow{\sim} k,~ x\mapsto x^q$ lifts to an automorphism $\sigma:R\xrightarrow{\sim}R$.  

	By a \textit{$k$-scheme} (resp. \textit{$R$-scheme}), we mean a separated scheme of finite type over $k$ (resp. over $R$). 
\end{secnumber}

\textbf{Acknowledgement.} 
	We would like to thank Benedict Gross, Shun Ohkubo, Daqing Wan, Liang Xiao and Zhiwei Yun for valuable discussions. X. Z. is partially supported by the National Science Foundation under agreement Nos. DMS-1902239.

\section{Review and complements on arithmetic $\mathscr{D}$-modules}
\label{Dmods}

\subsection{Overconvergent ($F$-)isocrystals and their rigid cohomologies} \label{review isocrystals}
In this subsection, we briefly recall the definition of overconvergent (resp. convergent) isocrystal following \cite{Ber96}. 

\begin{secnumber} \label{overconvergent function}
	Let $X$ be a $k$-scheme. 
	\textit{A frame of $X$} is a quadruple $(Y,j,\PP,i)$ (written as $(Y,\PP)$ for short) consisting of an open immersion $j:X\to Y$ of $k$-schemes, and a closed immersion $i:Y\to \PP$, where $\PP$ is a separated formal $R$-scheme which is
	smooth over $\Spf(R)$ in a neighborhood of $X$. 
	We denote by $\PP^{\rig}$ the rigid analytic space associated to $\PP$ and by $]X[_{\PP}$, $]Y[_{\PP}$ the tube of $X,Y$ in $\PP^{\rig}$ respectively (\cite{Ber96} \S~1).
		
	\textit{A strict neighborhood of $]X[_{\PP}$ in $]Y[_{\PP}$} is an admissible open subspace $V$ of $]Y[_{\PP}$ such that $V\cup ]Y\setminus X[_{\PP}$ 
	forms an admissible covering of $]Y[_{\PP}$ (\cite{Ber96} 1.2). 
	There exists an exact functor $j^{\dagger}$ from the category $\Ab(V)$ of abelian sheaves on $V$ to itself (\cite{Ber96} 2.1.1), 	
	defined by 
	\begin{equation} \label{functor j dagger}
		j^{\dagger}E=\varinjlim_{U} j_{U,V*}j_{U,V}^{-1}(E),
	\end{equation}
where the inductive limit is taken over all strict neighborhoods $j_{U,V}:U\to V$ of $]X[_{\PP}$ in $V$. 
	It is known that $j^{\dagger}\mathscr{O}_{V}$ is coherent as a sheaf of rings. 
	
	The notion of a morphism of frames is naturally defined, and a morphism $u:(Y',\PP')\to (Y,\PP)$ of frames induces a tensor functor:
	\begin{displaymath}
		u^{*}: \textnormal{ (Coherent $j^{\dagger}\mathscr{O}_{]Y[_{\PP}}$-modules) } \to 
			\textnormal{ (Coherent $j^{\dagger}\mathscr{O}_{]Y'[_{\PP'}}$-modules) }.
	\end{displaymath}	

	We denote by $\Conn(j^{\dagger}\mathscr{O}_{]Y[_{\PP}})$ the category of coherent $j^{\dagger}\mathscr{O}_{]Y[_{\PP}}$-modules $\mathscr{M}$ equipped with a $K$-linear morphism $\nabla:\mathscr{M}\to \mathscr{M}\otimes_{j^{\dagger}\mathscr{O}_{]Y[_{\PP}}}j^{\dagger}\Omega_{]Y[_{\PP}}^1$ satisfying the Leibniz rule and the usual integrability condition. 
\end{secnumber}

\begin{secnumber}\label{overcon isoc}
	For $n=1,2,3$, we denote by $\PP^n$ the fiber product of $n$-copies of $\PP$ over $\Spf(R)$. 
	Then $(Y,\PP^n)$ is a frame of $X$ and we have projections $p_i:(Y,\PP^2)\to (Y,\PP)$ ($i=1,2$), $p_{ij}:(Y,\PP^{3})\to (Y,\PP^{2})$ ($1\le i< j\le 3$) and the diagonal morphism $\Delta:(Y,\PP)\to (Y,\PP^2)$. 

	We denote by $\Iso^{\dagger}(X,Y/K)$ the category of pairs $(\mathscr{M},\varepsilon)$ consisting of a coherent $j^{\dagger}\mathscr{O}_{]Y[_{\PP}}$-module $\mathscr{M}$ and an isomorphism
		\begin{displaymath}
			\varepsilon: p_2^{*}(\mathscr{M})\xrightarrow{\sim} p_1^{*}(\mathscr{M})
		\end{displaymath}
	satisfying $\Delta^{*}(\varepsilon)=\id$ and $p_{13}^{*}(\varepsilon)=p_{12}^{*}(\varepsilon)\circ p_{23}^{*}(\varepsilon)$. Such a pair is called an \textit{isocrystal on $X$ overconvergent along $Y-X$ (relative to $K$)}. 
	This category is independent of the choice of $\PP$ up to canonical equivalences (\cite{Ber96} 2.3.1). 
	
	When $Y=X$, we have $j^{\dagger}\mathscr{O}_{]Y[_{\PP}}=\mathscr{O}_{]X[_{\PP}}$. Such a pair is also called a \textit{convergent isocrystal on $X/K$}. 
	The category $\Iso^{\dagger}(X,X/K)$ is simply denoted by $\Iso(X/K)$. 
	
	When $\overline{X}$ is a compactification of $X$ over $k$, the category $\Iso^{\dagger}(X,\overline{X}/K)$ is independent of the choice of $\overline{X}$ up to canonical equivalences (\cite{Ber96} 2.3.5) and is simply called the category of \textit{overconvergent isocrystals on $X/K$}, denoted by $\Iso^{\dagger}(X/K)$. 
\end{secnumber}

\begin{secnumber} \label{overconvergent connection}
	There exists an exact and fully faithful functor (\cite{Ber96} 2.2.5, 2.2.7)
	\begin{displaymath}
		\Iso^{\dagger}(X,Y/K)\to \Conn(j^{\dagger}\mathscr{O}_{]Y[_{\PP}}),\qquad (\mathscr{M},\varepsilon)\mapsto (\mathscr{M},\nabla).
	\end{displaymath}

	When $X=Y$ (resp. $Y=\overline{X}$ is a compactification of $X$), we say that $(\mathscr{M},\nabla)$ is \textit{convergent} (resp. \textit{overconvergent}) if it is contained in the essential image of the above functor.

	Let $\mathscr{M}$ be an overconvergent isocrystal on $X$ relative to $K$ and $\mathscr{M}\otimes_{j^{\dagger}\mathscr{O}_{]\overline{X}[_{\PP}}}j^{\dagger}\Omega_{]\overline{X}[_{\PP}}^{\bullet}$ the associated de Rham complex with respect to a frame $(\overline{X},\PP)$. 
	The \textit{rigid cohomology} $\bR \Gamma_{\rig}(X/K,\mathscr{M})$ is defined by 
	\begin{equation}
		\bR \Gamma_{\rig}(X/K,\mathscr{M})=
		\bR \Gamma(]\overline{X}[_{\PP},\mathscr{M}\otimes_{j^{\dagger}\mathscr{O}_{]\overline{X}[_{\PP}}} j^{\dagger}\Omega_{]\overline{X}[_{\PP}}^{\bullet}). 
		\label{rig coh def general}
	\end{equation}
\end{secnumber}

\begin{secnumber} \label{overconv F-iso}
	The category $\Iso^{\dagger}(X/K)$ is functorial with respect to pullbacks (\cite{Ber96} 2.3.6). 
	The absolute $s$-th Frobenius morphism $F_X:X\to X$ and endomorphism $\sigma:K\to K$ induce the Frobenius pullback functor:
	\begin{equation} \label{Frobenius pullback iso}
		F_X^{*}:\Iso^{\dagger}(X/K)\to \Iso^{\dagger}(X/K).
	\end{equation}

	\textit{An overconvergent $F$-isocrystal on $X/(K,\sigma)$ (or simply $X/K$)} is an overconvergent isocrystal $\mathscr{M}$ together with an isomorphism $\varphi:F_X^{*}(\mathscr{M})\xrightarrow{\sim} \mathscr{M}$, called \textit{($s$-th) Frobenius structure of $\mathscr{M}$}. 

	We denote by $\Fr\Iso^{\dagger}(X/K)$ the category of overconvergent $F$-isocrystals on $X/K$ and by $\Iso^{\dagger\dagger}(X/K)$ the thick full subcategory of $\Iso^{\dagger}(X/K)$ generated by those that can be endowed with an $s'$-th Frobenius structure for some integer $s|s'$. 
\end{secnumber}

\begin{secnumber} \label{Dwork isocrystal}
	In the following, we explain some examples of overconvergent isocrystals.  

	(i) \textbf{Dwork $F$-isocrystal.}
	Let $k=\mathbb{F}_p$ (i.e. $s=1$), $K=\mathbb{Q}_p(\mu_p)$, $R=\mathscr{O}_K$ and $\sigma=\id$. 
	We choose $\pi\in K$ such that $\pi^{p-1}=-p$ and take $\mathscr{P}=\widehat{\mathbb{P}}^1_R,Y=\P1_k,X=\A1_k$. Then $]Y[=\widehat{\mathbb{P}}^{1,\rig}_{R}$ and $]X[$ is the closed unit disc. 
		If $t$ denotes a coordinate of $\A1$, the connection on $j^{\dagger}\mathscr{O}_{]Y[}$ defined by
		\begin{displaymath}
			\nabla = d +\pi dt,
		\end{displaymath}
		is overconvergent and is called \textit{Dwork isocrystal}, denoted by $\mathscr{A}_{\pi}$. 

	By considering the lifting of the Frobenius of $\P1$ to $R$ given by $t\to t^p$, $F_{\A1_k}^{*}(\mathscr{A}_{\pi})$ is the module $j^{\dagger}\mathscr{O}_{]Y[}$ equipped with the connection $\nabla^{\sigma}$ defined by 
		\begin{displaymath}
			\nabla^{\sigma} = d + \pi p t^{p-1} dt.
		\end{displaymath}
	We define a Frobenius structure $\varphi:F_{\A1_k}^{*}(\mathscr{A}_{\pi})\to \mathscr{A}_{\pi}$ by the multiplication by $\theta_{\pi}(x)=\exp(\pi(x-x^{p}))$, which is a section of $j^{\dagger}\mathscr{O}_{]Y[}$. 
	This gives \textit{Dwork $F$-isocrystal} associated to $\pi$ on $\A1_k/K$. 

	There exists a unique nontrivial additive character $\psi:\mathbb{F}_p\to K^{\times}$ satisfying
	\begin{displaymath}
		\psi(1)=1+\pi \mod \pi^{2}.
	\end{displaymath}
	For each $x\in \mathbb{F}_{p}$, we denote by $\widetilde{x}$ the Teichmüller lifting of $x$ in $\mathbb{Q}_p$. Then $\theta_{\psi}(\widetilde{x})=\psi(x)$ (\cite{Ber84} 1.4).  So
	the Frobenius trace function of $\mathscr{A}_{\pi}$ is equal to $\psi\circ\Tr_{\mathbb{F}_q/\mathbb{F}_p}(-)$. 
	We also denote $\mathscr{A}_{\pi}$ by $\mathscr{A}_{\psi}$, as it plays a similar role of Artin-Schreier sheaf associated to $\psi$ in the $\ell$-adic theory. 

	(ii) \textbf{Kummer $F$-isocrystal.} Let $k$ be a finite field with $q=p^s$ elements. Set $K=\mathbb{Q}_q$, $R=\mathscr{O}_K$ and $\sigma=\id$. 
	We choose $a\in R$ and take $\PP=\widehat{\mathbb{P}}^1_R$, $Y=\mathbb{P}^1_k$ and $X=\mathbb{G}_{m,k}$. 
	If $x$ denotes a coordinate of $\Gm$, the connection on $j^{\dagger}\mathscr{O}_{]Y[}$ defined by 
		\begin{displaymath}
			\nabla = d - a\frac{dx}{x},
		\end{displaymath}
		is overconvergent, denoted by $\mathscr{K}_{a}$. With the lifting of Frobenius as in (i), $F_{\Gm}^{*}(\mathscr{K}_{a})$ is the module $j^{\dagger}\mathscr{O}_{]Y[}$ equipped with connection $\nabla^{\sigma}$ defined by 
		\begin{displaymath}
			\nabla^{\sigma} = d - a p \frac{dx}{x}.	
		\end{displaymath}
		The isocrystal $\mathscr{K}_a$ has a Frobenius structure if and only if $a\in \frac{1}{q-1}\mathbb{Z}$. 
		This Frobenius structure is given by multiplication by $t^{(q-1)a}$. Then we obtain the \textit{Kummer $F$-isocrystal} $\mathscr{K}_a$. 

		Let $\chi$ be a character of $k^{\times}$ such that $\chi(x)=\widetilde{x}^{(q-1)a}$, where $\widetilde{x}$ denotes the Teichmüller lifting of $x$. 
		We also denote $\mathscr{K}_{a}$ by $\mathscr{K}_{\chi}$ because the Frobenius trace function of $\mathscr{K}_{a}$ is equal to $\chi\circ \Nm_{k'/k}$. 
\end{secnumber}

\subsection{(Co)specialization morphism for de Rham and rigid cohomologies} 
\label{Sp cosp}

In this subsection, we review the specialization and cospecialization morphisms between the de Rham and rigid cohomology following (\cite{BB04} \S~1) and show the compatibility of these two morphisms in proposition \ref{compatibility sp cosp}. 

The results of this subsection will mainly be used in subsection \ref{Frob str Bessel}. 

\begin{secnumber} \label{dagger construction}
	In this subsection, $X$ denotes a smooth $R$-scheme of pure relative dimension $d$ and $X_{k}$ (resp. $X_{K}$) its special (resp. generic) fiber. 
	We use the corresponding calligraphic letter $\mathcal{X}$ to denote the rigid analytic space $X_K^{\an}$ associated to $X_{K}$ and the corresponding gothic letter $\XX$ to denote the $p$-adic completion of $X$. We denote by $\XX^{\rig}$ the rigid generic fiber of $\XX$. 
Let $\varepsilon:\mathcal{X}\to X_{K}$ denote the canonical morphism of topoi.

Let $(M,\nabla)$ be a coherent $\mathscr{O}_{X_K}$-module endowed with an integrable connection (relative to $K$). 
	We denote by $(M^{\an},\nabla^{\an})$ its pullback to $\mathcal{X}$ along $\varepsilon$.
	Then the canonical morphism of de Rham complexes $\varepsilon^{-1}(M\otimes_{\mathscr{O}_{X_K}}\Omega_{X_K}^{\bullet})\to M^{\an}\otimes_{\mathscr{O}_{\mathcal{X}}}\Omega_{\mathcal{X}}^{\bullet}$ induces a morphism from algebraic de Rham cohomology to analytic de Rham cohomology
	\begin{equation} \label{alg to an dR}
		\bR\Gamma_{\dR}(X_K,(M,\nabla))=\bR \Gamma(X_{K},M\otimes_{\mathscr{O}_{X_{K}}}\Omega_{X_{K}}^{\bullet})\to 
		\bR \Gamma(\mathcal{X},M^{\an}\otimes_{\mathscr{O}_{\mathcal{X}}}\Omega_{\mathcal{X}}^{\bullet})=\bR\Gamma_{\an}(\mathcal{X},(M^{\an},\nabla^{\an})).
	\end{equation}
\end{secnumber}

\begin{secnumber} \label{setting strict neighborhood}
	We assume that there exists a smooth proper $R$-scheme $\overline{X}$ and an open immersion $j:X\to \overline{X}$. 
	Let $\overline{\XX}$ be the $p$-adic completion of $\overline{X}$. Then the two rigid spaces $\overline{\XX}^{\rig}$ and $\overline{\mathcal{X}}=\overline{X}_{K}^{\an}$ are isomorphic, and $\XX^{\rig}$ is the tube $]X_{k}[_{\overline{\XX}}$ of $X_{k}$ in $\overline{\mathcal{X}}$. 
	In particular, $\mathcal{X}$ is a strict neighborhood of $\XX^{\rig}$ in $\overline{\XX}^{\rig}$. 
	We denote by $\Conn(X_{K})$ (resp. $\Conn(\mathcal{X})$) the category of coherent $\mathscr{O}_{X_K}$-modules with an integrable connection. 	

	We associate to $M^{\an}$ a $j^{\dagger}\mathscr{O}_{\overline{\XX}^{\rig}}$-module $M^{\dagger}=j^{\dagger}(M^{\an})$ \eqref{functor j dagger}, endowed with the corresponding connection. 
	In this setting, we have the following diagram \eqref{overconvergent connection}: 
	\begin{equation} \label{various categories}
		\xymatrix{
			\Conn(X_{K}) \ar[r]_{(-)^{\an}} \ar@/^2pc/[rr]^{(-)^{\dagger}} & 
			\Conn(\mathcal{X}) \ar[r]_{j^{\dagger}} & 
			\Conn(j^{\dagger}\mathscr{O}_{\overline{\XX}^{\rig}}) \ar[r]^{ |_{\XX^{\rig}}} & 
			\Conn(\mathscr{O}_{\XX^{\rig}}) \\
			\Fr\Iso^{\dagger}(X/K)\ar[r]&\Iso^{\dagger\dagger}(X/K) \ar@{^{(}->}[r]& \Iso^{\dagger}(X/K) \ar[r]^{|_{\XX^{\rig}}} \ar@{^{(}->}[u] &
			\Iso(X/K) \ar@{^{(}->}[u]
		}
	\end{equation}
	where the vertical arrows are fully faithful \eqref{overconvergent connection}. 
	When $\overline{X}_k\setminus X_k$ is a divisor, the functor $|_{\XX^{\rig}}$ is exact and faithful (\cite{Ber96II} 4.3.10). 
\end{secnumber}

\begin{secnumber} \label{dR rig coh}
	In the following, we assume moreover that the connection on $M^{\dagger}$ is \textit{overconvergent} (see \ref{overconvergent connection} for the definition). 
	The rigid cohomology $\bR \Gamma_{\rig}(X_{k}/K,M^{\dagger})$ \eqref{rig coh def general} can be calculated by 
	\begin{equation}
		\bR \Gamma_{\rig}(X_{k}/K,M^{\dagger})\xrightarrow{\sim} \bR \Gamma(\mathcal{X},M^{\dagger}\otimes_{\mathscr{O}_{\mathcal{X}}}\Omega_{\mathcal{X}}^{\bullet}). \label{rig coh def}
	\end{equation}

	The adjoint morphism $\id \to j^{\dagger}$ \eqref{functor j dagger} induces a canonical morphism on $\mathcal{X}$
	\begin{equation}
		M^{\an}\otimes_{\mathscr{O}_{\mathcal{X}}}\Omega_{\mathcal{X}}^{\bullet} \to 
		M^{\dagger}\otimes_{\mathscr{O}_{\mathcal{X}}}\Omega_{\mathcal{X}}^{\bullet}.
		\label{an to rig}
	\end{equation}
	By composing with \eqref{alg to an dR}, we deduce a canonical morphism, denoted by $\rho_{M}$ and called \textit{specialization morphism} for de Rham and rigid cohomologies: 
	\begin{equation} \label{specialization morphism}
		\rho_{M}: \bR \Gamma_{\dR}(X_{K},(M,\nabla)) \to \bR \Gamma_{\rig}(X_{k}/K,M^{\dagger}).
	\end{equation}

	Let $\mathbb{R}\underline{\Gamma}_{]X_k[}$ be the (derived) functor of local sections supported in the tube $]X_k[_{\XX}$ on $\mathcal{X}$ (or on $\overline{\mathcal{X}}$) (\cite{Ber96} 2.1.6). 
	The rigid cohomology with compact supports and coefficients in $M^{\dagger}$ is defined as: 
	\begin{eqnarray}
		\bR\Gamma_{\rig,\rc}(X_{k}/K,M^{\dagger})=\bR\Gamma(\mathcal{X}, \bR\underline{\Gamma}_{]X_k[} (M^{\an}\otimes\Omega_{\mathcal{X}}^{\bullet})). 	\label{def rig c}
	\end{eqnarray}
	
	The canonical morphism 
	\begin{equation} \label{local sections dR to dR}
		\bR\underline{\Gamma}_{]X_k[} (M^{\an}\otimes\Omega_{\mathcal{X}}^{\bullet})\to M^{\an}\otimes\Omega_{\mathcal{X}}^{\bullet}
	\end{equation}
	and \eqref{an to rig} induce a morphism
	\begin{equation} \label{forget map rig}
		\iota_{\rig}: \bR\Gamma_{\rig,\rc}(X_k/K,M^{\dagger})\to \bR \Gamma_{\rig}(X_k/K,M^{\dagger}).
	\end{equation}
\end{secnumber}

\begin{secnumber}
	We recall the definition of de Rham cohomology with compact supports and coefficients in $(M,\nabla)$ and the cospecialization morphism, following (\cite{BB04} 1.8 and \cite{AB} Appendix D.2). 
	
	Let $I$ be the ideal sheaf of the reduced closed subscheme $\overline{X}_{K}-X_{K}$ in $\overline{X}_{K}$. Take a coherent $\mathscr{O}_{\overline{X}_{K}}$-module $\overline{M}$ extending $M$. 
	The connection $\nabla$ extends to a connection on the pro-$\mathscr{O}_{\overline{X}_{K}}$-module $(I^{n}\overline{M})_{n}$ (\cite{AB} D.2.12). 
	This allows us to define the de Rham pro-complex $I^{\bullet}\overline{M}\otimes_{\mathscr{O}_{\overline{X}_K}} \Omega_{\overline{X}_{K}}^{\bullet}:= (I^{n}\overline{M})_{n}\otimes \Omega_{\overline{X}_K}^{\bullet}$. 
	The algebraic de Rham cohomology with compact supports and coefficients in $(M,\nabla)$, denoted by $\bR \Gamma_{\dR,\rc}(X_{K},(M,\nabla))$, is defined as (\cite{AB} D.2.16)
	\begin{eqnarray}
		\bR \Gamma_{\dR,\rc}(X_{K},(M,\nabla)) &=& \bR \Gamma(\overline{X}_{K}, \bR\varprojlim I^{\bullet}\overline{M}\otimes\Omega_{\overline{X}_{K}}^{\bullet}) 	\label{alg dR c}\\
		&\simeq& \bR \varprojlim \bR \Gamma(\overline{X}_K,I^{\bullet}\overline{M}\otimes\Omega_{\overline{X}_{K}}^{\bullet}). \nonumber
	\end{eqnarray}
	
	Let $j_K$ denote the open immersion $X_{K}\to \overline{X}_{K}$. There exists a canonical isomorphism on $X_{K}$:
	\begin{equation}
		j_{K}^{*}(\bR \varprojlim (I^{\bullet} \overline{M}\otimes \Omega_{\overline{X}_K}^{\bullet}))\xrightarrow{\sim} M\otimes \Omega_{X_K}^{\bullet}.
	\end{equation}
	We deduces from its adjoint $\bR \varprojlim (I^{\bullet} \overline{M}\otimes \Omega_{\overline{X}_K}^{\bullet}) \to \bR j_{K*}(M\otimes \Omega_{X_K}^{\bullet})$ a canonical morphism:
	\begin{equation}
		\iota_{\dR}: \bR \Gamma_{\dR,\rc}(X_K,(M,\nabla)) \to \bR \Gamma_{\dR}(X_K,(M,\nabla))
		\label{alg dRc to dR}
	\end{equation}

	By the rigid GAGA, there are canonical isomorphisms
	\begin{eqnarray}
		\qquad \bR \varprojlim \bR \Gamma(\overline{X}_K,I^{\bullet}\overline{M}\otimes \Omega_{\overline{X}_K}^{\bullet}) \xrightarrow{\sim}
		\bR \varprojlim \bR \Gamma(\overline{\mathcal{X}},I^{\bullet}\overline{M}^{\an}\otimes \Omega_{\overline{\mathcal{X}}}^{\bullet}) \xrightarrow{\sim} 
		\bR \Gamma(\overline{\mathcal{X}}, \bR \varprojlim I^{\bullet}\overline{M}^{\an}\otimes \Omega_{\overline{\mathcal{X}}}^{\bullet}). \label{alg dR an dR an}
	\end{eqnarray}
	We denote the right hand side by $\bR\Gamma_{\an,\rc}(\mathcal{X},(M^{\an},\nabla^{\an}))$. 
	Let $j^{\an}$ be the inclusion $\mathcal{X}\to \overline{\mathcal{X}}$. 
	Similarly, there exists a canonical morphism
	\begin{equation} \label{map define iota an}
		\bR \varprojlim (I^{\bullet} \overline{M}^{\an}\otimes \Omega_{\overline{\mathcal{X}}}^{\bullet}) \to 
		\bR j^{\an}_{*}(M^{\an}\otimes \Omega_{\mathcal{X}}^{\bullet}),
	\end{equation}
	which induces a morphism on analytic de Rham cohomologies
	\begin{equation}
		\iota_{\an}: \bR\Gamma_{\an,\rc}(\mathcal{X},(M^{\an},\nabla^{\an})) \to \bR\Gamma_{\an}(\mathcal{X},(M^{\an},\nabla^{\an})).
	\end{equation}

	Since $(\mathcal{X},](\overline{X}-X)_k[_{\XX})$ is an admissible covering of $\overline{\mathcal{X}}$, the canonical morphisms
		\begin{equation} \label{jstar compact support}
			\bR\underline{\Gamma}_{]X_k[}(\bR j^{\an}_{*}(E)) \to 
			\bR j^{\an}_{*}(\bR\underline{\Gamma}_{]X_k[} (E)), \quad
			\bR\underline{\Gamma}_{]X_k[}(E) \to 
			\bR\underline{\Gamma}_{]X_k[} \bR j^{\an}_{*}(j^{\an *}(E))
		\end{equation}
		are isomorphic for any complex of abelian sheaves $E$ on $\mathcal{X}$ (resp. $\overline{\mathcal{X}}$). 
		Then \eqref{map define iota an} induces an isomorphism
		\begin{equation} \label{iso support tube}
			\bR\underline{\Gamma}_{]X_k[}(\bR \varprojlim (I^{\bullet} \overline{M}^{\an}\otimes \Omega_{\overline{\mathcal{X}}}^{\bullet}) )
				\xrightarrow{\sim}
				\bR\underline{\Gamma}_{]X_k[} (\bR j^{\an}_{*}(M^{\an}\otimes \Omega_{\mathcal{X}}^{\bullet})).
		\end{equation}

		The \textit{cospecialization morphism}, denoted by $\rho_{\rc,M}$, is defined as the composition 
	\begin{eqnarray}
		\quad \rho_{\rc,M}: \bR\Gamma_{\rig,\rc}(X_{k}/K,M^{\dagger})& \stackrel{\eqref{jstar compact support}}{\simeq} & \bR\Gamma(\overline{\mathcal{X}}, \bR\underline{\Gamma}_{]X_k[} \bR j^{\an}_*(M^{\an}\otimes\Omega_{\mathcal{X}}^{\bullet}))
			\label{cospecialization morphism} \\
			& \stackrel{\eqref{iso support tube} }{\simeq} &  \bR\Gamma(\overline{\mathcal{X}}, \bR\underline{\Gamma}_{]X_k[}(\bR \varprojlim (I^{\bullet} \overline{M}^{\an}\otimes \Omega_{\overline{\mathcal{X}}}^{\bullet}) ))  \nonumber \\
			&\to& \bR\Gamma(\overline{\mathcal{X}}, \bR \varprojlim (I^{\bullet} \overline{M}^{\an}\otimes \Omega_{\overline{\mathcal{X}}}^{\bullet}))
			~(=\bR\Gamma_{\an,\rc}(\mathcal{X},(M^{\an},\nabla^{\an}))) \nonumber \\
			&\simeq & \bR\Gamma_{\dR,\rc}(X_{K},(M,\nabla)). \nonumber
	\end{eqnarray}
\end{secnumber}

\begin{prop} \label{compatibility sp cosp}	
	With the above notation and assumption, the following diagram is commutative:
	\begin{displaymath} 
		\xymatrix{
			\bR \Gamma_{\rig,\rc}(X_{k}/K,M^{\dagger}) \ar[r]^{\iota_{\rig}} \ar[d]_{\rho_{\rc,M}}&
			\bR \Gamma_{\rig}(X_{k}/K,M^{\dagger})\\
			\bR \Gamma_{\dR,\rc}(X_{K},(M,\nabla)) \ar[r]^{\iota_{\dR}} & 
			\bR \Gamma_{\dR}(X_{K},(M,\nabla)) \ar[u]_{\rho_{M}}.		
		} 
	\end{displaymath}
\end{prop}
\begin{proof}
	The algebraic de Rham cohomology with compact supports is isomorphic to the analytic one \eqref{alg dR an dR an}. 
	It suffices to show the following diagram is commutative 
	\begin{equation} \label{rig an diag}
		\xymatrix{
			\bR \Gamma_{\rig,\rc}(X_{k}/K,M^{\dagger}) \ar[r]^{\iota_{\rig}} \ar[d]_{\rho_{\rc,M}}&
			\bR \Gamma_{\rig}(X_{k}/K,M^{\dagger}) \\
			\bR \Gamma_{\an,\rc}(\mathcal{X}, (M^{\an},\nabla^{\an})) \ar[r]^{\iota_{\an}} & 
			\bR \Gamma_{\an}(\mathcal{X}, (M^{\an},\nabla^{\an})) \ar[u].		
		}
	\end{equation}
	where right vertical arrow is induced by \eqref{an to rig}. 
	
	The morphism $\bR \Gamma_{\rig,\rc}(X_{k}/K,M^{\dagger})\to \bR \Gamma_{\an}(\mathcal{X}, (M^{\an},\nabla^{\an}))$ is induced by the composition on $\overline{\mathcal{X}}$:
	\begin{displaymath}
		\bR\underline{\Gamma}_{]X_k[}(\bR \varprojlim (I^{\bullet} \overline{M}^{\an}\otimes \Omega_{\overline{\mathcal{X}}}^{\bullet}) ) \to 
		\bR \varprojlim (I^{\bullet} \overline{M}^{\an}\otimes \Omega_{\overline{\mathcal{X}}}^{\bullet}) \xrightarrow{\eqref{map define iota an}}
		\bR j^{\an}_{*}(M^{\an}\otimes \Omega_{\mathcal{X}}^{\bullet}). 
	\end{displaymath}
	The restriction of the above morphism to $\mathcal{X}$ coincides with the canonical morphism \eqref{local sections dR to dR}, which induces $\iota_{\rig}$ \eqref{forget map rig}. Then the commutativity of \eqref{rig an diag} follows.
\end{proof}

\subsection{Six functors formalism for arithmetic $\mathscr{D}$-modules} \label{section six functors formalism}
	Rigid cohomology theory is a $p$-adic Weil cohomology for a variety in characteristic $p$. 
	Overconvergent $F$-isocrystals are ``local systems'' in the coefficients theory of rigid cohomology. 
	However, the category of overconvergent $F$-isocrystals is not stable under certain cohomological operators. 
	Inspired by the theory of algebraic $\mathscr{D}$-modules, Berthelot introduced the notion of arithmetic $\mathscr{D}$-modules \cite{Ber96II, Ber02}. A six functor formalism for these coefficients is recently achieved by Caro, Abe and etc. 

	We use the notation of arithmetic $\mathscr{D}$-modules \cite{Ber02}. 
	For a smooth formal $R$-scheme $\XX$ and a divisor $Z$ of the special fiber of $\XX$, let $\mathscr{O}_{\XX,\mathbb{Q}}(^{\dagger}Z)$ (resp. $\mathscr{D}_{\XX,\mathbb{Q}}^{\dagger}(^{\dagger}Z)$) denote the sheaf of rings of functions (resp. differential operators) on $\XX$ with singularities overconvergent along $Z$ (\cite{Ber96II} 4.2.4). 
	Note that $\mathscr{O}_{\XX,\mathbb{Q}}(^{\dagger}Z)$ is isomorphic to $\Sp_{*}(j^{\dagger}\mathscr{O}_{\XX^{\rig}})$ for a frame $(\XX_{k},\XX)$ of $\XX_{k}-Z$ (see \ref{overconvergent function}) (\cite{Ber96II} 4.3.2). 
	We omit $(^{\dagger}Z)$ if $Z$ is empty. We denote $\mathscr{D}_{\XX,\mathbb{Q}}^{\dagger}(^{\dagger}Z)$ by $\mathscr{D}_{\XX,\mathbb{Q}}^{\dagger}(Z)$ (or $\mathscr{D}_{\XX,\mathbb{Q}}^{\dagger}(\infty)$) for short. 

\begin{secnumber} \label{basic def}
	Let us begin by recalling basic notions of $p$-adic coefficients used in \cite{Abe18}. 
	Let $L$ be an extension of $K$ in $\overline{K}$ and $\mathfrak{T}=\{k,R,K,L\}$ the associated geometric base tuple (\cite{Abe18} 1.4.10, 2.4.14). 

	We will also work in the arithmetic setting ($p$-adic coefficients with Frobenius structure). For this purpose, we need to assume moreover that there exists an automorphism $L\to L$ extending $\sigma:K\to K$ that we still denote by $\sigma$, and that there exists a sequence of finite extensions $M_n$ of $K$ in $L$ satisfying $\sigma(M_n)\subset M_n$ and $\cup_{n} M_n=L$. 
	Then we obtain an arithmetic base tuple $\mathfrak{T}_F=\{k,R,K,L,s,\sigma\}$ (\cite{Abe18} 1.4.10, 2.4.14). 
	We set $L_0=L^{\sigma=1}$. 

	Let $X$ be a $k$-scheme. There exists an $L$-linear (resp. $L_0$-linear) triangulated category $\rD(X/L)$ (resp. $\rD(X/L_F)$) relative to the geometric base tuple $\mathfrak{T}$ (resp. arithmetic base tuple $\mathfrak{T}_F$). 
	This category is denoted by $\rD^{\rb}_{\hol}(X/\mathfrak{T})$ or $\rD^{\rb}_{\hol}(X/L)$ (resp. $\rD^{\rb}_{\hol}(X/\mathfrak{T}_F)$ or $\rD^{\rb}_{\hol}(X/L_F)$) in (\cite{Abe18} 1.1.1, 2.1.16). 
	When $L=K$ and $X$ is quasi-projective, there exists a classical description of $\rD(X/K)$ in terms of arithmetic $\mathscr{D}$-modules introduced by Berthelot \cite{Ber02}: 
	If $X\to \mathscr{P}$ is an immersion into a smooth proper formal $R$-scheme $\mathscr{P}$, then $\rD(X/K)$ is a full subcategory of $\rD_{\coh}^{\rb}(\mathscr{D}_{\mathscr{P},\mathbb{Q}}^{\dagger})$ with objects satisfy certain finiteness condition called \textit{overholonomicity}, certain support condition, and can be equipped with some Frobenius structure 
	\footnote{To define Frobenius structure on objects of $\rD^{\rb}_{\coh}(\mathscr{D}_{\mathscr{P},\mathbb{Q}}^{\dagger})$ and the category $\rD(X/L)$, we need to assume the existence of a pair $(s,\sigma)$ (as in \ref{basic notation}). 
	However, the category $\rD(X/L)$ is independent of the choice of data $(s,\sigma)$ up to equivalences (\cite{Abe18} 1.1.2).} (cf. \cite{Abe18} 1.1.1, \cite{AC17}).
	
	The category $\rD(X/L)$ (resp. $\rD(X/L_F)$) is equipped with a t-structure, called \textit{holonomic t-structure}, whose heart is denoted by $\Hol(X/L)$ (resp. $\Hol(X/L_F)$), called \textit{category of holonomic modules}. These categories are analogue to the category of perverse sheaves in the $\ell$-adic cohomology theory. 
	The category $\Hol(X/L)$ is Noetherian and Artinian (\cite{Abe18} 1.2.7). 
	We denote by $\hH^*$ the cohomological functor for holonomic t-structure. 

	When $X=\Spec(k)$, there exists an equivalence of monoidal categories between $\rD(X/L)$ and the derived category of bounded complexes of $L$-vector spaces with finite dimensional cohomology.  
\end{secnumber}

\begin{secnumber} \label{Basic properties}
	The six functor formalism for $\rD(X/L)$ (resp. $\rD(X/L_F)$) has been established recently. We refer to \cite{AC17,AC18} and (\cite{Abe18} 2.3) for details and to (\cite{Abe18} 1.1.3) for a summary. Here we only collect some results needed in the sequel.

	Let $f:X\to Y$ be a morphism of $k$-schemes. For $\blacktriangle\in \{\emptyset,F\}$, there exist triangulated functors
	\begin{equation} \label{pullback pushforward}
		f_+,f_!:\rD(X/L_{\blacktriangle})\to \rD(Y/L_{\blacktriangle}),\qquad f^+,f^!:\rD(Y/L_{\blacktriangle})\to \rD(X/L_{\blacktriangle}),
	\end{equation}
	such that $(f^+,f_+)$, $(f_!,f^!)$ are adjoint pairs. These functors satisfy following properties:

	(i) The category $\rD(X/L_{\blacktriangle})$ is a closed symmetric monoidal category, namely it is equipped with a tensor product functor $\otimes$ and the unit object $L_X=\pi^+(L)$, where $\pi:X\to \Spec(k)$ is the structure morphism and $L$ is the constant module in degree $0$. 
	The functor $\otimes$ admits a left adjoint functor $\FHom_X$, called the \textit{internal Hom}. 
	The functor $f^+$ is monoidal. 

	(ii) There exists a duality functor $\mathbb{D}_{X}=\FHom_{X}(-,p^{!}L):\rD(X/L_{\blacktriangle})^{\circ}\to \rD(X/L_{\blacktriangle})$ (\cite{Abe18} 1.1.4). The canonical morphism $\id\to \mathbb{D}_X\circ \mathbb{D}_X$ is an isomorphism. We set $(-)\widetilde{\otimes}(-)= \mathbb{D}_X ( \mathbb{D}_X(-)\otimes \mathbb{D}_X(-))$. 

	(iii) There exists a canonical morphism of functors $f_!\to f_+$, which is an isomorphism if $f$ is proper. 

	(iv) (Base change). Consider the following Cartesian diagram of $k$-schemes
	\begin{equation}
		\xymatrix{
			X'\ar[r]^{g'} \ar[d]_{f'} & X \ar[d]^{f} \\
			Y' \ar[r]^{g} & Y.
		}
		\label{base change diagram}
	\end{equation}
	Then we have a canonical isomorphism $g^+f_!\simeq f'_!g'^+$. When $f$ is proper, this isomorphism is the base change homomorphism defined by the adjointness of $(f^+,f_+)$. 

	(v) (Berthelot-Kashiwara's theorem). 
	Let $i$ be a closed immersion. Then $i_+$ is exact and fully faithful. The restriction of $i^!$ to the essential image of $i_+$ is exact and is a quasi-inverse to $i_+$ (\cite{AC18} 1.3.2(iii)). 

	(vi) Let $i$ be a closed immersion of $k$-schemes and $j$ the open immersion defined its complement. There exists a canonical isomorphism $j^+\xrightarrow{\sim} j^!$. 
	We have distinguished triangles (\cite{Abe18} 1.1.3(10), 2.2.9):
	\begin{displaymath}
		j_!j^+\to \id\to i_+i^+\to, \qquad i_+i^!\to \id\to j_+j^+\to ,
	\end{displaymath}
	where the first and second morphisms are defined by adjunctions. 

	(vii) (Poincar\'e duality) We refer to (\cite{Abe18} 1.4.13) for the definition of Tate twist functor $(-)$. 
	Let $f:X\to Y$ be a smooth morphism of relative dimension $d$. Then there exists a canonical isomorphism $\varphi:f^+(d)[2d]\xrightarrow{\sim} f^!$ (\cite{Abe18} 1.5.13). 
	Moreover, the functors $f^+[d],f^![-d]$ are exact.

	(viii) There exists a canonical equivalence of categories $\rD(X/L_{\blacktriangle})\simeq \rD^{\rb}(\Hol(X/L_{\blacktriangle}))$ (\cite{AC17}, \cite{Abe18} 2.2.26).

	(ix) Let $X_1,X_2$ be two $k$-schemes and $p_i:X_1\times X_2 \to X_i$ the projection for $i=1,2$. 
	There exists a canonical isomorphism of functors $p_1^+(-)\otimes p_2^+(-)\simeq p_1^!(-)\widetilde{\otimes}p_2^!(-)$ (ii), denoted by $-\boxtimes -$ and called \textit{external tensor product}. This functor is exact (\cite{AC18} 1.3.3). 
\end{secnumber}

\begin{rem} \label{base tuple}
	(i) If $L$ is a finite extension of $K$, an object of $\Hol(X/L)$ is defined as an object of $\Hol(X/K)$ equipped with an \textit{$L$-structure} (\cite{Abe18} 1.4.1). 
	In general, Abe used the 2-inductive limit method of Deligne to construct $\rD(X/L)$ (\cite{Abe18} 2.4.14). 
	If $L'$ is an algebraic extension of $L$ in $\overline{K}$, we have an \textit{extension of scalars functor} $\iota_{L'/L}:\rD(X/L)\to \rD(X/L')$, which is exact and commutes with cohomological functors. 
	
	(ii) The category $\rD(X/\mathfrak{T})$ does not depend on the choice of the base field $k$ under certain conditions. 	More precisely, if $\mathfrak{T}'=\{k',R',K',L\}$ is another geometric base tuple over $\mathfrak{T}$. Then there exists a canonical equivalence (\cite{Abe18} 1.4.11):
	\begin{displaymath}
		\rD(X\otimes_k k'/\mathfrak{T}') \xrightarrow{\sim} \rD(X/\mathfrak{T}),
	\end{displaymath}
	which commutes with cohomological functors.

	(iii) Let $\mathfrak{T}_{F}$ be an arithmetic base tuple. 
	The $s$-th Frobenius morphism $F_X:X\to X$ induces a $\sigma$-semi-linear equivalence of categories $F_X^{*}:\rD(X/L)\xrightarrow{\sim} \rD(X/L)$ commuting with cohomological functors, called \textit{($s$-th) Frobenius pullback} (\cite{Abe18} 1.1.3 lemma). 
	An object of $\Hol(X/L_F)$ is an object $\mathscr{E}$ of $\Hol(X/L)$ equipped with a\textit{(n $s$-th) Frobenius structure} $\varphi:F_{X}^*(\mathscr{E})\xrightarrow{\sim} \mathscr{E}$ (cf. \cite{Abe18} 1.4). 
\end{rem}

\begin{secnumber} \label{generality overconv isocrystal}
	Let $X$ be a smooth $k$-scheme of dimension $d:\pi_0(X)\to \mathbb{N}$. 
	There exists a full subcategory $\Sm(X/L_{\blacktriangle})$ of $\Hol(X/L_{\blacktriangle})[-d] \subset \rD(X/L)$ consisting of \textit{smooth objects} (\cite{Abe18} 1.1.3(12) and 2.4.15). 
	In general, we say a complex $\mathscr{M}\in \rD(X/L_{\blacktriangle})$ is \textit{smooth} if $\hH^i(\mathscr{M})[-d]$ belongs to $\Sm(X/L_{\blacktriangle})$ for every $i$. 
	
	When $L=K$, there exists an equivalence $\widetilde{\Sp}_{*}$ between $\Sm(X/K)$ (resp. $\Sm(X/K_{F})$) and $\Iso^{\dagger\dagger}(X/K)$ (resp. $\Fr\Iso^{\dagger}(X/K)$) \eqref{overconv F-iso}. 
	If $X$ admits a smooth compactification $\overline{X}$ with a smooth lifting $\overline{\XX}$ to $\Spf(R)$, this equivalence is induced by the specialisation morphism $\Sp_{*}:\overline{\XX}^{\rig}\to \overline{\XX}$:
	\begin{equation} 
		\widetilde{\Sp}_{*}= \Sp_{*}(-d)[-d]: \Iso^{\dagger\dagger}(X/K) ~(\textnormal{resp. } \Fr\Iso^{\dagger}(X/K) )\xrightarrow{\sim} \Sm(X/K_{\blacktriangle}) \subset \Dbhol(X/K_{\blacktriangle}).
		\label{mod specialisation}
	\end{equation}
	In the following, we identify these two categories by $\widetilde{\Sp}_{*}$ and we use alternatively these two notations. 

	Let $f:X\to Y$ be a morphism between smooth $k$-schemes. 
	Via $\widetilde{\Sp}_{*}$, we can identify the functor $f^{+}$ and the pullback functor of overconvergent ($F$-)isocrystals $f^{*}$ (\cite{Abe18} 2.4.15). If $d$ denotes $\dim(X)-\dim(Y)$, for any object $M$ of $\Sm(X/L_{\blacktriangle})$, there exists a canonical isomorphism:
	\begin{equation}
		f^{!}(M)\simeq f^+(M)(d)[2d].
		\label{pullback of smooth mods}
	\end{equation}
\end{secnumber}

\begin{secnumber} \label{constructible t-structure}
	Let $X$ be a $k$-scheme. There exists a \textit{constructible t-structure} (\textit{c-t-structure} in short) on $\rD(X/L)$ (cf. \cite{Abe18} 1.3, 2.2.23). 
	When $X=\Spec(k)$, the constructible t-structure coincides with the holonomic one \eqref{Basic properties}. 
	If $X$ a smooth $k$-scheme, any object of $\Sm(X/L)$ is constructible. 

	The heart of c-t-structure is denoted by $\Con(X)$, called the \textit{category of constructible modules}, and is analogue to the category of constructible sheaves in the $\ell$-adic theory. 
	The cohomology functor of c-t-structure is denoted by $\cH^*$. 
	
	Let $f:X\to Y$ be a morphism between $k$-schemes. The functor $f^+$ is c-t-exact and $f_+$ is left c-t-exact. 
	If $i$ is a closed immersion, then $i_+$ is c-t-exact. If $j$ is an open immersion, then $j_!$ is c-t-exact (\cite{Abe18} 1.3.4).

	A constructible module $\mathscr{M}$ on $X$ is zero if and only if $i_x^+\mathscr{M}=0$ for any closed point $i_x:x\to X$ (\cite{Abe18} 1.3.7). 
\end{secnumber} 

\begin{secnumber}	
In the end, we present a generalization of the specialization morphism \eqref{specialization morphism} in a relative situation using the direct image of arithmetic $\mathscr{D}$-modules. 
 
	Let $f:X=\Spec(B)\to S=\Spec(A)$ be a smooth morphism of affine smooth $R$-schemes of relative dimension $d$ and let $(M,\nabla)$ be a coherent $\mathscr{O}_{X_K}$-module endowed with an integrable connection relative to $K$. 
	Consider $M$ as a $\cD_{X_K}$-module. 
	The direct image $f_{+}^{\dR}(M)$ of $\cD$-modules is calculated by the relative de Rham complex $M\otimes \Omega_{X/S}^{\bullet}$. 
	Since $f$ is affine, the above complex is calculated by
	\begin{equation}
		\Gamma(S,f_{+}^{\dR}(M))\simeq \DR_{B/A}(M,\nabla)= M\to M\otimes_{B}\Omega_{B/A}^1\to \cdots,
	\end{equation}
	where we denote abusively by $M$ the global section $\Gamma(X_K,M)$. 
\end{secnumber}

\begin{secnumber} \label{rel rig coh}
	We assume moreover that $f$ admits a \textit{good compactification}, i.e. $f$ can be extended to a smooth morphism $\overline{f}:\overline{X}\to \overline{S}$ of smooth \textit{projective} $R$-schemes $\overline{X},\overline{S}$ such that $\overline{X}_k-X_k$, $\overline{S}_k-S_k$ are ample divisors. 
	We keep the notation of \ref{setting strict neighborhood} and assume that $M^{\dagger}=j^{\dagger}(M^{\an})$ is overconvergent as in \ref{dR rig coh}. 
	We denote abusively the $\mathscr{D}_{\overline{\XX},\mathbb{Q}}^{\dagger}(\infty)$-module $\Sp_{*}(M^{\dagger})$ \eqref{generality overconv isocrystal} by $M^{\dagger}$. 
	The direct image of $M^{\dagger}$ along $f_{k}:X_{k}\to S_{k}$ is calculated by a relative de Rham complex: 
	\begin{eqnarray} \label{complex pushforward rig}
		f_{k,+}(M^{\dagger})\xrightarrow{\sim} 
		\bR \overline{f}_{k,*}( \Sp_{*}(M^{\dagger}\otimes_{\mathscr{O}_{\mathcal{X}}}\Omega_{\mathcal{X}/\mathcal{S}}^{\bullet})).
	\end{eqnarray}
	The above complex is a complex of overholonomic (and hence coherent) $\mathscr{D}^{\dagger}_{\overline{\mathfrak{S}},\mathbb{Q}}(\infty)$-modules.

	We set $A^{\dagger}=\Gamma(\overline{\mathfrak{S}},\mathscr{O}_{\overline{\mathfrak{S}},\mathbb{Q}}(^{\dagger}\infty))$, $B^{\dagger}=\Gamma(\overline{\XX},\mathscr{O}_{\overline{\XX},\mathbb{Q}}(^{\dagger}\infty))$
	and $D^{\dagger}_{\overline{\mathfrak{S}}}(\infty)=\Gamma(\overline{\mathfrak{S}},\mathscr{D}^{\dagger}_{\overline{\mathfrak{S}},\mathbb{Q}}(\infty))$ \eqref{basic notation}. 
	By $\mathscr{D}^{\dagger}$-affinity (\cite{Huy98} 5.3.3), the complex \eqref{complex pushforward rig} is equivalent to a complex of coherent $D^{\dagger}_{\overline{\SS}}(\infty)$-modules:
	\begin{eqnarray*}
		\bR\Gamma(\overline{\mathfrak{S}}, f_{k,+}(M^{\dagger})) &\simeq& 
		\bR\Gamma(\overline{\XX},\Sp_{*}(M^{\dagger}\otimes_{\mathscr{O}_{\mathcal{X}}}\Omega_{\mathcal{X}/\mathcal{S}}^{\bullet})) \\
		&\simeq& (M\otimes_{B_K}B^{\dagger})\otimes_{B} \Omega^{\bullet}_{B/A}.
	\end{eqnarray*}
	We denote the complex in the second line by $\DR^{\dagger}_{B/A}(M^{\dagger})$. 
	Note that this complex is $A^{\dagger}$-linear.

	If we set $D_{S_K}=\Gamma(S_K,\mathscr{D}_{S_K})$, there exists a canonical $D_{S_K}$-linear morphism, called the \textit{(relative) specialisation morphism}
	\begin{equation} \label{rel sp map}
		\DR_{B/A}(M,\nabla)\to \DR^{\dagger}_{B/A}(M^{\dagger}). 
	\end{equation}
\end{secnumber}

\subsection{Complements on the cohomology of arithmetic $\mathscr{D}$-modules}
\begin{secnumber} \label{cohomology arith mods}
	Let $f:X\to \Spec(k)$ be a $k$-scheme and $\mathscr{F}$ an object of $\rD(X/L)$. 
	We set
	\begin{equation}
		\rH^*(X,\mathscr{F})=\hH^*f_+(\mathscr{F}),\qquad \rH^*_{\rc}(X,\mathscr{F})=\hH^*f_!(\mathscr{F}),
		\label{coh def}
	\end{equation}
	and call them \textit{cohomology groups of $\mathscr{F}$, compact support cohomology groups of $\mathscr{F}$}, respectively.
	Note that they are finite dimensional $L$-vector spaces. 
	If $\mathscr{F}$ is an object of $\rD(X/L_F)$, then above cohomology groups are equipped with a Frobenius structure.
	If there is no confusion, we simply write $\rH^{*}(X,L)$ for $\rH^{*}(X,L_X)$.
	We collect some properties that we will use in the following:

	(i) If $X$ has dimension $\le d$, then for any $\mathscr{M}\in \Con(X)$, the compact support cohomology groups $\rH^i_{\rc}(X,\mathscr{M})$ are concentrated in degrees $0\le i\le 2d$ (\cite{Abe18} 1.3.8). 
		
	(ii) Suppose $X$ admits a smooth compactification $\overline{X}$ such that $\overline{X}$ possesses a smooth lifting over $R$ and that $\overline{X}- X$ is a divisor.
	Given an object $M$ of $\Iso^{\dagger\dagger}(X/K)$ (resp. $\Fr\Iso^{\dagger}(X/K)$) \eqref{overconv F-iso}, we have canonical isomorphisms (\cite{Abe14} 5.9): 
	\begin{equation} \label{coh Dmods rig coh}
	\rH_{\rig}^{*}(X,M)\simeq \rH^{*}(X,\widetilde{\Sp}_{*}(M)),\qquad 
	\rH_{\rig,\rc}^{*}(X,M)\simeq \rH^{*}_{\rc}(X,\widetilde{\Sp}_{*}(M)),
	\end{equation}
as objects of $\Vect_K$ (resp. $\Fr\Vect_{K}$). 
	Via \eqref{coh Dmods rig coh}, the canonical morphism $\rH^{*}_{\rc}(X,\widetilde{\Sp}_{*}(M))\to \rH^{*}(X,\widetilde{\Sp}_{*}(M))$ induced by $f_{!}\to f_{+}$ is compatible with $\iota_{\rig}$ \eqref{forget map rig}.

	In particular, we have $\rH^{0}(\mathbb{A}^n,L)\simeq L$, $\rH^i(\mathbb{A}^n,L)=0$ for $i\neq 0$ and $\rH^{2n}(\mathbb{A}^n,L)\simeq L$, $\rH^i(\mathbb{A}^n,L)=0$ for $i\neq 2n$. 

	(iii) If $X$ is smooth over $k$, then the dimension of $\rH^0(X,L_X)$ is equal to the number of geometrically connected components of $X$. 
	The Frobenius acts on $\rH^0(X,L_X)$ as identity. 
\end{secnumber}

\begin{secnumber}
	Let $Y$ be a closed subscheme of $X$ and $\mathscr{F}$ an object of $\rD(X/L)$. In view of the distinguished triangle \ref{Basic properties}(iv), there exists a long exact sequence of cohomology groups: 
	\begin{equation} \label{long exact Hc}
		\cdots\xrightarrow{\partial} \rH^{i}_{\rc}(X-Y,\mathscr{F})\to \rH^{i}_{\rc}(X,\mathscr{F})\to \rH^{i}_{\rc}(Y,\mathscr{F})\xrightarrow{\partial}\cdots
	\end{equation}

	In general, suppose that there exists a finite filtration of closed subschemes $\{X_{i}\}_{i\in\mathbb{Z}}$ of $X$, with closed immersions $X_{i+1}\hookrightarrow X_{i}$ such that $X_{i}=X$ for $i$ small enough and $X_{i}=\emptyset$ for $i$ big enough. 
	Then we deduce a spectral sequence (cf. \cite{Del77} *2.5)
	\begin{equation}
		\rE_{1}^{ij} = \rH_{\rc}^{i+j}(X_i-X_{i+1},\mathscr{F})\Rightarrow \rH_{\rc}^{i+j}(X,\mathscr{F}). \label{ss stratification}
	\end{equation}
\end{secnumber}

\begin{coro} \label{top coh dim}
	Let $d$ be the dimension of $X$. 
	Then the dimension of the top degree compact support cohomology $\rH^{2d}_{\rc}(X,L_X)$ is equal to the number of geometrically irreducible components of $X$. The Frobenius on $\rH^{2d}_{\rc}(X,L_X)$ acts by multiplication by $q^d$. 
\end{coro}
\begin{proof}
	We denote by $X_{\sm}$ (resp. $X_{\sing}$) the smooth (resp. singular) locus of $X$. 
	Then the assertion follows from the long exact sequence \eqref{long exact Hc} for $(X_{\sm},X_{\sing},X)$, Poincar\'e duality and \ref{cohomology arith mods}(iii).
\end{proof}

We show an analogue of (\cite{BBDG} 4.2.5) for arithmetic $\mathscr{D}$-modules.

\begin{prop} \label{BBD 4.2.5}
	Let $f:X\to Y$ be a smooth morphism of $k$-scheme of relative dimension $d$ with geometrically connected fibers. 
	Then the functor $f^+[d]$ induces a fully faithful functor $\Hol(Y/L_{\blacktriangle})\to \Hol(X/L_{\blacktriangle})$ for $\blacktriangle \in \{\emptyset, F\}$. 
\end{prop}

\begin{lemma}
	Let $\mathscr{M}$ be an object of $\rD^{\le 0}(X/L)$ and $\mathscr{N}$ an object of $\rD^{\ge 0}(X/L)$. Then $\FHom_X(\mathscr{M},\mathscr{N})$ belongs to $^{\rc}\rD^{\ge 0}(X/L)$ \eqref{constructible t-structure}. \label{lemma Hom constructible degrees}
\end{lemma}
\begin{proof}
	We prove by induction on the dimension of $X$. The assertion is clear if $\dim X=0$. 
	To prove the assertion, we can reduce to the case where $\mathscr{M},\mathscr{N}\in \Hol(X/L)$. 
	Then there exists a dense smooth open subscheme $j:U\to X$ such that $\mathscr{M}|_U,\mathscr{N}|_U$ are smooth. Let $i:Z\to X$ be the complement of $U$ and consider the triangle
	\begin{displaymath}
		i_+i^!\FHom_X(\mathscr{M},\mathscr{N})\to \FHom_X(\mathscr{M},\mathscr{N}) \to j_+j^+\FHom_X(\mathscr{M},\mathscr{N})\to .
	\end{displaymath}
	Since $i^!\FHom_X(\mathscr{M},\mathscr{N})\simeq \FHom_X(i^+\mathscr{M},i^!\mathscr{N})$ (\cite{Abe18} 1.1.5), the first term belongs to $^{\rc}\rD^{\ge 0}(X/L)$ by induction hypotheses. 
	Note that $\FHom_U(\mathscr{M}|_U,\mathscr{N}|_U) \simeq \mathbb{D}_U(\mathscr{M}|_U\otimes \mathbb{D}_U(\mathscr{N}|_U))$ is a smooth module and of constructible degree $0$. Then $j_+j^+\FHom_X(\mathscr{M},\mathscr{N})$ belongs to $^{\rc}\rD^{\ge 0}(X/L)$ and the assertion follows.
\end{proof}

\begin{secnumber}
	\textit{Proof of proposition \ref{BBD 4.2.5}}. 
	Since Frobenius pullback induces an equivalence of categories, it suffices to show the assertion for $\Hol(-/L)$. 
	Let $\mathscr{M},\mathscr{N}$ be two objects of $\Hol(Y/L)$. 
	Since $f$ is smooth, we deduce from $f^!\FHom_Y(\mathscr{M},\mathscr{N})\simeq \FHom_{X}(f^+\mathscr{M},f^!\mathscr{N})$ (\cite{Abe18} 1.1.5) an isomorphism
	\begin{displaymath}
		f^+\FHom_Y(\mathscr{M},\mathscr{N})\xrightarrow{\sim} \FHom_{X}(f^+\mathscr{M},f^+\mathscr{N}).
	\end{displaymath}
	By applying $\cH^0f_+\cH^0(-)$ to the above isomorphism and lemma \ref{lemma Hom constructible degrees}, we have
	\begin{equation}
		\cH^0f_+ f^+ \bigl(\cH^0 (\FHom_Y(\mathscr{M},\mathscr{N}))\bigr)
		\xrightarrow{\sim} \cH^0f_+\cH^0(\FHom_{X}(f^+\mathscr{M}[d],f^+\mathscr{N}[d])).
		\label{iso pullback pushforward h0}
	\end{equation}

	We claim that for any constructible module $\mathscr{F}$ on $Y$, there is a canonical isomorphism
	\begin{equation}
		\mathscr{F}\xrightarrow{\sim} \cH^0f_+ f^+ \mathscr{F}.
		\label{iso adj constructible H0}
	\end{equation}
	Then, by \ref{lemma Hom constructible degrees}, the proposition follows by applying $\rH^0(Y,-)$ to the composition of \eqref{iso pullback pushforward h0} and \eqref{iso adj constructible H0}. 

	By smooth base change and \ref{constructible t-structure}, to prove \eqref{iso adj constructible H0}, we can reduce to the case where $Y$ is a point. 
	After extending the scalar $L$ and the base field $k$ \eqref{base tuple}, we may assume moreover that $Y=\Spec(k)$. 
	In this case, the isomorphism \eqref{iso adj constructible H0} follows from \ref{cohomology arith mods}(iii).\hfill $\qed$
\end{secnumber}

\subsection{Equivariant holonomic $\mathscr{D}$-modules}

In this subsection, we study the notion of \textit{equivariant holonomic $\mathscr{D}$-modules} over a $k$-scheme (or an ind-scheme). 
We write simply $\rD(X)$ (resp. $\Hol(X)$) for $\rD(X/L)$ or $\rD(X/L_F)$ (resp. $\Hol(X/L)$ or $\Hol(X/L_F)$). 

\begin{secnumber} \label{equivariant object}
	Let $X\to S$ be a morphism of $k$-schemes, $H$ a smooth affine group scheme over $S$ with geometrically connected fibers and $\act:H\times_S X\to X$ an action of $H$ on $X$. We denote by $\pr_2:H\times_S X\to X$ the projection. 
	We define the category $\Hol_{H}(X)$ of $H$-equivariant holonomic modules on $X$ as follow. An object of $\Hol_{H}(X)$ is a pair consisting of a holonomic module $\mathscr{M}$ on $X$ and an isomorphism $\theta:\act^{+}(\mathscr{M})\xrightarrow{\sim} \pr_2^{+}(\mathscr{M})$ in $\rD(H\times_S X)$, satisfying:

	\begin{itemize}
		\item[(i)] $e^{+}(\theta)=\id$, where $e:X\to H\times_S X$ is induced by the unit section of $H$; 

		\item[(ii)] a cocycle condition on $H\times_S H\times_S X$. 
	\end{itemize}
	A morphism between $(\mathscr{M}_1,\theta_1)$ and $(\mathscr{M}_2,\theta_2)$ is a morphism $\varphi:\mathscr{M}_1\to \mathscr{M}_2$ of $\Hol(X)$ such that 
\begin{equation}
	\pr_2^{+}(\varphi)\circ \theta_{1} \simeq \theta_2\circ \act^{+}(\varphi).
	\label{phi H equivariant}
\end{equation}
	It is clear that $\Hol_H(X)$ is an abelian subcategory of $\Hol(H)$. 
	
	Suppose that $[X/H]$ is represented by a separated scheme of finite type $\overline{X}$ over $S$. 
	By smooth descent of holonomic modules (\cite{Abe18} 2.1.13), the pullback functor along the canonical morphism $q:X\to \overline{X}$ induces an equivalence of categories:
	\begin{equation} \label{quotient descent}
		q^{+}[d_{H}]:\Hol(\overline{X})\xrightarrow{\sim} \Hol_H(X).
	\end{equation}
\end{secnumber}

\begin{lemma}
	The canonical functor $\Hol_H(X)\to \Hol(X)$ is fully faithful. 
	\label{lemma ff forget H}
\end{lemma}
\begin{proof}
	Given two objects $(\mathscr{M}_1,\theta_1)$, $(\mathscr{M}_2,\theta_2)$ of $\Hol_H(X)$ and a morphism $\varphi:\mathscr{M}_1\to \mathscr{M}_2$ of $\Hol(X)$, one need to show \eqref{phi H equivariant}
	\begin{displaymath}
		\theta_2\circ \act^{+}(\varphi)\circ (\theta_{1})^{-1} \simeq \pr_2^{+}(\varphi).
	\end{displaymath}
	By proposition \ref{BBD 4.2.5}, $\pr_2^+$ is fully faithful. To show the above isomorphism, it suffices to show $e^+(\theta_2^{-1}\circ \act^{+}(\varphi)\circ \theta_{1}) \simeq e^+(\pr_2^{+}(\varphi))$, which follows from \ref{equivariant object}(i).
\end{proof}

\begin{lemma} \label{lemma descent cat equivariant}
	Let $H_1\subset H$ be a closed normal subgroup scheme over $S$. Suppose that $H/H_1, H_1$ are smooth over $S$ and that the action of $H$ on $X$ factors through $H/H_1$. Then, the canonical functor
	\begin{equation}
		\Hol_{H/H_1}(X) \to \Hol_H(X)
	\end{equation}
	is an equivalence of categories.
\end{lemma}
\begin{proof}
	The essential surjectivity follows from smooth descent (\cite{Abe18} 2.1.13). The full faithfulness follows from \ref{lemma ff forget H}. 
\end{proof}

\begin{secnumber} \label{twisted prod}
	Keep the notation of \ref{equivariant object}. 
	Let $Y$ be a separated $S$-scheme of finite type and $\varpi:E\to Y$ an $H$-torsor over $S$ with trivial action of $H$ on $Y$. 
	We denote by $Y\widetilde{\times}_{S}X$ the quotient of $E\times_S X$ by $H$, where $H$ acts on $E\times_S X$ diagonally.  
	
	Let $\mathscr{M}$ be a holonomic module on $Y$ and $\mathscr{N}$ an $H$-equivariant holonomic module on $X$. 
	Assume that $\mathscr{M}\boxtimes_S \mathscr{N}$ is a holonomic module on $Y\times_S X$ (Note that it is true if the base $S=\Spec(k)$).
	Then $(\varpi^+\mathscr{M}[\dim H]) \boxtimes_S \mathscr{N}$ is holonomic on $E\times_S X$ and is $H$-equivariant by construction. 
	By \eqref{quotient descent}, it descents to a holonomic module on $Y\widetilde{\times}_S X$, denoted by $\mathscr{M}\widetilde{\boxtimes}_S \mathscr{N}$ and called the \textit{twisted external product} of $\mathscr{M}$ and $\mathscr{N}$. 
\end{secnumber}

\begin{secnumber}
	We say an fpqc sheaf $\mathcal{X}$ on the category of $k$-algebras is a \textit{(strict) ind-scheme over $k$} if there exists an isomorphism of fpqc-sheaves $\mathcal{X}\simeq \varinjlim_{i\in I} X_i$ for a filtered inductive system $(X_i)_{i\in I}$ of $k$-schemes, whose transition morphisms are closed immersion. 
	The inductive system $(X_i)_{i\in I}$ is called \textit{an ind-presentation of $\mathcal{X}$}. We have following properties:

	(i) If $Z$ is a $k$-scheme and $u:Z\to \mathcal{X}$ is a closed subfunctor, then there exists an index $i$ such that $u$ factors through $Z\to X_i$. 

	(ii) If $\mathcal{X}\simeq \varinjlim_{j\in J} X'_j$ is another ind-presentation, the for any $i$, there exists an index $j$ such that $X_i$ is a closed subscheme of $X_{j}'$ and vice versa.

	Given an ind-scheme $\mathcal{X}=\varinjlim_{i\in I} X_i$, we denote by $\mathcal{X}_{\red}=\varinjlim_{i\in I} X_{i,\red}$ the reduced ind-subscheme of $\mathcal{X}$.

	For a transition morphism $\varphi:X_i\to X_j$, the functor $\varphi_+:\rD(X_i)\to \rD(X_j)$ is exact and fully faithful. 
	We define a triangulated category $\rD(\mathcal{X})$ as the 2-inductive limit 
	\begin{displaymath}
		\rD(\mathcal{X})=\varinjlim_{i\in I} \rD(X_i).
	\end{displaymath}
	The definition is independent of the choice of a ind-presentation of $\mathcal{X}$. 
	Since $\varphi_+$ is exact, $\rD(\mathcal{X})$ is also equipped with a t-structure, whose heart is denoted by $\Hol(\mathcal{X})$. 
	Note that $\Hol(\mathcal{X})$ coincides with the full abelian subcategory $\varinjlim_{i\in I} \Hol(X_i)$ of $\rD(\mathcal{X})$. 
	
	Given a morphism $f=(f_i)_{i\in I}: \mathcal{X}=\varinjlim X_i\to S$ to a $k$-scheme $S$, the cohomology functors $f_{i,!}$'s and $f_{i,+}$'s allow us to define $f_!,f_+:\rD(\mathcal{X})\to \rD(S)$. 
\end{secnumber}

\begin{secnumber} \label{equiv hol indscheme}
	Let $\mathcal{X}=\varinjlim_{i\in I} X_i$ be an ind-scheme and $f:\mathcal{X}\to S$ a morphism to a $k$-scheme. 
	Let $(H_j)_{j\in J}$ be a projective system of smooth affine $S$-group schemes with geometrically connected fibers, whose transition morphisms are quotient. We set $H=\varprojlim_{j\in J} H_j$ its projective limit, which is an affine group scheme over $S$. 
	Assume that there exists an action of $H$ on $f:\mathcal{X}\to S$ such that it stabilizes each subfunctor $f|_{X_i}$ and that the $H$-action factors through a quotient $H_{j_i}$ on $X_i\to S$ for each $i\in I$. Then we define $\Hol_{H}(X_i)$ to be $\Hol_{H_{j_i}}(X_i)$. By lemma \ref{lemma descent cat equivariant}, the category $\Hol_{H}(X_i)$ is independent of the choice of $H_{j_i}$ up to canonical equivalences. 
	Therefore, for $i\le j$, we have a fully faithful functor $\Hol_H(X_i)\to \Hol_H(X_j)$. 
	We define the \textit{category $\Hol_H(\mathcal{X})$ of $H$-equivariant holonomic modules on $\mathcal{X}$} as the inductive limit:
	\begin{displaymath}
		\Hol_H(\mathcal{X})=\varinjlim_{i\in I} \Hol_{H}(X_i).
	\end{displaymath}

	Let $\mathcal{Y}=\varinjlim_{i\in I} Y_i$ an ind-scheme over $S$ and $\varpi: E \to \mathcal{Y}$ an $H$-torsor. 
	We can define an ind-scheme $\mathcal{Y}\widetilde{\times}_S \mathcal{X}$ as follows. 
	For $i,l\in I$, we denote by $E_{l,j_i}$ the $H_{j_i}$-torsor $E|_{Y_l}\times^H H_{j_i}\to Y_l$ and by $Y_l\widetilde{\times}_S X_i = E_{l,j_i}\times_S X_i /H_{j_i}$ the twisted product \eqref{twisted prod}. 
	For a surjection $H_{j'}\twoheadrightarrow H_{j_i}$, there exists a canonical isomorphism $E_{l,j'}\times_S X_i/H_{j'}\xrightarrow{\sim} E_{l,j_i}\times_S X_i/H_{j_i}$. 
	Then this allows us to represent the fpqc sheaf $\mathcal{Y}\widetilde{\times}_{S} \mathcal{X}$ as an inductive limit of $Y_l\widetilde{\times}_S X_i$. 

	Let $\mathscr{M}$ be an object of $\Hol(\mathcal{Y})$ supported in $Y_l$ and $\mathscr{N}$ an object of $\Hol_{H}(\mathcal{X})$ supported in $X_{i}$. 
	Assume that $\mathscr{M}\boxtimes_S \mathscr{N}$ is a holonomic module on $Y\times_S X$. 
	Then we can define an object $\mathscr{M}\widetilde{\boxtimes}_S \mathscr{N}$ in $\Hol(Y_l\widetilde{\times}_S X_i)$ \eqref{twisted prod} and then in $\Hol(\mathcal{Y}\widetilde{\times}_S \mathcal{X})$. 
	The construction is independent of the choice of $i,l\in I$. 
\end{secnumber}

\subsection{Intermediate extension and the weight theory}

\begin{secnumber} \label{intermediate extension}
	Let $u:Y\to X$ be a locally closed immersion. Then the functor $u_+$ (resp. $u_!$) is left exact (resp. right exact) (\cite{AC18} 1.3.13). 
	For $\blacktriangle \in \{\emptyset, F\}$ and  $\mathscr{E}\in \Hol(Y/L_{\blacktriangle})$, we consider the homomorphism $\theta^0_{u,\mathscr{E}}:\hH^0(u_!\mathscr{E})\to \hH^0(u_+\mathscr{E})$ and we define $u_{!+}(\mathscr{E})$ to be (\cite{AC18} 1.4.1)
	\begin{equation}
		u_{!+}(\mathscr{E})=\im(\theta^0_{u,\mathscr{E}}:\hH^0(u_!\mathscr{E})\to \hH^0(u_+\mathscr{E})).
		\label{IC def}
	\end{equation}
	This defines a functor $u_{!+}:\Hol(Y/L_{\blacktriangle})\to \Hol(X/L_{\blacktriangle})$, called the \textit{intermediate extension functor}. We recall the following results and refer to (\cite{AC18} \S 1.4) for general properties of this functor:

	(i) (\cite{AC18} 1.4.7) Suppose $\mathscr{E}$ is irreducible. Then, $u_{!+}(\mathscr{E})$ is the unique irreducible subobject of $\hH^0(u_+\mathscr{E})$ (resp. irreducible quotient of $\hH^0(u_!\mathscr{E})$) in $\Hol(X/L_{\blacktriangle})$.

	(ii) (\cite{AC18} 1.4.9)
	Let $\mathscr{F}$ be an irreducible object of $\Hol(X/L_{\blacktriangle})$. 
	Then there exists a locally closed immersion $u:Y\to X$ from a smooth $k$-scheme $Y$ and a smooth holonomic module $\mathscr{E}$ on $Y$ such that $\mathscr{F}\simeq u_{!+}(\mathscr{E})$. 
\end{secnumber}

\begin{coro}
	Let $j:U\to X$ be an open subscheme of $X$ and $i:Z\to X$ its complement.

	\textnormal{(i)} Given a holonomic module $\mathscr{E}$ on $U$, $j_{!+}(\mathscr{E})$ is the unique extension $\mathscr{F}$ of $\mathscr{E}$ to $\Hol(X/L_{\blacktriangle})$ such that $i^+\mathscr{F}\in \rD^{\le -1}(Z/L_{\blacktriangle})$ and that $i^!\mathscr{F}\in \rD^{\ge 1}(Z/L_{\blacktriangle})$. 

	\textnormal{(ii)} If $X$ is smooth and $\mathscr{F}$ is a smooth holonomic module on $X$, then $j_{!+}(\mathscr{F}|_{U})\simeq \mathscr{F}$. 
	\label{lemma uniqueness intermediate extension}
\end{coro}
\begin{proof}
	(i) Since $j_!,i^+$ are right exact (\cite{AC18} 1.3.2), $\hH^0i^+(\hH^0(j_!(\mathscr{E})))=0$. By applying $i^+$ to $0\to \Ker(\theta_{j,\mathscr{E}}^0)\to \hH^0(j_!(\mathscr{E}))\to j_{!+}(\mathscr{E})\to 0$, we obtain $i^+(j_{!+}(\mathscr{E}))\in \rD^{\le -1}(Z/L)$. We prove $i^!\mathscr{F}\in \rD^{\ge 1}(Z/L)$ in a dual way. 

	Conversely, given such an extension $\mathscr{F}$, we can prove that the adjunction morphism $\hH^0j_!(\mathscr{E})\to \mathscr{F}$ (resp. $\mathscr{F}\to \hH^0j_+(\mathscr{E})$) is surjective (resp. injective) by the Berthelot-Kashiwara theorem. The assertion follows.

	(ii) The intermediate extension is stable under composition (\cite{AC18} 1.4.5). Then we can reduce to the case where $Z$ is smooth over $k$. 
	In this case, assertion (ii) follows from (i) and \eqref{pullback of smooth mods}. 
\end{proof}

\begin{secnumber} \label{def weight}
	We briefly recall the theory of weights for holonomic $F$-complexes developed by Abe and Caro \cite{AC18}.

	In the rest of this subsection, we assume $k$ has $q = p^s$ elements and we consider the arithmetic base tuple $\mathfrak{T}_F=\{k,R,K,L,s,\sigma=\id\}$ \eqref{basic def}. 
	We fix an isomorphism $\iota:\overline{K}\simeq \mathbb{C}$. 
	We refer to (\cite{AC18} 2.2.2, \cite{Abe18} 2.2.30) for the notion of being \textit{$\iota$-mixed} (resp. \textit{$\iota$-mixed of weight $\le w$, $\iota$-mixed of weight $\ge w$, $\iota$-pure}) for $\mathscr{M}\in \rD(X/L_F)$. 

	The weight behaves like the one in the $\ell$-adic theory:
	
	(i) (\cite{AC18} 4.1.3) The six operations preserve weights. More precisely, given a morphism $f:X\to Y$ of $k$-schemes, $f_+,f^!$ send $\iota$-mixed $F$-complexes of weight $\ge w$ to those of weight $\ge w$, $f_!,f^+$ send $\iota$-mixed $F$-complexes of weight $\le w$ to those of weight $\le w$. The dual functor $\mathbb{D}_X$ exchanges $\iota$-mixed $F$-complexes of weight $\le w$ to $\ge w$ and $\otimes$ sends $\iota$-mixed $F$-complexes of weight $(\le w, \le w')$ to $\le w+w'$. 
	
	(ii) (\cite{AC18} 4.2.4) Intermediate extension functor of an immersion preserves pure $F$-complexes and weights. 
	
	Moreover, we have a decomposition theorem for pure holonomic $F$-module.
\end{secnumber}

\begin{theorem}[\cite{AC18} 4.3.1, 4.3.6] \label{decomposition thm}
	Let $X$ be a $k$-scheme. 
	
	\textnormal{(i)} An $\iota$-pure $F$-holonomic module $\mathscr{E}$ on $X$ is semisimple in the category $\Hol(X/L)$ (not in $\Hol(X/L_F)$). 	

	\textnormal{(ii)} An $\iota$-pure $F$-holonomic complex $\mathscr{F}$ is isomorphic, in $\rD(X/L)$ to $\oplus_{n\in \mathbb{Z}} \hH^n(\mathscr{F})[n]$.  
	\label{decomposition}
\end{theorem}

The original form of (\cite{AC18} 4.3.1, 4.3.6) states the decomposition in the category of overholonomic modules (resp. complexes) over $X$. 
We remark that the same argument shows the decomposition in the category $\Hol(X/L)$ (resp. $\rD(X/L)$).

\subsection{Nearby and vanishing cycles} \label{nearby vanishing cycle}

In a recent preprint \cite{Abe18II}, Abe formulated the nearby and vanishing cycle functors for holonomic arithmetic $\mathscr{D}$-modules, based on the unipotent nearby and vanishing cycle functors introduced by himself and Caro in \cite{AC17}. 
	We briefly recall these constructions in this subsection. 

	We write simply $\rD(X)$ for $\rD(X/L)$ or $\rD(X/L_F)$. 
	The construction are parallel in two cases. When $\rD(X)=\rD(X/L)$, the Tate twist $(n)$ denotes the identity functor. 

\begin{secnumber}\label{def Psi un}
	Let $f:X\to \mathbb{A}_k^1$ be a morphism of $k$-schemes. We denote by $j:U=f^{-1}(\mathbb{G}_m)\to X$ the open immersion and by $i:X_{0}=X-U\to X$ its complement. 
	Following Beilinson \cite{Bei}, Abe and Caro constructed the unipotent nearby and vanishing cycle functors (\cite{AC17}, \S 2)
	\begin{equation}
		\Psi_{f}^{\un}:\Hol(U)\to \Hol(X_{0}),\quad \Phi_{f}^{\un}:\Hol(X)\to \Hol(X_0).  \label{unipotent nearby}
	\end{equation}
	
	We briefly recall the definition of $\Psi_{f}^{\un}$. We denote $\mathscr{O}_{\widehat{\mathbb{P}}_R^1,\mathbb{Q}}(^{\dagger}\{0,\infty\})$ simply by $\mathcal{O}_{\Gm}$ (see \ref{basic notation}).
	For $n\ge 0$, we define a free $\mathcal{O}_{\Gm}$-module $\Log^n$ of rank $n$
	\begin{displaymath}
		\Log^n=\oplus_{i=0}^{n-1} \mathcal{O}_{\Gm}\cdot \log^{[i]},
	\end{displaymath}
	generated by the symbols $\log^{[i]}$. There exists a unique $\mathscr{D}_{\widehat{\mathbb{P}}^1_R,\mathbb{Q}}^{\dagger}(\{0,\infty\})$-module structure on $\Log^n$ defined for $i\ge 0$ and $g\in \mathcal{O}_{\Gm}$ by
	\begin{displaymath}
		\nabla_{\partial_t}(g\cdot \log^{[i]})= \partial_t(g)\cdot \log^{[i]} + \frac{g}{t} \cdot \log^{[i-1]},
	\end{displaymath}
	where $t$ is the local coordinate of $\Gm$ and $\log^{[j]}=0$ for $j<0$.
	There exists a canonical Frobenius structure on $\Log^n$ sending $\log^{[i]}$ to $q^i \log^{[i]}$. 
	This defines an overconvergent $F$-isocrystal on $\Gm$ and then a smooth object of $\Hol(\Gm/K_F)$.
	We still denote by $\Log^n$ the extension of scalars $\iota_{L/K}(\Log^n)$ in $\Hol(\Gm)$. 

	We set $\Log^n_{f}= f^+ \Log^n \in \Hol(U)$ and define for $\mathscr{F}\in \Hol(U)$ \footnote{We adopt the definition of \cite{Abe18II}, which is different from that of \cite{AC17} by a Tate twist.}:
	\begin{equation}
		\Psi_{f}^{\un}(\mathscr{F})= \varinjlim_{n\ge 0} \Ker(j_!(\mathscr{F}\otimes \Log^n_f)\to j_+(\mathscr{F}\otimes \Log^n_f)).
		\label{def unipotent nearby Abe}
	\end{equation}
	This limit is representable in $\Hol(X_0)$ by (\cite{AC17} lemma 2.4). 

	The functors $\Psi_f^{\un},\Phi_f^{\un}$ are exact (\cite{AC17} 2.7) and extend to triangulated categories. 
	There exists a distinguished triangle $i^+ [-1]\to \Psi_{f}^{\un}\to \Phi_{f}^{\un}\xrightarrow{+1}$. 
\end{secnumber}

\begin{prop}[\cite{AC17} 2.5] \label{nearby vanishing un dual prop}
	There exist canonical isomorphisms:
	\begin{equation}
		(\mathbb{D}_{X_0}\circ \Psi_{f}^{\un})(1)\simeq \Psi_{f}^{\un}\circ \mathbb{D}_{U}, \qquad \mathbb{D}_{X_0}\circ \Phi_{f}^{\un}\simeq \Phi_{f}^{\un}\circ \mathbb{D}_{X}.
		\label{nearby vanishing un dual}
	\end{equation}
\end{prop}

\begin{secnumber}
	To define the full nearby/vanishing cycle functors of a morphism over a henselian trait, one need to extend the definition of holonomic arithmetic $\mathscr{D}$-modules to a larger class of schemes, which are closed under henselization. 
	We denote by $\Pro(k)$ the full subcategory of Noetherian schemes over $k$ which can be representable by a projective limit of a projective system of $k$-schemes whose transition morphisms are affine and \'etale. 
	In the rest of this subsection, we will work with schemes in the category $\Pro(k)$. 
	
	Given a morphism of finite type $X\to S$, if $S$ is an object of $\Pro(k)$ then so is $X$.
	The category $\Pro(k)$ is closed under henselization (resp. strict henselization) (\cite{Abe18II} 1.3). 

	Let $X=\varprojlim_{i\in I} X_{i}$ be an object of $\Pro(k)$ with a representation by $k$-schemes $X_{i}$. 
	For each transition morphism $\varphi:X_i\to X_{j}$ (which is affine and \'etale), we have a canonical isomorphism $\varphi^{+}\simeq \varphi^{!}:\rD(X_{j})\to \rD(X_{i})$. We define a triangulated category $\rD(X)$ of arithmetic $\mathscr{D}$-modules on $X$ as an inductive limit $\rD(X):=\varinjlim_{i\in I^{\circ}}\rD(X_{i})$.
	Since $\varphi^{+}$ is exact, $\rD(X)$ is equipped with a t-structure whose heart is denoted by $\Hol(X)$. 
	Moreover, we can extend the definition of cohomological functors \eqref{pullback pushforward} to $\rD(X)$ (cf. \cite{Abe18II} 1.4). 
\end{secnumber}

\begin{secnumber} \label{nearby vanishing}
	Let $(S,s,\eta)$ be a strict henselian trait in $\Pro(k)$ and $f:X\to S$ a morphism of finite type. With above preparations, we can define the unipotent nearby and vanishing cycles functors for $f$ (cf. \cite{Abe18II} 1.7,1.8)
	\begin{equation} \label{unipotent Psi Phi trait}
		\Psi_{f}^{\un}, \Phi_{f}^{\un}:\Hol(X)\to \Hol(X_{s}).
	\end{equation}
	
	We denote by $\Hen(S)$ the category of henselian traits over $S$ which is generically \'etale over $S$. Given an object $h:S'\to S$ of $\Hen(S)$, we denote abusively by $h$ the canonical morphism $X_{S'}\to X$, by $h_s:X_{s'}\xrightarrow{\sim} X_{s}$ the isomorphism on the special fibers and by $f':X_{S'}\to S'$ the base change of $f$ by $h$.

	Using \eqref{unipotent Psi Phi trait}, the full nearby and vanishing cycle can be defined as (cf. \cite{Abe18II} 1.9): 
	\begin{equation}
		\Psi_{f}=\varinjlim_{(S',h)\in \Hen(S)} h_{s+}\circ \Psi_{f'}^{\un}\circ h^{+}, \quad
		\Phi_{f}=\varinjlim_{(S',h)\in \Hen(S)} h_{s+}\circ \Phi_{f'}^{\un}\circ h^{+}.
		\label{nearby}
	\end{equation}
	By \cite{Abe18II} 2.2, they are well-defined functors
	\begin{displaymath}
		\Psi_{f}, \Phi_{f}: \Hol(X)\to \Hol(X_{s}).
	\end{displaymath}
\end{secnumber}

\subsection{Universal local acyclicity}
\label{subsection ULA}
\begin{secnumber} \label{LA setting}	
	Following Braverman-Gaitsgory (\cite{BG}, 5.1), we propose a notion of \textit{(universal) local acyclicity} for arithmetic $\mathscr{D}$-modules with respect to a morphism to a smooth target. 
	
	For a smooth $k$-scheme $X$, we denote by $d_{X}:\pi_{0}(X)\to \mathbb{N}$ the dimension of $X$. 
	Let $g:X_1\to X_2$ be a morphism of $k$-schemes and $\mathscr{F},\mathscr{F}'$ two objects of $\rD(X_2)$. We consider the composition 
	\begin{displaymath}
		g_{!}(g^{+}(\mathscr{F})\otimes g^{!}(\mathscr{F}'))\simeq \mathscr{F}\otimes g_{!}(g^{!}(\mathscr{F}'))\to \mathscr{F}\otimes\mathscr{F}'
	\end{displaymath}
	and its adjunction: 
	\begin{equation}
		g^{+}(\mathscr{F})\otimes g^{!}(\mathscr{F}')\to g^{!}(\mathscr{F}\otimes\mathscr{F}').
		\label{adjunction map}
	\end{equation}

	Now let $S$ be a smooth $k$-scheme and $f:X\to S$ a morphism of $k$-schemes. We set $X_{1}=X$, $X_{2}=X\times S$, $\mathscr{F}'=L_{X_2}$ and take $g$ to be the graph of $f$. 
	By Poincar\'e duality, we have $L_{X_{1}}(-d_S)[-2d_S]\xrightarrow{\sim} g^{!}(L_{X_{2}})$. Then, we obtain a canonical morphism
	\begin{equation}
		g^{+}(\mathscr{F}) \to g^{!}(\mathscr{F}) (d_S)[2d_S].
	\end{equation}
	By taking $\mathscr{F}$ to be $\mathscr{M}\boxtimes \mathscr{N}$, we obtain a canonical morphism (\ref{Basic properties}(ix))
	\begin{equation} \label{LA map}
		\mathscr{M}\otimes f^{+}(\mathscr{N}) \to (\mathscr{M}\widetilde{\otimes} f^{!}(\mathscr{N})) (d_S)[2d_S].
	\end{equation}
\end{secnumber}

\begin{definition} \label{def LA}
	Let $S$ be a smooth $k$-scheme and $f:X\to S$ a morphism of $k$-schemes.
	We say that an object $\mathscr{M}$ of $\rD(X)$ is \textit{locally acyclic (LA) with respect to $f$}, if the morphism \eqref{LA map} is an isomorphism for any object $\mathscr{N}$ of $\rD(S)$. 
	We say that $\mathscr{M}$ is \textit{universally locally acyclic (ULA) with respect to $f$}, if for any morphism of smooth $k$-schemes $S'\to S$, the $+$-inverse image of $\mathscr{M}$ to $X\times_S S'$ is locally acyclic with respect to $X\times_S S'\to S'$. 
\end{definition}

\begin{prop} \label{basic ULA}
	Keep the notation of \ref{def LA} and let $\mathscr{M}$ be an object of $\rD(X)$. 

	\textnormal{(i)} Any object $\mathscr{M}$ of $\rD(X)$ is ULA with respect to the structure morphism $X\to \Spec(k)$. 

	\textnormal{(ii)} Let $g:Y\to X$ be a smooth (resp. smooth surjective) morphism. Then $g^+(\mathscr{M})$ on $Y$ is LA with respect to $f\circ g$ if (resp. if and only if) $\mathscr{M}$ is LA with respect to $f$. 
	
	\textnormal{(iii)} If $g:S\to S'$ is a smooth morphism of smooth $k$-schemes and $\mathscr{M}$ is LA with respect to a morphism $f:X\to S$, then $\mathscr{M}$ is LA with respect to $g\circ f$. 
	
	\textnormal{(iv)} Let $h:Y\to S$ be a morphism of finite type and $g:X\to Y$ a proper $S$-morphism (resp. a closed immersion). Then $g_+(\mathscr{M})$ is LA with respect to $h$ if (resp. if and only if) $\mathscr{M}$ is LA with respect to $f$. 

	\textnormal{(v)} If $\mathscr{M}$ is LA with respect to $f$, then so is its dual $\mathbb{D}_X(\mathscr{M})$. 
\end{prop}
\begin{proof}
	(i) Let $S$ be a smooth $k$-scheme and $\mathscr{N}$ an object of $\rD(S)$. We need to show that the canonical morphism
	\begin{displaymath}
		(\id_X\times \Delta)^{+}(\mathscr{M}\boxtimes p_2^+(\mathscr{N}))\to (\id_X\times \Delta)^!(\mathscr{M}\boxtimes p_2^+(\mathscr{N}))(d_S)[2d_S]
	\end{displaymath}
	is an isomorphism, where $\Delta:S\to S\times S$ is the diagonal map and $p_2:S\times S\to S$ is the projection in the second component. 
	Then we reduce to show that the canonical morphism 
	\begin{displaymath}
		\mathscr{N}\to \Delta^{!}(p_2^{+}(\mathscr{N}))(2d_S)[2d_S]
	\end{displaymath}
	is an isomorphism. After taking dual functor, the assertion follows from (\cite{Abe18} 1.5.14). 

	Assertions (ii) and (iii) follow from (\ref{Basic properties}(vii)) and the smooth descent for $\rD(X)$ (\cite{Abe18} 2.1.13). 
	
	Assertion (iv) follows from the projection formula (\cite{Abe18} 1.1.3(9)) and the Berthelot-Kashiwara theorem. 

	If we apply the dual functor $\mathbb{D}_{X}$ to the morphism \eqref{LA map}, then we obtain the morphism \eqref{LA map}
	\begin{displaymath}
		\mathbb{D}_X(\mathscr{M})\otimes f^{+}(\mathbb{D}_S(\mathscr{N})) \to (\mathbb{D}_X(\mathscr{M})\widetilde{\otimes} f^{!}(\mathbb{D}_S(\mathscr{N}))) (d_S)[2d_S],		
	\end{displaymath}
	for the pair $(\mathbb{D}_X(\mathscr{M}),\mathbb{D}_S(\mathscr{N}))$. Then assertion (v) follows. 
\end{proof}

\begin{rem} \label{ULA indscheme}
	Let $S$ be a smooth $k$-scheme and $f:\mathcal{X}\to S$ a morphism from an ind-scheme to $S$. 
	In view of proposition \ref{basic ULA}(iv), we can define the notion of \textit{LA (resp. ULA) with respect to $f$} for objects of $\rD(\mathcal{X})$. 
\end{rem}

\begin{prop} \label{ULA hol}
	Keep the notation of \ref{def LA} and let $D$ be a smooth effective divisor in $S$, $i:Z=f^{-1}(D)\to X$ the closed immersion and $j:U\to X$ its complement. 
	Let $\mathscr{M}$ be an object of $\rD(X)$ such that it is LA with respect to $f$ and that $\mathscr{M}|_U$ is holonomic. 

	\textnormal{(i)} There exists canonical isomorphisms:
	\begin{equation} \label{intermediate extension ULA}
		\mathscr{M}\simeq j_{!+}(\mathscr{M}|_U),\qquad 
		i^{+}\mathscr{M}[-1]\xrightarrow{\sim} i^{!}\mathscr{M}(1)[1].
	\end{equation}
	In particular, $\mathscr{M}$ and $i^{+}\mathscr{M}[-1]$ are holonomic. 
	
	\textnormal{(ii)} The holonomic module $i^{+}\mathscr{M}[-1]$ is LA with respect to $f\circ i$ and $f|_Z:Z\to D$. 
\end{prop}
\begin{proof}
	(i) By \'etale descent for holonomic modules (\cite{Abe18} 2.1.13), we may assume that there is a smooth morphism $g:S\to \mathbb{A}^1$ such that $D=g^{-1}(0)$. 
	By proposition \ref{basic ULA}(iii), $\mathscr{M}$ is LA with respect to $g\circ f:X\to \mathbb{A}^1$. 
	Then we can reduce to the case $f:X\to \mathbb{A}^1$ and $Z=f^{-1}(0)$. 
	
	We will show that $\Phi_{f}^{\un}(\mathscr{M})=0$, i.e. the canonical morphism 
	\begin{equation} \label{ULA implies vanish}
		i^{+}\mathscr{M}[-1]\to \Psi_{f}^{\un}(\mathscr{M})
	\end{equation}
	is an isomorphism. 

	We denote by $\overline{j}:\Gm\to \mathbb{A}^1$ be the canonical morphism and abusively by $f$ the restriction $f|_U:U\to \mathbb{G}_m$. 
	By the projection formula, we have
	\begin{displaymath}
		j_{!}(\mathscr{M}|_U\otimes f^{+}\Log^n)\xrightarrow{\sim} \mathscr{M}\otimes j_{!}f^+ \Log^{n}\simeq \mathscr{M}\otimes f^+ \overline{j}_!\Log^n.
	\end{displaymath}
	On the other hand, by the projection formula and the LA property of $\mathscr{M}$, we have 
	\begin{eqnarray*}
		j_+(\mathscr{M}|_U\otimes (f^+ \Log^n)) &\xrightarrow{\sim}& j_+(\mathscr{M}|_U \widetilde{\otimes} (f^! \Log^n))(d_X)[2d_X] \\
		&\xrightarrow{\sim} & \mathscr{M}\widetilde{\otimes} (j_+f^! \Log^n) (d_X)[2d_X] \\\
		&\simeq & \mathscr{M}\widetilde{\otimes} (f^!\overline{j}_+ \Log^n) (d_X)[2d_X] \\
		&\simeq & \mathscr{M}\otimes (f^+\overline{j}_+ \Log^n).
	\end{eqnarray*}
	
	Via the above isomorphisms, the canonical morphism $j_!(\mathscr{M}|_U\otimes (f^+ \Log^n))\to j_+(\mathscr{M}|_U\otimes (f^+ \Log^n))$ coincides with the canonical morphism 
	\begin{displaymath}
		\mathscr{M}\otimes (f^+(\overline{j}_! \Log^n \to \overline{j}_+\Log^n)).
	\end{displaymath}

	To prove that \eqref{ULA implies vanish} is an isomorphism, we can reduce to the case where $f$ is the identity map of $\mathbb{A}^1$ and $\mathscr{M}$ is the constant module $L_{\mathbb{A}^1}[1]$ on $\mathbb{A}^1$. 
	If we denote by $N_{n}$ the action induced by $t\partial_t$ on the fiber $(\Log^{n})_0$ of $\Log^n$ at $0$ (called \textit{residue morphism} in \cite{AC18} 3.2.11), then $\Ker(j_!(\Log_n)\to j_+(\Log_n))$ is isomorphic to $\Ker(N_n)$ (cf. \cite{AC17} proof of lemma 2.4). 
	In this case, \eqref{ULA implies vanish} is an isomorphism.	 
	
	In particular $i^{+}\mathscr{M}[-1]$ is holonomic.
	By propositions \ref{nearby vanishing un dual prop} and \ref{basic ULA}(v), the second isomorphism of \eqref{intermediate extension ULA} follows from \eqref{ULA implies vanish}: 
	\begin{equation}
		\Psi_f^{\un}(\mathscr{M})\simeq \mathbb{D}_{X_0} \Psi_{f}^{\un}(\mathbb{D}_{X}(\mathscr{M}))(1) \xrightarrow{\sim} i^!\mathscr{M}(1)[1].
		\label{ULA implies vanish dual}
	\end{equation}
	Then we deduce $\mathscr{M}\simeq j_{!+}(\mathscr{M}|_U)$ by (\ref{lemma uniqueness intermediate extension}(i)).	

	Assertion (ii) follows from the six functor formalism. We left the proof to readers. 
\end{proof}

\begin{coro} \label{ULA vanishing cycle}
	If an object $\mathscr{M}$ of $\rD(X)$ is ULA with respect to $f$, then, for any strict henselian trait $T$ and any morphism $g:T\to S$, we have $\Phi_{f_T}^{\un}(\mathscr{M}|_{X_{T}})=0$ and $\Phi_{f_T}(\mathscr{M}|_{X_{T}})=0$, where $f_T:X_T\to T$ is the base change of $f$ by $g$. 
\end{coro}
\begin{proof}
	By definition, it suffices to show that the unipotent vanishing cycle $\Phi^{\un}_{f_T}(\mathscr{M}|_{X_T})$ vanishes. 
	
	By (\cite{EGAIV} 8.8.2), there exists a smooth $k$-scheme $S'$, a smooth effective divisor $D$ of $S'$ with generic point $\eta_D$ and a morphism $h:S'\to S$ such that the strict henselization of $S'$ at $\eta_D$ is isomorphic to $T$ and that $g$ is induced by $h$. 	
	We denote by $f_{S'}:X_{S'}\to S'$ the base change of $f$ by $h$. After shrinking $S'$, we may assume that there exists a smooth morphism $\pi:S'\to \mathbb{A}_k^1$ with $D=\pi^{-1}(0)$. 

	By definition (cf. \cite{Abe18II} 1.7-1.8), we reduce to show that $\Phi_{\pi\circ f_{S'}}^{\un}(\mathscr{M}|_{X_{S'}})=0$. 
	But this follows from proposition \ref{basic ULA}(iii) and the proof of \eqref{ULA implies vanish}. 
	Then the assertion follows. 
\end{proof}

\begin{coro} \label{ULA id local system}
	Let $X$ be a smooth $k$-scheme.
	If an object $\mathscr{M}$ of $\rD(X)$ is ULA with respect to the identity morphism, then each constructible cohomology module $\cH^{i}(\mathscr{M})$ is smooth (resp. each cohomology module $\hH^i(\mathscr{M})$ is smooth).
\end{coro}

\begin{proof}
	When $\mathscr{M}$ is constructible, it follows from \ref{ULA vanishing cycle} and (\cite{Abe18II} 3.8).
	We prove the general case by induction on the cohomological amplitude of $\mathscr{M}$.
\end{proof}
         
\begin{secnumber} \label{hol mods stack}
	In \ref{construction HNY}, we will use the notion of holonomic modules over a stack and apply cohomological functors of a \textit{schematic} morphism of algebraic stacks, that we briefly explain in the following. 
	Let $\XX$ be an algebraic stack of finite type over $k$. 	
	We refer to (\cite{Abe18} 2.1.16) for the definition of category $\Hol(\XX)$ of holonomic modules on $\XX$ and the category $\rD(\XX)$ (corresponds to the category $\rD^{\rb}_{\hol}(\XX)$ in \textit{loc. cit}). 
	The dual functor $\mathbb{D}_{\XX}$ is defined in (\cite{Abe18} 2.2). 
	Let $f:\XX\to \YY$ be a schematic morphism, $Y_{\bullet}\to \YY$ a simplicial algebraic space presentation. 
	By pullback, we obtain a simplicial presentation $X_{\bullet}\to \XX$ and a Cartesian morphisms $f_{\bullet}:X_{\bullet}\to Y_{\bullet}$. 
	Then the constructions of (\cite{Abe18} 2.1.10 and 2.2.14) allow us to define cohomological functors:
	\begin{displaymath}
		f_{+}:\rD(\XX)\simeq \rD_{\hol}^{\rb}(X_{\bullet})\rightleftarrows \rD_{\hol}^{\rb}(Y_{\bullet})\simeq \rD(\YY): f^{!}.
	\end{displaymath}

	Given a object $\mathscr{M}$ of $\rD(\XX)$ and a morphism $g:\XX\to S$ to a smooth $k$-scheme $S$, we say $\mathscr{M}$ is ULA with respect to $g$ if its $+$-pullback to a presentation $U\to \XX$ is ULA with respect to $U\to S$. 

	Suppose $S$ is moreover a curve. Let $s$ be a closed point of $S$ and $S_{(s)}$ the strict henselian at $s$. 
	Since nearby/vanishing cycle functors commute with smooth pullbacks, we can extend the definition of nearby/vanishing cycle functors for $g\times_{S}S_{(s)}$.   
\end{secnumber}

\subsection{Local monodromy of an overconvergent $F$-isocrystal}

\begin{secnumber} \label{review MC Swan}
	We briefly recall the local monodromy group of $p$-adic differential equations over the Robba ring following \cite{And02II,Mat02}. 
	
	We denote by $\mathcal{R}_K$ the Robba ring over $K$, by $\MC(\mathcal{R}_K/K)$ (resp. $\MC(\mathcal{R}/\overline{K})$) the category of $\nabla$-modules of finitely presented over $\mathcal{R}_K$ (resp. over $\mathcal{R}=\mathcal{R}\otimes_{K}\overline{K}$). 
	Each object of $\MC(\mathcal{R}/\overline{K})$ comes from the extension of scalar of an object of $\MC(\mathcal{R}_L/L)$ for some finite extension $L$ of $K$. 
	We denote by $\MC^{\uni}(\mathcal{R}/\overline{K})$ the full Tannakian subcategory of $\MC(\mathcal{R}/\overline{K})$ consisting of unipotent objects, i.e. objects which are isomorphic to successive extension of the trivial object (cf. \cite{Mat02} \S~4). 

	There is an equivalence between the category $\Vect^{\nil}_{\overline{K}}$ of finite dimensional $\overline{K}$-vector space with a nilpotent endomorphism and $\MC^{\uni}(\mathcal{R}/\overline{K})$, given by the functor $(V_0,N)\mapsto (V_0\otimes_{\overline{K}} \mathcal{R},\nabla_N)$, where the connection $\nabla_{N}$ is defined by $\nabla_N(v\otimes 1)=Nv \otimes dx/x$ (\cite{Mat02} 4.1).
	In particular, the Tannakian group of $\MC^{\uni}(\mathcal{R}/\overline{K})$ over $\overline{K}$ is isomorphic to $\Ga$. 

	We denote by $\MCF(\mathcal{R}/\overline{K})$ the full subcategory of $\MC(\mathcal{R}/\overline{K})$ consisting of objects admitting a Frobenius structure (\cite{And02II} 3.4). 
	The category $\MCF(\mathcal{R}/\overline{K})$ is a Tannakian category over $\overline{K}$, whose Tannakian group is denoted by $\mathcal{G}$.	
	Christol and Mebkhout introduce the notion of \textit{$p$-adic slope} for objects of $\MCF(\mathcal{R}/\overline{K})$ and show a Hasse-Arf type result \cite{CM00}. 
	This allows one to define a Hasse-Arf type filtration on $\MCF(\mathcal{R}/\overline{K})$ and then a decreasing filtration of closed normal subgroups $\{\mathcal{G}^{>\lambda}\}_{\lambda\ge 0}$ of $\mathcal{G}$ (cf. \cite{And02II} \S~1, 3.4). 
	If we denote by $I$ (resp. $P$) the inertia (resp. wild inertia) subgroup of $\Gal(k ( (t))^{\sep}/k( (t)))$, regarded as pro-algebraic groups, then there exist canonical isomorphisms of affine $\overline{K}$-groups (\cite{And02II} 7.1.1)
	\begin{equation}
		\mathcal{G}\simeq I\times \Ga,\quad \mathcal{G}^{>0}\simeq P.
		\label{local monodromy Tannakian}
	\end{equation}

	The local monodromy theorem says that any object of $\MCF(\mathcal{R}/\overline{K})$ is quasi-unipotent \cite{And02II,Ked04,Meb02}. 
	Given an object $M$ of $\MCF(\mathcal{R}/\overline{K})$, the action of $P$ on $M$ factors through a finite quotient. 
	By a theorem of Matsuda-Tsuzuki \cite{Tsu98,Mat02} (cf. \cite{And02II} 7.1.2), the irregularity of $M$, defined by $p$-adic slopes, is equal to the Swan conductor of the representation of $I$ on a fiber of $M$. 
\end{secnumber}

\begin{secnumber} \label{analytic solution}
	We denote by $K\{x\}$ the $K$-algebra of analytic functions on the open unit disc $|x|<1$, i.e.
	\begin{equation} \label{functions on open disc}
		K\{x\}=\{\sum_{n\ge 0} a_n x^n\in K\llbracket x\rrbracket;~ |a_n|_p\rho^n\to 0~ (n\to \infty)~ \forall \rho\in [0,1)\}.
	\end{equation}
	Let $\Omega_{K\{x\}}^1(\log)$ be the free $K\{x\}$-module of rank $1$ with basis $dx/x$ and consider the following canonical derivation $d:K\{x\}\to \Omega^1_{K\{x\}}(\log),~ f\mapsto x f'(x) dx/x$.
	An unipotent object $(M,\nabla)$ of $\MC(\mathcal{R}_K/K)$ extends to a \textit{log $\nabla$-module} over $K\{x\}$. 
	Let $(V,N)$ be the object of $\Vect^{\nil}_K$ associated to $(M,\nabla)$. 
	In view of \ref{review MC Swan}, there exists a canonical isomorphism between $\Coker(N)$ and the solution space $\Sol(M)$ of $(M,\nabla)$:
	\begin{equation}
		\Coker(N)\xrightarrow{\sim} \Sol(M)=\Hom_{K\{x\}}((M,\nabla),(K\{x\},d))^{\nabla=0}.
		\label{solution space}
	\end{equation}
	When the connection $\nabla$ is defined by a differential operator $D$, then $\Sol(M)$ is isomorphic to the solution space of $D$. 
\end{secnumber}

\begin{secnumber} \label{log isocrystal}
	Let $X$ be a smooth curve over $k$, $\overline{X}$ a smooth compactification of $X$ and $x$ a $k$-point in the boundary $\overline{X}\setminus X$.
	There exists a canonical functor defined by restriction at $x$:
	\begin{equation}\label{res to x}
		|_{x}:\Iso^{\dagger}(X/K)\to \MC(\mathcal{R}_K/K).
	\end{equation}

	We refer to \cite{Shiho} and (\cite{Ked07} \S~6) for the definition of log convergent ($F$-)isocrystals on $Y=X\cup\{x\}$ with a log pole at $x$. 
	Let $\mathscr{E}$ be an object of $\Iso^{\dagger\dagger}(X/K)$ (resp. $\Fr\Iso^{\dagger}(X/K)$). 
	A log-extenbility criterion of Kedlaya (\cite{Ked07} 6.3.2) says that if $\mathscr{E}|_{x}$ is unipotent, then $\mathscr{E}$ extends to a \textit{log convergent isocrystal (resp. $F$-isocrystal)} $\mathscr{E}^{\log}$ on $Y$ with a log pole at $x$. 

	The fiber $\mathscr{E}^{\log}_{x}$ of $\mathscr{E}^{\log}$ at $x$ is a $K$-vector space equipped with a nilpotent operator.  
	If $\mathscr{E}$ moreover has a Frobenius structure, then $\mathscr{E}^{\log}_{x}$ is a \textit{$(\varphi,N)$-module}, that is a $K$-vector space $V$ equipped with a nilpotent operator $N:V\to V$ and a $\sigma$-semilinear automorphism $\varphi:V\to V$ such that $\varphi^{-1}N\varphi=qN$. 
	We can describe it in terms of the nearby cycle of $\mathscr{E}$ around $x$. 
\end{secnumber}

\begin{prop} \label{nearby cycle log extension}
	Let $X=\Gm,\overline{X}=\P1$ and $x=0$.  
	Suppose that $\mathscr{E}$ is unipotent at $0$ and let $\Psi$ be the nearby cycle functor defined by $\id$ \eqref{nearby vanishing}. Then there exists a canonical isomorphism of $K$-vector spaces (resp. $K$-vector spaces with Frobenius structure):
	\begin{equation} \label{compare nearby log-extension}
		\mathscr{E}^{\log}_{0}\xrightarrow{\sim} \Psi(\mathscr{E}). 
	\end{equation}
\end{prop}

\begin{proof}
	The argument of (\cite{AC17} 2.4(1)) implies the assertion, that we briefly explain in the following. 
	Since $\mathscr{E}$ is unipotent at $x$, we have $\Psi^{\un}(\mathscr{E})\simeq \Psi(\mathscr{E})$. 
	By (\cite{AC18} 3.4.19, cf. \cite{AC17} 2.4(1)), we have a Frobenius equivariant isomorphism of $K$-vector spaces: 
	\begin{displaymath}
		\Psi^{\un}(\mathscr{E})\simeq \varinjlim_{n\ge 0} \Ker(N^{n}: (\mathscr{E}^{\log}\otimes \Log^n)_0 \to  (\mathscr{E}^{\log}\otimes \Log^n)_0),
	\end{displaymath}
	where $\Log^{n}$ is the log convergent $F$-isocrystal on $(\A1,0)$ defeind in \ref{def Psi un} and $N^n=N_{\mathscr{E}^{\log}_0}\otimes\id +\id \otimes N_{\Log^n_0}$ is the tensor product of two nilpotent operators. 
	Then the isomorphism \eqref{compare nearby log-extension} follows. 
\end{proof}

\subsection{Hyperbolic localization for arithmetic $\mathscr{D}$-modules} \label{Braden thm sec}

\begin{secnumber} \label{basic notation Braden}
	Let $X$ be a quasi-projective $k$-scheme such that $X\otimes_k \overline{k}$ is connected and normal.
	We suppose that there exists an action $\mu:\Gm\times X\to X$ of the torus $\Gm$ over $k$. 
	Following \cite{DG}, we denote by $X^0$ the closed subscheme of fixed points of $X$ (\cite{DG} 1.3), by $X^+$ (resp. $X^-$) the attractor (resp. repeller) of $X$ (\cite{DG} 1.4, 1.8). 
	We have a commutative diagram 
	\begin{equation} \label{diagram hyperbolic}
		\xymatrix{
			& X^+ \ar@/_/[ld]_{\pi} \ar[rd]^{g} & \\
			X^0 \ar[ru]_{f} \ar[rd]^{f'} & & X \\
			& X^- \ar@/^/[lu]^{\pi'} \ar[ru]_{g'} & 	
		}
	\end{equation}
	where $f,f'$ are closed immersions and are sections of $\pi,\pi'$, respectively, the restriction of  $g$ (resp. $g'$) to each connected component of $X^+$ (resp. $X^-$) is a locally closed immersion (\cite{DG} 1.6.8).
	Note that each connected component of $X^+$ is the preimage of a connected component of $X^0$ under $\pi$.

	We define \textit{hyperbolic localization} functors $(-)^{!+},(-)^{+!}:\rD(X)\to \rD(X^0)$, for $\mathscr{F}\in \rD(X)$ by:
	\begin{equation}
		\mathscr{F}^{!+}=f^!(g^+(\mathscr{F})),\qquad \mathscr{F}^{+!}=f'^+(g'^{!}(\mathscr{F})). \label{hyperbolic loc}
	\end{equation}
	
	We say an object $\mathscr{F}$ of $\rD(X)$ is \textit{weakly equivariant} if there exists an isomorphism $\mu^{+}(\mathscr{F})\simeq \mathscr{L}[-1]\boxtimes \mathscr{F}$ for some smooth module $\mathscr{L}$ on $\Gm$. 
\end{secnumber}

\begin{theorem}[Braden \cite{Br}] \label{Braden thm}
	\textnormal{(i)} There exists a canonical morphism $\iota_{\mathscr{F}}:\mathscr{F}^{+!}\to \mathscr{F}^{!+}$, which is an isomorphism if $\mathscr{F}$ is weakly equivariant. 

	\textnormal{(ii)} The canonical morphisms $\pi_! \to f^!$, $\pi'_+ \to f'^{+}$ induce morphisms 
	\begin{equation} \label{another hyperbolic}
		\pi_! g^+ \mathscr{F} \to \mathscr{F}^{!+}, \qquad \pi'_+g'^{!} \to \mathscr{F}^{+!},
	\end{equation}
	which are isomorphisms if $\mathscr{F}$ is weakly equivariant. 
\end{theorem}

\begin{secnumber}
	Recall the construction of $\iota_{\mathscr{F}}$. 
	The canonical morphism $i=(f,f'):X^0\to Z=X^+\times_X X^-$ is both an open immersion and a closed one (\cite{DG} 1.9.4). We denote by $h:Z\to X^+,h':Z\to X^-$ the canonical morphisms.

	We set $\mathscr{F}^+=g_+(g^+(\mathscr{F}))$ and denote by $\beta:\mathscr{F}\to \mathscr{F}^+$ the adjunction morphism.
	By the base change, there exists a canonical isomorphism
	\begin{equation} \label{F+ base change}
		(\mathscr{F}^+)^{+!}=f'^{+}g'^{!}g_{+}g^{+}(\mathscr{F})\simeq f'^+h'_+h^!g^+(\mathscr{F})\simeq \mathscr{F}^{!+}.
	\end{equation}
	Then we define the morphism $\iota_{\mathscr{F}}$ to be the composition of \eqref{F+ base change} and $\beta^{+!}:\mathscr{F}^{+!}\to (\mathscr{F}^+)^{+!}$. 

	By the base change, the morphism $\iota_{\mathscr{F}}$ is compatible with inverse image by the inclusion of a $\Gm$-equivariant open subscheme and with direct image by the inclusion of a $\Gm$-equivariant closed subscheme.
\end{secnumber}

\begin{secnumber}
	To prove canonical morphisms in \eqref{Braden thm} are isomorphisms, we can extend the scalar and assume that $L$ is an extension of the maximal unramified extension of $\mathbb{Q}_p$ \eqref{base tuple}. 

	Let $Y_0$ be a $k$-scheme and $Y=Y_0\otimes_k \overline{k}$. The category $\rD(Y_0/L)$ is independent of the choice of base field $k$ \eqref{base tuple} and we denote it by $\rD(Y/L)$. 
	Given a morphism $f:Y\to Z$ of $\overline{k}$-schemes, it descents to a morphism of $k'$-schemes for some finite extension $k'$ of $k$. This allows us to define the cohomological functors between $\rD(Y/L)$ and $\rD(Z/L)$. 
	To prove \ref{Braden thm}, we can replace the involved schemes by their base change to $\overline{k}$. 

By a result of Sumihiro \cite{Sum}, we may assume moreover that $X$ is isomorphic to an affine space over $\overline{k}$, equipped with a linear $\Gm$-action. 
\end{secnumber}

\begin{secnumber}
	Let $Y$ be a $\overline{k}$-scheme on which $\Gm$ acts trivially and $W$ a $\Gm$-equivariant vector bundle over $Y$. 
	Suppose that there exists a decomposition of vector bundles $W\simeq W_+\oplus W_-$ such that all the weights of $W_+$ are larger than all the weights of $W_-$.

	We set $E=\mathbb{P}(W)-\mathbb{P}(W_+)$ and $B=\mathbb{P}(W_-)$, where $\mathbb{P}(-)$ denotes the associated projective bundle over $Y$. 
	We denote by $p:E\to B$ the canonical morphism defined by $p([w_+,w_-])=[0,w_-]$, by $i:B\to E$ the canonical morphism, which is a section of $p$ and by $\varphi:B\to Y$ the projection.

	Based on the same argument of (\cite{Br} lemma 6), we can show a similar result. 
\end{secnumber}
\begin{lemma}[\cite{Br} lemma 6]\label{lemma 6}
	Let $\mathscr{F}$ be a weakly equivariant object of $\rD(E)$. There exists canonical isomorphisms
	\begin{equation}
		h_+p_+\mathscr{F}\simeq h_+i^+\mathscr{F},\qquad h_+p_!\mathscr{F}\simeq h_+i^!\mathscr{F}.
	\end{equation}
\end{lemma}

By \ref{lemma 6}, we deduce \ref{Braden thm}(ii). 
Using \ref{lemma 6} and \cite{Sum}, we prove \ref{Braden thm}(i) in the same way as in (\cite{Br} \S~4). 

\section{Geometric Satake equivalence for arithmetic $\cD$-modules} \label{sec Sat}
In this section, we establish the geometric Satake equivalence for arithmetic $\cD$-modules. 

	We assume that $k$ is a finite field with $q=p^s$ elements and keep the assumption and notation in \S~\ref{Dmods}. 
	We work with holonomic modules (resp. complexes) over the geometric base tuple $\mathfrak{T}=\{k,R,K,L\}$ and we omit $/L$ from the notations $\Hol(-/L),\rD(-/L)$ for simplicity. 

	Let $G$ denote a split reductive group over $k$. 
	Let $T$ be the abstract Cartan of $G$, which is defined up to a canonical isomorphism as the quotient of a Borel subgroup $B$ by its unipotent radical.
	We denote by $\wX=\wX(T)$ the weight lattice and by $\cwX=\cwX(T)$ the coweight lattice. 
	Let $\Phi\subset \wX$ (resp. $\Phi^{\vee}\subset \cwX$) the set of roots (resp. coroots). 
	Let $\Phi^+\subset \Phi$ be the set of positive roots and $\cwX(T)^+\subset \cwX(T)$ the semi-group of dominant coweights, determined by a choice of $B$. (But they are independent of the choice of $B$.)  
	Given $\lambda,\mu\in \cwX(T)$, we define $\lambda\le \mu$ if $\mu-\lambda$ is a non-negative integral linear combinations of simple coroots and $\lambda< \mu$ if $\lambda\le \mu$ and $\lambda\neq \mu$. This defines a partial order on $\cwX(T)$ (and on $\cwX(T)^+$).
	We denote by $\rho\in \wX(T)\otimes\mathbb{Q}$ the half sum of all positive roots.

\subsection{The Satake category}

\begin{secnumber}
	We briefly recall affine Grassmannians following (\cite{Zhu16} \S~1, \S~2). 
	The \textit{loop group} $LG$ (resp. \textit{positive loop group} $L^+G$) is the fpqc sheaf on the category of $k$-algebras associated to the functor $R\mapsto G(R ((t)))$ (resp. $ R\mapsto G(R\llbracket t \rrbracket)$ ).
	Then $L^+G$ is a subsheaf of $LG$ and the \textit{affine Grassmannian} $\Gr_G$ is defined as the fpqc-quotient 
	\begin{displaymath}
		\Gr_G=LG/L^+G.
	\end{displaymath}
	The sheaf $\Gr_G$ is represented by an ind-projective ind-scheme over $k$. 
	We write simply $\Gr$ instead of $\Gr_{G}$, if there is no confusion.

	For any dominant coweight $\mu\in \cwX(T)^+$, we denote by $\Gr_{\mu}$ the corresponding $(L^+G)$-orbit, which is smooth quasi-projective over $k$ of dimension $2\rho(\mu)$ (\cite{Zhu16} 2.1.5). 
	Let $\Gr_{\le \mu}$ be the reduced closure of $\Gr_{\mu}$ in $\Gr$, which is equal to $\cup_{\lambda\le \mu}\Gr_{\lambda}$. Let $j_{\mu}:\Gr_{\mu}\to \Gr_{\le \mu}$ be the open inclusion. 
	We have an ind-presentation $\Gr_{\red}\simeq \varinjlim_{\mu\in \cwX(T)^+}\Gr_{\le \mu}$. 
	Since we will work with holonomic modules, we can replace $\Gr$ by its reduced ind-subscheme (\cite{Abe18} 1.1.3 lemma), and omit the subscript $_{\red}$ to simplify the notation. 

	For $i\ge 0$, let $G_i$ be the $i$-th jet group defined by the functor $R\mapsto G(R[t]/t^{i+1})$. Then $G_i$ is representable by a smooth geometrically connected affine group scheme over $k$ and we have $L^+G\simeq \varprojlim_{i} G_i$. If we consider the left action of $L^+G$ on $\Gr$, then the action on $\Gr_{\le \mu}$ factors through $G_i$ for some $i$.
	We can define the category of $(L^+G)$-equivariant holonomic modules on $\Gr$ (see \ref{equiv hol indscheme}), denoted as $\Sat_{G}$ and called \textit{Satake category}. 
	It is a full subcategory of $\Hol(\Gr)$ (\ref{lemma ff forget H}). 
\end{secnumber}

\begin{prop} \label{Sat G semisimple}
	The category $\Sat_{G}$ is semisimple with simple objects $\IC_{\mu}:=j_{\mu,!+}(L_{\Gr_{\mu}}[2\rho(\mu)])$ \eqref{intermediate extension}. 
\end{prop}

\begin{lemma} \label{IsoGrmu semisimple}
	For $\mu\in \cwX(T)^+$, the category $\Sm(\Gr_{\mu})$ \eqref{generality overconv isocrystal} is semisimple with simple object $L_{\Gr_{\mu}}$. 
\end{lemma}
\begin{proof}
	The $(L^+G)$-orbit $\Gr_{\mu}$ is geometrically connected and satisfies $\pi_{1}^{\et}(\Gr_{\mu}\otimes_k \overline{k})\simeq \{1\}$ (cf. \cite{Ric14} proof of proposition 4.1). 
	Every irreducible object $\mathscr{M}$ of $\Sm(\Gr_{\mu})$ has a Frobenius structure with respect to the arithmetic base tuple $\mathfrak{T}_{F}=\{k,R,K,L,s,\id\}$ with finite determinant (\cite{Abe15} 6.1).
	By the companion theorem for overconvergent $F$-isocrystals over a smooth $k$-scheme (\cite{AE16} 4.2) and \v{C}ebotarev density (\cite{Abe18} A.4), we deduce that $\mathscr{M}\simeq L_{\Gr_{\mu}}$ in the category $\Sm(\Gr_{\mu})$. 
	
	To show the semisimplicity, it suffices to show that $\rH^1(\Gr_{\mu},L)=0$. 
	There exists a morphism $\pi:\Gr_{\mu}\to G/P_{\mu}$ realizing $\Gr_{\mu}$ as an affine bundle over the partial flag variety $G/P_{\mu}$, where $P_{\mu}$ is the parabolic subgroup containing $B$ and associated with $\{\alpha\in \Phi, (\alpha,\mu)=0\}$. 
	In view of \eqref{pullback of smooth mods} and the cohomology of affine spaces (\ref{cohomology arith mods}(ii)), the canonical morphism $L_{G/P_{\mu}}\to \pi_+(L_{\Gr_{\mu}})$ is an isomorphism. 
	Then the cohomology $\rH^{i}(\Gr_{\mu},L)$ is isomorphic to $\rH^{i}(G/P_{\mu},L)$.
	Since the partial flag variety admits a stratification of affine spaces, we deduce that $\rH^{i}(G/P_{\mu},L)=0$ if $i$ is odd by \eqref{ss stratification}. 
	Then the assertion follows. 
\end{proof}

To prove proposition \ref{Sat G semisimple}, we need a parity result on the constructible cohomology of $\IC_{\mu}$. 

\begin{lemma} \label{parity stalk cohomology of IC}
	The constructible module $\cH^{i}(\IC_{\mu})$ vanishes unless $i\equiv \dim (\Gr_{\mu})~(\textnormal{mod}~2)$. 
\end{lemma}
\begin{proof}
	We follow the argument of Gaitsgory (\cite{Gai01} A.7, cf. \cite{BR} \S 4.2 for a detailed exposition) in the $\ell$-adic case, whose proof is based on following ingredients: 1) the decomposition theorem; 2) the fiber of the Bott-Samelson resolution of a cell in affine flag variety is paved by affine spaces. 
	
	In our case, the assertion follows from the same argument using the decomposition theorem \eqref{decomposition}, the spectral sequence \eqref{ss stratification} and the parity of the compact support $p$-adic cohomology of affine spaces (\ref{cohomology arith mods}(ii)).
\end{proof}

\begin{secnumber} \textit{Proof of proposition \ref{Sat G semisimple}}. \label{proof semisimple}
	We follow the same line as in the $\ell$-adic case (cf. \cite{Gai01} prop. 1).
	By \ref{intermediate extension}(i), holonomic modules $\IC_{\mu}$ are irreducible of $\Sat_G$. 
	Let $\mathscr{E}$ be an irreducible object of $\Sat_G$. There exists an $(L^+G)$-orbit $\Gr_{\mu}$ such that $\mathscr{E}|_{\Gr_{\mu}}$ is a smooth object. 
	By \ref{lemma ff forget H} and \ref{IsoGrmu semisimple}, we deduce that $\mathscr{E}$ is isomorphic to $\IC_{\mu}$. 
	
	To prove the semisimplicity, it suffices to show that for $\lambda,\mu \in \cwX(T)^+$, we have 
	\begin{equation} \label{Ext 1 vanish}
		\Ext^{1}_{\Hol(\Gr)}(\IC_{\lambda},\IC_{\mu})=\Hom_{\rD(\Gr)}(\IC_{\lambda},\IC_{\mu}[1])=0.
	\end{equation}

	(i) In the case $\lambda=\mu$, \eqref{Ext 1 vanish} follows from $\Ext_{\Hol(\Gr_{\mu})}^{1}(L_{\Gr_{\mu}},L_{\Gr_{\mu}})=\rH^1(\Gr_{\mu},L)=0$ (see the proof of lemma \ref{IsoGrmu semisimple}). 

	(ii) Then we consider the case either $\lambda< \mu$ or $\mu< \lambda$. 
	Since the dual functor $\mathbb{D}$ induces an anti-equivalence, we may assume that $\mu< \lambda$. We denote by $i:\Gr_{\le \mu}\to \Gr_{\le \lambda}$ the close immersion and we have 
	\begin{displaymath}
		\Hom_{\rD(\Gr)}(\IC_{\lambda},i_+\IC_{\mu}[1])\simeq \Hom_{\rD(\Gr_{\le \mu})}(i^+\IC_{\lambda},\IC_{\mu}[1]).
	\end{displaymath}
	Note that $i^+\IC_{\lambda}$ has cohomological degrees $\le -1$ (\ref{lemma uniqueness intermediate extension}(i)). 
	Each $(L^+G)$-equivariant holonomic module $\hH^i(i^+\IC_{\lambda}|_{\Gr_{\mu}})$ is smooth and hence is constant \eqref{IsoGrmu semisimple}. 
	If there existed a non-zero morphism $g:i^+\IC_{\lambda}\to\IC_{\mu}[1]$, then it would induce a non-zero morphism $h:\hH^{-1}(i^+\IC_{\lambda}|_{\Gr_{\mu}})\to L_{\Gr_{\mu}}[2\rho(\mu)]$. 
	Since $i^+$ is c-t-exact, it contradicts to \ref{parity stalk cohomology of IC}. The equality \eqref{Ext 1 vanish} in this case follows. 

	(iii) In the case $\lambda\nleq \mu$ and $\mu\nleq \lambda$, we prove \eqref{Ext 1 vanish} by base change in the same way as in (\cite{BR} 4.3). \hfill $\qed$ 
\end{secnumber}

\begin{secnumber}
	We consider the action of $L^+G$ on $LG\times \Gr$ defined by $a(g,[h])=(ga^{-1},[ah])$, where $[h]$ denotes the coset $h\cdot L^+G$ of an element $h\in LG$. 
	We denote by $\Gr\widetilde{\times}\Gr$ the twisted product $LG\times \Gr/L^+G$ \eqref{twisted prod}.
	The morphism $LG\times \Gr\to \Gr$, defined by $(g,[h])\mapsto [gh]$, induces an ind-proper morphism
	\begin{equation} \label{convolution map}
		m:\Gr\widetilde{\times} \Gr \to \Gr,
	\end{equation}
	called the \textit{convolution morphism}.
	The morphism $m$ is $(L^+G)$-equivariant with respect to the left $(L^+G)$-actions. 
	
	Given two objects $\mathcal{A}_1,\mathcal{A}_2$ of $\Sat_G$, we denote by $\mathcal{A}_1\widetilde{\boxtimes}\mathcal{A}_2$ their external twisted product on $\Gr\widetilde{\times}\Gr$ (see \ref{twisted prod} and \ref{equiv hol indscheme}), and
 define the \textit{convolution product} by
	\begin{equation}
		\mathcal{A}_1\star \mathcal{A}_2 =m_{+}(\mathcal{A}_1\widetilde{\boxtimes} \mathcal{A}_2).
		\label{def convolution Sat}
	\end{equation}
	Similarly, there exists an $n$-fold twisted product $\Gr\widetilde{\times}\cdots \widetilde{\times} \Gr$ and a convolution morphism $m:\Gr\widetilde{\times}\cdots \widetilde{\times} \Gr \to \Gr$ (\cite{Zhu16} 1.2.15). 
	This allows us to define the $n$-fold convolution product $\mathcal{A}_1 \star \cdots \star \mathcal{A}_n$. 
	
	We will show that $\mathcal{A}_1\star \mathcal{A}_2$ is an object of $\Sat_G$ and that $\star$ defines a symmetric monoidal structure on $\Sat_G$.  
	To do it, we will interpret the convolution product as the specialization of a fusion product on Beilinson-Drinfeld Grassmannians in the next subsection. 
\end{secnumber}

\subsection{Fusion product}
\begin{secnumber} \label{BD Grass}
	Let $X$ be a smooth, geometrically connected curve over $k$, $n$ an integer $\ge 1$ and $X^n$ the $n$-folded self product of $X$ over $k$. 
	We briefly recall the definition of Beilinson-Drinfeld Grassmannians (\cite{Zhu16} \S~3). 

	For any $k$-algebra $R$ and any point $x=(x_i)_{i=1}^n\in X^n(R)$, we set $\Gamma_{x}=\cup_{i=1}^n \Gamma_{x_i}$ the closed subscheme of $X_R$ defined by the union of graphs $\Gamma_{x_i}\hookrightarrow X_R$ of $x_i:\Spec(R)\to X$. 
	The \textit{Beilinson-Drinfeld Grassmannian $\Gr_{G,X^n}$ (associated to $G$ over $X^n$)} is the functor which associates to every $k$-algebra $R$ the groupoid $\Gr_{G,X^n}(R)$ of triples $(x,\mathscr{E},\beta)$ 
	\begin{displaymath}
		\{x\in X^n(R),\quad 
		\textnormal{$\mathscr{E}$ a $G$-torsor on $X_R$}, \quad
		\beta: \mathscr{E}|_{X_R-\Gamma_x} \xrightarrow{\sim} \mathscr{E}^0:=G\times (X_R-\Gamma_x)~ \textnormal{a trivialisation}\}.
	\end{displaymath}
	The above functor is represented by an ind-projective ind-scheme over $X^n$ (\cite{Zhu16} 3.1.3). We denote by $q^n:\Gr_{G,X^n}\to X^n$ the canonical morphism. 
	If there is no confusion, we will write simply $\Gr_{X^n}$ instead of $\Gr_{G,X^n}$. 
	Note that the fiber of $\Gr_{X}$ at a closed point $x$ of $X$ is isomorphic to the affine Grassmannian.

	We refer to (\cite{Zhu16} 3.1) the definition of global loop groups $(L^+G)_{X^n}$ and $(LG)_{X^n}$ over $X^n$. 
	The sheaf $(L^+G)_{X^n}$ is represented by a projective limit of smooth affine group scheme over $X^n$. 
	There exists a canonical isomorphism of fpqc-sheaves $(LG)_{X^n}/(L^+G)_{X^n}\xrightarrow{\sim} \Gr_{G,X^n}$.
	We consider the left action of $(L^+G)_{X^n}$ on $\Gr_{G,X^n}$ over $X^n$ and denote by $\Hol_{(L^+G)_{X^n}}(\Gr_{X^n})$ the category of $(L^+ G)_{X^n}$-equivariant holonomic modules on $\Gr_{X^n}$ \eqref{equiv hol indscheme}.  
\end{secnumber}

\begin{secnumber} \label{globalisation functor}
	In the following, we take the curve $X$ to be the affine line $\mathbb{A}^1_k$. Then there exists an isomorphism $\Gr_X\simeq \Gr\times X$. 
	Given a holonomic module $\mathcal{A}$ on $\Gr$, the holonomic module $\mathcal{A}_X=\mathcal{A}\boxtimes L_X[1]$ is ULA with respect to $q:\Gr_X\to X$ \eqref{ULA indscheme}.
	If $\mathcal{A}$ is moreover $(L^+G)$-equivariant, then $\mathcal{A}_{X}$ is $(L^+G)_X$-equivariant.

	By proposition \ref{BBD 4.2.5}, we obtain a fully faithful functor 
	\begin{equation} \label{Sat G to X}
		\iota:\rD(\Gr)\to \rD(\Gr_{X}),\quad \mathcal{A}\mapsto \mathcal{A}_X.
	\end{equation}
	We denote the essential image of $\Sat_{G}$ via $\iota$ by $\Sat_{X}$, which is a full subcategory of $\Hol_{(L^+G)_X}(\Gr_{X})$. 
\end{secnumber}

\begin{secnumber} \label{vphi base change}
	To define the fusion product on $\Sat_X$, we will use the factorization structure of Beilinson-Drinfeld Grassmannians. 
	
	The diagonal immersion $\Delta:X\to X^2$ a morphism $\Gr_X \to \Gr_{X^2}$ sending $(x,\mathscr{E},\beta)$ to $(\Delta(x), \mathscr{E}, \beta)$, which is compatible with $(LG)_X$-actions. 
	Moreover, it induces a canonical isomorphism
	\begin{equation}
		\Delta: \Gr_{X}\xrightarrow{\sim}\Gr_{X^2}\times_{X^2,\Delta}X. 		\label{phi base change BD Grass}
	\end{equation}
	
	Let $U$ be complement of $\Delta:X\to X^2$.  
	Then there exists a canonical isomorphism, called the \textit{factorization isomorphism} (\cite{Zhu16} 3.2.1(iii))
	\begin{equation}
		c:\Gr_{X^2}\times_{X^2}U \xrightarrow{\sim} (\Gr_{X}\times \Gr_{X})\times_{X^2} U.
		\label{factorization iso}
	\end{equation}
	The involution $\sigma:X^2\to X^2, ~ (x,y)\mapsto (y,x)$, induces an involution $\Delta(\sigma):\Gr_{X^2}\to \Gr_{X^2}$. 
	If we consider the $S_2$-action on $\Gr_X\times \Gr_X$ by the permutation of two factors, then $c$ is moreover $S_2$-equivariant.
\end{secnumber}

\begin{secnumber} \label{twisted prod BD Gr}
	The convolution morphism \eqref{convolution map} also admits a globalization. 
	The \textit{convolution Grassmannian} $\Gr_{X}\widetilde{\times}\Gr_{X}$ is the ind-scheme representing
	\begin{displaymath}
		R\mapsto \bigg\{(x,(\mathscr{E}_i,\beta_i)_{i=1,2}) \bigg| 
		\txt{	$x=(x_1,x_2)\in X^2(R)$, \quad $\mathscr{E}_1,\mathscr{E}_2$ are $G$-torsors on $X_R$\\
		$\beta_{i}:\mathscr{E}_i|_{X_R-\Gamma_{x_i}}\xrightarrow{\sim} \mathscr{E}_{i-1}|_{X_R-\Gamma_{x_{i}}}$ where $\mathscr{E}_0=\mathscr{E}^{0}$ is trivial} \bigg\}.
	\end{displaymath} 
	There exists a convolution morphism
	\begin{equation} \label{fusion m}
		m:\Gr_{X}\widetilde{\times}\Gr_{X}\to \Gr_{X^2}, \qquad (x,(\mathscr{E}_i,\beta_i)_{i=1,2}) \mapsto (x,\mathscr{E}_2,\beta_1\circ \beta_2).
	\end{equation}
	By definition, the restriction of $m$ on $U$ is an isomorphism.  	

	We can view $\Gr_X\widetilde{\times}\Gr_X$ as a twisted product \eqref{twisted prod} in the following way. 	
	There exists a $(L^+G)_X$-torsor $\mathbb{E}\to \Gr_X \times X$ classifying 
	\begin{displaymath}
		R\mapsto \mathbb{E}(R)=\bigg\{(x_1,\beta_1,\mathscr{E}_1)\in \Gr_X(R);~ x_2\in X(R);~  \eta:\mathscr{E}^0\xrightarrow{\sim} \mathscr{E}_1|_{\widehat{\Gamma}_{x_2}}\bigg\}, 
	\end{displaymath} 
	where $\eta$ is a trivialisation of $\mathscr{E}_1$ on the formal completion $\widehat{\Gamma}_{x_2}$ of $X_R$ along $\Gamma_{x_2}$. 
	Using the torsor $\mathbb{E}$, we can identify $\Gr_X\widetilde{\times} \Gr_X$ with the twisted product $(\Gr_X\times X)\widetilde{\times}_X \Gr_X$ \eqref{equiv hol indscheme}. 
	In summary, we have the following diagram over $X^2$
	\begin{equation}
		\Gr_X\times \Gr_X =(\Gr_X\times X)\times_X \Gr_X \leftarrow \mathbb{E}\times_X \Gr_X \rightarrow \Gr_X\widetilde{\times} \Gr_X \xrightarrow{m} \Gr_{X^2}. \label{convolution diagram}
	\end{equation}

	Let $\mathcal{A}_{1},\mathcal{A}_2$ be two objects of $\Sat_{X}$. 
	Note that $(\mathcal{A}_1\boxtimes L_X)\boxtimes_X \mathcal{A}_2 \simeq \mathcal{A}_1\boxtimes \mathcal{A}_2$ is holonomic. 
	We denote by $\mathcal{A}_1\widetilde{\boxtimes}\mathcal{A}_2$ the twisted product of $\mathcal{A}_1\boxtimes L_X$ and $\mathcal{A}_2$ on $\Gr_X\widetilde{\times}\Gr_X$ \eqref{equiv hol indscheme}. 
\end{secnumber}

\begin{prop} \label{prop fusion prod}
	\textnormal{(i)} There exists a canonical isomorphism of holonomic modules on $\Gr_{X^2}$:
	\begin{equation} \label{fusion product iso}
		m_{+}(\mathcal{A}_{1}\widetilde{\boxtimes}\mathcal{A}_2)\simeq j_{!+}(\mathcal{A}_1\boxtimes\mathcal{A}_2|_{U}). 
	\end{equation}
	The left hand side, denoted by $\mathcal{A}_1\boxast\mathcal{A}_2$, is ULA with respect to $q^2:\Gr_{X^2}\to X^2$ and is $(L^+G)_{X^2}$-equivariant. 

	\textnormal{(ii)} There exists a canonical isomorphism of holonomic modules on $\Gr_{X}$: 
	\begin{displaymath}
		\Delta^{+}[-1](\mathcal{A}_{1}\boxast \mathcal{A}_{2}) \xrightarrow{\sim} \Delta^{!}[1](\mathcal{A}_{1}\boxast \mathcal{A}_{2}).
	\end{displaymath}
	We denote one of the above module by $\mathcal{A}_1\circledast \mathcal{A}_2$ and call it \textit{fusion product of $\mathcal{A}_1,\mathcal{A}_2$}.
	This holonomic module is ULA with respect to $q:\Gr_X\to X$. 
\end{prop} 
\begin{proof}
	(i) The holonomic module $\mathcal{A}_1\boxtimes \mathcal{A}_2$ on $\Gr_{X}\times \Gr_X$ is the inverse image of a holonomic module on $\Gr\times \Gr$ and hence is ULA with respect to the projection $\Gr_X\times \Gr_X \to X^2$. 
	Recall that $\mathcal{A}_1\widetilde{\boxtimes} \mathcal{A}_2$ is constructed by descent along a quotient by a smooth group scheme over $X$ (\ref{equiv hol indscheme}, \ref{convolution diagram}). Hence it is ULA with respect to the projection to $X^2$ by proposition \ref{basic ULA}(iii). 
	Since $m$ is ind-proper, then $m_{+}(\mathcal{A}_{1}\widetilde{\boxtimes}\mathcal{A}_2)$ is ULA with respect to $q^2:\Gr_{X^2}\to X^2$.

	Since $m|_{U}$ is an isomorphism, under the isomorphism \eqref{factorization iso} we have
	\begin{displaymath}
		\mathcal{A}_1\widetilde{\boxtimes}\mathcal{A}_{2}|_{U}=\mathcal{A}_1\boxtimes\mathcal{A}_2|_{U},
	\end{displaymath}
	which is holonomic. 
	Then we deduce the isomorphism \eqref{fusion product iso} from proposition \ref{ULA hol}(i).
	The morphism $m$ is $(L^+G)_{X^2}$-equivariant. By proper base change, we deduce that $m_{+}(\mathcal{A}_{1}\widetilde{\boxtimes}\mathcal{A}_2)$ is $(L^+G)_{X^2}$-equivariant. 

	Assertion (ii) follows from proposition \ref{ULA hol}.
\end{proof}

\begin{coro} \label{monidal structure}
	Let $\mathcal{A}_1,\mathcal{A}_2$ be two objects of $\Sat_{G}$.

	\textnormal{(i)} There exists a canonical isomorphism on $\Gr_X$ \eqref{def convolution Sat}
	\begin{equation} \label{convolution and fusion}
		(\mathcal{A}_1\star \mathcal{A}_2)_{X} \simeq \mathcal{A}_{1,X}\circledast \mathcal{A}_{2,X}. 
	\end{equation}

	\textnormal{(ii)} The convolution product $\mathcal{A}_1\star \mathcal{A}_2$ is still holonomic and belongs to $\Sat_{G}$.

	\textnormal{(iii)} The category $\Sat_G$ (resp. $\Sat_{X}$) equipped with the bifunctor $\star$ (resp. $\circledast$) and the unit object $\IC_0$ (resp. $\IC_{0,X}$) forms a monoidal category. 
\end{coro}
\begin{proof}
	(i) There exists a canonical isomorphism
	\begin{displaymath}
		(\Gr\widetilde{\times}\Gr)\times X \simeq (\Gr_X\widetilde{\times}\Gr_{X})\times_{X^2,\Delta} X,
	\end{displaymath}
	compatible with projections to $\Gr_{X}$. 
	Via the above isomorphism, we have $(\mathcal{A}_1\widetilde{\boxtimes}\mathcal{A}_2)_X \simeq \Delta^{+}[-1](\mathcal{A}_{1,X}\widetilde{\boxtimes}\mathcal{A}_{2,X})$. 
	Then the isomorphism of \eqref{convolution and fusion} follows. 

	(ii) Taking a $k$-point $i_x:x\to X$ and applying the functor $i_x^+[-1]$ to \eqref{convolution and fusion}, we deduce that $\mathcal{A}_1\star \mathcal{A}_2$ is holonomic by propositions \ref{ULA hol} and \ref{prop fusion prod}. 

	(iii) It suffices to show the assertion for $\Sat_G$. 
	By identifying $(\mathcal{A}_1\star \mathcal{A}_2)\star \mathcal{A}_2$ and $\mathcal{A}_1\star (\mathcal{A}_2\star \mathcal{A}_3)$ with $\mathcal{A}_1\star \mathcal{A}_2 \star \mathcal{A}_3$, we obtain the associative constraint. 
	We can verify the pentagon axiom and the unit axiom by $n$-fold convolution product. 	
\end{proof}

\subsection{Hypercohomology functor}

\begin{secnumber}
	We define the hypercohomology functor $\rH^{*}$ by 
	\begin{eqnarray} \label{hypercoh Gr}
		\rH^{*}:\Sat_{G} \to \Vect_{L}, \qquad
		\mathcal{A} \mapsto  \bigoplus_{n\in \mathbb{Z}} \rH^{n}(\Gr,\mathcal{A}).
	\end{eqnarray}
	Since $\Sat_G$ is semisimple \eqref{Sat G semisimple}, $\rH^{*}$ is exact and faithful. 

	Let $\mathcal{A}$ be an object of $\Sat_G$ and $\pi:\Gr\to \Spec(k)$ the structure morphism. 
	By the Künneth formula (\cite{Abe18} 1.1.7), there exists a canonical isomorphism
	\begin{equation}
		q_+(\mathcal{A}_X)[-1]\simeq \pi_{+}(\mathcal{A})\boxtimes L_X.
	\label{constant module}
	\end{equation}
\end{secnumber} 

\begin{lemma} \label{Coh functor monoidal}
	Given two objects $\mathcal{A}_1,\mathcal{A}_2$ of $\Sat_{X}$, there exists a canonical isomorphism 
	\begin{equation}
		q_+(\mathcal{A}_1\circledast \mathcal{A}_2) [-1] \simeq (q_+(\mathcal{A}_1)[-1])\otimes (q_+(\mathcal{A}_2)[-1]). 
	\end{equation}
\end{lemma}
\begin{proof}
	It suffices to construct a canonical isomorphism
	\begin{equation} \label{q_+ boxtimes}
		q_+^2(\mathcal{A}_1\boxast \mathcal{A}_2)\simeq q_+(\mathcal{A}_1)\boxtimes q_+(\mathcal{A}_2).
	\end{equation}
	By \eqref{fusion product iso} and the Künneth formula (\cite{Abe18} 1.1.7), such an isomorphism exists on $U=X^2-\Delta(X)$. 
	
	Let $\tau:X^2\to X$ be the morphism sending $(x,y)$ to $x-y$. 
	Both sides of \eqref{q_+ boxtimes} are ULA with respect to $\tau$ by propositions \ref{basic ULA} and \ref{prop fusion prod}. 
	By proposition \ref{ULA hol}, we deduce a canonical isomorphism on $X$
	\begin{displaymath}
		\Delta^!\bigl(q_+^2(\mathcal{A}_1\boxast \mathcal{A}_2)\bigr)\simeq \Delta^{!}\bigl(q_+(\mathcal{A}_1)\boxtimes q_+(\mathcal{A}_2)\bigr).
	\end{displaymath}
	Then the isomorphism \eqref{q_+ boxtimes} follows from the distinguished triangle $\Delta_+\Delta^!\to \id \to j_+j^+ \to $. 
\end{proof}

By \eqref{convolution and fusion}, \eqref{constant module} and lemma \ref{Coh functor monoidal}, we deduce that: 

\begin{coro} \label{faithful exact mono}
	The functor $\rH^{*}$ is monoidal. 
\end{coro}

\begin{rem} \label{coh functor Frob}
	Let $\mathcal{A}_1,\mathcal{A}_2$ be two objects of $\Sat_G$, both equipped with a Frobenius structure with respect to the arithmetic tuple $\mathfrak{T}_{F}=\{k,R,K,L,s,\id_{L}\}$. 
	The proof of corollary \ref{faithful exact mono} applies to arithmetic $\mathscr{D}$-modules with Frobenius structures. 
	Then we deduce that the following isomorphism is compatible with Frobenius structure
	\begin{displaymath}
		\rH^{*}(\mathcal{A}_1\star \mathcal{A}_2)\simeq \rH^{*}(\mathcal{A}_1)\otimes \rH^{*}(\mathcal{A}_2).
	\end{displaymath}
\end{rem}

\begin{secnumber}\label{constraint}
	In the following, we will construct a commutativity constraint on $(\Sat_G,\star)$ which makes the functor $\rH^*$ into a tensor functor. 
	
	The permutation $\sigma:\{1,2\}\to \{1,2\}$ induces an involution $\Delta(\sigma):\Gr_{X^2}\to \Gr_{X^2}$ over the involution $\sigma:X^2\to X^2,~ (x,y)\mapsto (y,x)$ \eqref{vphi base change}. 
	Let $\mathcal{A}_1,\mathcal{A}_2$ be two objects of $\Sat_G$. 
	We deduce from \eqref{factorization iso} and \eqref{fusion product iso} a canonical isomorphism
	\begin{equation} \label{constraint BD Gr}
		\Delta(\sigma)^{+}(\mathcal{A}_{1,X}\boxast \mathcal{A}_{2,X})\xrightarrow{\sim} \mathcal{A}_{2,X}\boxast \mathcal{A}_{1,X}. 
	\end{equation}
	We denote by $c_{\gr}$ the isomorphism which fits into the following diagram
	\begin{displaymath}
		\xymatrix{
			\sigma^{+}(q_+^2(\mathcal{A}_1\boxast \mathcal{A}_2))\ar[r]^-{\sim} \ar[d]_{\wr}& q_+^2(\Delta(\sigma)^{+}(\mathcal{A}_1\boxast\mathcal{A}_2)) \ar[r]^-{\sim} & q_+^2(\mathcal{A}_2\boxast \mathcal{A}_1) \ar[d]^{\wr} \\
			\sigma^{+}(q_+(\mathcal{A}_1)\boxtimes q_+(\mathcal{A}_2))\ar[rr]^{c_{\gr}} && q_+(\mathcal{A}_2)\boxtimes q_+(\mathcal{A}_1)
		}
	\end{displaymath}

	The cohomology $\rH^n(i_{(x,x)}^+(c_{\gr}))$ of the fiber of $c_{\gr}$ at $(x,x)\in X^2$ is the composition
	\begin{displaymath}
		\bigoplus_{i+j=n}\rH^{i}(\Gr,\mathcal{A}_1)\otimes\rH^j(\Gr,\mathcal{A}_2)\simeq 
		\rH^{n}(\Gr\times \Gr,\mathcal{A}_1\boxtimes \mathcal{A}_2)\simeq 
		\rH^{n}(\Gr\times \Gr,\mathcal{A}_2\boxtimes \mathcal{A}_1) \simeq 
		\bigoplus_{i+j=n} \rH^{j}(\Gr,\mathcal{A}_2)\otimes\rH^i(\Gr,\mathcal{A}_1),
	\end{displaymath}
	where the first and third isomorphisms are given by the Künneth formula and the second one is induced by the symmetry of external tensor products. 
	It sends $s\otimes t$ to $(-1)^{ij} t\otimes s$ for $s\in \rH^i(\Gr,\mathcal{A}_1)$ and $t\in \rH^{j}(\Gr,\mathcal{A}_2)$.
	
	Taking the fiber of \eqref{constraint BD Gr} at $(x,x)$, we obtain a canonical isomorphism 
	\begin{equation} \label{constraint c'}
		c'_{\mathcal{A}_1,\mathcal{A}_2}:\mathcal{A}_1\ast \mathcal{A}_2 \simeq \mathcal{A}_2\ast \mathcal{A}_1,
	\end{equation}
	which fits into a commutative diagram
	\begin{equation} \label{commutative c'}
	\xymatrix{
		\rH^{*}(\mathcal{A}_1\ast \mathcal{A}_2) \ar[r]^{c'_{\mathcal{A}_1,\mathcal{A}_2}} \ar[d]^{\wr} & \rH^{*}(\mathcal{A}_2\ast \mathcal{A}_1) \ar[d]^{\wr} \\
		\rH^{*}(\mathcal{A}_1)\otimes \rH^{*}(\mathcal{A}_2) \ar[r]^{c_{\gr}} & \rH^{*}(\mathcal{A}_2)\otimes \rH^{*}(\mathcal{A}_1).
}
\end{equation}

We regard $\rH^{*}$ as a functor from $\Sat_G$ to the category $\Vect_{L}^{\gr}$ of $\mathbb{Z}$-graded vector spaces over $L$ by considering the $\mathbb{Z}$-grading on cohomology degree \eqref{hypercoh Gr}. 
The above diagram means $\rH^{*}$ is compatible with the constraint $c'_{\mathcal{A}_1,\mathcal{A}_2}$ on $\Sat_G$ and the supercommutativity constraint $c_{\gr}$ on $\Vect_L^{\gr}$.
In \ref{modify constraint}, we will modify the constraint $c'$ and make it compatible with the usual constraint on $\Vect_{L}$. 
\end{secnumber}

\subsection{Semi-infinite orbits}
In this subsection, we study the $p$-adic cohomology of objects of $\Sat_G$ on semi-infinite orbits of $\Gr_G$ following Mirkovi\'c and Vilonen \cite{MV}. 

\begin{secnumber} \label{notation parabolic ind}
	Let $B^{\op}$ be the opposite Borel subgroup. 
	The inclusion $B,B^{\op}\to G$ and projections $B,B^{\op}\to T$ induce morphisms
	\begin{equation}
		\Gr_{T} \xleftarrow{\pi} \Gr_{B}\xrightarrow{i} \Gr_{G}, \qquad \Gr_{T} \xleftarrow{\pi'} \Gr_{B^{\op}}\xrightarrow{i'} \Gr_{G}.
		\label{hyperbolic localisation Gr}
	\end{equation}

	Via $i$, each connected component of $(\Gr_B)_{\red}$ is locally closed in $\Gr_{G}$. 
	To simplify the notation, we will omit the subscript $_{\red}$ in the following. 
	The affine Grassmannian $\Gr_T$ is discrete, whose $k$-points are given by $L_{\lambda}=t^{\lambda}T(k\llbracket t\rrbracket)/T(k\llbracket t\rrbracket)\in \Gr_{T}(k), \ \lambda\in \cwX(T)$. 
	For $\lambda\in \cwX(T)$, we define ind-subschemes $S_{\lambda}$ and $T_{\lambda}$ of $\Gr_{G}$ to be
	\begin{equation}
		S_{\lambda}=i(\pi^{-1}(L_{\lambda})),\qquad 
		T_{\lambda}=i'(\pi'^{-1}(L_{\lambda})).
	\end{equation}
	For $i\in \mathbb{Z}$, we set cohomology functors $\rH^{i}_{\rc}(S_{\lambda},-)$ and $\rH^{i}_{T_{\lambda}}(\Gr_{G},-)$ to be
	\begin{displaymath}
		\rH_{\rc}^{i}(S_{\lambda},-)=\rH^{i}( (\pi_{!}i^{+}(-))_{\lambda}),\quad 
		\rH_{T_{\lambda}}^{i}(\Gr_{G},-)=\rH^{i}( \pi'_{+}i'^{!}(-))_{\lambda}).
	\end{displaymath}
\end{secnumber}

\begin{prop}[\cite{MV} 3.1, 3.2, \cite{Zhu16} 5.3.6] \label{geo Slammda}
	\textnormal{(i)} The union $S_{\le \lambda}=\cup_{\lambda'\le \lambda}S_{\lambda'}$ is closed in $\Gr_G$ and $S_{\lambda}$ is open and dense in $S_{\le \lambda}$.

	\textnormal{(ii)} For $\mu\in\cwX(T)^{+}$, the intersection $\Gr_{G,\mu}\cap S_{\lambda}$ (resp. $\Gr_{G,\mu}\cap T_{\lambda}$) is non-empty if and only if $L_{\lambda}\in \Gr_{G,\le \mu}$ (equivalently there exists $w\in W$ such that $w\lambda\le \mu$). 
	In the non-empty case, $\Gr_{G,\mu}\cap S_{\lambda}$ (resp. $\Gr_{G,\mu}\cap T_{\lambda}$) has pure dimension $\rho(\lambda+\mu)$.
\end{prop}

\begin{prop} \label{decomposition CT}
	\textnormal{(i)} For any object $\mathcal{A}$ of $\Sat_G$, there exists a functorial isomorphism 
	\begin{equation}
		\rH_{\rc}^{i}(S_{\lambda},\mathcal{A})\simeq \rH_{T_{\lambda}}^{i}(\Gr_{G},\mathcal{A}). \label{Braden thm SatG}
	\end{equation}
	Both sides vanish if $i\neq 2\rho(\lambda)$. 

	\textnormal{(ii)} For $\mu\in \cwX(T)^+$, the dimension of $\rH_{\rc}^{2\rho(\lambda)}(S_{\lambda},\IC_{\mu})$ is equal to the number of geometrically irreducible components of $S_{\lambda}\cap \Gr_{G,\mu}$.
	If we work with the arithmetic base $\mathfrak{T}_F=\{k,R,K,L,s,\id_L\}$, the Frobenius acts on $\rH_{\rc}^{2\rho(\lambda)}(S_{\lambda},\IC_{\mu})$ by multiplication by $q^{\rho(\lambda+\mu)}$. 

	\textnormal{(iii)} For any integer $i$, there exists a functorial isomorphism
	\begin{equation}
		\rH^{i}(\Gr_G,\mathcal{A})\simeq \bigoplus_{\lambda\in \cwX(T)} \rH^{i}_{\rc}(S_{\lambda},\mathcal{A}).
		\label{decomposition H}
	\end{equation}
\end{prop}

The proposition can be proved in the same way as in (\cite{MV} 3.5, 3.6) by Braden's theorem \eqref{Braden thm}. 

\begin{prop} \label{tensor CT}
	Given two objects $\mathcal{A}_1,\mathcal{A}_2$ of $\Sat_{G}$, there exists a canonical isomorphism
	\begin{equation}
		\rH^{2\rho(\lambda)}_{\rc}(S_{\lambda},\mathcal{A}_1\ast \mathcal{A}_2)\simeq \bigoplus_{\lambda_1+\lambda_2=\lambda} \rH^{2\rho(\lambda_1)}_{\rc}(S_{\lambda_1},\mathcal{A}_1)\otimes \rH^{2\rho(\lambda_2)}_{\rc}(S_{\lambda_2},\mathcal{A}_2).
	\end{equation}
\end{prop}

\begin{secnumber}
	To prove the above proposition, we need to extend semi-infinite orbits $S_{\lambda},T_{\lambda}$ to Beilinson-Drinfeld Grassmannians $\Gr_{G,X^n}$. 
	For simplicity, we only consider the case where $X=\mathbb{A}^1$ and $n=1,2$. 

	When $n=1$, we have ind-representation $\Gr_{G,X}\simeq \varinjlim \Gr_{G,\le\mu,X}$, where $\Gr_{G,\le \mu,X}\simeq \Gr_{G,\le \mu}\times X$ is normal for $\mu\in \cwX(T)^+$. 
	When $n=2$, for $\mu,\nu\in \cwX(T)^+$, we denote by $\Gr_{G,\le (\mu,\nu),X^2}$ the closure of $\Gr_{G,\le \mu,X}\times \Gr_{G,\le \nu,X}|_U$ in $\Gr_{G,X^2}$. These closed subschemes form an ind-representation of $\Gr_{X^2}$. 
	The scheme $\Gr_{G,\le (\mu,\nu),X^2}$ is flat over $X^2$ and satisfies $\Gr_{G,\le (\mu,\nu),X^2}\times_{X^2,\Delta}X\simeq \Gr_{G,\le \mu+\nu, X}$ (\cite{Zhu09} 1.2, \cite{Zhu16} 3.1.14). 
	The composition $\Gr_{G,\le (\mu,\nu),X^2}\to X^2\to X$ with $X^2\to X,~(x,y)\mapsto x-y$, is flat with reduced fiber at $0$ and is normal on $X-\{0\}$. Then we deduce that $\Gr_{G,\le (\mu,\nu),X^2}$ is normal (cf. \cite{PZ} 9.2). 
	
	We consider the action of $\Gm$ on $\Gr_{G,X^n}$ induced by $2\check{\rho}$, which is compatible with the action of $\Gm$ on $\Gr_{G}$ on each fiber of $x\in |X^n|$ (\ref{geo Slammda}(i)). 
	Then $\Gr_{T,X^n}$ is the ind-subscheme of fixed points. 
	For $\lambda\in \cwX(T)$, we set $C_{\lambda}(X^2)=\Gr_{T,\le \lambda,X^2}-\Gr_{T,<\lambda,X^2}$. Its fiber at $x=(x,x)\in \Delta(X)\subset X^2$ is isomorphic to $\{L_{\lambda}\}$ and its fiber at $x=(x,y)\in X^2-\Delta(X)$ is isomorphic to $\prod_{\lambda_1+\lambda_2=\lambda}\{L_{\lambda_1}\}\times \{L_{\lambda_2}\}$. 
	Connected components of $\Gr_{T,X^2}$ are parametrized by $\{C_{\lambda}(X^2)\}_{\lambda\in \cwX(T)}$. 

	We denote by $S_{\lambda}(X^n)$ (resp. $T_{\lambda}(X^n)$) the connected component of $\Gr_{G,X^n}^+$ (resp. $\Gr_{G,X^n}^{-}$) corresponding to $C_{\lambda}(X^2)$. (See \ref{basic notation Braden} for the notation). 
	The fiber of $S_{\lambda}(X^2)$ (resp. $T_{\lambda}(X^2)$) at $x=(x,x)\in \Delta(X)\subset X^2$ is isomorphic to $S_{\lambda}$ (resp. $T_{\lambda}$) and its fiber at $x=(x,y)\in X^2-\Delta(X)$ is isomorphic to $\prod_{\lambda_1+\lambda_2=\lambda}S_{\lambda_1}\times S_{\lambda_2}$ (resp. $\prod_{\lambda_1+\lambda_2=\lambda}T_{\lambda_1}\times T_{\lambda_2}$).
\end{secnumber}

\begin{secnumber}
	\textit{Proof of proposition \ref{tensor CT}}. 
	Let $\mathcal{A}_1,\mathcal{A}_2$ be two objects of $\Sat_{G}$, $\mathcal{A}_{1,X},\mathcal{A}_{2,X}$ their extensions to $\Gr_{G,X}$ \eqref{Sat G to X} and $\mathcal{B}=\mathcal{A}_{1,X}\boxast \mathcal{A}_{2,X}$. 
	Consider the following diagram of ind-schemes:
	\begin{equation} \label{stratification S lambda curve}
		\xymatrix{
			S_{\lambda}(X^2)\ar[r]^{j} \ar@/_1pc/[rr]_{i_{\lambda}} & S_{\le \lambda}(X^2) \ar[r]^{\overline{i_{\lambda}}} & \Gr_{G,X^2} \ar[r]^{q^2}& X^2
		}
	\end{equation}
	
	For $i\in \mathbb{Z}$, we define the constructible module $\mathcal{L}^i_{\lambda}(\mathcal{A}_1,\mathcal{A}_2)$ on $X^2$ to be 
	\begin{equation}
		\mathcal{L}^i_{\lambda}(\mathcal{A}_1,\mathcal{A}_2)=
		\cH^i (q^2_{+}(i_{\lambda,!}(i_{\lambda}^{+}\mathcal{B}))) \simeq 
		\cH^i (q^2_{+}(i_{\lambda,+}'(i'^{!}_{\lambda}\mathcal{B}))),
	\end{equation}
	where the second isomorphism follows from Braden's theorem \eqref{Braden thm}. 
	By the base change, the stalk of $\mathcal{L}^i_{\lambda}(\mathcal{A}_1,\mathcal{A}_2)$ at a $k$-point $(x_1,x_2)$ of $X^2$ is isomorphic to
	\begin{equation} \label{Calcul fibers Lilambda}
		\mathcal{L}^i_{\lambda}(\mathcal{A}_1,\mathcal{A}_2)_{(x_1,x_2)}\simeq \left\{
	\begin{array}{ll}
		\rH_{\rc}^{i}(S_{\lambda},\mathcal{A}_1\star \mathcal{A}_2) & \text{if } x_1=x_2,\\
		\bigoplus_{\lambda_1+\lambda_2=\lambda} \rH_{\rc}^{i}(S_{\lambda_1}\times S_{\lambda_2}, \mathcal{A}_1\boxtimes \mathcal{A}_2) & \text{if } x_1\neq x_2.
\end{array} \right.
	\end{equation}
	By \ref{decomposition CT}, $\mathcal{L}^i_{\lambda}(\mathcal{A}_1,\mathcal{A}_2)$ vanishes unless $i=2\rho(\lambda)$ and we deduce from the Künneth formula that
	\begin{equation} \label{Kunneth formula}
		\rH_{\rc}^{2\rho(\lambda)}(S_{\lambda_1}\times S_{\lambda_2}, \mathcal{A}_1\boxtimes \mathcal{A}_2) \simeq  
		\rH_{\rc}^{2\rho(\lambda_1)}(S_{\lambda_1}, \mathcal{A}_1)\otimes \rH_{\rc}^{2\rho(\lambda_2)}(S_{\lambda_1}, \mathcal{A}_2).
	\end{equation}
	
	The adjunction morphisms $\id\to \overline{i}_{\lambda,+}\overline{i}_{\lambda}^{+}$ and $j_{!}j^{+}\to \id$ \eqref{stratification S lambda curve} induce canonical morphisms
	\begin{equation} \label{projection coh}
		\cH^{2\rho(\lambda)}(q_+^2(\mathcal{A}_{1,X}\boxast \mathcal{A}_{2,X})) \twoheadrightarrow \cH^{2\rho(\lambda)}( (q^2\circ \overline{i}_{\lambda})_+ \overline{i}_{\lambda}^+ \mathcal{B}) \xleftarrow{\sim} 
		\mathcal{L}^{2\rho(\lambda)}(\mathcal{A}_1,\mathcal{A}_2),
	\end{equation}
	where the first arrow is an epimorphism and the second arrow is an isomorphism in view of the calculation of their fibers \eqref{decomposition CT}. 
	
	By Braden's theorem and a dual argument for $T_{\lambda}(X^2)$, we obtain a section of \eqref{projection coh}:
	\begin{displaymath}
		\mathcal{L}^{2\rho(\lambda)}(\mathcal{A}_1,\mathcal{A}_2) \to 
		\cH^{2\rho(\lambda)}(q_+^2(\mathcal{A}_{1,X}\boxast \mathcal{A}_{2,X})).
	\end{displaymath}
	
	In view of proposition \ref{decomposition CT}, we deduce a decomposition 
	\begin{equation} \label{decomposition CT X2}
		\cH^{i}(q_+^2(\mathcal{A}_{1,X}\boxast \mathcal{A}_{2,X})\simeq 
		\bigoplus_{2\rho(\lambda)=i} \mathcal{L}_{\lambda}^{i}(\mathcal{A}_1,\mathcal{A}_2).
	\end{equation}
	The left hand side is a constant module with value $\rH^{i}(\Gr,\mathcal{A}_1\star \mathcal{A}_2)$ by \eqref{constant module}, \eqref{q_+ boxtimes}.  
	Then each summand $\mathcal{L}_{\lambda}^{i}(\mathcal{A}_1,\mathcal{A}_2)$ is also a constant module. 
	Hence fibers of $\mathcal{L}_{\lambda}^{i}(\mathcal{A}_1,\mathcal{A}_2)$ \eqref{Calcul fibers Lilambda}, \eqref{Kunneth formula} are isomorphic. 
	The proposition follows. \hfill$\qed$
\end{secnumber}

\begin{secnumber} \label{modify constraint}
	We modify the constraints $c'_{\mathcal{A}_1,\mathcal{A}_2}$ \eqref{constraint c'} by a sign as follows (see \cite{MV} after Remark 6.2) and make it compatible with the usual constraint $c_{\Vect}$ on $\Vect_{L}$ defined by $s\otimes t \mapsto t \otimes s$.

	The morphism $p:\cwX(T)\to \mathbb{Z}/2\mathbb{Z},~ \mu\mapsto (-1)^{2\rho(\mu)}$ defines a $\mathbb{Z}/2\mathbb{Z}$-grading on simple objects of $\Sat_G$. 
	By propositions \ref{geo Slammda} and \ref{decomposition CT}, we have 
	\begin{equation}
		\rH^i(\Gr,\IC_{\mu})=0,\quad \textnormal{if } i\neq 2\rho(\mu)~ (\textnormal{mod}~2).
		\label{coh parity}
	\end{equation}
	Given two simple objects $\mathcal{A}_1,\mathcal{A}_2$ of $\Sat_G$, we define a new constraint $c_{\mathcal{A}_1,\mathcal{A}_2}$ to be
	\begin{equation} \label{constraint c}
		c_{\mathcal{A}_1,\mathcal{A}_2}=(-1)^{p(\mathcal{A}_1)p(\mathcal{A}_2)} c'_{\mathcal{A}_1,\mathcal{A}_2}.
	\end{equation}
	Since $\Sat_G$ is semisimple, the definition of $c_{\mathcal{A}_1,\mathcal{A}_2}$ extends to any pair $(\mathcal{A}_1,\mathcal{A}_2)$ of objects of $\Sat_G$. 
	By \eqref{commutative c'} and \eqref{coh parity}, the following diagram is commutative
	\begin{equation} \label{Hstar tensor}
	\xymatrix{
		\rH^{*}(\mathcal{A}_1\ast \mathcal{A}_2) \ar[r]^{c_{\mathcal{A}_1,\mathcal{A}_2}} \ar[d]^{\wr} & \rH^{*}(\mathcal{A}_2\ast \mathcal{A}_1) \ar[d]^{\wr} \\
		\rH^{*}(\mathcal{A}_1)\otimes \rH^{*}(\mathcal{A}_2) \ar[r]^{c_{\Vect}} & \rH^{*}(\mathcal{A}_2)\otimes \rH^{*}(\mathcal{A}_1),
}
\end{equation}
where the isomorphism $c_{\Vect}$ is the usual commutativity constraint on vector spaces, i.e $c_{\Vect}(v\otimes w)=w\otimes v$.
\end{secnumber}

\begin{prop}[\cite{Zhu16} 5.2.9] \label{sym monoidal}
	The monoidal category $\Sat_G$ equipped with the constraints $c$ forms a symmetric monoidal category. The functor $\rH^{*}$ \eqref{hypercoh Gr} is a tensor functor. 
\end{prop}
\begin{proof}
	We need to verify $c_{\mathcal{A}_1,\mathcal{A}_2}\circ c_{\mathcal{A}_2,\mathcal{A}_1}=\id$ and the hexagon axiom. 
	Since the functor $\rH^*$ is faithful, it suffices to prove these assertions after applying $\rH^*$. 
	By \eqref{Hstar tensor} and the fact that $c_{\Vect}^2=\id$, we deduce that $c_{\mathcal{A}_1,\mathcal{A}_2}\circ c_{\mathcal{A}_2,\mathcal{A}_1}=\id$.
	We verify the hexagon axiom in a similar way.
	The second assertion follows from corollary \ref{faithful exact mono} and \eqref{Hstar tensor}.
\end{proof}

\subsection{Tannakian structure and the Langlands dual group}
\label{Langlands dual subsec}
\begin{theorem}
	The symmetric monoidal category $(\Sat_{G},\IC_0,\ast,c)$ \eqref{sym monoidal}, equipped the hypercohomology functor $\rH^{*}$ \eqref{hypercoh Gr} forms a neutral Tannakian category over $L$. 
	\label{thm Tannakian cat}
\end{theorem}

We prove it in the same way as in (\cite{Zhu16} 5.2.9) using proposition \ref{decomposition CT}(ii). 

\begin{prop}
	The Tannakian group $\widetilde{G}=\Aut^{\otimes} \rH^{*}$ of the Tannakian category $\Sat_{G}$ is a connected reductive group scheme over $L$.
\end{prop}
\begin{proof}
	For $\mu_1,\mu_2\in \cwX(T)^+$, $\IC_{\mu_1}\star\IC_{\mu_2}$ is defined by direct image through the birational morphism 
	\begin{displaymath}
		\Gr_{\le \mu_1}\widetilde{\times}\Gr_{\le \mu_2}\to \Gr_{\le \mu_1+\mu_2}.
	\end{displaymath}
	Hence it is supported on $\Gr_{\le \mu_1+\mu_2}$ and is isomorphic to $L_{\Gr_{\mu_1+\mu_2}}[2\rho(\mu_1+\mu_2)]$ on $\Gr_{\mu_1+\mu_2}$.
	Then by decomposition theorem \eqref{decomposition thm}, $\IC_{\mu_1+\mu_2}$ is a direct summand of $\IC_{\mu_1}\star\IC_{\mu_2}$. 
	Hence the semisimple category $\Sat_G$ is generated by $\{\IC_{\mu_i}\}_{i\in I}$ with a finite set of generators of $\cwX(T)^+$.
	Then $\widetilde{G}$ is algebraic by (\cite{DM82} 2.20). 
	There is no tensor subcategory which contains only direct sums of finite collection of $\IC$'s. Then $\widetilde{G}$ is connected (\cite{DM82} 2.22).
	Finally, since $\Sat_G$ is semisimple, $\widetilde{G}$ is reductive (\cite{DM82} 2.23). 
\end{proof}

\begin{theorem}
	The reductive group $\widetilde{G}$ is the Langlands dual group of $G$ over $L$. More precisely, the root datum of $\widetilde{G}$ with respect to a maximal torus $\widetilde{T}$ is dual to that of $(G,T)$. 
	\label{thm Langlands dual}
\end{theorem}

Since $\Gr_{T,\red}$ is a discrete set of points indexed by $\cwX(T)$ \eqref{notation parabolic ind}, we have:

\begin{lemma} \label{Langlands dual torus}
	In the case $G=T$ is a torus, theorem \ref{thm Langlands dual} holds.
\end{lemma}

\begin{secnumber}
	We denote by $\CT:\Sat_{G}\to \Sat_{T}$ the functor
	\begin{equation}
		\mathcal{A}\mapsto (\rH_{\rc}^{*}(S_{\lambda},\mathcal{A}))_{\lambda\in \cwX(T)},
	\end{equation}
	and by $\rH^{*}_G$ (resp. $\rH^{*}_T$) the fiber functor of the Tannakian category $\Sat_G$ (resp. $\Sat_T$). 
	By \ref{decomposition CT}, there exists a canonical isomorphism of functors: 
	\begin{equation} \label{iso fiber functors}
		\rH^*_G\simeq \rH^*_T\circ\CT :\Sat_G\to \Vect_{L}.
	\end{equation}
	By \ref{tensor CT}, $\CT$ is a tensor functor and therefore induces a homomorphism:
	\begin{equation} \label{That to Gtilde} 
		\check{T}\simeq \widetilde{T} \to \widetilde{G}.
	\end{equation}
	By (\cite{DM82} 2.21(b)), $\check{T}$ is a closed sub-torus in $\widetilde{G}$.

Using \ref{decomposition CT}, we prove the following in the same way as in the $\ell$-adic case (cf. \cite{Zhu16} 5.3.17, \cite{BR} \S 9.1). 	
\end{secnumber}

\begin{lemma}
	The torus $\check{T}$ is a maximal torus of $\widetilde{G}$. 
	\label{lemma max torus}
\end{lemma}

\begin{secnumber}
	\textit{Proof of theorem \ref{thm Langlands dual}}. 
	We take a Borel subgroup $\widetilde{B}\subset \widetilde{G}$ containing $\check{T}$ such that $2\rho\in \wX(T)$ is a dominant coweight for the choice positive roots of $\widetilde{G}$ with respect to $\widetilde{B}$.
	Then we can show that the set of dominant weights $\wX(\check{T})^+$ with respect to $\widetilde{B}$ is equal to the set of dominant coweights $\cwX(T)^+$ of $T$ by proposition \ref{decomposition CT}(ii) (cf. \cite{Zhu16} 5.3, and \cite{BR} 9.5 for more details).  
	In particular, $\widetilde{B}$ is uniquely determined. 

	We denote by $\widetilde{Q}^{+}$ the semi-subgroup of $\wX(\check{T})$ generated by positive roots of $\widetilde{G}$. A weight $\lambda$ belongs to $\widetilde{Q}^{+}$ if and only if there exists a highest weight representation $V_{\mu}(=\rH^*(\IC_{\mu}))$ such that $\mu-\lambda$ is also a weight of $V_{\mu}$. 
	By proposition \ref{decomposition CT}, this is equivalent to $L_{\mu-\lambda}\in \Gr_{\le \mu}$, and equivalent to $\lambda$ being a sum of positive coroots of $G$.
	Therefore the semigroup $(Q^{\vee})^{+}\subset \cwX(T)=\wX(\check{T})$ generated by positive coroots of $G$ coincides with the semigroup $\widetilde{Q}^{+}$. 
	Then, the set of simple coroots of $G$ coincide with the set of simple roots of $\widetilde{G}$. 

	The theorem follows from the fact that a root datum is uniquely determined by the semigroup $(\wX)^{+}$ of dominant weights and the set $\Delta$ of simple roots. \hfill$\qed$
\end{secnumber}

\subsection{The full Langlands dual group} 

For our applications of the geometric Satake equivalence for arithmetic $\cD$-modules, it is important to consider the Frobenius structure on the Satake category. 
In this subsection, we study the full Langlands dual group constructed by the Satake category equipped with Frobenius structures. 

\begin{secnumber}
	We suppose that the geometric base tuple $\{k,R,K,L\}$ is underlying to an arithmetic base tuple $\{k,R,K,L,t,\sigma\}$, where $t$ is an integer (which may be different from the degree $s$ of $k$ over $\mathbb{F}_p$) and $\sigma$ is an automorphism of $L$ and extends a lifting of $t$-th Frobenius automorphism on $k$ to $K$ \eqref{basic def}. 
	
	The Frobenius pullback functor $F^*_{\Gr}:\Hol(\Gr/L)\xrightarrow{\sim} \Hol(\Gr/L)$ \eqref{base tuple} induces a $\sigma$-semi-linear equivalence of tensor categories $F^*_{\Gr}: \Sat_G\xrightarrow{\sim} \Sat_G$. 
	We denote by $\Fr\Sat_G$ the category of pairs $(X,\varphi)$ consisting of an object $X$ of $\Sat_G$ and a Frobenius structure $\varphi: F^*_{\Gr}X\xrightarrow{\sim} X$. Morphisms are morphisms of $\Sat_G$ compatible with $\varphi$ (cf. \cite{Abe18} 1.4.6). 
	We will show that $\Fr\Sat_G$ is a Tannakian category. 	
\end{secnumber}

\begin{secnumber} \label{general Tannakian Z}
	We first study some general constructions in the Tannakian formalism following \cite{RZ}.

	For $n\in \mathbb{Z}$, we denote abusively by $\sigma^n$ the equivalence of categories $(-)\otimes_{L,\sigma^n}L:\Vect_L\xrightarrow{\sim} \Vect_L$. 

	Let $(\mathcal{C},\omega)$ be a neutralized Tannakian over $L$. 
	We suppose that, for each $n\in \mathbb{Z}$, there exists a $\sigma^n$-semi-linear equivalence of tensor categories 
	\begin{displaymath}
		\tau_n:\mathcal{C}\to \mathcal{C}
	\end{displaymath}
	and an isomorphism of tensor functors $\alpha_n: \omega\circ \tau_n \xrightarrow{\sim} \sigma^n\circ \omega$. 
	For any pair $n,m\in \mathbb{Z}$, we suppose moreover that there exists an isomorphism of tensor functors $\varepsilon: \tau_m\circ \tau_n \simeq \tau_{m+n}$ such that 
	\begin{displaymath}
		(\id\circ \alpha_n) \circ ( \alpha_m \circ \id)= 
		\alpha_{m+n}\circ \omega(\varepsilon): 
		\omega\circ \tau_m\circ \tau_n \simeq \sigma^{m+n}\circ \omega.  
	\end{displaymath}
	Since $\omega$ is faithful, such an isomorphism $\varepsilon$ is unique. 
	
	Let $H$ be the Tannakian group of $(\mathcal{C},\omega)$. 
	The above structure defines a homomorphism
	\begin{equation} \label{action of Z on H}
		\iota:\mathbb{Z}\to \Aut(H(L)),
	\end{equation}
	by letting $\iota(n)$ send $h:\omega\to \omega$ to 
	\begin{displaymath}
		\omega \xrightarrow{\alpha_n^{-1}} \sigma^{-n}\circ \omega\circ \tau_n \xrightarrow{ h\circ \id} \sigma^{-n}\circ \omega\circ \tau_n \xrightarrow{\alpha_n} \omega.
	\end{displaymath}
	
	We define the category $\mathcal{C}^{\mathbb{Z}}$ of $\mathbb{Z}$-equivariant objects in $\mathcal{C}$ as follows. 
	An object $(X,\{c_n\}_{n\in \mathbb{Z}})$ consists of an object $X$ of $\mathcal{C}$ and isomorphisms $c_n:\tau_n(X)\xrightarrow{\sim} X$ satisfying cocycle conditions $c_{n+m}=c_n\circ \tau_n(c_m)$. 
	A morphism between $(X,\{c_n\}_{n\in \mathbb{Z}})$ and $(X',\{c'_n\}_{n\in \mathbb{Z}})$ is a morphism of $\mathcal{C}$ compatible with $c_n,c'_n$. 
\end{secnumber}

\begin{secnumber} \label{notation semilinear L}
	Let $\Gamma$ be an abstract group and $\varphi:\Gamma\to \mathbb{Z}$ a homomorphism.
	We say an action of $\Gamma$ on an $L$-vector space $V$ is \textit{$\sigma$-semi-linear (with respect to $\varphi$)} if it is additive and satisfies $\gamma(a v)=\sigma^{\varphi(\gamma)}(a) \gamma(v)$ for $\gamma\in \Gamma, a\in L$ and $v\in V$.  
	We denote by $\Rep_{L,\sigma}(\Gamma)$ the category of $\sigma$-semi-linear representations of $\Gamma$ on finite dimensional $L$-vector spaces. 

	We denote by $H(L)\rtimes \mathbb{Z}$ the semi-direct product of $H(L)$ and $\mathbb{Z}$ via $\iota$ \eqref{action of Z on H}. The short exact sequence $1\to H(L)\to H(L)\rtimes \mathbb{Z}\to \mathbb{Z} \to 1$ allows us to define the category $\Rep_{L,\sigma}(H(L)\rtimes \mathbb{Z})$. 
\end{secnumber}

\begin{prop} \label{fullyfaithful alg to L points}
	Let $H$ be a split reductive group over $L$, $\Rep_{L}(H)$ the category of algebraic representations of $H$ and $\Rep_{L}(H(L))$ the category of finite dimensional representations of the abstract group $H(L)$.  
	Then the following canonical functor is fully faithful:
\begin{displaymath}
	\Rep_L(H)\to \Rep_L(H(L)),\qquad \rho\mapsto \rho(L).
\end{displaymath}
\end{prop}

\begin{prop} \label{prop Tannakian Z}
	Keep the assumption and notation as above. 

	\textnormal{(i)} The category $\mathcal{C}^{\mathbb{Z}}$ is a Tannakian category over $L_0=L^{\sigma=1}$ neutralized by $\omega$ over $L$ \textnormal{(\cite{DM82} \S~3)}.

	\textnormal{(ii)} Suppose that the Tannakian group $H$ of $(\mathcal{C},\omega)$ is a split reductive group over $L$. Then $\omega$ induces an equivalence of tensor categories
	\begin{equation}
		\mathcal{C}^{\mathbb{Z}} \xrightarrow{\sim} \Rep_{L,\sigma}^{\circ}(H(L)\rtimes \mathbb{Z}),
		\label{equivalence CZ}
	\end{equation}
	where $\Rep_{L,\sigma}^{\circ}(H(L)\rtimes \mathbb{Z})$ is the full subcategory of $\Rep_{L,\sigma}(H(L)\rtimes \mathbb{Z})$ \eqref{notation semilinear L} consisting of representations whose restriction to $H(L)$ is algebraic. 
\end{prop}
\begin{proof}
	(i) We define a monoidal structure on $\mathcal{C}^{\mathbb{Z}}$ by letting
	\begin{displaymath}
		(X,\{c_n\})\otimes (X',\{c'_n\})= (X'',\{c''_{n}\}),
	\end{displaymath}
	where $X''=X\otimes X'$ and $c''_n$ is the composition
	\begin{displaymath}
		\tau_n(X'')\simeq \tau_n(X)\otimes \tau_n(X') \xrightarrow{c_n\otimes c'_n} X\otimes X'.
	\end{displaymath}
	This defines a structure of symmetric monoidal category on $\mathcal{C}^{\mathbb{Z}}$. 

	We apply (\cite{Del90} 2.5) to show that $(\mathcal{C}^{\Gamma},\otimes)$ is rigid. 
	Given an object $(X,\{c_n\})$ of $\mathcal{C}^{\mathbb{Z}}$, we denote by $X^{\vee}$ be the dual of $X$ in $\mathcal{C}$ and then we have $\tau_n(X^{\vee})\simeq \tau_n(X)^{\vee}$. For each $n$, we have an isomorphism
	\begin{displaymath}
		c_{n}^{\vee}: X^{\vee} \xrightarrow{\sim} (\tau_n(X))^{\vee} \simeq \tau_n(X^{\vee}).
	\end{displaymath}
	Then we define $(X^{\vee}, \{ (c_n^{\vee})^{-1}\})$ to be the dual of $(X,\{c_n\})$ in $\mathcal{C}^{\mathbb{Z}}$.
	In view of (\cite{DM82} 1.6.5), the evaluation and coevaluation morphisms of $X$ and of $\tau_n(X)$ are compatible via $\tau_n$. 
	Then we obtain the evaluation and the coevaluation morphisms of $(X,\{c_{n}\})$ in $\mathcal{C}^{\mathbb{Z}}$ satisfying the axiom of (\cite{Del90} 2.1.2). 
	Hence $\mathcal{C}^{\mathbb{Z}}$ is a rigid abelian tensor category. 

	Since $\tau_n$ is $\sigma^n$-semi-linear, we have $\End(\id_{\mathcal{C}^{\mathbb{Z}}})\simeq L_0$. 
	The forgetful tensor functor $\mathcal{C}^{\mathbb{Z}}\to \mathcal{C}$ is exact and faithful. 
	Hence the fiber functor $\omega$ of $\mathcal{C}$ defines a fiber functor $\omega: \mathcal{C}^{\Gamma}\to \Vect_{L}$ (\cite{DM82} 3.1). 
	Then the assertion follows from (\cite{Del90} 1.10-1.13, see also \cite{DM82} footnote 12). 

	(ii) It suffices to construct an equivalence of tensor categories 
	\begin{equation}
		\Psi: \Rep_{L}(H)^{\mathbb{Z}} \xrightarrow{\sim} \Rep_{L,\sigma}^{\circ}(H(L)\rtimes \mathbb{Z}).
		\label{Rep Z to Rep}
	\end{equation}
	Let $((V,\rho),\{c_n\})$ be an object of $\Rep_L(H)^{\mathbb{Z}}$. Then we define a representation $(V,\widetilde{\rho})$ of $\Rep_{L,\sigma}^{\circ}(H(L)\rtimes \mathbb{Z})$, for any element $(h,n)\in H(L)\rtimes \mathbb{Z}$, by letting $\widetilde{\rho}(h,n)$ to be the composition 
	\begin{equation} \label{def rho tilde}
		\sigma^n(\omega (V,\rho))\xrightarrow{\alpha_{n}^{-1}} 
		\omega(\tau_n(V,\rho)) \xrightarrow{h\circ \id} 
		\omega(\tau_n(V,\rho))\xrightarrow{c_n} 
		\omega(V,\rho).
	\end{equation}
	Using the cocycle condition, one checks that the above formula defines a representation. Then we obtain the functor $\Psi$ \eqref{Rep Z to Rep}. 
	By \ref{fullyfaithful alg to L points}, the canonical morphism
	\begin{displaymath}
		\Hom_{H}(\rho,\rho')\to \Hom_{H(L)}(\rho(L),\rho'(L))
	\end{displaymath}
	is bijective. In view of \eqref{def rho tilde}, we deduce that $\Psi$ is fully faithful. 
	We leave the verification of the essential surjectivity to readers. 
\end{proof}

\begin{secnumber}
	The Frobenius pullback functor $F^*_{\Gr}=F_{\Gr/k}^+\circ \sigma^*:\Sat_G\xrightarrow{\sim} \Sat_G$ satisfies $\rH^{*}\circ F^*_{\Gr}\simeq \sigma \circ \rH^{*}$. 
	We take for every integer $n$ the tensor equivalence $\tau_n$ on $\Sat_G$ to be $|n|$-th composition of $F^*_{\Gr}$ (or a quasi-inverse of $F^*_{\Gr}$ if $n<0$) \eqref{general Tannakian Z}. 
	These functors satisfy the assumption of \ref{general Tannakian Z}. 
	With the notation of \ref{general Tannakian Z}, $\Fr\Sat_G$ is equivalent to the category $\Sat_G^{\mathbb{Z}}$. 
	In this case, we obtain the following result by \ref{prop Tannakian Z}. 
\end{secnumber}

\begin{theorem} \label{Satake arith}
	\textnormal{(i)} The category $\Fr\Sat_G$ is a Tannakian category over $L_0$, neutralized by the fiber functor $\rH^{*}$ over $L$. 
	If $t=s$ and $\sigma=\id_L$, then $\Fr\Sat_G$ is a neutral Tannakian category. 

	\textnormal{(ii)} There exists a canonical equivalence of tensor categories 
	\begin{equation}
		\Fr\Sat_G\xrightarrow{\sim} \Rep_{L,\sigma}^{\circ}(\check{G}(L)\rtimes \mathbb{Z}),
	\end{equation}
	compatible with fiber functors. 
\end{theorem}

\begin{secnumber}
	We work with the arithmetic tuple $\mathfrak{T}_{F}=\{k,R,K,L,s,\id_{L}\}$ and we suppose there exists a square-root $p^{1/2}$ of $p$ in $L$. 
	This allows to define half Tate twist functor $(\frac{n}{2})$ for $n\in \mathbb{Z}$ by sending each object $\mathscr{M}\in \rD(X/L_{F})$, equipped with the Frobenius structure $\Phi$, to $(\mathscr{M},p^{-sn/2}\cdot \Phi)$. 
	
	For $\mu\in \cwX(T)$, we denote by $\IC_{\mu}^{\Weil}=j_{\mu,!+}(L_{\Gr_{\mu}})[2\rho(\mu)](\rho(\mu))$ the holonomic module in $\Fr\Sat_G$ with weight $0$, and by $\mathcal{S}$ the full subcategory of $\Fr\Sat_G$ consisting of direct sums of $\IC_{\mu}^{\Weil}$'s. 
	
	The category $\mathcal{S}$ is closed under the convolution on $\Fr\Sat_{G}$, i.e. $\IC_{\lambda}^{\Weil}\star \IC_{\mu}^{\Weil}$ is isomorphic to a direct sum of $\IC_{\nu}^{\Weil}$. 
	Indeed, by proposition \ref{decomposition CT}(ii), the Frobenius acts on the total cohomology $\rH^{*}(\IC_{\mu}^{\Weil})$ by a diagonalizable automorphism with eigenvalues $q^{n/2}$, $n\in \mathbb{Z}$. 
	Since $\rH^{*}$ is compatible with Frobenius structure \eqref{coh functor Frob}, so is the Frobenius action on $\rH^{*}(\IC_{\lambda}^{\Weil}\star\IC_{\mu}^{\Weil})$. 
	We have a decomposition $\IC_{\lambda}\star\IC_{\mu}\simeq \oplus \IC_{\nu}$. Then the claim follows from the fact that the the action of Frobenius on cohomology determines the isomorphism class of an object of $\Fr\Sat_G$ whose underlying holonomic module is isomorphic to a direct sum of $\IC_{\nu}$'s.

	The canonical functor $\Fr\Sat_G\to \Sat_{G}$ induces an equivalence of tensor categories $\mathcal{S}\xrightarrow{\sim} \Sat_G$.
	In particular, we obtain equivalences of tensor categories
	\begin{equation}
		\Sat: \Rep_L(\cG)\simeq \Sat_G \simeq \mathcal{S}.
		\label{Satake wt 0}
	\end{equation}
\end{secnumber}

\begin{secnumber}
	In the end, we briefly review the action of outer automorphism group of $G$ on the Satake category  $\Sat_G$ (resp. $\mathcal{S}$).
	
	Let $(\mathcal{C},\omega)$ be a Tannakian category over $L$ and $H$ the associated Tannakian group. 
	We denote by $\Aut^{\otimes}(\mathcal{C},\omega)$ the set of isomorphism classes of pairs $(\tau,\alpha)$ of a tensor equivalence $\tau:\mathcal{C}\xrightarrow{\sim} \mathcal{C}$ and an isomorphism of functors $\alpha:\omega \xrightarrow{\sim} \omega \circ \tau$. 
	This set has a natural group structure. A similar construction as in \ref{general Tannakian Z} defines a canonical morphism $\Aut^{\otimes}(\mathcal{C},\omega)\to \Aut(H)$, which is an isomorphism (\cite{HNY} lemma B.1). We apply this to the Satake category $\mathcal{S}$ (or $\Sat_{G}$) equipped with the fiber functor $\rH^{*}$. The action of $\Aut(G)$ on $\Gr_G$ induces an action on $(\mathcal{S},\rH^{*})$, and therefore an action of $\Aut(G)$ on $\cG$, i.e. a homomorphism $\iota: \Aut(G)\to\Aut(\cG)$.
\end{secnumber}

\begin{lemma}\label{action Aut Satake}
        There is a natural pinning $(\check{B},\check{T}, N)$ of $\cG$ such that that map $\iota$ factors as $\Aut(G)\twoheadrightarrow \Out(G)\xrightarrow{\sim} \Aut^{\dagger}(\cG,\check{B},\check{T}, N)\subset\Aut(\cG)$.
\end{lemma}

The lemma can be shown in the same way as in (\cite{HNY} lemma B.2 or \cite{RZ} lemma A.6).
In particular, for $\sigma\in \Aut(G)$ and $V\in \Rep(\cG)$, we have $\sigma^*\Sat(V)\simeq \Sat( \iota(\sigma) V)$.

\section{Bessel $F$-isocrystals for reductive groups} \label{Bessel isoc}
        In this section, we construct Bessel $F$-isocrystals for reductive groups and calculate their monodromy groups. 
        We use notations from \ref{basic notation}, with $k$ being a finite field of $q=p^s$ elements.
	We assume moreover that there exists an element $\pi\in K$ satisfying $\pi^{p-1}=-p$ and a square root of $p$ in $K$. 
	We fix an arithmetic base tuple $\{k=\mathbb{F}_q, R, K, L, s, \id_{L}\}$ \eqref{basic def} and an isomorphism $\overline{K}\simeq \mathbb{C}$ (in order to talk about weight). 
	
	We fix $\{0,\infty\}\subset \P1$ (over some base that we will specify in each subsection), and set $X=\P1-\{0,\infty\}$. Although $X\simeq\Gm$, it is more convenient to regard $X$ as an algebraic curve equipped with a simply transitive action of $\Gm$.
        
        Throughout this section, let $G$ be a split reductive group (over some base). We fix a Borel subgroup $B\subset G$ and a maximal torus $T\subset B$. Let $U\subset B$ be the unipotent radical of $B$, and $U^{\op}\subset B^{\op}$ the opposite Borel and its unipotent radical. Let $T_{\ad}\subset B_{\ad}\subset G_{\ad}$ denote the quotients of $T\subset B\subset G$ by the center $Z(G)$ of $G$.
	We denote by $(\cG,\cB,\cT)$ the Langlands dual group of $G$ over $L$, constructed by the geometric Satake equivalence \eqref{Langlands dual subsec}.

\subsection{Kloosterman $F$-isocrystals for reductive groups} \label{construction HNY} 

	In this subsection, we follow the method of Heinloth-Ng\^o-Yun \cite{HNY} to produce overconvergent $F$-isocrystals on $X$ by applying the geometric Langlands correspondence. 

	We work with schemes over $k$. 
	We will consider with both geometric coefficients and arithmetic coefficients, but for simplicity, 
	we omit $L_{\blacktriangle}$ from the notation $\Hol(-/L_{\blacktriangle}),\rD(-/L_{\blacktriangle})$ and $L$ from $\Rep_L(-)$. 

\begin{secnumber}
	Let $\mathcal{G}=G\times \P1$. 
	For a coordinate $x$ on $\P1$, so $y=x^{-1}$ is a local coordinate around $\infty$, we denote by
	\begin{eqnarray*}
		I(0)&=& \{g\in G(k\llbracket y\rrbracket)\mid g(0)\in B\} \quad \textnormal{the Iwahori subgroup}, \\
		I(1)&=& \{g\in G(k\llbracket y\rrbracket)\mid g(0)\in U\} \quad \textnormal{the unipotent radical of $I(0)$},\\
		Z(G)(1)&=& \{g\in Z(G)(k\llbracket y\rrbracket)\mid g(0)\equiv 1 \mod y\},\\
		I(2)&=&Z(G)(1)[I(1),I(1)],\\
		I(i)^{\op}&\subset& G(k\llbracket x\rrbracket) \textnormal{ the analogous groups obtained by opposite Borel subgroup}.
	\end{eqnarray*}	
	If $G$ is semisimple, $I(2)=[I(1),I(1)]$. On the other hand, if $G$ is a torus, then $I(2)=I(1)$. (So our definition of $I(2)$ is slightly different from \cite{HNY} 1.2 when $G$ is not semisimple, but for $G=\GL_n$ coincides with the one in \cite{HNY} 3.1.) These groups are independent of the choice of $x$.
	
	By abuse of notations, we use the same notations for the corresponding (ind)-group schemes over $k$. Then 
	\begin{equation}\label{affine simple}
	        I(1)/I(2)\simeq \bigoplus_{\alpha \textnormal{ affine simple}} U_{\alpha}
	\end{equation} 
	where $U_{\alpha}(k)\subset G(k\llbracket s\rrbracket)$ is the root subgroup corresponding to $\alpha$. We also write
	\begin{eqnarray*}
	        \Omega&=&N_{G(k(\!(x)\!))}(I(0)^{\op})/I(0)^{\op},
        	\end{eqnarray*}
	which is regarded as a discrete group over $k$.

	We denote by $\Gmn$ the group scheme over $\P1$ such that (\cite{HNY} 1.2)
	\begin{displaymath}
		\Gmn|_{X} = G\times X,\quad \Gmn(\mathscr{O}_{0}) = I(m)^{\op}\subset G(\mathscr{O}_{0}), \quad \Gmn(\mathscr{O}_{\infty}) = I(n) \subset G(\mathscr{O}_{\infty}). 
	\end{displaymath}
	
	We denote by $\Bun_{\Gmn}$ the moduli stack of $\Gmn$-bundles on $\P1$. Let $\Bun_{\Gmn}^{0}$ denote its connected component containing the trivial $\Gmn$-bundle $\star: \Spec(k) \to \Bun_{\Gmn}$.
	For each $\gamma\in \Omega$, there is a canonical isomorphism $\Hk_{\gamma}:\Bun_{\mathcal{G}(0,n)}\simeq \Bun_{\mathcal{G}(0,n)}$ given by the Hecke modification of $\mathcal{G}(0,n)$-bundles at $0\in\P1$ corresponding to $\gamma$ (\cite{HNY} Corollary 1.2). This induces a canonical bijection between $\Omega$ and the set of connected components of $\Bun_{\mathcal{G}(0,n)}$ (and therefore all $\Bun_{\Gmn}$). Let $\Bun_{\Gmn}^{\gamma}$ denote the connected component corresponding to $\gamma$ under the bijection. For $\gamma\in\Omega$, let $i_{\gamma}=\Hk_{\gamma}(\star): \Spec(k)\to \Bun_{\mathcal{G}(0,n)}^{\gamma}$.
	
	There is also the action of $I(1)/I(2)$ on $\Bun_{\Gzt}$ by modifying $\Gzt$-bundles at $\infty$. Let
	\begin{equation}\label{big open cell}
	         j: \Omega\times I(1)/I(2)\to \Bun_{\Gzt},
	\end{equation}         
         be the open immersion of the big cell, defined by applying the action of $I(1)/I(2)\times\Omega$ to the trivial $\Gzt$-bundle (\cite{HNY} Corollary 1.3). Let $j_{\gamma}:I(1)/I(2)\to \Bun_{\Gzt}^{\gamma}$ 
	denote its restriction to the component corresponding to $\gamma$.
\end{secnumber}

\begin{secnumber}
	The stack of Hecke modifications of $\Gmn$-torsors (over $X$) is:
	\begin{displaymath}
		\Hke_{\Gmn}^{X}(S):= \bigg\{(\mathscr{E}_1,\mathscr{E}_2,x,\beta) \bigg| 
		\txt{	$\mathscr{E}_i\in \Bun_{\Gmn}(S)$, \quad $x:S\to X$\\
		$\beta:\mathscr{E}_1|_{X_S-\Gamma_{x}}\xrightarrow{\sim} \mathscr{E}_{2}|_{X_S-\Gamma_{x}}$ } \bigg\}.
	\end{displaymath}
	There exist natural morphisms
	\begin{equation} \label{def Hk opr}
		\xymatrix{
			& \Hke_{\Gmn}^{X} \ar[ld]_{\pr_1} \ar[rd]^{\pr_2} \ar[r]^{q} & X \\
			\Bun_{\Gmn} & & \Bun_{\Gmn}\times X,
		}
	\end{equation}
	where $\pr_1$ (resp. $\pr_2$, resp. $q$) sends $(\mathscr{E}_1,\mathscr{E}_2,x,\beta)$ to $\mathscr{E}_1$ (resp. $(\mathscr{E}_2,x)$, resp. $x$). 

	Following \cite{HNY}, we denote by $\GR$ the Beilinson-Drinfeld Grassmannian of $\Gmn$ with modifications on $X$.  Note that $\GR\simeq \Gr_{G,X}\simeq \Gr_G\times X$ and therefore is independent of $(m,n)$.
	There exists a smooth atlas $\varpi: U\to \Bun_{\Gmn}$ such that 
\begin{eqnarray}
	U\times_{\Bun_{\Gmn},\pr_1} \Hke^{X}_{\Gmn} \simeq U\times \GR,  \label{altas pr1}\\
	(U\times \Gm)\times_{(\Bun_{\Gmn}\times X),\pr_2} \Hke^{X}_{\Gmn} \simeq U\times \GR. \label{altas pr2} 
\end{eqnarray}
	
For $V\in \Rep(\cG)$, we associate a holonomic module $\Sat(V)$ on $\Gr_G$ by the geometric Satake equivalence \eqref{Satake wt 0}. 
We denote abusively by $\IC_{V}$ the holonomic module on $\Hke_{\Gmn}^{X}$ defined by smooth descent of $K_{U\times X}\boxtimes \Sat(V)$ on $U\times X \times \Gr_G$ (supported in a subscheme $U\times X\times \Gr_{G,V}$). Then $\IC_{V}$ is supported in a substack $\Hke_{\Gmn,V}^{X}$ of $\Hke_{\Gmn}^{X}$. 

	The geometric Hecke operators is defined as a functor 
	\begin{eqnarray} \label{def Hk op}
			\Hk: \Rep(\check{G}) \times \rD(\Bun_{\Gmn}) &\to& \rD(\Bun_{\Gmn} \times X), \\ 
			(V, \mathscr{M}) &\mapsto& \Hk_V(\mathscr{M}) :=\pr_{2,!} \bigl( \pr_{1,V}^{+}(\mathscr{M}) \otimes \IC_{V}\bigr). \nonumber
	\end{eqnarray}
	Here $\pr_{1,V}:\Hke_{\Gmn,V}^{X}\to \Bun_{\Gmn}$ and $\pr_{2}|_{\Hke_{\Gmn,V}^{X}}:\Hke_{\Gmn,V}^{X}\to \Bun_{\Gmn}\times X$ are schematic (\ref{altas pr1}, \ref{altas pr2}), which allows us to apply cohomological functors of $\pr_{1,V},\pr_2$ \eqref{hol mods stack}. 
	
	We call a tensor functor
	\begin{displaymath}
		E:\Rep(\cG)\to \Sm(X/L) \quad \textnormal{(resp. $\Sm(X/L_F)$)}
	\end{displaymath}
	\textit{$\cG$-valued overconvergent isocrystal (resp. $F$-isocrystal) $E$ on $X$}. 
	We denote by $E_V$ its value on $V\in \Rep(\cG)$. \textit{A Hecke eigen-module with eigenvalue} $E$ is a holonomic module $\mathscr{M}$ on $\Bun_{\Gmn}$ together with isomorphisms 
	\[
	      \Hk_V(\mathscr{M})\xrightarrow{\sim} \mathscr{M}\boxtimes E_V, \quad V\in \Rep(\cG),
	\]      
	which are compatible with tensor structure on $\Rep(\cG)$ and composition of Hecke operator.  We refer to \cite[5.4.2]{BD99} for the precise definition and detailed discussions.
\end{secnumber}

\begin{secnumber}
	We take a non-trivial additive character $\psi:\mathbb{F}_{p}\to K^{\times}$ and denote by $\pi\in K$ the associated element satisfying $\pi^{p-1}=-p$ \eqref{Dwork isocrystal}. Let $\mathscr{A}_{\psi}$ be the Dwork $F$-isocrystal on $\mathbb{A}^1$ \eqref{Dwork isocrystal}.
	
	We fix a generic linear function $\phi$ of $I(1)/I(2)$, that is, a homomorphism $\phi: I(1)/I(2)\to \mathbb{A}^1$ of algebraic group over $k$ whose restriction to each $U_{\alpha}$ is an isomorphism 
	\begin{equation}\label{phialpha}
	     \phi_{\alpha}:=\phi|_{U_{\alpha}}: U_{\alpha}\simeq \mathbb{A}^1.
	\end{equation} 	
	Let $\mathscr{A}_{\psi\phi}=\phi^{+}(\mathscr{A}_{\psi})$. 
        (Note that our notation is slightly abusive as this sheaf depends only on the character $\psi\circ\mathrm{tr}_{k/\mathbb F_p}\circ\phi$ of $I(1)/I(2)$ as a $p$-group).
	We denote by $\Hol(\Bun_{\Gzt})^{I(1)/I(2),\mathscr{A}_{\psi\phi}}$ the category of holonomic modules on $\Bun_{\Gzt}$ which are $(I(1)/I(2),\mathscr{A}_{\psi\phi})$-equivariant.

	By repeating the argument of (\cite{HNY} 2.3), we obtain a parallel result for holonomic modules. 	
\end{secnumber}

\begin{lemma}[\cite{HNY} 2.3] \label{lemma clean}
	\textnormal{(i)} The canonical morphism $j_{\gamma,!}(\mathscr{A}_{\psi\phi})\xrightarrow{\sim}j_{\gamma,+}(\mathscr{A}_{\psi\phi})$ is an isomorphism.

	\textnormal{(ii)} The functor 
	\begin{eqnarray*}
		\Hol(X) &\to & \Hol(\Bun^{\gamma}_{\Gzt}\times X)^{I(1)/I(2),\mathscr{A}_{\psi\phi}} \\
		\mathscr{M} & \mapsto & j_{\gamma,!}(\mathscr{A}_{\psi\phi})\boxtimes \mathscr{M}
	\end{eqnarray*}
	is an equivalence of categories, with 
	a quasi-inverse given by 
	\begin{displaymath}
		\mathscr{N}\mapsto (i_{\gamma}\times \id_{X})^+(\mathscr{N}) \simeq 
		(i_{\gamma}\times \id_{X})^!(\mathscr{N}).
	\end{displaymath}
\end{lemma}

We denote by $A_{\psi\phi}$ the object of $\Hol(\Bun_{\Gzt})^{I(1),\mathscr{A}_{\psi\phi}}$ defined by $(j_{\gamma,!}(\mathscr{A}_{\psi\phi})[\dim \Bun_{\Gzt}])_{\gamma\in \Omega}$.

\begin{theorem} \label{thm eigensheaves}
	\textnormal{(i)} For $(m,n)=(0,2)$, the holonomic module $A_{\psi\phi}$ \eqref{lemma clean} is a Hecke eigen-module with Hecke eigenvalue a $\cG$-valued overconvergent $F$-isocrystal
\begin{equation}
	\Kl^{\rig}_{\check{G}}(\psi\phi):\Rep(\check{G})\to \Sm(X/L_{F}).
	\label{Kl rig}
\end{equation}

\textnormal{(ii)} For every representation $V$ of $\cG$, $\Kl_{\cG,V}^{\rig}(\psi\phi)$  is pure of weight zero.
\end{theorem}

If $\psi$ (resp. $\psi$ and $\phi$) is clear from the context, we simply write $\Kl^{\rig}_{\check{G}}(\psi\phi)$ by $\Kl^{\rig}_{\cG}(\phi)$ (resp. $\Kl^{\rig}_{\cG}$).
In the remainder of this section, we prove the above theorem by repeating the strategy in the $\ell$-adic case, following \cite{HNY}.
The first step is to show holonomicity. 

\begin{lemma}
	For every $V\in \Rep(\check{G})$, the complex $\Hk_V(A_{\psi\phi})[1]$ is holonomic.  
	\label{lemma hol Hk}
\end{lemma}
\begin{proof}
For $\gamma\in \Omega$, we denote by $j_{\gamma}:I(1)/I(2)\to \Bun_{\Gzt}$ the open immersion and by $j_{\gamma}':\pr_{1,V}^{-1}(I(1)/I(2))\to \Hke_{\Gzt}^{X}$ the base change of $j_{\gamma}$ to Hecke stack. Then the restriction of $\pr_2$ to $\pr_{1,V}^{-1}(I(1)/I(2))$
\begin{displaymath}
	\pr_2:\pr_{1,V}^{-1}(I(1)/I(2)) \to \Bun_{\Gzt}\times X
\end{displaymath}
is affine (\cite{HNY} remark 4.2). 
We claim that the canonical morphism
\begin{equation}
	j'_{\gamma,!}(\pr_{1,V}^+(\mathscr{A}_{\psi\phi})\otimes\IC_{V}) \xrightarrow{\sim}j'_{\gamma,+}(\pr_{1,V}^+(\mathscr{A}_{\psi\phi})\otimes\IC_{V})
	\label{cleanness on Hk}
\end{equation}
is an isomorphism. Indeed, since both $j_{\gamma,!}$ and $j_{\gamma,+}$ commute with smooth base change, it suffices to show the isomorphism after taking inverse image to $U\times \GR$ \eqref{altas pr1}. In this case, the morphism $\pr_{1,V}$ corresponds to the projection $U\times \GR\to U$. 
Then the isomorphism \eqref{cleanness on Hk} follows from $j_{\gamma,!}(\mathscr{A}_{\psi\phi})\xrightarrow{\sim}j_{\gamma,+}(\mathscr{A}_{\psi\phi})$ \eqref{lemma clean}. 

Then we deduce that $j'_{\gamma,!}(\pr_{1,V}^+(\mathscr{A}_{\psi\phi})\otimes\IC_V)[1]$ is holonomic and we have
\begin{eqnarray} \label{cleanness HkV}
	\Hk_V(A_{\psi\phi})|_{\Bun_{\Gzt}^{\gamma}\times X} &\simeq & (\pr_2 \circ j'_{\gamma})_!(\pr_{1,V}^+(\mathscr{A}_{\psi\phi})\otimes \IC_V) \\
	&\simeq & (\pr_2 \circ j'_{\gamma})_+(\pr_{1,V}^+(\mathscr{A}_{\psi\phi})\otimes \IC_V). \nonumber
\end{eqnarray}
Since $(\pr_2\circ j'_{\gamma})$ is affine, $(\pr_2\circ j'_{\gamma})_+$ (resp. $(\pr_2\circ j'_{\gamma})_!$) is right (resp. left) exact (\cite{AC18} 1.3.13). Then the assertion follows. 
\end{proof}

\begin{secnumber}
	\textit{Proof of \ref{thm eigensheaves}}(i). 
	The action of $I(1)/I(2)$ on $\Bun_{\Gzt}$ extends to an action on the diagram \eqref{def Hk opr}. 
	For each $\gamma\in \Omega$, $\Hk_V(A_{\psi\phi})|_{\Bun_{\Gzt}^{\gamma}\times X}$ is $(I(1)/I(2),\mathscr{A}_{\psi\phi})$-equivariant. 
	By \ref{lemma clean}, for each $\gamma\in \Omega$, we have
	\begin{displaymath}
		\Hk_V(A_{\psi\phi})|_{\Bun_{\Gzt}^{\gamma}\times X}\simeq A_{\psi\phi}^{\gamma}\boxtimes E_{V}^{\gamma},
	\end{displaymath}
	where $E_{V}^{\gamma}[1]$ is a holonomic module on $X$. 
	By the same argument as in (\cite{HNY} 4.2), we show that $E_{V}^{\gamma}$ is canonically isomorphic to $E_{V}^{0}$. So we will drop the index $\gamma$ in the following.

	Since $\IC_V$ is ULA with respect to the projection $\GR\simeq \Gr_{G,X}\to X$ \eqref{globalisation functor}, we have $\Phi(\IC_V)=0$ \eqref{ULA vanishing cycle}. 
	Since taking vanishing cycle functor commutes with smooth pull-back and proper push-forward (\cite{Abe18II} 2.6), we deduce that
	\begin{displaymath}
		A_{\psi\phi}\boxtimes \Phi(E_V) \simeq \Phi(A_{\psi\phi}\boxtimes E_V)\simeq \pr_{2,!}(\Phi(\pr_{1,V}^+(A_{\psi\phi})\otimes \IC_V)) \simeq \pr_{2,!}(\pr_{1,V}^+(A_{\psi\phi})\otimes \Phi(\IC_V))=0.
	\end{displaymath}
	By \ref{ULA id local system}, $E_V$ is smooth. 
	Then the assertion follows. \hfill\qed
\end{secnumber}

\begin{secnumber} \label{geo picture Kl}
	\textit{Proof of \ref{thm eigensheaves}}(ii). 
	In the following, we present a concrete way to calculate the Hecke eigenvalue.
	
	We denote by $\star\in \Bun_{\Gzt}$ the base point corresponding to the trivial bundle $\Gzt$. The base change of convolution diagram \eqref{def Hk opr} to $\star\times X$ can be written as
	\begin{equation}\label{Hk vs BD}
		\xymatrix{
			& \GR \ar[ld]_{p_1} \ar[rd]^{p_2} & \\
			\Bun_{\Gzt} & & X.
		}
	\end{equation}
	We denote by $\GR_{V}\subset \GR\simeq \Gr\times X$ the support of $\Sat(V)\boxtimes L_{X}$,
	by $\GR^{\circ}$ the inverse image of the big cell $j(I(1)/I(2)\times\Omega)$ by $p_1$, and by $\GR_{V}^{\circ}=\GR_{V}\cap \GR^{\circ}$. Consider the following diagram: 
	\begin{equation} \label{triangle}
		\xymatrix{
			& & \GR_{V}^{\circ}  \ar[ld]_{p_{1,V}^{\circ}} \ar@{^{(}->}[r]_-{j'} \ar@/^3pc/[rrd]^{p_2^{\circ}} & \GR_{V} \ar[ld]^{p_{1,V}} \ar[rd]_{p_2} & \\
			\mathbb{A}^1 & I(1)/I(2)\times\Omega \ar[l]_-{\phi} \ar@{^{(}->}[r]^-{j} & \Bun_{\Gzt} & & X.
		}
	\end{equation}
	By the base change and \eqref{cleanness HkV}, we have
	\begin{equation}
		E_{V}\simeq p_{2,!}^{\circ}(p_{1,V}^{\circ,+}(\mathscr{A}_{\psi\phi})\otimes\IC_V|_{\GR^{\circ}}).
	\label{cal Hk eigenvalue}
	\end{equation}

	By cleanness \eqref{cleanness on Hk}, $E_V$ can be calculated by either $+$ or $!$ pushforward. More precisely, the following canonical morphism is an isomorphism
	\begin{equation} \label{cleanness on GR}
		p_{2,!}^{\circ}(p_{1,V}^{\circ,+}(\mathscr{A}_{\psi\phi})\otimes\IC_V|_{\GR^{\circ}}) \xrightarrow{\sim}
		p_{2,+}^{\circ}(p_{1,V}^{\circ,+}(\mathscr{A}_{\psi\phi})\otimes\IC_V|_{\GR^{\circ}}).
	\end{equation}
	In particular, the overconvergent $F$-isocrystal $E_{V}$ is pure of weight zero. Theorem \ref{thm eigensheaves}(ii) follows. \hfill$\qed$
\end{secnumber}

\begin{secnumber} \label{triv functoriality}
        There is the following ``trivial'' functoriality between Kloosterman $F$-isocrystals. We fix $\psi$. Let $G'\to G$ be a homomorphism of reductive groups induces the same adjoint quotient $G'_{\mathrm{ad}}\simeq G_{\mathrm{ad}}$. 
        Then it induces an isomorphism $I'(1)/I'(2)\simeq I(1)/I(2)$, and therefore we can abusively use the notation $\phi$ to denote the ``same'' linear functions on these spaces under the identification.
        On the other hand, it induces a homomorphism of dual groups $\cG\to \cG'$ and therefore a tensor functor $\mathrm{Res}: \Rep(\cG')\to \Rep(\cG)$ by restrictions. Then $\Kl^{\rig}_{\cG'}$ is the push-out of $\Kl^{\rig}_{\cG}$ along $\cG\to \cG'$. Concretely, this means that there is a canonical isomorphism of tensor functors (we omit both $\psi$ and $\phi$ from the notations)
        \[
              \Kl^{\rig}_{\cG'}\simeq  \Kl^{\rig}_{\cG}\circ \mathrm{Res}:\Rep(\cG')\to \Sm(X/L_F)
        \]      
        This allows use to reduce certain questions of $\Kl_{\cG}^{\rig}$ to the case when $\cG$ is simply-connected. We also obtain the following exceptional isomorphisms (due to coincidences of Dynkin diagrams in low rank cases)
         \begin{eqnarray}             
               \Kl_{\SL_2,\Sym^2}^{\rig} & \simeq &\Kl_{\SO_3,\Std}^{\rig},  \label{SL2 vs SO3}  \\
               \Kl_{\Sp_4,\ker(\wedge^2\to \mathbf{1})}^{\rig} & \simeq & \Kl_{\SO_5,\Std}^{\rig}, \label{Sp4 vs SO5} \\               
	       \Kl_{\SO_4,\Std}^{\rig} & \simeq & \Kl_{\SL_2\times\SL_2,\Std\boxtimes\Std}^{\rig}, \label{SO4 vs SL2}\\
                \Kl_{\SO_6,\Std}^{\rig}& \simeq &\Kl_{\SL_4,\wedge^2}^{\rig}, \label{SO6 vs SL4}
         \end{eqnarray}  
         where $\mathbf{1}$ denotes the trivial representation, $\Std$ the standard representation, $\Sym^\bullet$ and $\wedge^\bullet$ the symmetric powers and wedge powers of the standard representation.
\end{secnumber}

\begin{secnumber}\label{symmetry of Kl} 
         There is a natural action of $\Gm$ on $X\subset \P1$. On the other hand, the group of automorphisms $\Aut(G,B,T)$ acts on $\Gmn$.
         It follows that $\Gm\times\Aut(G,B,T)$ acts on \eqref{def Hk opr}, and therefore on \eqref{Hk vs BD}. It also acts on $I(1)/I(2)\times\Omega$ as group automorphisms such that the open embedding \eqref{big open cell} is $\Gm\times\Aut(G,B,T)$-equivariant.  
	Recall that the natural action of $\Aut(G)$ on the Satake category induces a homomorphism $\iota:\Aut(G)\to \Aut(\cG,\cB,\cT,N)$ \eqref{action Aut Satake}.
	Given $\delta=(a,\sigma)\in (\Gm\times \Aut(G,B,T))(k)$ and $V\in \Rep(\cG)$, 
	then there is a canonical isomorphism 
        \begin{equation}\label{Aut of diag}
		\Kl^{\rig}_{\cG,V}(\psi(\phi\circ\delta))\simeq a^+\Kl^{\rig}_{\cG, \iota(\sigma^{-1}) V}(\psi\phi),
         \end{equation}
         given by the composition
         \begin{eqnarray*} 
                             p_{2,!}(p_{1,V}^+( j_!(\phi\circ\delta)^+\mathscr{A}_{\psi})\otimes \IC_V)            
			    & \simeq & p_{2,!}(\delta^+p_{1,V}^+( j_!\phi^+\mathscr{A}_{\psi})\otimes \IC_V)\\                                                                
			    & \simeq &  a^+p_{2,!}(p_{1,V}^+( j_!\phi^+\mathscr{A}_{\psi})\otimes (\delta^{-1})^+\IC_V) \\
			    & \simeq &  a^+p_{2,!}(p_{1,V}^+( j_!\phi^+\mathscr{A}_{\psi})\otimes (\IC_{\iota(\sigma^{-1}) V})). 
         \end{eqnarray*}
         
         In particular, given $t\in T_{\mathrm{ad}}(k)\subset \Aut(G,B,T)$, the element $\delta=(1,t)$ induces an isomorphism
         \begin{equation}\label{dependence phi}
               \Kl_{\cG}^{\rig}(\psi(\phi\circ \delta))\simeq \Kl_{\cG}^{\rig}(\psi\phi).
         \end{equation}       
         That is, $\Kl_{\cG}^{\rig}(\psi\phi)$ depends only on the $T_{\ad}$-orbit of $\phi$. 
         On the other hand, let $a$ be an element of $\Gm(k)$, $\psi_a$ the additive character defined by $\psi_a(-)=\psi(a-)$, $t_a\in T_{\mathrm{ad}}$ the unique element such that $\alpha(t_a)=a$ for every simple root $\alpha$ of $G$ and $h$ the Coxeter number of $G$. 
	 By applying $\delta=(a^h,t_a)$ in \eqref{Aut of diag}, we deduce that 
         \begin{equation}\label{dependence psi}
                  \Kl_{\cG}(\psi_a\phi) \simeq \Kl_{\cG}(\psi(\phi\circ\delta)) \simeq (a^h)^+\Kl_{\cG}(\psi\phi).
         \end{equation}
         
	 In addition, given a generic linear function $\phi$ of $I(1)/I(2)$, the collection $\{\phi_{\alpha}\}$ from \eqref{phialpha} for those $\alpha$ being simple roots of $G$, provide a pinning of $(G,B,T)$, and therefore induces a splitting $\mathrm{Out}(G)\to \Aut(G,B,T)$.         
         If $G$ is almost simple, not of type $A_{2n}$, then every element $\sigma\in \mathrm{Out}(G)$ fixes the remaining $\phi_{\alpha}$. 
	 If $G$ is of type $A_{2n}$, the unique non-trivial element $\sigma_0\in \mathrm{Out}(G)$ send the remaining $\phi_{\alpha}$ to $-\phi_{\alpha}$.
         Therefore, if either $\cG$ is almost simple not of type $A_{2n}$, or if $p=2$, then for every $\sigma\in \mathrm{Out}(G)$, we have $\phi\circ(1,\sigma)=\phi$ and a canonical isomorphism
         \begin{equation}\label{inv outer automorphism} 
		 \Kl_{\cG,V}^{\rig}(\psi\phi)\simeq \Kl_{\cG,\iota(\sigma^{-1}) V}^{\rig}(\psi\phi), \quad 
         \end{equation}       
         compatible with the tensor structures. On the other hand, if $G$ is almost simple of $A_{2n}$ and if $p>2$, then the element $\delta=(-1,\sigma_0)$ induces a canonical isomorphism \eqref{Aut of diag}
         \begin{equation}\label{inv outer automorphism A2n} 
		 \Kl_{\cG,V}^{\rig}(\psi\phi)\simeq (-1)^+\Kl_{\cG,V^{\vee}}^{\rig}(\psi\phi),
         \end{equation}  
	 where $V^{\vee}$ denotes the dual representation of $V$,
	 compatible with the tensor structures.
         \end{secnumber}

\begin{secnumber}
	There is a variant with multiplicative characters, which slightly generalizes $A_{\psi\phi}$. Note that $T\simeq I(0)^{\op}/I(1)^{\op}$. Let 
	\[
	    \widetilde T= N_{G(k(\!(x)\!))}(I(0)^{\op})/I(1)^{\op},
	 \]   
	which fits into the exact sequence $1\to T\to \widetilde T\to \Omega\to 1$. The group $\widetilde T$ acts on $\Bun_{\Got}$ by modifying $\Got$-bundles at $0$. 
	Then the analogue of \eqref{big open cell} in this case is the open embedding
	\[
	   j:\widetilde T\times I(1)/I(2)\to \Bun_{\Got}.
	\]
	We choose a splitting $s: T\rtimes \Omega\xrightarrow{\sim}\widetilde{T}$.
	For $\gamma\in \Omega$, $j$ sends $T\times \gamma \times I(1)/I(2)$ to the connected component $\Bun_{\Got}^{\gamma}$ and we denote it by
	\begin{displaymath}
		j_{\gamma}: T\times \gamma \times I(1)/I(2) \to \Bun_{\Got}^{\gamma}.
	\end{displaymath}
	A character $\chi:T(k)\to K^{\times}$ defines a rank one overconvergent $F$-isocrystal $\mathscr{K}_{\chi}$ on the torus $T$ (cf. \ref{Dwork isocrystal}(ii)). 
	If $\chi^{\gamma}:T(k)\to K^{\times}$ denotes the character defined by $\chi^{\gamma}(t)=\chi(\Ad_{\gamma}(t))$, then lemma \ref{lemma clean} also holds for $(T\times I(1)/I(2), \mathscr{K}_{\chi^{\gamma}}\boxtimes A_{\varphi})$-equivariant holonomic modules on $\Bun_{\Got}^{\gamma}$.

	We denote by $A_{\psi\phi,\chi,s}$ the holonomic module on $\Bun_{\Got}$ defined by $(j_{\gamma,!}(\mathscr{K}_{\chi^{\gamma}}\boxtimes A_{\psi\phi})[\dim \Bun_{\Got}])_{\gamma\in \Omega}$. By replacing $(0,2)$ by $(1,2)$ in theorem \ref{thm eigensheaves} and repeating the arguments, we obtain a $\cG$-valued overconvergent $F$-isocrystal 
	\[
	     \Kl_{\cG}^{\rig}(\psi\phi,\chi,s):\Rep(\check{G})\to \Sm(X/L_{F}),
	\]
such that for every representation $V$ of $\cG$, $\Kl_{\cG,V}^{\rig}(\psi\phi,\chi)$ is pure of weight zero. Note that by \eqref{cal Hk eigenvalue}, $\Kl_{\cG}^{\rig}(\psi\phi)=\Kl_{\cG}^{\rig}(\psi\phi,\mathbf{1},s)$ for the trivial character $\mathbf{1}$ and does not depend on the choice of the splitting $s$.
\end{secnumber}

\begin{secnumber} \label{HNY Kl}
	Let $\ell$ be a prime different from $p$. We take an isomorphism $\iota: \overline{K}\simeq \overline{\mathbb{Q}}_{\ell}$. 
	Using the $\ell$-adic Artin-Schreier sheaf $\AS_{\psi}$ on $\mathbb{A}^{1}_k$ associated to $\psi$, and the Kummer local system on $\mathbb{G}_{m,k}$ associated to $\chi$, Heinloth, Ng\^o and Yun construct a $\ell$-adic $\cG$ local system 
	\begin{equation}
		\Kl_{\cG}^{\et,\ell}(\psi\phi,\chi,s): \Rep(\cG)\to \Loc(X).
		\label{l Kl}
	\end{equation} 
	By the trace formula (\cite{ELS}, \cite{AC18} 4.3.9) and Gabber-Fujiwara's $\ell$-independence (\cite{AC18} 4.3.11), the Frobenius traces of $\Kl_{\cG,V}^{\et,\ell}(\psi\phi,\chi,s)$ and of $\Kl_{\cG,V}^{\rig}(\psi\phi,\chi,s)$ at each closed point of $X_{k}$ coincide via $\iota$.

	When $\chi$ is the trivial character, we omit $\chi$ and $s$ from the notation.
\end{secnumber}
\begin{secnumber} \label{Kl dR}
	There is a variant of Heinloth-Ng\^o-Yun's construction using algebraic $\mathscr{D}$-modules instead of $\ell$-adic sheaves to produce a $\cG$-connection on $X_K$ in zero characteristic (\cite{HNY} 2.6). 
	Note that all the geometric objects used in the above construction can be also defined over $K$. 
	We choose a generic linear function $\phi: I(1)/I(2)\to \A1$ over $K$ and a linear function $\chi:\Lie(T)\to K$.
	Via an isomorphism $T_K\simeq \mathbb{G}_{m,K}^{r}$, 
	\begin{displaymath}
		\boxtimes_{i=1}^{r} \bigl(K\langle x,x^{-1},\partial_x\rangle/(x\partial_x -\chi(\mathbf{1}_{i})) \bigr)
	\end{displaymath}
	defines an algebraic $\mathscr{D}$-module on $T_{K}$, which is independent of the choice of trivialisation that we denote by $\mathsf{K}_{\chi}$. 
	We replace the Artin-Schreier sheaf $\AS_{\psi}$ on $\mathbb{A}^1_{k}$ by the \textit{exponential $\mathscr{D}$-module}
	\begin{equation}\label{exp Dmod}
		\EE_{\lambda}=K\langle x, \partial_x \rangle/(\partial_x-\lambda),\quad \lambda\in K,
	\end{equation}          
	on $\mathbb{A}^{1}_{K}$. 
	We choose a splitting $s:T\rtimes \Omega\simeq \widetilde{T}$ over $K$. 
	Then we obtain a tensor functor 
	\begin{equation} \label{notation Kl dR}
		\Kl_{\cG}^{\dR}(\lambda\phi,\chi,s): \Rep(\cG)\to \Conn(X_K),
	\end{equation}
	where the target denotes the category of vector bundles with connection on $X_K$. 
	Here we identify homomorphisms $\phi: I(1)/I(2)\to \mathbb{A}^1$ of algebraic group over $K$ with $\Hom_K(\Lie I(1)/I(2), K)$ via differentiation, so $\lambda\phi$ is regarded as a linear function on $\mathrm{Lie}(I(1)/I(2))$.

	When $\chi=0$, we omit $\chi$ and $s$ from the notation and $\Kl_{\cG}^{\dR}(\lambda\phi)$ is constructed in the same way as $\Kl^{\rig}_{\cG}(\phi)$ by $\mathsf{E}_{\lambda}$.  
\end{secnumber}

\subsection{Comparison between $\Kl_{\cG}^{\dR}$ and $\Kl_{\cG}^{\rig}$} \label{Frob str Bessel} 
	In this subsection, we work with schemes over $R$ and we keep the notation of \ref{construction HNY}. 
	We say a linear function $\phi:I(1)/I(2)\to \A1$ over $R$ is \textit{generic}, if it is generic modulo the maximal ideal of $R$. 
	We take such a function $\phi$ and we denote abusively its base change to $k$ (resp. $K$) by $\phi$. 
	Let $\overline{\chi}:T(k)\to K^{\times}$ be a character. 
	There exists a homomorphism $\chi:T\to \Gm$ such that $\chi(\widetilde{x})=\overline{\chi}(x)$ for $x\in T(k)$ and some lifting $\widetilde{x}$ of $x$ in $T(K)$. 
	We denote abusively $\chi:\Lie(T)\to K$ the differential of $\chi$. 
	We choose a splitting $s:T\times \Omega\simeq \widetilde{T}$ over $R$ and we denote abusively its base change to $k$ (resp. $K$) by $s$.

	We denote the sheaf $\mathscr{O}_{\widehat{\mathbb{P}}^1,\mathbb{Q}}(^{\dagger}\{0,\infty\})$ \eqref{section six functors formalism} by $\mathcal{O}_{X}$ for short. 
	The following theorem is our main result of this subsection. 

\begin{theorem}
	We set $L=K$. For every representation $V$ of $\cG$, there exists a canonical isomorphism of $\mathcal{O}_{X}$-modules with connection \eqref{setting strict neighborhood}
	\begin{equation} \label{dagger dR iso rig}	
		\iota_{V}: (\Kl^{\dR}_{\cG,V}(-\pi\phi,\chi,s))^{\dagger}\xrightarrow{\sim} \Kl^{\rig}_{\cG,V}(\phi,\overline{\chi},s),
	\end{equation}
	compatible with tensor structures. 
\label{main thm Frob FG}
\end{theorem}

In the following, we will present the proof in the case where $\overline{\chi}$ is trivial for simplicity and the general case follows from the same argument. 
We will omit $-\pi\phi,\chi,\phi,\overline{\chi},s$ from the notation.

\begin{secnumber}
	We first consider the case where $V$ is associated to a minuscule coweight $\lambda$. 
	In this case, $\Gr_{\lambda}$ is isomorphic to a partial flag variety and is smooth and projective, and $\IC_{V}$ is isomorphic to $K_{\Gr_{\lambda}}[\dim \Gr_{\lambda}]$ supported on $\GR_{V}\simeq \Gr_{\lambda}\times X$. 
	We show the above theorem by comparing the relative twisted de Rham cohomologies and the relative twisted rigid cohomologies along the morphism
	\[
	      p_2^{\circ}:\GR^{\circ}_{V}\to X
	\] 
	in \eqref{triangle}. To do it, we first show that the associated de Rham and rigid cohomologies at each fiber of $X$ are isomorphic. 

	We regard \eqref{triangle} as a diagram of schemes over $\Spec(R)$. 
	We denote $M:=p_1^+(\mathsf{E}_{-\pi})[\dim \Gr_{\lambda}]$, which is a line bundle with connection on $\GR^{\circ}_{V,K}$. 
	With the notation of \ref{setting strict neighborhood}, the bundle with connection $M^{\dagger}$ on $(\GR^{\circ}_{V,K})^{\an}$ is overconvergent and underlies to the arithmetic $\mathscr{D}$-module $p_1^+(\mathscr{A}_{\psi})[\dim \Gr_{\lambda}]$ on $\GR^{\circ}_{V,k}$, denoted by $\mathscr{M}$. 
\end{secnumber}

\begin{lemma} \label{compare dR rig fiber}
	Let $s$ be a point of $X(k)$. We choose a lifting in $X(R)$ and still denote it by $s$. 	
	The specialisation morphism \eqref{specialization morphism} on the fiber $\GR_{V,s}^{\circ}$ of $\GR_{V}^{\circ}$ above $s$
	\begin{equation} \label{sp map GrV}
		\rH^{*}_{\dR}( (\GR^{\circ}_{V,s})_K,M_s)\to \rH^{*}_{\rig}( (\GR^{\circ}_{V,s})_k,\mathscr{M}_s) 
	\end{equation}
	is an isomorphism. Moreover, these cohomology groups vanish except for the middle degree $0$. 
\end{lemma}

\begin{proof}
	We set $Y=\GR^{\circ}_{V,s}$ and we write $M$ (resp. $\mathscr{M}$) instead of $M_s$ (resp. $\mathscr{M}_s$). 
	Since $Y$ admits a smooth compactification $\Gr_{\lambda}$ whose boundary is a divisor, we can calculate above cohomology groups by direct image of corresponding algebraic (resp. arithmetic) $\mathscr{D}$-modules \eqref{cohomology arith mods}. 
	Note that $\Kl_{\cG,V}^{\dR}$ (resp. $\Kl_{\cG,V}^{\rig}$) is a bundle with connection (resp. overconvergent $F$-isocrystal) of rank $\dim V$. 
	By the base change, cohomology groups in \eqref{sp map GrV} vanish except for the middle degree and have dimension $\dim V$ in the middle degree. 
	By (\ref{cleanness on GR}), the canonical morphism $\iota_{\rig}:\rH^{*}_{\rig,\rc}(Y_k,\mathscr{M})\to \rH^{*}_{\rig}(Y_k,\mathscr{M})$ is an isomorphism. 
	In view of proposition \ref{compatibility sp cosp}, we deduce that the specialisation morphism \eqref{sp map GrV} is surjective. Then the assertion follows. 
\end{proof}

\begin{secnumber} \label{proof thm min}
	\textit{Proof of theorem \ref{main thm Frob FG} in the minuscule case}. 
	Now we use the relative specialization morphism \eqref{rel sp map} to compare $(\Kl_{\cG,V}^{\dR})^{\dagger}$ and $\Kl_{\cG,V}^{\rig}$. 
	Let $\Gr_{\P1}\to \P1$ be the Beilinson--Drinfeld Grassmannian of $G$ over $\P1$ and $\varpi:\Gr_{\lambda,\P1}\to \P1$ the closed subscheme associated to $\lambda$. 
	Note that $\varpi$ is a locally trivial fibration over $\P1$ with smooth projective fibers $\Gr_{\lambda}$ and defines a good compactification of $p_{2}^{\circ}$ \eqref{rel rig coh}. 

	We take again the notation of \ref{rel rig coh} for the smooth $R$-morphism $p_2^{\circ}$. 
	We set $A=\Gamma(X,\mathscr{O}_X)$, $A_{K}=A[\frac{1}{p}]$, $A^{0}=\widehat{A}[\frac{1}{p}]$ the ring of analytic functions on $\widehat{X}^{\rig}$ and $A^{\dagger}=\Gamma(\mathbb{P}^1_{k},\mathcal{O}_{\Gm})$ the ring of analytic functions on $\widehat{\mathbb{P}}^{1}$ overconvergent along $\{0,\infty\}$. 
	We have inclusions $A_{K}\subset A^{\dagger}\subset A^{0}$. 
	If $D_{X_K}$ denotes the ring of algebraic differential operators on $X_K$, there exists a canonical $D_{X_{K}}$-linear specialization morphism \eqref{rel sp map}
	\begin{equation} \label{rel specialization}
		\Gamma(X_{K},\Kl_{\cG,V}^{\dR}) \to \Gamma(X_k,\Kl_{\cG,V}^{\rig}),
	\end{equation}
	where the left (resp. right) hand side is coherent over $A_K$ (resp. $A^{\dagger}$). 
	The above morphism induces a horizontal $A^{\dagger}$-linear morphism
	\begin{displaymath}
		\iota_{V}:\Gamma(X_{K},\Kl_{\cG,V}^{\dR})\otimes_{A_{K}}A^{\dagger} \to \Gamma(X_{k},\Kl_{\cG,V}^{\rig}), 
	\end{displaymath}
	which gives rise to the morphism \eqref{dagger dR iso rig}. 
	Recall that the homomorphism $A^{\dagger}\to A^{0}$ is faithfully flat (\cite{Ber96II} 4.3.10). To prove $\iota_V$ is an isomorphism, it suffices to show that the induced horizontal $A^0$-linear morphism: 
	\begin{equation} \label{rel specialization 0}
		\iota_V\otimes_{A^{\dagger}}A^0: 
		\Gamma(X_{K},\Kl_{\cG,V}^{\dR})\otimes_{A_K}A^{0} \to \Gamma(X_{k},\Kl_{\cG,V}^{\rig})\otimes_{A^{\dagger}}A^{0}
	\end{equation}
	is an isomorphism. 
	Let $\widehat{A}\to R$ be a continuous homomorphism and $s:A\to R$ the associated $R$-point of $\Gm$. 
	By \eqref{pullback of smooth mods} and the base change, the fiber $\iota\otimes_{A^{\dagger}} K$ coincides with the morphism \eqref{sp map GrV} associated to the point $s\in X(R)$ and is an isomorphism \eqref{compare dR rig fiber}. 
	Since both sides of \eqref{rel specialization 0} are coherent $A^{0}$-modules, the morphism $\iota_V\otimes_{A^{\dagger}}A^0$ is an isomorphism and the assertion follows. \hfill$\qed$
\end{secnumber}

\begin{secnumber}
	Next, we consider the case where $V$ is associated to the quasi-minuscule coweight $\lambda$. 
	In this case, $\Gr_{\le \lambda}$ contains a smooth open subscheme $\Gr_{\lambda}$ whose complement is isomorphic to $\Spec(R)$, and admits a desingularisation $\widetilde{\Gr}_{\le \lambda}$ (cf. \cite{NP} \S~7). 
	We take an isomorphism $\GR_{V}\simeq X\times \Gr_{\le \lambda}$ and set $\GR_{V}^{\circ\circ}=\GR_{V}^{\circ}\cap (X\times \Gr_{\lambda})$ to be the smooth locus of $\GR_{V}^{\circ}$ \eqref{geo picture Kl}. 
	We denote by $j:\GR^{\circ\circ}_{V}\to \GR^{\circ}_{V}$ the open immersion and by 
	\begin{equation}
		\tau=p_2^{\circ}\circ j:\GR_{V}^{\circ\circ}\to X
	\end{equation}
	the canonical morphism, which admits a good compactification $\widetilde{\Gr}_{\le \lambda}\times \P1 \to \P1$ in the sense of \ref{rel rig coh}.

	We denote by $M$ the line bundle with connection $p_1^+(\mathsf{E}_{-\pi})[\dim \Gr_{\lambda}]|_{\GR^{\circ\circ}_{V,K}}$ and by $\mathscr{M}$ the smooth arithmetic $\mathscr{D}$-module $p_1^+(\mathscr{A}_{\psi})[\dim \Gr_{\lambda}]|_{\GR^{\circ\circ}_{V,k}}$.  
	The holonomic module $\IC_{V}$ is constant on $\GR_{V}^{\circ\circ}$. 
	Then we deduce that
	\begin{displaymath}
		j_{!+}(M)\simeq p_1^+(\mathsf{E}_{-\pi})\otimes\IC_{V}|_{\GR^{\circ}_{V,K}}, \quad j_{!+}(\mathscr{M})\simeq p_1^+(\mathscr{A}_{\psi})\otimes\IC_{V}|_{\GR^{\circ}_{V,k}}.
	\end{displaymath}
	Note that $j_{!+}(M)[1],j_{!+}(\mathscr{M})[1]$ are holonomic. 
\end{secnumber}

\begin{lemma} \label{compare dR rig fiber quasi-min}
	\textnormal{(i)} The complex $\tau_{k,+}(\mathscr{M})[1]$ (resp. $\tau_{K,+}(M)[1]$) is holonomic. 

	\textnormal{(ii)} 
	Let $s$ be a point of $X(k)$. We choose a lifting in $X(R)$ and still denote it by $s$.
	If we denote by $M_s$ (resp. $\mathscr{M}_{s}$) the $+$-pullback of $M$ (resp. $\mathscr{M}$) along the fiber at $s$, then the specialisation morphism \eqref{specialization morphism}
	\begin{equation} 
		\rH^{*}_{\dR}( (\GR_{V,s}^{\circ\circ})_K,M_s)\to \rH^{*}_{\rig}((\GR_{V,s}^{\circ\circ})_k,\mathscr{M}_s) 
	\end{equation}
	induces an isomorphism
	\begin{equation} \label{sp quasi-min}
		\rH^{0}_{\dR}( (\GR_{V,s}^{\circ})_{K}, j_{!+}(M_s))\xrightarrow{\sim} \rH^{0}( (\GR_{V,s}^{\circ})_{k}, j_{!+}(\mathscr{M}_{s})).
	\end{equation}
\end{lemma}

\begin{proof}
	(i) Let $i:Z\to \GR_{V}^{\circ}$ be the complement of $\GR_{V}^{\circ\circ}$ in $\GR_{V}^{\circ}$, which is isomorphic to $X$. 
	Consider the distinguished triangle on $\GR_{\le \lambda,k}^{\circ}$
	\begin{displaymath}
		j_{!+}(\mathscr{M})[1]\to j_+(\mathscr{M})[1]\to C\to .
	\end{displaymath}
	By \ref{lemma uniqueness intermediate extension}, $C\simeq i^{!}(j_{!+}(\mathscr{M}))[2]$ has degree $\ge 0$ and is supported on $Z$. Applying $p_{2,+}^{\circ}$ to the above triangle, we obtain
	\begin{displaymath}
		p_{2,+}^{\circ}(j_{!+}(\mathscr{M}))[1]\to \tau_+(\mathscr{M})[1] \to p_{2,+}^{\circ}(C)\to ,  
	\end{displaymath}
	where the first term is holonomic (cf. \ref{lemma hol Hk}), and the second term has cohomological degrees $\le 0$ because $\tau$ is affine and the last term has cohomological degrees $\ge 0$ since $p_{2}^{\circ}|_Z$ is the identity. Then we deduce that each term in the above triangle is holonomic.  

	(ii) We set $Y=\GR_{V,s}^{\circ}$, $U=\GR_{V,s}^{\circ\circ}$ and we write simply $M$ (resp. $\mathscr{M}$) instead of $M_s$ (resp. $\mathscr{M}_s$). 
	By applying the argument of (i), we deduce that the canonical morphism of cohomology groups $\rH^{0}(Y_{k},j_{!+}(\mathscr{M})) \to \rH^{0}_{\rig}(U_{k},\mathscr{M}) $ is injective. 
	By a dual argument, we deduce that the canonical morphism $\rH^{0}_{\rig,\rc}(U_{k},\mathscr{M}) \to \rH^{0}_{\rc}(Y_{k},j_{!+}(\mathscr{M}))$ is surjective. 
	In summary, we have a sequence: 
	\begin{equation} \label{rig coh quasi-min case}
		\rH^{0}_{\rig,\rc}(U_{k},\mathscr{M}) \twoheadrightarrow
		\rH^{0}_{\rc}(Y_{k},j_{!+}(\mathscr{M})) \xrightarrow{\sim} 
		\rH^{0}(Y_{k},j_{!+}(\mathscr{M})) \hookrightarrow  
		\rH^{0}_{\rig}(U_{k},\mathscr{M}),
	\end{equation}
	where the middle isomorphism is due to the cleanness \eqref{cleanness on GR} and the composition is the canonical morphism $\iota_{\rig}$. 

	We construct an analogue sequence of \eqref{rig coh quasi-min case} for de Rham cohomology of $M$ on $U_K$. These two sequences fit into a commutative diagram \eqref{compatibility sp cosp}
	\begin{displaymath}
		\xymatrix{
		\rH^{0}_{\dR,\rc}(U_{K},M) \ar@{->>}[r] &
		\rH^{0}_{\dR,\rc}(Y_{K},j_{!+}(M)) \ar[r]^{\sim} & 
		\rH^{0}_{\dR}(Y_{K},j_{!+}(M)) \ar@{^{(}->}[r] &
		\rH^{0}_{\dR}(U_{k},M) \ar[d]^{\rho_M}
		\\
		\rH^{0}_{\rig,\rc}(U_{k},\mathscr{M}) \ar@{->>}[r] \ar[u]_{\rho_{M,\rc}} &
		\rH^{0}_{\rc}(Y_{k},j_{!+}(\mathscr{M})) \ar[r]^{\sim} & 
		\rH^{0}(Y_{k},j_{!+}(\mathscr{M})) \ar@{^{(}->}[r] &
		\rH^{0}_{\rig}(U_{k},\mathscr{M})
		}
	\end{displaymath}
	Let $E$ be the image of $\rH^{0}_{\rig,\rc}(U_k,\mathscr{M})\to \rH^{0}_{\dR}(Y_{K},j_{!+}(M))$. 
	Then the specialisation morphism $\rho_{M}$ sends $E$ surjectively to the subspace $\rH^{0}(Y_{k},j_{!+}(\mathscr{M}))$.  
	Since $\dim E\le \dim \rH^{0}_{\dR}(Y_{K},j_{!+}(M))=\dim \rH^{0}(Y_{k},j_{!+}(\mathscr{M}))$, we deduce that $E=\rH^{0}_{\dR}(Y_{K},j_{!+}(M))$ and that $\rho_{M}$ induces an isomorphism \eqref{sp quasi-min}. 
\end{proof}

\begin{secnumber}
	\textit{Proof of theorem \ref{main thm Frob FG} in the quasi-minuscule case}. 
	By \ref{compare dR rig fiber quasi-min}(i), we have a diagram of $D_{X_{K}}$-modules  
	\begin{equation}
		\xymatrix{
			\Gamma(X_{K},\Kl_{\cG,V}^{\dR}) \ar[r] & 
			\Gamma(X_{K},\tau_{K,+}(M)) \ar[d] & \\
			\Gamma(X_{k},\Kl_{\cG,V}^{\rig}) \ar[r] &
			\Gamma(X_{k},\tau_{k,+}(\mathscr{M}))
		}
	\end{equation}
	where the vertical arrow is the relative specialization morphism \eqref{rel sp map}. 
	Let $U$ be an open dense subscheme of $X_{k}$ such that $\tau_{k,+}(\mathscr{M})|_{U}$ is smooth, $\mathfrak U$ the corresponding formal open subscheme of $\widehat{X}$ and $Z=\mathbb{P}^1_k\setminus U$. 

	We denote the sheaf of rings $\mathscr{O}_{\widehat{\mathbb{P}}^1,\mathbb{Q}}(^{\dagger} Z)$ by $\mathcal{O}_{U}$ for short. 
	By \ref{compare dR rig fiber quasi-min} and the same argument of \ref{proof thm min}, the above diagram induces an injective morphism of $\mathcal{O}_{U}$-modules with connection $(\Kl_{\cG,V}^{\dR})^{\dagger}\otimes_{\mathcal{O}_{X}}\mathcal{O}_U\to \tau_+(\mathscr{M})\otimes_{\mathcal{O}_{X}}\mathcal{O}_U$ and then induces an isomorphism of $\mathcal{O}_{U}$-modules with connection:
	\begin{equation} \label{iso conv isocrystals Y}
		(\Kl_{\cG,V}^{\dR})^{\dagger}\otimes_{\mathcal{O}_{X}}\mathcal{O}_U\xrightarrow{\sim} \Kl_{\cG,V}^{\rig}\otimes_{\mathcal{O}_{X}}\mathcal{O}_U. 
	\end{equation}
	In particular, the left hand side is overconvergent along $Z$. 
	Since the convergency of an $\mathscr{O}_{\widehat{X}^{\rig}}$-module with connection can be checked by restricting to a dense open subscheme of $X_{k}$ (\cite{Ogus84} 2.16), the $\mathscr{O}_{\widehat{X}^{\rig}}$-module with connection $(\Kl_{\cG,V}^{\dR})^{\dagger}|_{\widehat{X}^{\rig}}$ is convergent. 
	Then we deduce that the $\mathcal{O}_{X}$-module with connection $(\Kl_{\cG,V}^{\dR})^{\dagger}$ is overconvergent along $\{0,\infty\}$. 
	The restriction functor $\Iso^{\dagger}(X_{k}/K)\to \Iso^{\dagger}(U/K)$ is fully faithful (cf. \cite{Ked07} 6.3.2). 
	Then the isomorphism \eqref{iso conv isocrystals Y} gives rise to an isomorphism \eqref{dagger dR iso rig} and the assertion follows. \hfill $\qed$
\end{secnumber}

\begin{secnumber}
	In the end, we show the general case of theorem \ref{main thm Frob FG}.
	Let $V_1,\cdots,V_n$ be minuscule and quasi-minuscule representations of $\cG$. Then we have a decomposition of representations
	\begin{equation} \label{decomp reps}
		V_1\otimes V_2 \otimes \cdots \otimes V_n\simeq \bigoplus_{W\in \Rep(\cG)} m_W W,
	\end{equation}
	where $m_W$ denotes the multiplicity of $W$. 
	Each representation $W$ of $\Rep(\cG)$ appears as a summand of the above decomposition for some minuscule and quasi-minuscule representations $V_1,\cdots,V_n$. 

	Then we obtain the associated decomposition of bundles with connection on $X_{K}$ and overconvergent $F$-isocrystals on $X_{K}$ respectively:
	\begin{eqnarray} \label{decomposition Kl tensor}
		\bigotimes_{i=1}^{n}\Kl_{\cG,V_i}^{\dR} \simeq \bigoplus_{W\in \Rep(\cG)} m_W \Kl_{\cG,W}^{\dR}, \\
		\bigotimes_{i=1}^{n} \Kl_{\cG,V_i}^{\rig} \simeq \bigoplus_{W\in \Rep(\cG)} m_W \Kl_{\cG,W}^{\rig}. \label{decomposition Kl tensor rig}
	\end{eqnarray}

	Theorem \ref{main thm Frob FG} in the minuscule and quasi-minuscule cases provides an isomorphism of overconvergent isocrystals
	\begin{equation}
		(\otimes_{i=1}^n\Kl_{\cG,V_i}^{\dR})^{\dagger} \xrightarrow{\sim} \otimes_{i=1}^n \Kl_{\cG,V_i}^{\rig}. 
		\label{iso analytification}
	\end{equation}
	By (\cite{Ber96} 2.2.7(iii)), the connection on left hand side, restricted on each component $(\Kl_{\cG,W}^{\dR})^{\dagger}$, is overconvergent. 
	We denote abusively the associated overconvergent isocrystal on $X_{k}$ by $(\Kl_{\cG,W}^{\dR})^{\dagger}$. 

	The isomorphism \eqref{iso analytification} induces a commutative diagram 
	\begin{equation} \label{diag End commute}
		\xymatrix{
			& \End_{\Rep(\cG)} (\bigotimes_{i=1}^{n} V_i) \ar[ld]_{\Kl_{\cG}^{\dR}} \ar[rd]^{\Kl_{\cG}^{\rig}} & \\
			\End_{\Conn(X_{K})}( \bigotimes_{i=1}^{n} \Kl_{\cG,V_i}^{\dR}) \ar[rr] && \End_{\Sm(X_{k}/K)}(\bigotimes_{i=1}^{n} \Kl_{\cG,V_i}^{\rig})
	}
	\end{equation}
	Indeed, choose a $k$-point $s$ of $X_{k}$ and a lift $\widetilde{s}$ to $X(K)$. The isomorphism \eqref{iso analytification} induces an isomorphism between fibers $(\Kl_{\cG,V_i}^{\dR})_{\widetilde{s}}$ and $(\Kl_{\cG,V_i}^{\rig})_{s}$. 
	The composition of the functor $\Kl_{\cG}^{\dR}$ (resp. $\Kl_{\cG}^{\rig}$) with the fiber functor at $\widetilde{s}$ (resp. $s$) is the forgetful functor $\Rep(\cG)\to \Vect_{K}$. 
	Since fiber functors are faithful, we deduce the commutativity of \eqref{diag End commute} by considering their fibers. 

	If $e$ denotes the idempotent of $\End_{\Rep(\cG)} (\otimes_{i=1}^n V_i)$ corresponding to a summand $W$, then its image via left (resp. right) vertical arrow is the idempotent corresponding to $\Kl_{\cG,W}^{\dR}$ (resp. $\Kl_{\cG,W}^{\rig}$) (\ref{decomposition Kl tensor}, \ref{decomposition Kl tensor rig}). 
	By \eqref{iso analytification} and \eqref{diag End commute}, we deduce a canonical isomorphism of overconvergent isocrystals on $X_{k}$
	\begin{displaymath}
		\iota_{W}: (\Kl_{\cG,W}^{\dR})^{\dagger} \xrightarrow{\sim} \Kl_{\cG,W}^{\rig}.
	\end{displaymath}
	One verifies that the above isomorphism is independent of the choice of idempotent $e$ and then of the choice of minuscule representations $\{V_{i}\}_{i=1}^{n}$. 
	Isomorphisms $\iota_{W}$ are compatible with tensor structures due to \eqref{iso analytification}.
	Now theorem \ref{main thm Frob FG} follows. \hfill $\qed$ 
\end{secnumber}

\subsection{Comparison between $\Kl_{\cG}^{\dR}$ and $\Be_{\cG}$} \label{compare Kl Be}
In this subsection, we recall the Bessel connection $\Be_{\cG}(\check{\xi})$ on $X$ constructed by Frenkel and Gross \cite{FG09} of $\cG$ and identify it with $\Kl_{\cG}^{\dR}(\phi)$ \eqref{Kl dR}.

	We work with schemes over $K$. 
	Let $(\check{\gg},\check{\mathfrak b},\check{\mathfrak t})$ denote the Lie algebras of $(\cG,\cB,\cT)$ over $K$.

\begin{secnumber} \label{def FG connection}
	Let $A_K$ denote the ring of algebraic functions of $X$. 
	There exists a grading on the affine Lie algebra $\check{\gg}_{\mathrm{aff}}:=\check{\gg}\otimes A_K$, which on $\check{\gg}$-part is given by $\Ad \rho(\Gm)$, and on $A_K$-part is given by the $\check{h}$-multiple of the grading induced by the natural action of $\Gm$ on $X$. 
	Here as before $\rho\in \mathbb{X}^\bullet(T)\otimes\mathbb{Q}$ is the half sum of positive roots of $G$ (and therefore is a cocharacter of $\cG_{\mathrm{ad}}$), and $\check{h}$ is the Coxeter number of $\cG$. 
	
	Let $\check{\gg}_{\mathrm{aff}}(1)\subset \check{\gg}_{\mathrm{aff}}$ be the subspace of degree $1$. Then
	\[
	      \check{\gg}_{\mathrm{aff}}(1)=\bigoplus_{\check{\alpha} \mbox{ affine simple}}\check{\gg}_{\mathrm{aff},\check{\alpha}},
	\]	
	where $\check{\gg}_{\mathrm{aff},\check{\alpha}}$ is the root subspace corresponding to the affine simple root $\check{\alpha}$ of $\check{\gg}_{\mathrm{aff}}$. 
	Let $\check{\xi}\in  \check{\gg}_{\mathrm{aff}}(1)$ be a \textit{generic} element, by which we mean each of its $\check{\alpha}$-component $\check{\xi}_{\check{\alpha}}\neq 0$. In \cite{FG09}, Frenkel and Gross defined a $\check{\gg}$-valued connection on the trivial $\check{G}$-bundle on $X$ by the following formula:
\begin{equation}
	\Be_{\cG}(\check{\xi}) = d +  \check{\xi} \frac{dx}{x}.
	\label{FG connection formula}
\end{equation}
        Here $x$ is a coordinate $x: X\cup\{0\}\simeq \A1$. Note that
        $\frac{dx}{x}$ itself is independent of the choice of the coordinate $x$, and is a generator of the module of log differentials on $X\cup\{0\}$ with logarithmic pole at $0$. 
        	
	We may write $N=\sum_{\check{\alpha}} \check{\xi}_{\check{\alpha}}$, where the sum is taken over simple roots of $\check{\gg}$ (instead of $\check{\gg}_{\mathrm{aff}}$). This is a principal nilpotent element of $\check{\gg}$. The remaining affine root subspaces are of the form $x\check{\gg}_{-\check{\theta}_i}$, where $x$ is a coordinate as above and $\check{\theta}_i$ is the highest root of the simple factor $\check{\gg}_i$ of $\check{\gg}$. So we may write the sum of the remaining affine root vectors as $x E$ for some $E\in \sum\check{\gg}_{-\check{\theta}_i}$. Then the connection can be written as
	\begin{equation}\label{Bessel connection form}
	          \Be_{\cG}(\check{\xi}) = d +  (N+ xE) \frac{dx}{x},
	  \end{equation}
	which is the form as used in \cite{FG09}. In particular, this connection is regular singular with a principal unipotent monodromy at $0$. On the other hand, it has an irregular singularity at $\infty$, with maximal formal slope $1/\check{h}$ (\cite{FG09} \S 5).

	We regard $\Be_{\cG}(\check{\xi})$ as a tensor functor from the category of representations of $\check{G}$ to the category of bundles with connection on $X$:
	\begin{equation} \label{FG connection}
		\Be_{\cG}(\check{\xi}): \Rep(\check{G}) \to \Conn(X). 
	\end{equation}
\end{secnumber}

\begin{secnumber} \label{action Gm Sigma}
	We will identify $\Kl_{\cG}^{\dR}(\lambda\phi)$ and $\Be_{\cG}(\check{\xi})$ as $\cG$-bundles with integrable connections on $X$. For this purpose, we need to discuss how these connections depend on parameters. 
        We identify the dual space $\gg_{\mathrm{aff}}^*$ of $\gg_{\mathrm{aff}}:=\gg\otimes A_K$ with $\gg^*\otimes A_K$ via
         the canonical residue pairing 
        \[
                   (\gg\otimes A_K)\otimes (\gg^*\otimes\omega_X)\to K, \quad (\xi\otimes f,\check{\xi}\otimes g)= (\xi,\check{\xi}) \mathrm{Res}_{x=\infty}fg\frac{dx}{x}.
        \]
        Recall that $\lambda\phi$ is a  linear function $\Lie(I(1)/I(2))\to K$. We identify $\Hom_K(\Lie I(1)/I(2), K)$ with
	\[
	     \gg^*_{\mathrm{aff}}(1)=\bigoplus_{\alpha \mbox{ affine simple}} \gg^*_{\alpha}.
	\]
	where $\gg^*_{\alpha}\subset \gg_{\mathrm{aff}}^*$ is the dual of the root subspace corresponding to $\alpha$.

	By \eqref{dependence phi} (applied to the $\mathscr{D}$-module setting), $\Kl_{\cG}^{\dR}(\lambda\phi)$ depends only on the $T_{\ad}$-orbit of this functional. In addition, $T_{\ad}$-orbits of generic linear functions on $\Lie(I(1)/I(2))$ are parameterized by the GIT quotient $ \gg^*_{\mathrm{aff}}(1)/\!\!/T_{\ad}$.
	
	On the other hand, the group $\Gm\times\Aut(\cG,\cB,\cT)$ acts on $\check{\gg}_{\mathrm{aff}}$ preserving the grading. For $\check{\delta}=(a,\check{\sigma})$, a simple gauge transform implies that the analogue of \eqref{Aut of diag} holds, namely 
	\begin{equation} \label{formula action Gm Sigma}
                \Be_{\cG,V}(\check{\delta}(\check{\xi}))\simeq a^+\Be_{\cG,\check{\sigma}V}(\check{\xi}).
        \end{equation}
	It follows that the analogue of \eqref{dependence phi} and of \eqref{dependence psi}  also hold for Bessel connections. 
        In particular, $\Be_{\cG}(\check{\xi})$ only depends on the $\cT_{\mathrm{ad}}$-orbit of $\check{\xi}$. 
        Again, $\cT_{\ad}$-orbits of generic $\check{\xi}$ are parameterized by the GIT quotient $\check{\gg}_{\mathrm{aff}}(1)/\!\!/\cT_{\mathrm{ad}}$. 
\end{secnumber}

        Here is the main theorem of this subsection.
\begin{theorem} \label{thm Zhu}
	There exists a canonical isomorphism of affine schemes 
	\begin{equation}\label{classical duality}
		\gg^*_{\mathrm{aff}}(1)/\!\!/T \xrightarrow{\sim} \check{\gg}_{\mathrm{aff}}(1)/\!\!/\check{T},
	\end{equation} 
	such that if the $T_{\ad}$-orbit through $\lambda \phi$ and the $\cT_{\mathrm{ad}}$-orbit through $\check{\xi}$ match under this isomorphism, then 
	\[
	      \Kl_{\cG}^{\dR}(\lambda\phi)\simeq \Be_{\cG}(\check{\xi})
	\]      
	as $\cG$-bundles with connection on $X$.
\end{theorem}
    
 If $\cG$ is of adjoint type, a weaker version of this theorem was the main result of \cite{Zhu17}.   
      
\begin{secnumber}  
        We first explain the isomorphism \eqref{classical duality}. Let $\omega_X$ denote the canonical bundle on $X$ and by abuse of notation, we sometimes also use it to denote the space of its global sections.  Via the open embedding $j_{\gamma}: I(1)/I(2)\hookrightarrow \Bun_{\Gzt}^{\gamma}$, we identify $I(1)/I(2)\times \gg^*_{\mathrm{aff}}(1)$ with $T^*\Bun_{\Gzt}^{\gamma}|_{j_{\gamma}(I(1)/I(2))}$. The Hitchin map (e.g. see \cite{BD99} Sect. 2, and \cite{Zhu17})
        \[
		h^{cl}: T^*\Bun^{\gamma}_{\Gzt}\to \mathrm{Hitch}(X):=\Gamma(X, \mathfrak c^*\times^{\Gm}\omega_X)~ 
		\footnote{Here $\Gamma(X, \mathfrak c^*\times^{\Gm}\omega_X)$ denotes abusively the affine space associated to the $K$-vector space $\Gamma(X, \mathfrak c^*\times^{\Gm}\omega_X)$. } 
        \] 
	induces a closed embedding $h^{cl}:\gg_{\aff}^*(1)/\!\!/T\hookrightarrow \mathrm{Hitch}(X)$, where  $\mathfrak c^*:={\gg}^*/\!\!/G$ is the GIT quotient of $\gg^*$ by the adjoint action of $G$, equipped with a $\Gm$-action induced by the natural $\Gm$-action on $\gg^*$. (For an explicit description of the image of the map when $\gg$ is simple, see  the discussions before \cite{Zhu17}  lemma 18). 
	        
        On the other hand, there exists a canonical morphism 
        \[
              \check{\gg}_{\mathrm{aff}}(1)\frac{dx}{x}\subset \check{\gg}\otimes\omega_X\to \Gamma(X, \check{\mathfrak c}\times^{\Gm}\omega_X)
        \] 
	where $\check{\mathfrak c}:=\check{\gg}/\!\!/\cG$, which also induces a closed embedding $\check{\gg}_{\mathrm{aff}}(1)/\!\!/\cT\to \Gamma(X, \check{\mathfrak c}\times^{\Gm}\omega_X)$. 
	The identification $(\Lie T)^*=\Lie\cT$ induces a canonical isomorphism $\mathfrak c^*\xrightarrow{\sim}\check{\mathfrak c}$. 
	One checks easily that there is a unique isomorphism $ \gg^*_{\mathrm{aff}}(1)/\!\!/T\xrightarrow{\sim} \check{\gg}_{\mathrm{aff}}(1)/\!\!/\check{T}$ that fits into the following commutative diagram
        \[\xymatrix{
              \gg_{\mathrm{aff}}^*(1)/\!\!/T\ar^-{\sim}[r]\ar@{^{(}->}[d] & \check{\gg}_{\mathrm{aff}}(1)/\!\!/\cT\ar@{^{(}->}[d]\\
              \Gamma(X, \mathfrak c^*\times^{\Gm}\omega_X) \ar^-{\sim}[r] & \Gamma(X, \check{\mathfrak c}\times^{\Gm}\omega_X)
        }\]
	where the bottom isomorphism is induced by $\mathfrak c^*\xrightarrow{\sim}\check{\mathfrak c}$. 
        
        In the case $G$ and $\cG$ are almost simple, unveiling the definition, we see that $\lambda\phi$ and $\check{\xi}$ match to each other if the following holds: Let $r$ be the rank of $G$ and $\cG$. 
        Recall that the ring of invariant polynomials on $\gg^*$ (resp. $\check{\gg}$) has a generator $P_r$ (resp. $\check{P}_r)$, homogeneous of degree $h=\check{h}$.      
        We choose them to match each other as functions on $\mathfrak c^*\simeq \check{\mathfrak c}$. Then  $\lambda\phi$ matches $\check{\xi}$ if and only if 
	\begin{equation}
		\lambda^hP_{r}(\phi)=P_{r}(\lambda\phi)=\check{P}_{r}(\check{\xi}).
		\label{invariant poly match}
	\end{equation}
	This condition is independent of the choice of $P_r$ and $\check{P}_r$ (as soon as they match to each other).
        
        For concrete computations, it is convenient to 
        fix a coordinate $x\in \A1\subset\P1$, and 
        a pinning $N=\sum_{\check{\alpha}\in\check{\Delta}} \check{\xi}_{\check{\alpha}}$ of $(\cG,\cB,\cT)$. Then we may rewrite \eqref{classical duality} as an isomorphism
        \begin{equation}\label{phi-theta match xi-theta}
                 \gg^*_{\mathrm{aff}}(1)/\!\!/T \simeq \check{\gg}_{\mathrm{aff}}(1)/\!\!/\check{T}\simeq N + x\sum_i\check{\gg}_{-\check{\theta}_i}\simeq x\sum_i\check{\gg}_{-\check{\theta}_i}.
        \end{equation}  
\end{secnumber}          
        
\begin{secnumber} 
        We prove theorem \ref{thm Zhu} by quantizing \eqref{classical duality} and applying the Galois-to-automorphic direction of geometric Langlands correspondence. 
        For this, we need to review the notion of $\check{\gg}$-opers  (\cite{BD99} \S 3). By descent, it suffices to prove the theorem after base change from $K$ to $\overline{K}$. So we assume that all the geometric objects below are defined over $\overline{K}$, and omit the subscript. Let $\cG_{\mathrm{ad}}$ denote the adjoint group of $\cG$. 
        
        Let $Y$ be a smooth curve over $\overline{K}$. Let $\mathrm{Op}_{\check{\gg}}(Y)$ denote the moduli spaces of $\cG_{\mathrm{ad}}$-opers on $Y$ (\cite{BD99} 3.1.11).        
        By (\cite{BD99} 3.1.11, 3.4.3), $\mathrm{Op}_{\check{\gg}}(Y)$ is an ind-affine scheme. There is a natural free  and transitive action of the (ind)-vector space $\Gamma(Y, \check{\mathfrak c}\times^{\Gm}\omega_Y)$ on  $\mathrm{Op}_{\check{\gg}}(Y)$ (\cite{BD99} 3.1.9). This induces a natural filtration on the ring of regular functions $\mathrm{Fun} \mathrm{Op}_{\check{\gg}}(Y)$, whose associated graded is the ring of regular functions $\mathrm{Fun} \Gamma(Y, \check{\mathfrak c}\times^{\Gm}\omega_Y)$.       
   
        Back to our case $Y=X$.        
        We consider the subscheme of $\mathrm{Op}_{\check{\gg}}:=\mathrm{Op}_{\check{\gg}}(\P1)_{(0,\varpi(0)),(\infty,1/\check{h})}\subset \mathrm{Op}_{\check{\gg}}(X)$, which is the moduli
        of $\cG_{\mathrm{ad}}$-opers on $X$ which are
        \begin{itemize}
                 \item regular singular with principal unipotent monodromy at $0$;
                 \item possibly irregular of maximal formal slope $\leq 1/\check{h}$ at $\infty$.
        \end{itemize}  
        See the discussions before (\cite{Zhu17} lemma 20) (where slightly different notations were used). 
        In this case, the action of  $\Gamma(X, \check{\mathfrak c}\times^{\Gm}\omega_X)$  on $\mathrm{Op}_{\check{\gg}}(X)$ induces a free and transitive action of $x\sum_i\check{\gg}_{-\check{\theta}_i}\simeq \check{\gg}_{\mathrm{aff}}(1)/\!\!/\cT$ \eqref{phi-theta match xi-theta} on $\mathrm{Op}_{\check{\gg}}$. In particular, $ \mathrm{Fun} \mathrm{Op}_{\check{\gg}}$ has a natural filtration whose associated graded is $(\mathrm{Fun} \check{\gg}_{\mathrm{aff}}(1))^{\cT} $.
        
        On the other hand, the space $\mathrm{Op}_{\check{\gg}}$ has a distinguished point, corresponding to the $\cG_{\mathrm{ad}}$-oper that is tame at both $0$ and $\infty$. Therefore, we obtain a canonical isomorphism $x\sum_i\check{\gg}_{-\check{\theta}_i}\in \check{\gg}_{\mathrm{aff}}(1)/\!\!/\cT\simeq \mathrm{Op}_{\check{\gg}}(X)$. Explicitly, this isomorphism sends $xE\in x\sum_i\check{\gg}_{-\check{\theta}_i}$ to the connection $d+ (N+xE) \frac{dx}{x}$ on the trivial $\cG$-bundle which has a natural oper form.       
        Now the quantization of \eqref{classical duality} gives a canonical isomorphism of filtered algebras (\cite{Zhu17} lemma 21)
        \begin{equation}\label{quantization}
             U( \Lie I(1)/I(2))^T\simeq   \mathrm{Fun} \mathrm{Op}_{\check{\gg}},
        \end{equation}  
        whose associated graded gives back to \eqref{classical duality}. 
        Here $U(V)$ is the universal enveloping algebra of $V=\Lie I(1)/I(2)$, equipped with the usual filtration. As $V$ is abelian, it is also canonically isomorphic to $ (\mathrm{Fun} V^*)^T $. Putting all the above isomorphisms together, we obtain the following commutative diagram        
        \begin{equation*}
        \xymatrix{    
                 (\mathrm{Fun} \gg_{\mathrm{aff}}^*(1))^T \ar^-\sim[r]\ar_-\sim[d] &  
		 (\mathrm{Fun} \check{\gg}_{\mathrm{aff}}(1))^{\cT} \ar^-\sim[d]\\
                 U( \Lie I(1)/I(2))^T                         \ar^-\sim[r]  
		 &    \mathrm{Fun} \mathrm{Op}_{\check{\gg}}
        }
        \end{equation*}
        Together with the main result of \cite{Zhu17}, we obtain the proof of theorem \ref{thm Zhu} in the case when $\cG=\cG_{\mathrm{ad}}$.
        \end{secnumber}
 
 \begin{secnumber}               
         Next, we explain how to extend it to allow $G$ to be a general semisimple group. 
        
         One approach is to generalize the work of \cite{BD99} to allow certain level structures, as what \cite{Zhu17} did for simply-connected groups. 
         In this approach, one must deal with the subtle question of the construction of ``square root'' of the canonical bundle on the moduli of $\mathcal{G}$-bundles. 
         
         In our special case, we have another short and direct approach,
         using the isomorphism $\Kl_{\cG_{\mathrm{ad}}}^{\dR}(\lambda\phi)\simeq \Be_{\cG_{\mathrm{ad}}}(\check{\xi})$ just established. 
         
         First, we claim that up to isomorphism, there exists a unique de Rham $\cG$-local system on $X$, which induces $\Be_{\cG_{\mathrm{ad}}}(\check{\xi})$, and has unipotent monodromy at $0$. Indeed, any two such de Rham $\cG$-local systems differ by a de Rham $\check{Z}$-local system on $X\cup\{0\}\simeq \mathbb{A}^1$ 
         (i.e. one is obtained from the other by twisting a de Rham $\check{Z}$-local system). As $\check{Z}$ is a finite group, 
         the wild part of the differential Galois group at $\infty$ of this local system must be trivial, and therefore this local system itself is trivial.
         
         Now since both $\Kl_{\cG}^{\dR}(\lambda\phi)$ and $\Be_{\cG}(\check{\xi})$ have the property as in the claim (to see that $\Kl_{\cG}^{\dR}(\lambda\phi)$ has unipotent monodromy at $0$, 
         one uses the same argument as \cite{HNY} theorem 1 (2)), they must be isomorphic.
 \end{secnumber}

\subsection{Bessel $F$-isocrystals for reductive groups} \label{Generalised Bessel}
         In this subsection, we construct Bessel $F$-isocrystals for reductive groups, by putting the above ingredients together. 
	 We keep the notation of \ref{Frob str Bessel}.

\begin{secnumber}\label{F on Be}         
	We take a non-trivial additive character $\psi:\mathbb{F}_p\to K^{\times}$ and a generic linear function $\phi:I(1)/I(2)\to \mathbb{A}^1$ over $R$ \eqref{Frob str Bessel}.
	We set $\lambda=-\pi\in K$ corresponding to $\psi$ (as in \ref{Dwork isocrystal}). 
	Let $\check{\xi}\in \check{\gg}_{\aff}(1)$ match $-\pi\phi$ under the isomorphism \eqref{classical duality}.
         
        We write $\Be_{\cG}(\check{\xi})$ more explicitly as follows. 
        Choose a coordinate $x$ of $X\cup\{0\}$ over $R$, and a pinning $N=\sum_{\check{\alpha}\in\check{\Delta}} \check{\xi}_{\check{\alpha}}$ of $(\cG,\cB,\cT)$. 
        By \eqref{phi-theta match xi-theta}, there is a unique element $E=E_{\phi}\in   \sum_i\check{\gg}_{-\check{\theta}_i}$ such that   
        \begin{equation}\label{normalization}
                   \Kl^{\dR}_{\cG}(1\cdot\phi) \simeq  d + (N + xE) \frac{dx}{x} ,
        \end{equation}           
	By \eqref{invariant poly match}, we deduce that
        \begin{equation}\label{Klpi vs Be}
                      \Kl^{\dR}_{\cG}(-\pi\phi) \simeq d + (N + (-\pi)^h x E) \frac{dx}{x} = \Be_{\cG}(\check{\xi}).
        \end{equation}               
 
         Now we can define the object appearing in the title of the paper. Let $\Be_{\cG}^{\dagger}(\check{\xi})$ denote the composition of $\Be_{\cG}(\check{\xi}): \Rep(\cG)\to \Conn(X_{K})$ with the $(-)^{\dagger}$-functor from \eqref{various categories}.
         By theorem \ref{main thm Frob FG}, a choice of isomorphism \eqref{Klpi vs Be} endows $\Be_{\cG}^{\dagger}(\check{\xi})$ with a Frobenius structure, 
         i.e. a lifting of $\Be_{\cG}^{\dagger}(\check{\xi})$ as a functor $\Rep(\cG)\to \Fr\Iso^{\dagger}(X_k/K)$, or alternatively, an isomorphism of tensor functors 
         \[
		 \varphi:F^{*}_{X_k}\circ \Be_{\cG}^{\dagger}(\check{\xi})
		 \xrightarrow{\sim} \Be_{\cG}^{\dagger}(\check{\xi}): \Rep(\cG)\to \Iso^{\dagger}(X_k/K),
         \] 
	 where $F^{*}_{X_k}:\Iso^{\dagger}(X_{k}/K)\to \Iso^{\dagger}(X_{k}/K)$ denotes the $s$-th Frobenius pullback functor \eqref{Frobenius pullback iso}. 
	 From the calculation of the differential Galois group of $\Be_{\cG}$ in \cite{FG09} coro. 9, coro. 10 (see \eqref{table monodromy}) that the automorphism group of $\Be_{\cG}$ is $Z_G(K)$. 
         Therefore, the Frobenius structure on $\Be_{\cG}^{\dagger}(\check{\xi})$  is independent of the choice of the isomorphism $\Be_{\cG}(\check{\xi})\simeq \Kl_{\cG}^{\dR}(\lambda\phi)$. 
	 We use $(\Be_{\cG}^{\dagger}(\check{\xi}),\varphi)$ (or simply $\Be_{\cG}^{\dagger}(\check{\xi})$ if there is no confusion) to denote the $\cG$-valued overconvergent $F$-isocrystal
         \begin{equation}\label{Be cG F}
		 (\Be_{\cG}^{\dagger}(\check{\xi}),\varphi): \Rep(\cG)\to \Fr\Iso^{\dagger}(X_k/K),
         \end{equation} 
	 which we call the \textit{Bessel $F$-isocrystal of $\cG$}.
\end{secnumber}

\begin{secnumber} \label{log-extension} 
	For each representation $\rho:\cG\to\GL(V)$, the restriction of $\Be_{\cG,V}^{\dagger}(\check{\xi})$ at $0$ defines an object $\Be_{\cG,V}^{\dagger}(\check{\xi})|_{0}$ of $\MCF(\mathcal{R}_K/K)$ \eqref{review MC Swan}, which is solvable at $1$ (\cite{Ked10} 12.6.1). 
	By \eqref{FG connection formula}, the $p$-adic exponents of $\Be_{\cG,V}^{\dagger}(\check{\xi})|_{0}$ are $0$. 
	Then it is equivalent to the connection $d+d\rho(N)$ over the Robba ring by (\cite{Ked10} 13.7.1). 
	Hence, $\Be_{\cG,V}^{\dagger}(\check{\xi})|_{0}$ satisfies the Robba condition (i.e. it has zero $p$-adic slope) and is unipotent. 
	
	We denote by $\Fr\Iso^{\log,\uni}\bigl( (\mathbb{A}^1_k,0)/K\bigr)$ the category of log convergent $F$-isocrystals on $\A1_k$ with a log pole at $0$ relative to $K$ and nilpotent residue, and are overconvergent along $\infty$ \eqref{log isocrystal}. 
	By (\cite{Ked07} 6.3.2), this category is equivalent to the full subcategory of $\Fr\Iso^{\dagger}(X_k/K)$ consisting of objects which are unipotent at $0$. 
	Then the $\cG$-valued overconvergent $F$-isocrystal $(\Be_{\cG}^{\dagger}(\check{\xi}),\varphi)$ \eqref{Be cG F} factors through:
\begin{equation}
	(\Be_{\cG}^{\dagger}(\check{\xi}),\varphi): \Rep(\cG) \to \Iso^{\log,\uni}\bigl( (\mathbb{A}^1_k,0)/K\bigr).
	\label{Be cG log}
\end{equation}     
\end{secnumber}

\begin{secnumber} \label{Frobenius Adagger} 
Here is a more concrete description of the Frobenius structure on $\Be_{\cG}^{\dagger}(\check{\xi})$. Note that its underlying bundles of  $\Be_{\cG,V}^{\dagger}(\check{\xi})$ are free $\mathscr{O}_{\widehat{\mathbb{P}}^1,\mathbb{Q}}(^{\dagger}\{\infty\})$-modules. 
If we set $A^{\dagger}=\Gamma(\mathbb{P}^1_{k}, \mathscr{O}_{\widehat{\mathbb{P}}^1,\mathbb{Q}}(^{\dagger}\{\infty\}))$, by the Tannakian formalism, the Frobenius structure on $\Be_{\cG}^{\dagger}(\check{\xi})$ is equivalent to an element $\varphi\in \cG(A^{\dagger})$ satisfying 
\begin{equation} \label{Frob str varphi}
		x \frac{d\varphi}{dx} \varphi^{-1} + \Ad_{\varphi} (N+(-\pi)^{h} x E) = q(N+(-\pi)^h x^q E).
	\end{equation}
	Given a point $a\in |\mathbb{A}_k^1|$ and $\widetilde{a}:A^{\dagger}\to \overline{K}$ its Teichmüller lifting, 
	we denote by $\varphi_a=\prod_{i=0}^{\deg(a)-1} \varphi(\widetilde{a}^{q^i})$. 
	When $a\neq 0$, the Frobenius trace of $(\Be_{\cG}^{\dagger}(\check{\xi}),\varphi)$ at $a$ can be calculated by the trace of $\varphi_a$. 
	Now we rephrase the above discussions as follows, which is the first main result of our article.
\end{secnumber}

\begin{theorem}\label{def Bessel crystal}
	There is a unique element $\varphi \in \cG(A^{\dagger})$ satisfying the differential equation \eqref{Frob str varphi} such that via a (fixed) isomorphism $\overline{K}\simeq \overline{\mathbb{Q}}_{\ell}$, for every $a\in |X|$ and $V\in \Rep(\cG)$      
         \begin{equation}\label{E: trace identity}
                \mathrm{Tr}( \varphi_a, V)= \mathrm{Tr} (\mathrm{Frob}_{a}, \Kl^{\et,\ell}_{\cG,V,\bar{a}}(\psi\phi)). 
         \end{equation}
\end{theorem}

When $a=0$, we can describe $\varphi_0$ more precisely.
\begin{prop} \label{Frobenius 0}
	Let $2\rho$ be the sum of positive coroots in $\cwX(\cT)$. Then $\varphi_0=2\rho(\sqrt{q})$ in the semisimple conjugacy classes $\Conj^{\textnormal{ss}}(\cG(\overline{K}))$ of $\cG(\overline{K})$. 
\end{prop}

\begin{proof}
	The Frobenius endomorphism $\varphi_{0}$ at $0$ satisfies $\varphi_0^{-1}N\varphi_{0}=qN$ \eqref{def FG connection}. 
	Since $N$ is a principal nilpotent element and $\Ad_{\rho(q)}N=q^{-1}N$,
	we deduce that $\varphi_0=\varepsilon\rho(q)$ in $\Conj^{\textnormal{ss}}(\cG(\overline{K}))$ for some element $\varepsilon$ in the center $Z_{\cG}(\overline{K})$. 

	To show $\varepsilon=\id$, it suffices to investigate Frobenius eigenvalues of $\Psi(\Be_{\cG,V}^{\dagger})$ \eqref{nearby cycle log extension} for $V\in \Rep(\cG)$, which is same as those of $\Psi(\Kl_{\cG,V}^{\et,\ell})$ by \ref{Kl dR} and Gabber-Fujiwara's $\ell$-independence (\cite{Abe18} 4.3.11). 
	By a result of G\"ortz and Haines \cite{GH07}, the $i$th graded piece of the weight filtration of $\Psi(\Kl_{\cG,V}^{\et,\ell})$ has the same dimension as the dimension of $\rH^{2i}(\Gr_{G},\IC_{V})$ and is equipped with a Frobenius action by $\times q^i$ (cf. \cite{HNY} 4.3).
	Then we deduce that $\varepsilon=\id$. 
\end{proof}

\subsection{Monodromy groups}
\label{sec monodromy}

\begin{secnumber} \label{setting monodromy gp}
	In this subsection, we keep the notation of \ref{Generalised Bessel} and we take $L$ to be $\overline{K}$. We drop $\phi\psi$ from the notation. 
	
	We denote by $\langle\Be_{\cG}^{\dagger}\rangle$ (resp. $\langle\Be_{\cG}^{\dagger},\varphi\rangle$, resp. $\langle\Be_{\cG}\rangle$) the full subcategory of $\Sm(X_{k}/\overline{K})$ (resp. $\Sm(X_{k}/\overline{K}_F)$, resp. $\Conn(X_{\overline{K}})$) whose objects are all the sub-quotients of objects $\Be_{\cG,V}^{\dagger}$ (resp. $(\Be_{\cG,V}^{\dagger},\varphi)$, resp. $\Be_{\cG,V}$) for $V\in \Rep(\cG)$. 
	Then $\langle\Be_{\cG}^{\dagger}\rangle$ (resp. $\langle\Be_{\cG}^{\dagger},\varphi\rangle$, resp. $\langle\Be_{\cG}\rangle$) forms a Tannakian category over $\overline{K}$ 
	and we denote by $G_{\geo}$ (resp. $G_{\mathrm{arith}}$, resp. $G_{\alg}$) the associated Tannakian group (with respect to a fiber functor $\omega$, but 
	is independent of the choice of the fiber functor up to isomorphism \cite{Del90}). 
	The tensor functors on the left side of the following diagrams induce closed immersions of algebraic groups on the right side
\begin{equation}\label{geo monodromy rig to dR}
        \xymatrix{
                      &\langle\Be_{\cG}^{\dagger},\varphi\rangle\ar[dl]&   &                           &G_{\mathrm{arith}}\ar[dr]&  \\
        \langle\Be_{\cG}^{\dagger}\rangle&     &\Rep(\cG)\ar[dl]\ar[ul]& G_{\geo}\ar[dr]\ar[ur] &                   & \cG\\
                      & \langle\Be_{\cG}\rangle\ar[ul] &                              &                                                       & G_{\alg}\ar[ur] &	}.
\end{equation}

	In (\cite{FG09} Cor. 9 and Cor. 10), Frenkel and Gross showed that the differential Galois group $G_{\alg}$ of the $\cG$-connection $\Be_{\cG}:\Rep(\cG)\to \Conn(X_{\overline{K}})$ is a connected closed subgroup of $\cG$ and explicitly calculated it when $\cG$ is almost simple. The result can be found in \eqref{table monodromy}.
	The main theorem of this subsection is as follows.
\end{secnumber}

\begin{theorem} \label{thm monodromy} 
	Let $G$ be a split almost simple group over $R$ and $\cG$ its Langlands dual group over $\overline{K}$. 
	We denote by $\Sigma$ the outer automorphism group of $\cG$ and by $\Out(\check{\gg})$ the outer automorphism group of $\check{\gg}$.

	\textnormal{(i)} If $\cG$ is not of type $A_{2n}$ or $\Char(k)>2$, then $G_{\geo}\to G_{\alg}$ is an isomorphism. In particular,
	\begin{itemize}
		\item $G_{\geo}\xrightarrow{\sim} \cG^{\Sigma,\circ}$, if $\cG$ is not type $A_{2n}$ ($n\ge 2$) or $B_{3}$ or $D_{2n}$ ($n\ge 2$) with $\Sigma\neq \Out(\check{\gg})$.  
	        \item $G_{\geo}=\cG$, if $\cG$ is of type $A_{2n}$, 
	        \item $G_{\geo}\xrightarrow{\sim} G_2$, if $\cG$ is of type $B_{3}$ or of type $D_4$. 
		\item $G_{\geo}\xrightarrow{\sim}\Spin_{4n-1}$ if $\cG$ is of type $D_{2n}$ with $\Sigma\simeq \{1\}$ ($n\ge 3$). 
	\end{itemize}
	
	\textnormal{(ii)} If $\cG=\SL_{2n+1}$ and $\Char(k)=2$, then $G_{\geo}(\Be^{\dagger}_{\SL_{2n+1}})=G_{\geo}(\Be^{\dagger}_{\SO_{2n+1}})$. 
	In particular,
	\begin{itemize}
		\item $G_{\geo}\xrightarrow{\sim}\mathrm{SO}_{2n+1}$, if $n\neq 3$,
		\item $G_{\geo}\xrightarrow{\sim}G_2$, if $n=3$. 
	\end{itemize}	
	In particular, $G_{\geo}\neq G_{\alg}$ in this case.
	
	\textnormal{(iii)} The map $G_{\geo}\to G_{\mathrm{arith}}$ is always an isomorphism.
	\label{thm geometric mono}
\end{theorem}

\begin{proof}We first study the local monodromy at $0$ and $\infty$.

	In view of \ref{log-extension}, the restriction functor at $0$ \eqref{res to x} induces
	\begin{displaymath}
		\Rep(\cG)\to \langle\Be_{\cG}^{\dagger}\rangle \xrightarrow{|_{0}} \MC^{\uni}(\mathcal{R}/\overline{K})\xrightarrow{\sim} \Vect^{\nil}_{\overline{K}}, 
	\end{displaymath}
	sending each representation $\rho:\cG\to \GL(V)$ to $(V,d\rho(N))\in \Vect^{\nil}_{\overline{K}}$. 
	Then, it induces closed immersions of Tannakian groups
	\begin{equation}
		\Ga\to G_{\geo}\to \cG,
		\label{composition groups}
	\end{equation}
	whose composition sends $1\in \overline{K}\simeq \Lie(\Ga)$ to $N\in \check{\gg}$.

\begin{lemma}\label{non-tame}
        The restriction functor $~|_{\infty}:\langle\Be_{\cG}^{\dagger}\rangle\to \MCF(\mathcal{R}/\overline{K})$ to $\infty\in\P1_k$ induces homomorphisms $I_{\infty}\times \Ga\to G_{\geo}$ which is non-trivial on $P_{\infty}$.
\end{lemma}

\begin{proof}
        If the image $P_{\infty}$ in $\cG_{\geo}$ were trivial,
	by the Grothendieck–Ogg–Shafarevich formula, $\Kl_{\cG}^{\et,\ell}$ would also be tame at $0,\infty$. 
	Then the associated $\ell$-adic representation $\pi_1(X_{\overline{k}})\to \cG$ would factor through the tame quotient $\pi_1^{\tame}(X_{\overline{k}})$, which is isomorphic to $I_{\infty}^{\tame}$ as $X\simeq \Gm$. 
	Since $\Kl_{\cG,V}^{\et,\ell}$ is pure of weight zero for every $V\in \Rep(\cG)$, the geometric monodromy group of $\Kl_{\cG}$ would be semisimple and then finite. 
	This contradicts to fact that $\Kl_{\cG}^{\et,\ell}$ has a principal unipotent monodromy at $0$ (\cite{HNY} Thm. 1). 
\end{proof}
\begin{secnumber} \label{list lie algebras}
	Since every overconvergent $F$-isocrystal $\Be_{\cG,V}^{\dagger}$ is pure of weight $0$ and is therefore geometrically semi-simple (\cite{AC18} 4.3.1), the neutral component $G_{\geo}^{\circ}$ is semi-simple \cite{Cr92}. 
	Therefore, \eqref{composition groups} implies that it contains a principal unipotent element and hence its projection to the adjoint group $\cG_{\mathrm{ad}}$ of $\cG$ contains a principal $\PGL_2$. 
	Then it is almost simple and its Lie algebra appears in one of the following chains: 
	\begin{eqnarray*} 
	&&\xymatrix{
		\mathfrak{sl}_2 \ar[r]& \mathfrak{sp}_{2n} \ar[r] & \mathfrak{sl}_{2n}
	} \\
	&&\xymatrix{
		& & \mathfrak{sl}_{2n+1} \\
		\mathfrak{sl}_2 \ar[r]& \mathfrak{so}_{2n+1} \ar[ru] \ar[rd] & \\
		& & \mathfrak{so}_{2n+2}
	} \\
	&&\xymatrix{
		& & & \mathfrak{sl}_{7} \\
		\mathfrak{sl}_2 \ar[r]& \mathfrak{g}_2 \ar[r]&  \mathfrak{so}_{7} \ar[ru] \ar[rd] & \\
		& & & \mathfrak{so}_{8}
	} \\
	&&\xymatrix{
		\mathfrak{sl}_2 \ar[r]& \mathfrak{f}_{4} \ar[r] & \mathfrak{e}_{6}
	} \\
	&&\xymatrix{
		\mathfrak{sl}_2 \ar[r]& \mathfrak{e}_{7}
	} \\
	&&\xymatrix{
		\mathfrak{sl}_2 \ar[r]& \mathfrak{e}_{8}
	}
\end{eqnarray*}
\end{secnumber}

\begin{lemma} \label{exclude sl2}
	If $\cG$ is not of type $A_1$, and not of type $A_2$ when $p=2$, the image $G_{\geo}\to \cG_{\mathrm{ad}}$ cannot be contained in a principal $\PGL_2$ of $\cG_{\mathrm{ad}}$.
\end{lemma}

\begin{proof}
	The image of the wild inertia group $P_{\infty}$ (resp. $I_{\infty}$) in $\PGL_{2}$ is a finite $p$-group (resp. a solvable group). 
	In view of the all possible finite groups contained in $\PGL_2$, there are two possibilities:

	(a) the image of $P_{\infty}$ is contained in $\Gm\subset\PGL_{2}$; 
	
	(b) $p=2$ and the image of $I_{\infty}$ (resp. $P_{\infty}$) is isomorphic to the alternative group $A_4$ (resp. the group $\mathbb{Z}/2\mathbb{Z}\times \mathbb{Z}/2\mathbb{Z}$).

	To prove the lemma, we follow a similar argument of (\cite{HNY} 6.8), but with the quasi-minuscule representation replaced by the adjoint representation $\Ad$. 
	In any case, by a result of Baldassarri \cite{Ba82} (cf. \cite{And02} 3.2), the maximal $p$-adic slope of $\Be_{\cG,\Ad}^{\dagger}$ is less or equal to the maximal formal slope $1/\check{h}$ of $\Be_{\cG,\Ad}$ \eqref{def FG connection}. 
	Let $r$ be the rank of $\cG$ and $h$ the Coxeter number of $\cG$. 
	Then we deduce that 
	\begin{equation} \label{inequality Irr}
		\Irr_{\infty}(\Be^{\dagger}_{\cG,\Ad})\le \frac{\rank \Ad}{\check{h}}=\frac{\check{h}+1}{\check{h}}r < r+1,
	\end{equation}
	and hence $\Irr_{\infty}(\Be^{\dagger}_{\cG,\Ad})\le r$. 

	On the other hand, we have a decomposition $\Ad\simeq \oplus_{i=1}^{r} S^{2\ell_i}$ as representations of principal $\PGL_2$, where $\{\ell_1+1,\cdots,\ell_{r}+1\}$ is the set of exponents of $\check{\mathfrak{\gg}}$. 
	
	Case (a). Since $\Irr_{\infty}(\Be^{\dagger}_{\cG})\neq 0$, the image of $P_{\infty}$ in $\PGL_2$ contains $\mu_p$ and the image of $I_{\infty}$ is contained in $N(\Gm)$.  
	By a similar argument of (\cite{HNY} 6.8), we deduce $\Irr_{\infty}(S^{2\ell})\ge \ell - \lfloor\ell/p\rfloor\ge 1$. Under our assumption, $\max_i\{\ell_i,p\}>2$, so there is least one $i$ such that  $\ell_i - \lfloor\ell_i/p\rfloor>1$. 
	Then $\Irr_{\infty}(\Be^{\dagger}_{\cG,\Ad}) > r$. Contradiction! 

	Case (b). Recall that there are four irreducible representations of $A_4$: $\id$, two non-trivial one dimensional representation $V_1',V_1^{''}$, the standard representation $V_3$. 
	Via the inclusion $A_4\to \PGL_2$, we have
	\begin{eqnarray*}
		&S^2\simeq V_3,\quad S^4\simeq V_1'\oplus V_1^{''}\oplus V_3, \quad S^6\simeq \id \oplus V_3^{\oplus 2}, \quad S^8\simeq \id\oplus V_1'\oplus V_1^{''}\oplus V_3^{\oplus 2}, &\\
		& S^{10}\simeq V_1'\oplus V_1^{''}\oplus V_3^{\oplus 3}, \quad
		S^{12}\simeq \id^{\oplus 2}\oplus V_1'\oplus V_1^{''}\oplus V_3^{\oplus 3},\quad
		S^{14}\simeq \id\oplus V_1'\oplus V_1^{''}\oplus V_3^{\oplus 4}. &
	\end{eqnarray*}
	In particular, we have $\Irr_{\infty}(S^{2\ell})\ge 2$ for $\ell=3,4,5,6,7$. 
	In general, $I_{\infty}$ acts non-trivially on $S^{2n}$ and we have $\Irr_{\infty}(S^{2\ell})\ge 1$. 
	Then we deduce that $\Irr_{\infty}(\Ad)\ge r(G)+1$. Contradiction!
\end{proof}

Now we prove theorem \ref{thm monodromy}. By the ``trivial'' functoriality \eqref{triv functoriality}, it is enough to prove the theorem when $\cG$ is simply-connected, so that $\cG^{\Sigma}$ is connected.
                
	(a) The case where $\cG$ is not of type $A_{2n}$. 
	In view of lemma \ref{exclude sl2}, and the calculation of $G_{\alg}$ \eqref{table monodromy}, we deduce that $G_{\geo}^{\circ}\to G_{\geo}\to G_{\alg}$ are isomorphisms. Using \ref{symmetry of Kl}, we see $G_{\mathrm{arith}}\subset \cG^{\Sigma}$. This implies that $G_{\mathrm{arith}}=G_{\geo}$ unless $\cG$ is of type $B_3$. In this last case, if $\cG=\Spin_7$, and $G_{\mathrm{arith}}\subset G_2\times Z(\cG)$. Taking into account of the Frobenius at $0$ \eqref{Frobenius 0}, we see that $G_{\mathrm{arith}}=G_{\geo}$.
	        
	(b) The case where $\cG$ is of type $A_{2n}$ and $p>2$. It suffices to exclude that $G_{\geo}$ is contained in $\SO_{2n+1}$. Suppose it is true by contrast. Let $\sigma_0$ be the generator of $\Sigma$ and $\check{\delta}=(-1,\sigma_0)$ in $\Gm\times \Aut(G,B,T)$. Then  we deduce isomorphisms of overconvergent isocrystals on $X_{k}$ 
	\begin{displaymath}
		\Be_{\SL_{2n+1},\Std}^{\dagger}(\check{\xi}) 		
		\simeq (-1)^+ \Be_{\SL_{2n+1},\Std^{\vee}}^{\dagger}(\check{\xi}) 
		\simeq (-1)^+ \Be_{\SL_{2n+1},\Std}^{\dagger}(\check{\xi}),
	\end{displaymath}
         where the first isomorphism follows from \eqref{inv outer automorphism A2n}, and
	 the second one is due to $\Std^{\vee}\simeq \Std$ as representations of $\SO_{2n+1}$. 	
	Since $\Char{k}>2$, this isomorphism provides a ``descent datum'' so that $\Be_{\SL_{2n+1},\Std}^{\dagger}(\check{\xi})$ descends to $\Gm/\mu_2$. It follows that its Swan conductor at $\infty$ is at least two, 
	if non-zero. 
	On the other hand, using lemma \ref{non-tame} and the result of Baldassarri \cite{Ba82} (cf. \cite{And02} 3.2) again, the Swan conductor of $\Be_{\SL_{2n+1},\Std}^{\dagger}(\check{\xi})$ at $\infty$ is $1$, contradiction! 

	(c) The case where $\cG$ is of type $A_{2n}$ and $p=2$. In appendix \eqref{Carlitz diff}, we will identify $\Be_{\SO_{2n+1},\Std}^{\dagger}$ with $\Be_{\SL_{2n+1},\Std}^{\dagger}$. Then we reduce to the case (a). 
\end{proof}

We end this section by some corollaries of our calculation of the monodromy groups. 

\begin{coro} \label{l-adic mono}
	Assume that $\cG$ is almost simple. The monodromy groups $G_{\geo}^{\ell}, G_{\arith}^{\ell}$ of the $\Kl_{\cG}^{\et,\ell}(\psi\phi)$ over $\overline{\mathbb{Q}}_{\ell}$ \eqref{l Kl} are calculated as in theorem \ref{thm geometric mono}. 
\end{coro}
Note that this gives a different proof of the main result of \cite{HNY} theorem 3 (where some explicit small $p$ are excluded).  Our method avoids analyzing some difficult geometry related to quasi-minuscule and adjoint Schubert varieties.
\begin{proof}	
	The monodromy group $G_{\arith}^{\ell}$ (resp. $G_{\arith}$) can be calculated by that of $\Kl_{\cG,V}^{\et,\ell}$ (resp. $\Be_{\cG,V}^{\dagger}$) for a faithful representation $V$ of $\cG$. 
	The semisimplification of $\Kl_{\cG,V}^{\et,\ell}$ and $\Be_{\cG,V}^{\dagger}$ are semi-simple and have same Frobenius traces. 
	Then by (\cite{DZ17} 4.1.1, 4.3.2), there exists a surjective morphism $G_{\arith}^{\ell}\twoheadrightarrow G_{\arith}$. 
	Since they are both closed subgroups of $\cG$, they must be isomorphic to each other and the assertion follows. 
\end{proof}

\begin{coro}\label{local monodromy}	
	Assume that $\cG$ is almost simple. Let $\Ad$ be the adjoint representation of $\cG$. 

	\textnormal{(i)} We have $\rH^{i}(\mathbb{P}^1,j_{!+}(\Be_{\cG,\Ad}^{\dagger}))=0$ for all $i$.
	
	\textnormal{(ii)} We have $\Irr_{\infty}(\Be_{\cG,\Ad}^{\dagger})=r(\cG)$, the rank of $\cG$. In addition, $\Ad^{I_{\infty}}=0$, and the nilpotent monodromy operator $N_\infty=0$ \eqref{review MC Swan}. Therefore, the local Galois representation $I_\infty\to \cG$ is a simple wild parameter in the sense of Gross-Reeder (\cite{GR10} \S~6).
\end{coro}
\begin{proof}
	The corresponding assertions for the algebraic connection $\Be_{\cG,\Ad}$ are proved in (\cite{FG09} \S 14). 
	Set $\mathscr{E}=\Be_{\cG,\Ad}^{\dagger}$, which is self dual. 
	We have $\rH^0(X,\mathscr{E})=\Ad^{G_{\geo}}=0$ and $\rH^{2}(X,\mathscr{E})=0$ by $\mathscr{D}^{\dagger}$-affinity. 
	We obtain $\rH^{i}_{\rc}(X,\mathscr{E})=0$ for $i=0,2$ by the Poincar\'e duality.  
	By the Grothendieck–Ogg–Shafarevich formula and \eqref{inequality Irr}, we have
	\begin{displaymath}
		 \rH^1_{\rc}(X,\mathscr{E})=\Irr_{\infty}(\mathscr{E})\le r(\cG).
	\end{displaymath} 

	Let $j:X\to \P1$ be the inclusion. 
	We have a distinguished triangle
	\begin{displaymath}
		j_!(\mathscr{E})\to j_{!+}(\mathscr{E})\to \rH^0i_0^+j_+(\mathscr{E})\oplus \rH^0i_{\infty}^{+}j_+(\mathscr{E})\to,
	\end{displaymath}
	which induces a long exact sequence: 
	\begin{eqnarray} \label{long exact sequence coh}
		0\to \rH^0(\P1,j_{!+}(\mathscr{E}))\to \rH^0i_0^+j_+(\mathscr{E})\oplus \rH^0i_{\infty}^{+}j_+(\mathscr{E})\xrightarrow{d} \rH^1_{\rc}(X,\mathscr{E})\to   \\
		\rH^1(\P1,j_{!+}(\mathscr{E})) \to 0 \to \rH^2_{\rc}(X,\mathscr{E})=0\to \rH^2(\P1,j_{!+}(\mathscr{E}))\to 0. \nonumber
	\end{eqnarray}
	By the Poincar\'e duality, we conclude that $\rH^i(\P1,j_{!+}(\mathscr{E}))=0$ for $i=0,2$. 

	For $x\in \{0,\infty\}$, the restriction of $\mathscr{E}$ at $x$ gives rise to an action of the inertia group $I_{x}$ on $\Ad$ and a commuting nilpotent monodromy operator $\mathcal{N}_x:\Ad\to \Ad$ \eqref{review MC Swan}. 
	Then $\rH^0i_x^+(j_{+}(\mathscr{E}))$ is calculated by
	\begin{displaymath}
		\Ad^{I_{x},\mathcal{N}_x}:=\Ker(\mathcal{N}_{x}:\Ad^{I_x}\to \Ad^{I_x}).
	\end{displaymath}
	
	The Bessel isocrystal is unipotent at $0$ with $\mathcal{N}_0=[-,N]$ \eqref{log-extension}. We have $\Ad^{I_0,\mathcal{N}_0}=\Ad^{N}$, which has dimension $r(\cG)$. 
	Then the morphism $d$ in \eqref{long exact sequence coh} is both injective and surjective. 
	We deduce that
	\begin{displaymath}
		\Ad^{I_{\infty},\mathcal{N}_{\infty}}=0,\qquad \rH^1(\P1,j_{!+}(\mathscr{E}))=0.
	\end{displaymath}
	Since $\mathcal{N}_{\infty}$ is still a nilpotent operator on $\Ad^{I_{\infty}}$, we conclude assertions (i) and (ii). 
\end{proof}

\begin{rem}
        (i) By corollary \ref{l-adic mono} and the same arguments, we recover \cite{HNY} prop. 5.3 on the analogous statements for $\Kl_{\cG}$ (and remove the restriction of the characteristic of $k$ in \emph{loc. cit.}).
        
        (ii)  It follows from \cite{GR10} prop. 5.6 that when $p$ does not divide the order $\sharp W$ of Weyl group, the only non-zero break of $\Be_{\cG,\Ad}^{\dagger}$ (and $\Kl_{\cG}$) at $\infty$ is $1/\check{h}$. Indeed, the local Galois representation $I_\infty\to \cG$ is described explicitly in \cite{GR10} prop. 5.6 and \S~6.2.
        
	(iii) It is expected that the description in (ii) of the local monodromy of $\Be_{\cG}^\dagger$ (and $\Kl_{\cG}$) at $\infty$ should hold when $(p,h)=1$. When $\cG=\GL_n$, this is indeed the case. For $\Kl_n$, this was proved by Fu and Wan (\cite{FW05} theorem 1.1). For $\Be_n^{\dagger}$, the can be shown
	 by studying the solutions of Bessel differential equation \eqref{Bessel connection intro} at $\infty$. We omit details and refer to (\cite{Ohk18} 6.7) for a treatment in the case when $n=2$.

	(iv) Using theorem \ref{thm monodromy} (ii), which will be proved in the appendix \ref{Carlitz diff}, we see that when $p=2$ and $n$ is an odd integer, the associated local Galois representation of $\Be_{\SO_n}^{\dagger}$ at $\infty$ coincides with the simple wild parameter constructed by Gross-Reeder in \cite{GR10} \S~6.3. 
	In particular, the image of the inertia group $I_{\infty}$ in the case $\cG=\SO_3$ is isomorphic to $A_4$. 
	Together with $\Be_{\SO_3,\Std}^{\dagger} \simeq \Be_{\SL_2,\Sym^2}^{\dagger}$ \eqref{SL2 vs SO3}, this allows us to recover Andr\'e's result on the local monodromy group of $\Be_{2}^{\dagger}$ at $\infty$ in the case $p=2$ (\cite{And02} \S~7, 8). 
\end{rem}

\section{Applications} \label{sec applications}
In this section, we give some applications of our study of Bessel $F$-isocrystals for reductive groups.

\subsection{Functoriality of Bessel $F$-isocrystals} \label{sec functoriality}         
		We may ask all possible Frobenius structure on $\Be_{\cG}^{\dagger}(\check{\xi})$ (not necessarily the one from \ref{F on Be}), i.e. all possible isomorphisms of tensor functors $\varphi:F^{*}_{X}\circ \Be_{\cG}^{\dagger}\xrightarrow{\sim} \Be_{\cG}^{\dagger}$. 
		
\begin{lemma} \label{Uniqueness Frobenius}
	The Frobenius structure on $\Be_{\cG}^{\dagger}(\check{\xi})$ is unique up to an element in the center $Z_{\cG}(\overline{K})$ of $\cG$. 
\end{lemma}

\begin{proof}
	Given two Frobenius structures $\varphi_1,\varphi_2$, $u:=\varphi_{2}\circ \varphi_{1}^{-1}$ is an isomorphism of tensor functors $\Be_{\cG}^{\dagger}(\check{\xi})\xrightarrow{\sim} \Be_{\cG}^{\dagger}(\check{\xi})$. 
	If $\omega$ denotes a fiber functor of $\langle\Be_{\cG}^{\dagger}(\check{\xi})\rangle$, then $\omega\circ u$ is an element in $\cG(\overline{K})$ commuting with $G_{\geo}(\overline{K})$ by the Tannakian formalism. Then the assertion follows from $Z_{\cG}(G_{\geo})=Z_{\cG}$.
\end{proof}

\begin{secnumber} \label{setting functoriality}
		 Let $G,G'$ be two split, almost simple groups over $R$ whose Langlands dual groups $\cG'\subset \cG$ over $\overline{K}$ appear in the same line in the left column of the \eqref{table monodromy}. 
		 Up to conjugation, we can assume that the inclusion $\cG'\subset \cG$ preserves the pinning. Then it induces a natural inclusion $\check{\gg}'_{\mathrm{aff}}(1)\subset \check{\gg}_{\mathrm{aff}}(1)$. 
		 Let $\phi'$ be a generic linear function of $G'$ over $R$ \eqref{F on Be} and $\check{\xi}$ the generic element in $\check{\gg}'_{\mathrm{aff}}(1)$ corresponding to $-\pi\phi'$ \eqref{F on Be}. 
		Note that $\check{\xi}$ is also a generic element in $\check{\gg}_{\mathrm{aff}}(1)$.
\end{secnumber}

\begin{prop}
	\textnormal{(i)} There exists a generic linear function $\phi$ of $G$ over $R$ such that $-\pi\phi$ matches $\check{\xi}\in \check{\gg}_{\mathrm{aff}}(1)$ under the isomorphism \eqref{classical duality}. 
	
	\textnormal{(ii)} Let $(\Be_{\cG'}^{\dagger}(\check{\xi}),\varphi')$ (resp. $(\Be_{\cG}^{\dagger}(\check{\xi}),\varphi)$) be the Bessel $F$-isocrystal of $\cG'$ (resp. $\cG$) constructed by $\phi'$ (resp. $\phi$) in \ref{F on Be}. 
	Then $(\Be_{\cG}^{\dagger}(\check{\xi}),\varphi)$ is the push-out of $(\Be_{\cG'}^{\dagger}(\check{\xi}),\varphi)$. 
	\label{functoriality}
\end{prop}

\begin{proof}
	(i) Let $\phi$ be the generic linear function of $G$ over $K$ such that $-\pi\phi$ corresponds to $\check{\xi}$ under the isomorphism \eqref{classical duality}. We will show that $\phi$ is naturally integral. 

	By construction, $\Be_{\cG}(\check{\xi})$ is the push-out of $\Be_{\cG'}(\check{\xi})$. In particular, for $V\in \Rep(\cG)$, the connection $(\Be_{\cG,V}(\check{\xi}))^{\dagger}$ has a Frobenius structure and is overconvergent. 
	Let $\chi$ be a generic linear function of $G$ over $R$ and $\check{\eta}\in \check{\gg}_{\mathrm{aff}}(1)$ the corresponding generic element. 
	Then there exists an element $c\in K^{\times}$ such that we can rewrite two Bessel connections for the adjoint representation of $\cG$ as follow \eqref{Bessel connection form}:
	\begin{equation}
		\Be_{\cG,\Ad}(\check{\eta})=d+(N+xE)\frac{dx}{x},\quad \Be_{\cG,\Ad}(\check{\xi})=d+(N+cxE)\frac{dx}{x}.
	\end{equation}
	Via \eqref{classical duality}, it suffices to show that $c\in R^{\times}$.
	
	Both the above two connections admit Frobenius structures and decompose in the categories $\Conn(X_{\overline{K}})$, $\Sm(X_{k}/\overline{K})$ and $\Sm(X_{k}/\overline{K}_F)$ in the same way (according to the decomposition of $\Ad$ in $\Rep(\cG')$ by theorem \ref{thm monodromy}). 
	Let $V$ be a non-trivial irreducible component of $\Ad$ in $\Rep(\cG')$ and $V(\check{\eta}), V(\check{\xi})$ the corresponding overconvergent $F$-isocrystal. 
	Since $V(\check{\eta})|_0$ is unipotent, if $\{e_i\}$ denotes a basis of $V$, there exists a solution 
	\begin{displaymath}
		u: e_i \mapsto f_i(x) \in \Sol(V(\check{\eta})|_0)\quad \eqref{solution space}
	\end{displaymath}
	whose convergence domain is the open unit disc of radius $1$. 
	Then $u_c: e_i \mapsto f_i(cx)$ belongs to $\Sol(V(\check{\xi})|_0)$ and has the same convergent radius. 
	If $c$ is not a $p$-adic unit, then $V(\check{\eta})$ (or $V(\check{\xi})$) admits the trivial overconvergent isocrystal on $X_{k}$ as a quotient, which contracts to their irreducibility. 
	The assertion follows. 
	
	(ii) By (i), the $\cG$-valued overconvergent isocrystal $\Be_{\cG}^{\dagger}(\check{\xi})$ is the push-out of $\Be_{\cG'}^{\dagger}(\check{\xi})$. 
	It remains to identify two Frobenius structures on $\cG$-valued overconvergent isocrystals $\Be^{\dagger}_{\cG}(\check{\xi})\simeq \Be^{\dagger}_{\cG'}(\check{\xi})\times^{\cG'}\cG$, which are different by an element $\varepsilon$ in the center $Z_{\cG}(\overline{K})$ by \eqref{Uniqueness Frobenius}. 
	Taking account of the extension of Frobenius structures to $0$ \eqref{Frobenius 0}, we deduce that $\varepsilon=\id$ and the assertion follows. 
\end{proof}

Now we can prove the following conjecture of Heinloth-Ng\^o-Yun (\cite{HNY} conjecture 7.3). 

\begin{theorem} \label{thm functoriality}
	We keep the notation of \ref{setting functoriality} and fix a non-trivial additive character $\psi$.
        Assume that $\cG'\subset \cG$ over $\overline{\mathbb Q}_\ell$ appear in the same line in the left column of the \eqref{table monodromy}.
        For every generic linear function $\phi'$ of $G'$ over $k$, there is a generic linear function $\phi$ of $G$ over $k$ such that
	$\Kl_{\cG}^{\et,\ell}(\psi\phi)$ is isomorphic to the push-out of  $\Kl_{\cG'}^{\et,\ell}(\psi\phi')$ along $\cG'\subset \cG$ as $\ell$-adic $\cG$-local systems on $X_k$. 
\end{theorem}

\begin{proof}
	By the ``trivial'' functoriality \eqref{triv functoriality}, we may assume that $\cG$ is simply connected. 
	We lift $\phi$ to be a generic linear function of $G'$ over $R$ and take $\phi'$ as in \ref{functoriality}. 
	We need to show that $\Kl_{\cG}^{\et,\ell}(\psi\phi)\simeq \Kl_{\cG'}^{\et,\ell}(\psi\phi')\times^{\cG'}\cG$ as $\cG$-local systems.
	It follows from theorem \ref{def Bessel crystal} and proposition \ref{functoriality} that for every representation $V\in \Rep(\cG)$, regarded as a representation of $\cG'$, and every $a\in |X_{k}|$, we have
        \[
             \Tr(\mathrm{Frob}_{a} | \Kl_{\cG,V,\bar{a}}^{\et,\ell}) =   \Tr(\mathrm{Frob}_{a} | \Kl_{\cG',V,\bar{a}}^{\et,\ell}). 
        \]
	
	Note that if $\Sigma$ is the group of pinned automorphisms of $\cG$, then the closed embedding $\cG^{\Sigma}\to \cG$ induces a surjective homomorphism of K-rings $K(\Rep(\cG))\otimes \overline{\mathbb Q}_\ell\to K(\Rep(\cG^{\Sigma}))\otimes \overline{\mathbb Q}_\ell$. 
	Then the homomorphism $K(\Rep(\cG))\otimes \overline{\mathbb Q}_\ell\to K(\Rep(G_{\geo}))\otimes \overline{\mathbb Q}_\ell$ is also surjective. 
	It follows that if we replace $V$ by any representation $W$ of $G_{\geo}$ ($\subset \cG'\subset \cG$), the above equality holds. 
	This implies that the Frobenius conjugacy classes of $\Kl_{\cG}^{\et,\ell}$ and of $\Kl_{\cG'}^{\et,\ell}$ have the same image in $G_{\geo}/\!\!/G_{\geo}$. 
	Now, for a faithful representation $W$ of $G_{\geo}$, two representations $\Kl_{\cG}^{\et,\ell},\Kl_{\cG'}^{\et,\ell}:\pi_{1}(X_{k},\overline{x})\to G_{\geo}(\overline{\mathbb{Q}}_{\ell})$ are conjugated in $\GL(W)$ by an element $g$. 
	This element $g$ induces an automorphism of $G_{\geo}$. It fixes every Frobenius conjugacy class and therefore fixes $G_{\geo}/\!\!/G_{\geo}$. 
	Then $g$ must be inner. That is these two representations are conjugate in $G_{\geo}$ and the assertion follows. 
\end{proof}

\subsection{Hypergeometric $F$-isocrystals} \label{sec hyper}
To describe Bessel $F$-isocrystals for classical groups, we need to review some basic facts about the hypergeometric $F$-isocrystals. 

\begin{secnumber} \label{hyp equation Frobenius}
	In (\cite{Katz90} 5.3.1), Katz interpreted hypergeometric $\mathscr{D}$-modules on $\Gm$ as the multiplicative convolution of hypergeometric $\mathscr{D}$-modules of rank one. 
	Besides the hypergeometric $\mathscr{D}$-modules, Katz also studied $\ell$-adic theory of hypergeometric sheaves using multiplicative convolution. 
	The Frobenius traces of these sheaves are called \textit{hypergeometric functions (over finite fields)} which generalize Kloosterman sums \eqref{Kl sum}. 
	
	Let $\psi$ be a non-trivial additive character on $\mathbb{F}_p$, $n$ an integer $\ge 1$ and $\underline{\rho}=(\rho_1,\cdots, \rho_{m})$ a sequence of multiplicative characters on $k^{\times}$. The hypergeometric function $\rH_{\psi}(n,\underline{\rho})$ \footnote{It corresponds to the hypergeometric function associated to $\psi$, $n$ trivial characters $\chi$'s and $m$ characters $\rho$'s defined in (\cite{Katz90} 8.2.7)} is defined for any finite extension $k'/k$ and $a\in k'^{\times}$ by
\begin{equation}
	\rH_{\psi}(n,\underline{\rho})(a)=\sum \psi\biggl(\Tr_{k'/\mathbb{F}_p}(\sum_{i=1}^{n}x_i-\sum_{j=1}^{m}y_j)\biggr) \cdot \prod_{j=1}^{m} \rho_{j}^{-1}(\Nm_{k'/k}(y_j)),
	\label{hyp sums}
\end{equation}
where the sum take over $(x_1,\cdots,x_n,y_1,\cdots,y_m)\in (k'^{\times})^{m+n}$ satisfying $\prod_{i=1}^{n}x_i=a \prod_{j=1}^{m}y_{j}$. 

	Recently, Miyatani studied the $p$-adic counterpart of this theory \cite{Miy}. Using the multiplicative convolution of arithmetic $\mathscr{D}$-modules, he constructed the Frobenius structure on hypergeometric $\mathscr{D}$-modules whose Frobenius traces are hypergeometric sums. 
	In the following, we briefly recall his results in some special cases. 

	Let $\mathscr{M},\mathscr{N}$ be two objects of $\rD(\mathbb{G}_{m,k}/L_{\blacktriangle})$ and $\mu:\Gm\times \Gm\to \Gm$ the multiplication morphism. Recall that the \textit{(multiplicative) convolution} $\star$ 
	is defined by 
	\begin{equation}
		\mathscr{M}\star \mathscr{N}= \mu_{!}(\mathscr{M}\boxtimes \mathscr{N}).
		\label{def convolution}
	\end{equation}

	Let $n>m$ be two non-negative integers, $\pi\in K$ associated to $\psi$ \eqref{Dwork isocrystal} and $\underline{\beta}=(\beta_1,\cdots,\beta_m)$ a sequence of elements of $\frac{1}{q-1}\mathbb{Z}-\mathbb{Z}$. 
	We denote by $\Hyp_{\pi}(n,\underline{\beta})$ the $p$-adic hypergeometric differential operator on $\Gm$ 
	\begin{equation}
		\Hyp_{\pi}(n,\underline{\beta})=\delta^{n}-(-1)^{n+mp}\pi^{n-m}x\prod_{i=1}^{m}(\delta-\beta_j),
		\label{p-adic hyp}
	\end{equation}
	where $x$ is a coordinate of $\Gm$ and $\delta=x\frac{d}{dx}$. 
	We denote by $\DHyp_{\pi}(n,\underline{\beta})$ the $\mathscr{D}^{\dagger}_{\widehat{\mathbb{P}}^1,\mathbb{Q}}(\{0,\infty\})$-module
	\begin{equation}
		\DHyp_{\pi}(n,\underline{\beta})= \mathscr{D}^{\dagger}_{\widehat{\mathbb{P}}^1,\mathbb{Q}}(\{0,\infty\})/ (\mathscr{D}^{\dagger}_{\widehat{\mathbb{P}}^1,\mathbb{Q}}(\{0,\infty\})\Hyp_{\pi}(n,\underline{\beta}) ).
		\label{D-mod hyp}
	\end{equation}
\end{secnumber}

\begin{theorem}[Miyatani \cite{Miy}]
	\label{theorem Miyatani}
	We fix an isomorphism $\overline{\mathbb{Q}}_p\simeq \mathbb{C}$. 

	\textnormal{(i)} The arithmetic $\mathscr{D}$-module $\DHyp_{\pi}(n,\underline{\beta})$ underlies to a pure overconvergent $F$-isocrystal on $\mathbb{G}_{m,k}$ of rank $n$ and weight $n+m-1$. 

	\textnormal{(ii)} The Frobenius structure on the overconvergent isocrystal $\DHyp_{\pi}(n,\underline{\beta})$ is unique (up to a scalar). 

	\textnormal{(iii)} The Frobenius trace of $\DHyp_{\pi}(n,\underline{\beta})$ on $\mathbb{G}_{m,k}$ is equal to the hypergeometric function $(-1)^{n+m-1}\rH_{\psi}(n,\underline{\rho})$ \eqref{hyp sums}, 
	where $\rho_i$ is defined for $\xi\in k^{\times}$ and $\widetilde{\xi}$ the Teichmüller lifting of $\xi$, by $\rho_i(\xi)=\widetilde{\xi}^{(q-1)\beta_i}$.
\end{theorem}

Assertions (i) and (iii) are stated in (\cite{Miy} Main theorem). One can apply the method of (\cite{Miy} 4.2.1) to show that $\DHyp_{\pi}(n,\underline{\beta})$ is irreducible in the category $\rD(\mathbb{G}_{m,k}/K)$ and hence assertion (ii). 

The arithmetic $\mathscr{D}$-module $\DHyp_{\pi}(n,\underline{\beta})$ depends only on $\psi$ and $\underline{\rho}$ (the class of $\underline{\beta}$ modulo $\mathbb{Z}$) that we also denote by $\DHyp_{\psi}(n,\underline{\rho})$. 

\begin{secnumber} \label{Frob trace Hyp}
	\textbf{Normalised Hypergeometric sum.}
	Let $\mathscr{F}$ be the hypergeometric $\ell$-adic sheaf on $\mathbb{G}_{m,k}$ associated to $\psi$, $n$ trivial multiplicative characters and non-trivial multiplicative characters $\underline{\rho}=(\rho_1,\cdots,\rho_m)$. 
	The space $\mathscr{F}^{I_{0}}$ of $I_{0}$-invariants is one-dimensional and $\Frob_{k}$ acts on it as the monomial in Gauss sums (\cite{Katz09} 2.6.1)
	\begin{displaymath}
		\alpha=(-1)^{m} \prod_{j=1}^{m}G(\psi^{-1},\rho_j^{-1}),
	\end{displaymath}
	where $G(\psi^{-1},\rho_{j}^{-1})$ denotes the Gauss sum associated to $\psi^{-1}$ and $\rho_{j}^{-1}$. 

	On the other hand, note that the action of $I_{0}$ is maximal unipotent. Any lifting $F_{0}$ in the decomposition group $D_{0}$ at $0$ of the Frobenius automorphism has eigenvalues set $\{\alpha, q\alpha,\cdots,q^{2n}\alpha\}$ (cf. \cite{Katz88} 7.0.7). 
	After twisting a geometrically constant lisse rank one sheaf (resp. overconvergent $F$-isocrystal), we denote by $\widetilde{\mathscr{F}}$ (resp. $\THyp_{\psi}(n,\underline{\rho})$) the normalised hypergeometric sheaf (resp. $F$-isocrystal) whose the Frobenius eigenvalues at $0$ is $\{q^{-(n-1)/2},\cdots.,q^{(n-1)/2}\}$. 
	Its Frobenius trace function, called the \textit{normalised hypergeometric sum} $\widetilde{\rH}_{\psi}(n,\underline{\rho})$ is defined for $a\in \mathbb{F}_q^{\times}$ by 
	\begin{equation} \label{Hyp sum normalised}
		\widetilde{\rH}_{\psi}(n,\underline{\rho})(a)
		=\frac{1}{(-\sqrt{q})^{n-1}\prod_{j=1}^m G(\psi^{-1},\rho^{-1})} \rH_{\psi}(n,\underline{\rho})(a).
	\end{equation}
	When $m=0$, we have $\widetilde{\mathscr{F}}=\Kl_n^{\et}$ \eqref{definition Kl} and $\THyp_{\psi}(n,\emptyset)=\Be_n^{\dagger}$ \eqref{Dwork thm intro}.
\end{secnumber}

\subsection{Bessel $F$-isocrystals for classical groups} \label{Kl classical} 
\begin{secnumber} \label{Bessel SL}  
	The Kloosterman sheaf and the Bessel $F$-isocrystal for $(G=\GL_n,\cG=\GL_n)$ have been extensively studied. As usual, let $E_{ij}$ denote the $n\times n$-matrix with the $(i,j)$-entry $1$ and all other entries $0$. 
	We choose the standard Borel $B$ of the upper triangular matrices and the standard torus $T$ of the diagonal matrices. 
	We choose a coordinate $x$ of $\A1$. 
	Then there is a canonical isomorphism 
	\[\Ga^n\simeq I(1)/I(2), \quad (a_1,\ldots,a_n)\mapsto  \sum_{i=1}^{n-1} a_iE_{i,i+1}+ a_n x^{-1}E_{n,1}.
	\]
	We choose $\phi:\Ga^n\to \Ga$ to be the addition map. Under the isomorphism \eqref{classical duality} and \eqref{invariant poly match}, $\phi$ corresponds to $\check{\xi}=N+Edx$ \eqref{normalization} with
	\begin{equation} \label{Bessel connection}
		N = \begin{pmatrix}
		0 & 1 & 0 & \dots & 0 \\
		0 & 0 & 1 & \dots & 0 \\
		\vdots & \vdots & \ddots & \ddots & \vdots \\
		0 & 0 & 0 & \dots & 1 \\
		0 & 0 & 0 & \dots & 0  \end{pmatrix}, \quad 
		E = \begin{pmatrix}
		0 & 0 & 0 & \dots & 0 \\
		0 & 0 & 0 & \dots & 0 \\
		\vdots & \vdots & \ddots & \ddots & \vdots \\
		0 & 0 & 0 & \dots & 0 \\
		1 & 0 & 0 & \dots & 0  \end{pmatrix}. 
	\end{equation}
	On the other hand, by \ref{triv functoriality} and (\cite{HNY} \S 3), we have 
         \begin{equation}\label{Kl explicit}
		 \Kl_{\SL_n,\Std}^{\et} \simeq \Kl_{\GL_n,\Std}^{\et} \simeq \Kl^{\et}_n.
         \end{equation}
	Indeed, diagram \eqref{triangle} reduces to diagram \eqref{Del diag} in this case.  Therefore, the Kloosterman connection is isomorphic to the classical Bessel connection (\ref{Bessel conn intro}, \ref{Dwork thm intro}) 
	\begin{equation}
		\Kl^{\dR}_{\SL_n,\Std}(\lambda\phi)\simeq \Be_n,\qquad
		\Kl^{\rig}_{\SL_n,\Std}(\phi)\simeq \Be_n^{\dagger}.
		\label{Kl dR rig SLn}
	\end{equation}
	Recall that the connection $\Be_{n}$ corresponds to the Bessel differential equation \eqref{Bessel connection intro}.
\end{secnumber}
         
\begin{secnumber}
         Consider 
         \[
		 G=\SO_{2n+1}, \quad \cG=\Sp_{2n}=\{ A\in \SL_{2n} \mid AJA^T=J\},
         \]
         where $J$ is the anti-diagonal matrix with $J_{ij}=(-1)^i\delta_{i,2n+1-j}$. 
	 Then matrices $(N,E)$ as in \eqref{Bessel connection} are in $\check{\gg}$ and $\Be_{\cG}(\check{\xi})$ is given by the same formula as $\GL_{2n}$ case. 
	 Then we deduce an isomorphism of overconvergent $F$-isocrystals $\Be_{\Sp_{2n},\Std}^{\dagger}(\check{\xi})\simeq \Be_{2n}^{\dagger}$ by \eqref{functoriality}. 
\end{secnumber}

\begin{secnumber}
	Consider 
	\[
		G=\SO_{2n}\ \cG=\SO_{2n}=\{A\in \SL_{2n}, AJA^T=J\},	
	\]
	where $J$ is the anti-diagonal matrix with $J_{ij}=(-1)^{\max\{i,j\}}\delta_{i,2n+1-j}$. 
	There exists a canonical isomorphism
	\begin{displaymath}
		\Ga^{n+1}\simeq I(1)/I(2),\quad (a_1,\cdots,a_{n+1})\mapsto \sum_{i=1}^{n-1} (E_{i,i+1}+E_{2n-i,2n-i+1}) + (E_{n-1,n+1}+E_{n,n+2})+ x^{-1}(E_{1,2n-1}+E_{2,2n}).
	\end{displaymath}
	Then we take $\phi:\Ga^{n+1}\to \Ga$ to be the addition map. When $n\ge 3$, under the isomorphism \eqref{classical duality} and \eqref{invariant poly match}, $\lambda\phi$ corresponds to $\check{\xi}=N+ \lambda^{2n-2} E x$ \eqref{normalization} with
	\begin{tiny}
\begin{equation} \label{Bessel connection SO2n matrix}
		N= \begin{pmatrix}
			0 & 1 & 0 & 0 & \dots & \dots & \dots & 0 \\
			\vdots & \ddots & \ddots & & & & \dots & \vdots \\
		\vdots & \ddots & 0 & 1 & 1 & 0 & \dots & 0 \\
		 &  & \dots & 0 & 0 & 1 & & \vdots \\
		 &  & &  & 0 & 1 &  & \vdots \\
		0 & &   &  &  & 0 & \ddots & \vdots \\
		0 & 0 &   &   &   &  & \ddots & 1 \\
		0 & 0 & 0 & \dots & \dots & &  & 0 
	\end{pmatrix}, \quad
	        E= \begin{pmatrix}
			0 & 0 & 0 & 0 & \dots & \dots & \dots & 0 \\
			\vdots & \ddots & \ddots & & & & \dots & \vdots \\
		\vdots & \ddots & 0 & 0 & 0 & 0 & \dots & 0 \\
		 &  & \dots & 0 & 0 & 0 & & \vdots \\
		 &  & &  & 0 & 0 &  & \vdots \\
		0 & &   &  &  & 0 & \ddots & \vdots \\
		1 & 0 &   &   &   &  & \ddots & 0 \\
		0 & 1 & 0 & \dots & \dots & &  & 0 
	\end{pmatrix}.
	\end{equation}
	\end{tiny}
	
	The corresponding Bessel connection is written as
\begin{equation} \label{Bessel connection SO2n}
	\Be_{\SO_{2n},\Std}(\check{\xi})=d + (N + \lambda^{2n-2} E x)\frac{dx}{x}.
\end{equation}
If $e_1,\cdots,e_{2n}$ denote a basis for the above connection matrix, the restriction of the above connection to the subbundle generated by $e_n-e_{n+1}$ is trivial. 
The other horizontal subbundle, generated by $e_n+e_{n+1}$ and other basis vectors, is isomorphic to the Bessel connection $\Be_{\SO_{2n-1},\Std}(\check{\xi})$ discussed below \eqref{decomposition SO2n+2 SO2n+1}. 
\end{secnumber}

\begin{secnumber}
	In \cite{LT17}, T. Lam and N. Templier identified the diagram \eqref{triangle} with the Laudau-Ginzburg model for quadrics \cite{PRW} and used it to calculate the associated Kloosterman $\mathscr{D}$-modules. 
	We briefly recall this construction following (\cite{PRW} \S~3). 
	Let $Q_{2n-2}=G/ P$ be the $(2n-2)$-dimensional quadric and let $(p_0:\cdots:p_{n-1}:p'_{n-1}:p_n:\cdots:p_{2n-2})$ be the Pl\"ucker coordinates of $Q_{2n-2}$ satisfying 
	\begin{equation} \label{quadric eq}
		p_{n-1}p'_{n-1}-p_{n-2}p_{n}+\cdots+(-1)^{n-1}p_{0}p_{2n-2}=0.
	\end{equation}
	Consider the open subscheme 
	\begin{equation}
		Q_{2n-2}^{\circ}=Q_{2n-2}-D,
		\label{open subscheme quadric}
	\end{equation}
	with the complement $D=D_0+D_1+\cdots+D_{n-1}+D_{n-1}'$, where $D_i$ is defined by 
	\begin{equation} \label{divsors SO2n}
	\begin{array}{l}
		{D_{0} :=\left\{p_{0}=0\right\}} \\ 
		{D_{\ell} :=\left\{\sum_{k=0}^{\ell}(-1)^{k} p_{\ell-k} p_{2 n-2-\ell+k}=0\right\} \text { for } 1 \leq \ell \leq n-3} \\ 
		{D_{n-2} :=\left\{p_{2 n-2}=0\right\}} \\ 
		{D_{n-1} :=\left\{p_{n-1}=0\right\}} \\ 
		{D'_{n-1} :=\left\{p'_{n-1}=0\right\}}
	\end{array}
\end{equation}
The divisor $D$ is anti-canonical in $Q_{2n-2}$. For simplicity, we set 
\begin{displaymath}
	\delta_{\ell}=\sum_{k=0}^{\ell}(-1)^{k} p_{\ell-k} p_{2n-2-\ell+k},\quad \text { for } 0 \leq \ell \leq n-3. 
\end{displaymath}
If $x$ denotes a coordinate of $\Gm$, we define a regular function $W:Q_{2n-2}^{\circ}\times \Gm\to \mathbb{A}^1$ to be
\begin{equation} \label{superpotential}
	W(p_i:p'_{n-1};x)=
	\frac{p_{1}}{p_{0}}+
	\sum_{\ell=1}^{n-3} \frac{p_{\ell+1} p_{2 n-2-\ell}}{\delta_{\ell}}+
	\frac{p_{n}}{p_{n-1}}+\frac{p_{n}}{p_{n-1}^{\prime}}+ x \frac{p_{1}}{p_{2 n-2}}.
\end{equation}

The Kloosterman overconvergent $F$-isocrystal and connection are calculated by
\begin{equation}
	\Kl_{\SO_{2n},\Std}^{\rig}(\phi) \simeq \pr_{2,!}(W^{*}(\mathscr{A}_{\psi}))[2(n-1)](n-1),\quad
	\Kl_{\SO_{2n},\Std}^{\dR}(\lambda\phi) \simeq \pr_{2,!}(W^{*}(\mathsf{E}_{\lambda}))[2(n-1)]. \label{Calculation LG model}
\end{equation}
	We deduce that the Frobenius trace $\Kl_{\SO_{2n},\Std}$ of $\Kl_{\SO_{2n},\Std}^{\rig}(\phi)$ is defined for $a\in \mathbb{F}_q^{\times}$ by 
	\begin{equation}
		\Kl_{\SO_{2n},\Std}(a)=\frac{1}{q^{n-1}}\sum_{(p_i,p_{n-1}')\in Q_{2n-2}^{\circ}(\mathbb{F}_q)}\psi\biggl( \Tr_{\mathbb{F}_q/\mathbb{F}_p} \biggl(\frac{p_{1}}{p_{0}}+
	\sum_{\ell=1}^{n-3} \frac{p_{\ell+1} p_{2 n-2-\ell}}{\delta_{\ell}}+
	\frac{p_{n}}{p_{n-1}}+\frac{p_{n}}{p_{n-1}^{\prime}}+ a \frac{p_{1}}{p_{2 n-2}} \biggr)
	\biggr).
		\label{Kl sum SO2n}
	\end{equation}
\end{secnumber}
\begin{prop} \label{cal Kl sum SO2n}
	\textnormal{(i)} When $n=2$, we have
	\begin{equation}
		\Kl_{\SO_4,\Std}(a)=\Kl(2;a)^2.
	\end{equation}

	\textnormal{(ii)} When $n\ge 3$, we can simplify above sum as
	\begin{equation} \label{Kl sum SO2n simple}
		\Kl_{\SO_{2n},\Std}(a)=\frac{1}{q^{n-1}}\biggl(\sum_{x_i\in \mathbb{F}_q^{\times}}\psi\biggl(\Tr_{\mathbb{F}_q/\mathbb{F}_p}(x_1+x_2+\cdots +x_{2n-2}+a\frac{x_1+x_2}{x_1x_2\cdots x_{2n-2}})\biggr)+(q-1)q^{n-2}\biggr).
	\end{equation}
\end{prop}
\begin{proof}
	Assertion (i) is easy to prove and is left to readers. It also follows from \eqref{SO4 vs SL2}. 

	(ii) The equality follows from subdividing the sum \eqref{Kl sum SO2n} in the following parts:

	(a) Case $p_{n},p_{n+1},\cdots,p_{2n-3}\neq 0$: we replace $p_i,p_{n-1}'$ by $x_i,y_i\in \mathbb{F}_q^{\times}$ as follows:
	\begin{eqnarray*}
		p_{k}&=&\left\{\begin{array}{ll} {1} & {\text { if } k=0} \\ {x_{1} \ldots x_{k-1}\left(x_{k}+y_{k}\right)} & {\text { if } 1 \leq k \leq n-2} \\ {x_{1} \ldots x_{n-2} x_{n-1}} & {\text { if } k=n-1} \\ {x_{1} \ldots x_{n-2} x_{n-1} y_{n-1}} & {\text { if } k=n} \\ {x_{1} \ldots x_{n-2} x_{n-1} y_{n-1} y_{n-2} \ldots y_{2 n-1-k}} & {\text { otherwise }} \end{array}\right. \\
			p_{n-1}'&=&x_1\cdots x_{n-2} y_{n-1}.
	\end{eqnarray*}
	Then the sum \eqref{Kl sum SO2n} becomes the toric exponential sum in \eqref{Kl sum SO2n simple}. 

	(b) Case $p_{n}=0$ and $p_{2n-2-\ell}\neq 0$ for some $\ell\in\{1,\cdots,n-3\}$: we assume $\ell$ is maximal. 
	By dividing $p_{n-1}'$, we consider the affine coordinates $p_0,\cdots,p_{2n-2}$ and we replace $p_{n-1}$ by the equation \eqref{quadric eq}. 
	Since $p_{n},\cdots,p_{2n-2-\ell-1}=0$, $p_{\ell+1}$ can be taken in $\mathbb{F}_q$ regardless of the condition $\delta_{\ell}\neq 0$. 
	Then we have $\sum_{p_{\ell+1}\in \mathbb{F}_q}\psi(\frac{p_{\ell+1}p_{2n-2-\ell}}{\delta_{\ell}})=0$ and that the sum \eqref{Kl sum SO2n} equals to zero in this case.

	(c) Case $p_{n}=p_{n+1}=\cdots = p_{2n-3}=0$: it is easy to show that the sum \eqref{Kl sum SO2n} equals to $\frac{q-1}{q}$, which is the constant part of \eqref{Kl sum SO2n simple}. 
\end{proof}

\begin{secnumber} \label{Bessel SO2n1}
         Consider 
         \[
		G=\Sp_{2n}, \quad \cG=\SO_{2n+1}=\{ A\in \SL_{2n+1} \mid AJA^T=J\},
         \]
	where $J$ is the anti-diagonal matrix with $J_{ij}=(-1)^i\delta_{i,2n+2-j}$.
	There exists a canonical isomorphism
	\begin{displaymath}
		\Ga^{n+1}\simeq I(1)/I(2),\quad (a_1,\cdots,a_{n+1})\mapsto \sum_{i=1}^{n-1} (E_{i,i+1}+E_{2n-i,2n-i+1}) + E_{n-1,n}+ x^{-1}E_{2n,1}. 
	\end{displaymath}
	Then we take $\phi:\Ga^{n+1}\to \Ga$ to be the addition map. 
	Under the isomorphism \eqref{classical duality} and \eqref{invariant poly match}, $\lambda\phi$ corresponds to $\check{\xi}=N+ \lambda^{2n} E x$ \eqref{normalization} with $N$ as in \eqref{Bessel connection}, which belongs to $\check{\gg}$, and 
       	\begin{equation} \label{SO2n+1 connection}
	 E  = \left( \begin{array}{ccccc}
		0 & 0 & \dots & \dots & 0 \\
		\vdots & \vdots & & & \vdots \\
		0 & 0 & \dots &  &  \\
		2 & 0 & \dots & \dots & 0 \\
		0 & 2 & 0 & \dots & 0  \end{array} \right) \quad \in \check{\gg}. 
	\end{equation}
	Then we can write the Bessel connection as
	\begin{equation}
		\Kl_{\SO_{2n+1},\Std}^{\dR}(\lambda\phi)\simeq \Be_{\SO_{2n+1},\Std}(\check{\xi})= d+ (N+ \lambda^{2n} E x)\frac{dx}{x}.
		\label{Bessel SO2n+1}
	\end{equation}
        After taking a gauge transformation by the matrix
	\begin{displaymath}
		\left( \begin{array}{ccccc}
		1 & 0 & \dots & \dots & 0 \\
		0 & 1 & \dots & \dots & 0 \\
		\dots & \dots & \dots & \dots & \dots \\
		0 & \dots & \dots & \dots & \dots \\
		2\lambda^{2n}x & 0 & \dots & 0 & 1  \end{array} \right).
	\end{displaymath}
	we obtain the scalar differential equation associated to $\Be_{\SO_{2n+1},\Std}(\check{\xi})$: 
	\begin{equation}
		(x\frac{d}{dx})^{2n+1}- \lambda^{2n}x(4x\frac{d}{dx}+2)=0. 
		\label{SO2n+1 equation}
	\end{equation}

	When $n\ge 2$, we can rewrite $\check{\xi}$ as 
	\begin{equation} \label{xi SO2n+1 normal}
		\check{\xi}= 
		\left( \begin{array}{ccccccc}
		0 & 1 & 0 & \dots & & 0 & 0  \\
		0 & \ddots & \ddots & & & & 0 \\
		\vdots &  & 0 & \sqrt{2} & 0 & & \vdots \\
		\vdots &  & 0 & 0 & \sqrt{2} & & \vdots \\
		\vdots&& 0 & 0 & 0 & \ddots & 0 \\
		0 & \dots &  &  & & & 1 \\ 
		0& 0& \dots && && 0	
	\end{array} \right)
		+
		\lambda^{2n}
		\left( \begin{array}{ccccccc}
		0 & 0 & \dots & \dots & & & 0 \\
		\vdots & \vdots & && & & \vdots \\
		\vdots&&&&&& \vdots \\
		\vdots&&&&&& \vdots \\
		0 &  \dots & & & & & 0 \\
		1 & 0 & \dots & \dots & & & 0 \\
		0 & 1 & 0 & \dots & & & 0  \end{array} \right)x, 
	\end{equation}
	where $\sqrt{2}$ is a square root of $2$ in $\overline{K}$ and appears in positions $(n,n+1)$ and $(n+1,n+2)$. 
	Via the natural inclusion $\so_{2n+1}\to \so_{2n+2}$ the above element $\check{\xi}\in (\so_{2n+1})_{\aff}(1)$ corresponds to $\check{\xi}\in (\so_{2n+2})_{\aff}(1)$ defined in \eqref{Bessel connection SO2n matrix}. 
	The standard $(2n+2)$-dimensional representation of $\so_{2n+2}$ decomposes as a direct sum of the trivial representation and the standard $(2n+1)$-dimensional representation of $\so_{2n+1}$ as representations of $\so_{2n+1}$. 
	Then we obtain decompositions of Bessel connections and Bessel $F$-isocrystals by proposition \ref{functoriality} 
\begin{equation} \label{decomposition SO2n+2 SO2n+1}
		\Be_{\SO_{2n+2},\Std}(\check{\xi})\simeq \Be_{\SO_{2n+1},\Std}(\check{\xi})\oplus (\mathscr{O}_{\mathbb{G}_{m,K}},d), \quad	
	\Be_{\SO_{2n+2},\Std}^{\dagger}(\check{\xi})\simeq \Be_{\SO_{2n+1},\Std}^{\dagger}(\check{\xi}) \oplus (\mathcal{O}_{\Gm},d).
	\end{equation}
	In the remaining ot this subsection, we omit $\check{\xi}$ from the notation. 
\end{secnumber}

\begin{rem}
	The fact that matrix $E$ in \eqref{SO2n+1 connection} takes value $2$ in its non-zero entries is delicate. On the one hand, it comes from the calculation of invariant polynomials. On the other hand,
	it ensures the existence of a Frobenius structure on the differential equation \eqref{SO2n+1 equation} with parameter $\lambda=-\pi$. 
	For instance, for every prime number $p$, the convergence domain of the unique solution of \eqref{SO2n+1 equation} ($\lambda=-\pi$) at $0$ : 
	\begin{displaymath}
		F(x)=\sum_{r\ge 0} \frac{(2r-1)!!}{(r!)^{2n+1}} (2\pi^{2n})^r x^{r},
	\end{displaymath}
	is the open unit disc of radius $1$. 
	In particular, $F(x)$ belongs to $K\{x\}$ \eqref{functions on open disc} and it justifies \eqref{solution space}.  
\end{rem}

\begin{prop} \label{Be SO2n+1 Std}
	\textnormal{(i)} When $p>2$, there exists an isomorphism of overconvergent $F$-isocrystals \eqref{Frob trace Hyp}
	\begin{equation} \label{SO2n+1 Hyp}
		\Be_{\SO_{2n+1},\Std}^{\dagger}\simeq [x\mapsto 4x]^{+}\THyp_{\psi}(2n+1,\rho),
	\end{equation}
	where $\rho$ denotes the quadratic character of $k^{\times}$.

	\textnormal{(ii)} When $p=2$, there exists an isomorphism of overconvergent $F$-isocrystals 
	\begin{equation} \label{SO2n+1 SL}
		\Be_{\SO_{2n+1},\Std}^{\dagger}\simeq \Be_{2n+1}^{\dagger}.
	\end{equation}
\end{prop}

\begin{proof}
	(i)
	If we rescale $x$ by $x\mapsto \frac{1}{4}x$, the differential equation \eqref{SO2n+1 equation} turns to the hypergeometric differential equation $\Hyp_{\psi}(2n+1;\rho)$ associated to $\rho$ \eqref{hyp equation Frobenius}. 
	Frobenius structures on two sides of \eqref{SO2n+1 Hyp} are of weight zero and have Frobenius eigenvalues $\{q^{-n},\cdots,q^{-1},0,q,\cdots,q^n\}$ at $0$ (\ref{Frobenius 0}, \ref{Hyp sum normalised}). 
	Then these two Frobenius structures coincide by theorem \ref{theorem Miyatani}(ii) and the isomorphism \eqref{SO2n+1 Hyp} follows. 	
	
	(ii) We will prove the assertion in Appendix \ref{appendix}. 
\end{proof}		
 	\begin{secnumber}
	    It follows that there exists an isomorphism of overconvergent $F$-isocrystals \eqref{def convolution}
	\begin{equation} \label{SO2n+1 conv}
		 \Be^{\dagger}_{\SO_{2n+1},\Std}\simeq \Be^{\dagger}_{\SO_3,\Std}\star \Be^{\dagger}_{2n-2}. 
         \end{equation}
         by the convolution interpretation of hypergeometric overconvergent $F$-isocrystals (\cite{Miy} Main theorem (ii) and 3.3.3). 
	\end{secnumber}

\begin{coro} \label{functoriality p=2}
	Suppose $p=2$. The $\SL_{2n+1}$-valued overconvergent $F$-isocrystals $\Be_{\SL_{2n+1}}^{\dagger}$  is the push-out of $\Be_{\SO_{2n+1}}^{\dagger}$ along $\SO_{2n+1}\to \SL_{2n+1}$. 
\end{coro}

It follows from \ref{Be SO2n+1 Std}(ii). 

\begin{coro} \label{identities exp sums}
	\textnormal{(i)} 
	The Frobenius trace function $\Kl_{\SO_{2n+1},\Std}$ of $\Be_{\SO_{2n+1},\Std}^{\dagger}$ is equal to 
	\begin{eqnarray}
		\Kl_{\SO_{2n+1},\Std}(a) &=& \sum_{x,y\in k^{\times},xy=a} \Kl_{\SO_3,\Std}(x)\Kl(2n-2;y) \label{Kl sum SO2n+1} \\
		&=& \left\{ 
			\begin{array}{ll}
			        \Kl(2n+1;a), & p=2, \\
				\widetilde{\rH}_{\psi}(2n+1;\rho)(4a), & p>2. 				
			\end{array}
			\right.
	\end{eqnarray}

	\textnormal{(ii)} 
	We have an identity of exponential sums \eqref{Kl sum SO2n simple}
	\begin{equation}
		\Kl_{\SO_{2n+2},\Std}(a)-1=\Kl_{\SO_{2n+1},\Std}(a).
		\label{identity exp sum SO2n}
	\end{equation}	
\end{coro}
\begin{proof}
	(i) The first equality follows from \eqref{SO2n+1 conv}. The second one follows from \ref{Be SO2n+1 Std}(i-ii). 

	(ii) It follows from proposition \ref{cal Kl sum SO2n} and \eqref{decomposition SO2n+2 SO2n+1}. 
\end{proof}

In particular, by \eqref{SL2 vs SO3}  and corollary \ref{identities exp sums}(i), we obtain \eqref{SO3 hyp intro}.
Using the triviality functoriality \ref{triv functoriality} and the exceptional isomorphism for groups of low ranks \eqref{SL2 vs SO3}-\eqref{SO6 vs SL4}, one can similarly obtain other identities between exponential sums, whose sheaf-theoretic incarnations were obtained by Katz \cite{Katz09}.

\subsection{Frobenius slopes of Bessel $F$-isocrystals}
\label{NT polygons}
\begin{secnumber} \label{def NT polygons}
	We first recall the definition of the \textit{Newton polygon} of a conjugacy class in $\cG(\overline{K})$. 
	Let $\cwX(\cT)^{+}$ be the set of dominant coweights of $\cG$ and $\cwX(\cT)_{\mathbb{R}}^+$ the positive Weyl chamber, equipped with the following partial order $\le$: $\mu\le \lambda$ if $\lambda-\mu$ can be written as a linear combination of positive coroots of $\cG$ with coefficients in $\mathbb{R}_{+}$.	
	We identify $(\cwX(\cT)\otimes_{\mathbb{Z}}\mathbb{R})/W$ and $\cwX(\cT)_{\mathbb{R}}^{+}$. Recall that $\rho$ denotes the half sum of positive roots of $G$
	\begin{displaymath}
		\rho=\frac{1}{2} \sum_{\alpha\in \Phi^{+}}\alpha \in \wX(T)=\cwX(\cT).
	\end{displaymath}

	Let $v:\overline{K}\to \mathbb{Q}\cup\{\infty\}$ be the $p$-adic order, normalised by $v(q)=1$. It induces a homomorphism of groups $v:\cT(\overline{K})\to \cwX(\cT)\otimes_{\mathbb{Z}}\mathbb{R}$. 
	By identifying $\cT(\overline{K})/W$ and the set of semisimple conjugacy classes $\Conj^{\SSS}(\cG(\overline{K}))$ in $\cG(\overline{K})$, we deduce a homomorphism:
	\begin{equation}
		\NP: \Conj^{\SSS}(\cG(\overline{K}))=\cT(\overline{K})/W \to (\cwX(\cT)\otimes_{\mathbb{Z}}\mathbb{R})/W=\cwX(\cT)_{\mathbb{R}}^{+}.
		\label{Newton polygon}
	\end{equation}

	In the case where $\cG=\GL_n$, $\NP$ is equivalent to the classical $p$-adic Newton polygon. Indeed, we have 
	\begin{displaymath}
		\cwX(\cT)_{\mathbb{R}}^{+}=\{(\lambda_1,\cdots,\lambda_n)\in \mathbb{R}^{n}, \lambda_1\le \cdots\le \lambda_n\},
	\end{displaymath}
	and we can associate to $(\lambda_1,\cdots,\lambda_n)$ a convex polygon with vertices $(i,\lambda_1+\cdots+\lambda_i)$ for $i\in \{1,\cdots,n\}$.  
	For $\lambda=(\lambda_1,\cdots,\lambda_n),\mu=(\mu_1,\cdots,\mu_n)$ in $\cwX(\cT)_{\mathbb{R}}^{+}$, $\mu\le \lambda$ if and only if the polygon associated to $\mu$ lies above that of $\lambda$ with the same endpoint. 
\end{secnumber}

\begin{theorem} \label{Frobenius slope}
	Let $x\in |\A1_{k}|$ be a closed point and $\varphi_{x}\in \cG(\overline{K})$ the Frobenius automorphism of $(\Be^{\dagger}_{\cG},\varphi)$ at $x$ \eqref{Frobenius Adagger}. 
	Let $v$ be the $p$-adic order normalised by $v(q^{\deg(x)})=1$ and $\NP$ defined as above.
	
	\textnormal{(i)} Except for finitely many closed points of $|\mathbb{A}^1_k|$, we have $\NP(\varphi_{x})=\rho$. 
	
	\textnormal{(ii)} Suppose that $\cG$ is of type $A_n,B_n,C_n,D_n$ or $G_2$, then we have $\NP(\varphi_x)=\rho$ for every $x\in|\A1_k|$. 
\end{theorem}

\begin{proof} (i) 
	In (\cite{VL11} 2.1), V. Lafforgue shows that the Newton polygon \eqref{Newton polygon} of the Hecke eigenvalue of a cuspidal function is $\le \rho$.  
	In particular, we deduce that $\NP(\varphi_x)\le \rho$ for all $x\in |\mathbb{G}_{m,k}|$. 
	By \ref{Frobenius 0}, we have $\NP(\varphi_0)=\NP(\rho(q))=\rho$. That is the Newton polygon achieves the upper bound $\rho$ at $0$.
	We take a finite set of tensor generators $\{V_1,\cdots,V_n\}$ of $\Rep(\cG)$. 
	Then the assertion follows by applying Grothendieck-Katz' theorem (cf. \cite{Cr86} 1.6) to log convergent $F$-isocrystals $\Be_{\cG,V_i}^{\dagger}$. 

	(ii) (a) The case where $\cG$ is of type $A_n,C_n$. By functoriality \eqref{functoriality}, we reduce to study the Frobenius slope of Bessel $F$-isocrystal $\Be_{n}^{\dagger}$ of rank $n$ \eqref{Dwork thm intro}. 
	After the work of Dwork, Sperber and Wan \cite{Dw74,Sp80,Wan93}, the Frobenius slope set of $\Be_n^{\dagger}$ (normalised to be weight $0$) at each closed point $x\in |\mathbb{A}_k^1|$ is equal to $\{-\frac{n-1}{2},-\frac{n-3}{2},\cdots,\frac{n-1}{2}\}$. 
	Then the assertion follows. 
	
	(b) The case where $\cG$ is of type $B_n,D_n,G_2$. 
	By functoriality \eqref{functoriality}, we reduce to show that the Frobenius slope set of $\Be_{\SO_{2n+1},\Std}^{\dagger}$ at each closed point is equal to $\{-n,-n+1,\cdots,n\}$. 
	If $p=2$, it follows from \ref{Be SO2n+1 Std}(ii) and the case (a). 
	If $p>2$, in view of \ref{Frob trace Hyp} and \ref{Be SO2n+1 Std}(i), it follows from the following lemma. 
\end{proof}

\begin{lemma} \label{Frob slope Hyp}	
	The Frobenius slope set of $\DHyp_{\psi}(2n+1;\rho)$ \eqref{theorem Miyatani} at each closed point is equal to 
	\begin{displaymath}
		\bigg\{\frac{1}{2},\frac{3}{2},\cdots,2n+\frac{1}{2}\bigg\}.
	\end{displaymath}
\end{lemma}
\begin{proof}	
	We deduce this fact from Wan's results on Frobenius slope of certain toric exponential sums \cite{Wan93,Wan04}. 

	For $a\in \mathbb{F}_q^{\times}$ and a divisor $d$ of $p-1$, consider the following Laurent polynomial in $\mathbb{F}_q[x_1^{\pm},\cdots,x_{2n+1}^{\pm}]$
	\begin{displaymath}
		f_d(x_1,\cdots,x_{2n+1})=x_1+\cdots+x_{2n}-x_{2n+1}^d+\frac{a x_{2n+1}^{d}}{x_1x_2\cdots x_{2n}}.
	\end{displaymath}
	For $m\ge 1$, we denote by $S_{m}(f_d)$ the exponential sum associated a Laurent polynomial:
	\begin{displaymath}
		S_{m}(f_d)=\sum_{x_i\in \mathbb{F}_{q^{m}}^{\times}}\psi\biggl(\Tr_{\mathbb{F}_{q^{m}}/\mathbb{F}_p}f_d(x_1,\cdots,x_{2n+1})\biggr).
	\end{displaymath}
	Then we have an identity
	\begin{equation} \label{identity Hyp toric}
		S_m(f_2)= S_m(f_1) +
		\sum_{\begin{subarray}{c} x_1\cdots x_{2n+1}=ay\\ x_i\in \mathbb{F}_{q^m}^{\times}\end{subarray}}
			\psi\biggl(\Tr_{\mathbb{F}_{q^m}/\mathbb{F}_p}(x_1+\cdots+x_{2n+1}-y)\biggr) \cdot \rho^{-1}\bigl( \Nm_{\mathbb{F}_{q^m/\mathbb{F}_q}}(y)\bigr), 
	\end{equation}
	where the last term is the Frobenius trace of $\DHyp(2n+1;\rho)$. 

	The L-function associated to these exponential sums is a rational function:
	\begin{displaymath}
		\rL(f_d,T)=\exp\biggl(\sum_{m\ge 1}S_{m}(f_d)\frac{T^m}{m}\biggr),\quad
	\end{displaymath}
	
	We denote by $\Delta(f_d)$ the convex closure in $\mathbb{R}^{2n+1}$ generated by the origin and lattices defined by exponents appeared in $f_d$:
	\begin{displaymath}
		\{(0,\cdots,0), (1,\cdots,0),\cdots, (0,\cdots,1,0), (0,\cdots,0,d), (-1,\cdots,-1,d)\}.
	\end{displaymath}
	The polyhedron $\Delta(f_d)$ is $(2n+1)$-dimensional and has volume $\frac{d}{2n!}$. 
	The Laurent polynomials $f_d$ is non-degenerate (cf. \cite{Wan04} Def. 1.1). After Adolphson-Sperber \cite{AS89}, the L-function $\rL(f_d,T)$ is a polynomial of degree $d(2n+1)$. 

	We denote by $\NP(f_d)$ the (Frobenius) Newton polygon associated to L-functions $\rL(f_d,T)$ (cf. \cite{Wan04} 1.1) and by $\HP(f_d)$ the Hodge polygon defined in term of the polyhedron $\Delta(f_d)$ (cf. \cite{Wan04} 1.2). 
	The (multi-)set of slopes of $\HP(f_d)$ is 
	\begin{equation} \label{slope HP}
		\bigg\{0,\frac{1}{d},\frac{2}{d},\cdots, 2n+\frac{d-1}{d}\bigg\}.
	\end{equation}

	The Newton polygon lies above the Hodge polygon \cite{AS89}. 
	A Laurent polynomial is called \textit{ordinary} if these two polygons coincide. 
	Let $\delta$ be a co-dimension $1$ face of $\Delta$ which does not contain the origin and $f^{\delta}_d$ the restriction of $f_d$ to $\delta$ which is also non-degenerate. 
	The Laurent polynomial $f^{\delta}_{d}$ is diagonal in the sense of (\cite{Wan04} \S~2). 
	If $V_1,\cdots,V_{2n+1}$ denote the vertex of $\delta$ written as column vectors, the set $S(\delta)$ of solutions of
	\begin{displaymath}
		(V_1,\cdots,V_{2n+1}) 
		\begin{pmatrix}
			r_1 \\ \vdots \\ r_{2n+1}
		\end{pmatrix}
		\equiv 0 ~(\textnormal{mod } 1),\quad r_i \textnormal{ rational, } 0\le r_i<1,
	\end{displaymath}
	forms an abelian group of order $d$ (cf. \cite{Wan04} 2.1). 
	Since $d$ is a divisor of $p-1$, we deduce that for each $\delta$, $f_d^{\delta}$ is ordinary by (\cite{Wan04} Cor. 2.6). 
	By Wan's criterion for the ordinariness \cite{Wan93} (cf. \cite{Wan04} Thm. 3.1), $f_d$ is ordinary. 
	
	In view of \eqref{identity Hyp toric} and the slope sets of $\HP(f_1), \HP(f_2)$ \eqref{slope HP}, the assertion follows. 
\end{proof}

\appendix

\renewcommand{\thesubsection}{\Alph{subsection}}

\counterwithout{theorem}{subsection}
\counterwithin{theorem}{section}
\counterwithout{equation}{theorem}
\counterwithin{equation}{theorem}

\section{A $2$-adic proof of Carlitz's identity and its generalization} \label{appendix}

As mentioned in introduction, Carlitz \cite{Car69} proved the following identity between Kloosterman sums:
\begin{displaymath}
	\Kl(3;a)=\Kl(2;a)^2-1,\qquad \forall~ a\in \mathbb{F}_{2^s}^{\times}.
\end{displaymath}
In this appendix, we reprove and generalize this identity by establishing an isomorphism between two Bessel $F$-isocrystals $\Be_{2n+1}^{\dagger}$ and $\Be_{\SO_{2n+1},\Std}^{\dagger}$. The following is a restatement of proposition \ref{Be SO2n+1 Std}(ii).

\begin{prop} \label{Carlitz diff}
	There exists an isomorphism between following two overcovergent $F$-isocrystals on $\mathbb{G}_{m,\mathbb{F}_2}$ \eqref{SO2n+1 equation}: 
	\begin{equation}
		\Be_{2n+1}^{\dagger}:(x\frac{d}{dx})^{2n+1}+2^{2n+1}x=0,\quad \Be_{\SO_{2n+1},\Std}^{\dagger}: (x\frac{d}{dx})^{2n+1} - 2^{2n+1} x (2x\frac{d}{dx}+1)=0.
		\label{Be3 BeSO3}
	\end{equation}
\end{prop}

Our strategy is first to show that their maximal slope quotient convergent $F$-isocrystals are isomorphic. 
Then we conclude the proposition by a dual version of a minimal slope conjecture (proposed by Kedlaya \cite{Ked16} and recently proved by Tsuzuki \cite{Tsu19}) that we briefly recall in the following. 

\begin{secnumber}
	We keep the notation of section \ref{sec applications}. 
	Let $X$ be a smooth $k$-scheme. Let $\mathscr{M}^{\dagger}$ be an overconvergent $F$-isocrystal on $X/K$. We denote the associated convergent $F$-isocrystal on $X/K$ by $\mathscr{M}$.
	When the (Frobenius) Newton polygons of $\mathscr{M}$ are constant on $X$, $\mathscr{M}$ admits a \textit{slope filtration}, that is an increasing filtration 
	\begin{equation} \label{slope filtration}
	0=\mathscr{M}_{0} \subsetneq \mathscr{M}_{1} \subsetneq \cdots \subsetneq \mathscr{M}_{r-1} \subsetneq \mathscr{M}_{r}=\mathscr{M}
	\end{equation}
 of convergent $F$-isocrystals on $X/K$ such that 
\begin{itemize}
	\item $\mathscr{M}_{i}/\mathscr{M}_{i-1}$ is isoclinic of slope $s_i$ and 
	\item $s_1< s_2<\cdots < s_r$.
\end{itemize}

By Grothendieck's specialization theorem, for any convergent $F$-isocrystal $\mathscr{M}$ on $X/K$, there exists an open dense subscheme $U$ of $X$ such that the Newton polygons of $\mathscr{M}$ are constant. 

We remark that for a log convergent $F$-isocrystal with constant Newton polygons over a smooth $k$-scheme with normal crossing divisor, such a slope filtration (of sub log convergent $F$-isocrystals) also exists. 

In a recent preprint, Tsuzuki showed a dual version of Kedlaya's minimal slope conjecture (\cite{Ked16} 5.14):
\end{secnumber}

\begin{theorem}[\cite{Tsu19} theorem 1.3]
	Let $X$ be a smooth connected curve over $k$. Let $\mathscr{M}^{\dagger},\mathscr{N}^{\dagger}$ be two irreducible overconvergent $F$-isocrystals such that the corresponding convergent $F$-isocrystals $\mathscr{M},\mathscr{N}$ admit the slope filtrations $\{\mathscr{M}_i\}$, $\{\mathscr{N}_i\}$ respectively. 
	We renumber the slope filtration by 
	\begin{equation} \label{dual slope filtration}
\mathscr{M}=\mathscr{M}^{0} \supsetneq \mathscr{M}^{1} \supsetneq \mathscr{M}^{2} \supsetneq \cdots \supsetneq \mathscr{M}^{r-1} \supsetneq \mathscr{M}^{r}=0
	\end{equation}
	with slopes $s^0>s^1>\cdots >s^{r-1}$. 
	Suppose there exists an isomorphism $h:\mathscr{N}/\mathscr{N}^1\xrightarrow{\sim} \mathscr{M}/\mathscr{M}^1$ of convergent $F$-isocrystals between the maximal slope quotients. 
	Then there exists a unique isomorphism $g^{\dagger}:\mathscr{N}^{\dagger}\xrightarrow{\sim} \mathscr{M}^{\dagger}$ of overconvergent $F$-isocrystals, which is compatible with $h$ as morphisms of convergent $F$-isocrystals.  
	\label{Tsuzuki thm}
\end{theorem}

\begin{secnumber} 
	Following Dwork's strategy (\cite{Dw68} \S~1-3), we study the maximal slope quotients of $\Be_{2n+1}^{\dagger}$ and of $\Be_{\SO_{2n+1},\Std}^{\dagger}$ in terms of their unique solutions at $0$. 

	In the following, we assume $k=\mathbb{F}_p$. 
	We first recall Dwork's congruences and show a refinement of his result in the $2$-adic case. 
	Consider for every $i\ge 0$, a map $B^{(i)}(-):\mathbb{Z}_{\ge 0}\to K^{\times}$ and the following congruence relation for $0\le a<p$ and $n,m,s \in \mathbb{Z}_{\ge 0}$: 
\begin{itemize}
	\item[(a)] $B^{(i)}(0)$ is a $p$-adic unit for all $i\ge 0$,
	
	\item[(b)] $ \displaystyle \frac{B^{(i)}(a+np)}{B^{(i+1)}(n)} \in R$ for all $i\ge 0$, 
	\item[(c)] $  \displaystyle \frac{B^{(i)}(a+np+mp^{s+1})}{B^{(i+1)}(n+mp^{s})} 
		\equiv \frac{B^{(i)}(a+np)}{B^{(i+1)}(n)} \mod p^{s+1}$ for all $i\ge 0$.
	\item[(c')] When $p=2$, $  \displaystyle \frac{B^{(i)}(a+n2+m2^{s+1})}{B^{(i+1)}(n+m2^{s})} 
		\equiv u(i,s,m) \frac{B^{(i)}(a+n2)}{B^{(i+1)}(n)} \mod 2^{s+1}$ for all $i\ge 0$, where $u(i,s,m)=1$ if $s\neq1$ and $u(i,1,m)=1$ or $-1$ depending on $i$ and $m$. 
\end{itemize}
If conditions (a-c) (or (a,b,c')) are satisfied, then $B^{(i)}(n)\in R$ for all $i,n\ge 0$. We set 
\begin{eqnarray*}
	&F^{(i)}(x)=\sum_{j=0}^{\infty} B^{(i)}(j) x^j  \quad \in K \llbracket x \rrbracket, &\\
	&F^{(i)}_{m,s}(x)=\sum_{j=mp^s}^{(m+1)p^{s}-1} B^{(i)}(j) x^j\quad \in K[x],& \quad s\ge 0. 
\end{eqnarray*}
We write $F^{(i)}_{0,s}$ by $F^{(i)}_{s}$ for simplicity. 
\end{secnumber}
\begin{theorem} \label{thm Dwork congruence}
	\textnormal{(i) [\cite{Dw68} theorem 2]} If conditions (a-c) are satisfied, then
	\begin{equation}
		F^{(0)}(x)F^{(1)}_{m,s}(x^{p})\equiv F^{(0)}_{m,s+1}(x) F^{(1)}(x^p) \quad \mod p^{s+1} B^{(s+1)}(m) \llbracket x \rrbracket.
		\label{Dwork congruence thm}
	\end{equation}

	\textnormal{(i')} If conditions (a,b,c') are satisfied (in particular $p=2$), then 
	\begin{equation}
		F^{(0)}(x)F^{(1)}_{m,s}(x^{2})\equiv F^{(0)}_{m,s+1}(x) F^{(1)}(x^2) \quad \mod 2^{s}B^{(s+1)}(m) \llbracket x \rrbracket.
		\label{Dwork congruence modified}
	\end{equation}

	\textnormal{(ii) [\cite{Dw68} theorem 3]} 
	Under the assumption of (i) or (i') and suppose moreover that
	\begin{itemize}
		\item[(d)] $B^{(i)}(0)=1$ for $i\ge 0$.
		\item[(e)] $B^{(i+r)}=B^{(i)}$ for all $i\ge 0$ and some fixed $r\ge 1$.
	\end{itemize}
	Let $U$ be the open subscheme of $\mathbb{A}^1_{k}$ defined by 
	\begin{equation} \label{open locus U}
		F_{1}^{(i)}(x)\neq 0,\quad \textnormal{for } i=0,1,\cdots,r-1. 
	\end{equation}
	Then the limit 
	\begin{equation} \label{limit def f(x)}
		f(x)=\lim_{s\to \infty} F^{(0)}_{s+1}(x)/F^{(1)}_{s}(x^p)
	\end{equation}
	defines a global function on the formal open subscheme $\UU$ of $\widehat{\mathbb{A}}^{1}_R$ associated to $U$, which takes $p$-adic unit value at each rigid point of $\UU^{\rig}$. 
\end{theorem}

We prove assertion (i') in the end \eqref{proof congruence}. 
We briefly explain Dwork's result (ii) in the language of formal schemes. 
	The condition \eqref{open locus U} implies that $F_s^{(i)}\neq 0$ on $U$ (cf. \cite{Dw68} 3.4). 
	For $s\ge 1$, the congruences \eqref{Dwork congruence thm} and \eqref{Dwork congruence modified} imply that
	\begin{displaymath}
		F_{s+1}^{(0)}(x)/ F_s^{(1)}(x^p)= F_s^{(0)}(x)/F_{s-1}^{(1)}(x^p) \quad \in \Gamma(\UU,\mathscr{O}_{\UU}/p^{s-1}\mathscr{O}_{\UU}).
	\end{displaymath}
	This allows us to use \eqref{limit def f(x)} to define a global function $f$ of $\mathscr{O}_{\UU}$.

\begin{secnumber} \label{function congruence}
	Let $F(x)=\sum_{j\ge 0} B(j) x^{j}$ be a formal power series in $R\llbracket x \rrbracket$.
	We say \textit{$F$ satisfies Dwork's congruences} if by setting $B^{(i)}(j)=B(j)$ for every $i\ge 0$, conditions of theorem \ref{thm Dwork congruence}(ii) are satisfied.
	
	We take such a function $F$ and then we obtain a function $f\in \Gamma(\UU,\mathscr{O}_{\UU})$ coinciding with $F(x)/F(x^p)$ in $K\{x\}$ \eqref{functions on open disc} (i.e. the open unit disc).
	Moreover, by (\cite{Dw68} lemma 3.4(ii)), there exists a function $\eta\in \Gamma(\UU,\mathscr{O}_{\UU})$ coinciding with $F'(x)/F(x)$ in $K\{x\}$ defined by
	\begin{displaymath}
		\eta(x)\equiv F'_{s+1}(x)/F_{s+1}(x) \mod p^{s}.
	\end{displaymath}
	The functions $f(x)$ and $\eta(x)$ satisfy a differential equation:
	\begin{displaymath}
		\frac{f'(x)}{f(x)}+px^{p-1}\eta(x^{p})=\eta(x).
	\end{displaymath}
	Note that $f(0)=F(0)/F(0)=1$. Then we deduce that the following corollary. 
\end{secnumber}

\begin{coro}
	The connection $d- \eta$ on the trivial bundle $\mathscr{O}_{\UU^{\rig}}$ and the function $f$ form a unit-root convergent $F$-isocrystal $\mathscr{E}_F$ on $U/K$, whose Frobenius eigenvalue at $0$ is $1$.
\end{coro}

\begin{secnumber}
	Let $\mathscr{M}^{\dagger}$ be an overconvergent $F$-isocrystal on $\mathbb{G}_{m,k}$ over $K$ of rank $r$ whose underlying bundle is trivial and the connection is defined by a differential equation: 
	\begin{equation} \label{diff equation}
		P(\delta)=\delta^r+p_{r}\delta^{r-1}+\cdots + p_{1}=0, 
	\end{equation}
	where $\delta=x\frac{d}{dx}$, $p_i\in \Gamma(\widehat{\mathbb{A}}_R^1,\mathscr{O}_{\widehat{\mathbb{A}}_R^1})[\frac{1}{p}]$.
	We assume moreover that $\mathscr{M}^{\dagger}$ is \textit{unipotent at $0$ with a maximal unipotent local monodromy}. 
	Then $\mathscr{M}^{\dagger}$ extends to a log convergent $F$-isocrystal $\mathscr{M}^{\log}$ on $(\mathbb{A}^1,0)$ and its Frobenius slopes at $0$ are
	\begin{displaymath}
		s_1<s_2=s_1+1<\cdots< s_r=s_1+r-1.
	\end{displaymath}
	Note that $\mathscr{M}^{\dagger}$ is indecomposable in $\Fr\Iso^{\dagger}(\mathbb{G}_{m,k}/K)$ and so is $\mathscr{M}$ in $\Fr\Iso(\mathbb{G}_{m,k}/K)$. 
	Then by Drinfeld-Kedlaya's theorem on the generic Frobenius slopes \cite{DK17}, we deduce property (i):

	\textnormal{(i)} The generic Frobenius slopes (mult-)set is $\{s_1,\cdots,s_{r}\}$ with $s_i=s_1+i-1$. 

	\textnormal{(ii)} In view of \eqref{solution space}, the differential equation $D=0$ admits a unique solution at $0$:
	\begin{displaymath}
		F(x)=\sum_{n\ge 0} A(n) x^n \quad \in K\{x\},\quad \textnormal{ with $A(0)=1$}.
	\end{displaymath}
\end{secnumber}
\begin{prop} \label{max isocrystal}
	Suppose the function $F(x)$ satisfies Dwork's congruences \eqref{function congruence} and let $\mathscr{E}_F$ be the associated unit-root convergent $F$-isocrystal on $U \subset \A1_k$. Then

	\textnormal{(i)} There exists an epimorphism of log convergent isocrystals $\mathscr{M}^{\log}\to \mathscr{E}_F$ on $(U,0)$. 
	
	\textnormal{(ii)} 
	As log convergent isocrystals, $\mathscr{E}_F$ coincides with the maximal slope quotient $\mathscr{M}^{\log}/\mathscr{M}^{\log,1}$ of $\mathscr{M}^{\log}$ \eqref{dual slope filtration}. 
\end{prop}
\begin{proof}
	(i) We set $A=\Gamma(\UU,\mathscr{O}_{\UU})[\frac{1}{p}]$. 
	We claim that there exists a decomposition of differential operators:
	\begin{equation} \label{decomposition P by eta}
		P(\delta)=Q(\delta)(\delta-x\eta),\quad Q(\delta)=\delta^{r-1}+q_{r-1}\delta^{r-2}+\cdots+ q_{1}, \quad q_i\in A.
	\end{equation}
	Indeed, by the Euclidean algorithm (\cite{Ked10} 5.5.2), there exists $r\in A$ such that $P=Q(\delta-x\eta)+r$. 
	By evaluating the above identity at $F$ (in the ring $K\{x\}$ containing $A$), we obtain
	\begin{displaymath}
		P(\delta)(F)=0=Q(\delta)(\delta-x\eta)(F)+rF=rF.
	\end{displaymath}
	Then we deduce $r=0$ and \eqref{decomposition P by eta} follows. 

	Let $e_1,\cdots,e_r$ be a basis of $\mathscr{M}$ such that $\nabla_{\delta}(e_i)=e_{i+1}, 1\le i\le r-1$ and $\nabla_{\delta}(e_r)=-(p_r e_r+\cdots+p_1 e_1)$. 
	We consider a free $\mathscr{O}_{\UU^{\rig}}$-module with a log connection $\mathscr{N}$ with a basis $f_1,\cdots,f_{r-1}$ and the connection defined by $\nabla_{\delta}(f_i)=f_{i+1}$, $\nabla_{\delta}(f_{r-1})=-(q_{r-1} f_{r-1}+\cdots+q_{1}f_{1})$. 
	By \eqref{decomposition P by eta}, the morphism $f_1\mapsto e_2- x\eta e_1$ induces a horizontal monomorphism $\mathscr{N}\to \mathscr{M}^{\log}$ whose cokernel is isomorphic to $\mathscr{E}_F$.

	(ii) Note that $\Pic(\UU^{\rig})\simeq \Pic(U)$ (\cite{VdP} 3.7.4) is trivial. 
	Then the rank one convergent isocrystal $\mathscr{M}^{\log}/\mathscr{M}^{\log,1}$ can be represented as a connection $d-\lambda$ on the trivial bundle $\mathscr{O}_{\UU^{\rig}}$. 
	
	Since $\mathscr{M}^{\log}$ has a maximal unipotent at $0$, the rank one quotient of the restriction $\mathscr{M}^{\log}|_0$ of $\mathscr{M}^{\log}$ at the open unit disc around $0$ is unique \eqref{review MC Swan}. 
	In particular, $d-\lambda$ kills the unique solution $F$ of $P(\delta)=0$. By analytic continuation, we have $\lambda=\eta$ and the assertion follows. 
\end{proof}

\begin{rem}
	The unique solution $F(x)$ belongs to the ring $K\llbracket x \rrbracket_0 = R \llbracket x \rrbracket \otimes_R K$ of bounded functions on open unit disc, which is a subring of $K\{x\}$. 
	Assertion (ii) can be viewed as an example of Dwork-Chiarellotto-Tsuzuki conjecture on the comparison between the log-growth filtration (of solutions) and Frobenius slope filtration \cite{CT09}. 
	This conjecture was recently proved by Ohkubo \cite{Ohk18}. 
\end{rem}

\begin{secnumber} \textit{Proof of proposition \ref{Carlitz diff}.}	
	We set $k=\mathbb{F}_2$ and apply the above discussions to overconvergent $F$-isocrystals $\mathscr{M}^{\dagger}=\Be_{2n+1}^{\dagger}$ and $\mathscr{N}^{\dagger}=\Be_{\SO_{2n+1},\Std}^{\dagger}$ on $\mathbb{G}_{m,\mathbb{F}_2}/K$ \eqref{Be3 BeSO3}. 
	Their unique solutions at $0$ are:
\begin{displaymath}
	F(x)= \sum_{r\ge 0} \frac{(-2)^{(2n+1)r}}{(r!)^{2n+1}} x^r,\qquad G(x)=\sum_{r\ge 0} \frac{2^{(2n+1)r}(2r-1)!!}{(r!)^{2n+1}} x^r. 
\end{displaymath}

In the following lemma, we show that $F$ and $G$ satisfy Dwork's congruences and that the associated maximal slope quotients $\mathscr{E}_F$ and $\mathscr{E}_G$ \eqref{max isocrystal} are isomorphic. 
Then proposition \ref{Carlitz diff} follows from theorem \ref{Tsuzuki thm} and the following lemma. \hfill $\qed$ 
\end{secnumber}
\begin{lemma}
	\textnormal{(i)} The functions $F(x)$ and $G(x)$ satisfy Dwork's congruences \eqref{function congruence} and define unit-root convergent $F$-isocrystals $\mathscr{E}_F$ and $\mathscr{E}_G$ on $\mathbb{A}_k^1$ respectively. 

	\textnormal{(ii)} The function $F(x)/G(x)$ extends to a global function of $\widehat{\mathbb{A}}^1_R$ and induces an isomorphism $\mathscr{E}_G\xrightarrow{\sim} \mathscr{E}_F$. 
\end{lemma}
\begin{proof}
	(i) Conditions (a,b,d,e) are easy to verified. The coefficients of $F(x)$ (resp. $G(x)$) satisfy condition (c') (resp. (c)), that is
	\begin{eqnarray*}
		\frac{(-2)^{(2n+1)(a+\ell 2+m2^{s+1})}/((a+\ell 2+m2^{s+1})!)^{2n+1}}{(-2)^{(2n+1)(\ell +m2^{s})}/((\ell +m2^{s})!)^{2n+1}} \equiv u(s,m) 
		\frac{ (-2)^{(2n+1)(a+\ell 2)}/ ((a+\ell 2)!)^{2n+1}}{(-2)^{(2n+1)\ell }/ (\ell !)^{2n+1}} \mod 2^{s+1}, 
	\end{eqnarray*}
	where $u(1,m)=(-1)^m$ and $u(s,m)=1$ if $s\neq 1$, and 
	\begin{eqnarray*}
		\frac{(2(a+\ell 2+m2^{s+1})-1)!! 2^{(2n+1)(a+\ell 2+m2^{s+1})}/((a+\ell 2+m2^{s+1})!)^{2n+1}}{(2(\ell +m2^{s})-1)!!2^{(2n+1)(\ell +m2^{s})}/((\ell +m2^{s})!)^{2n+1}} \equiv  \\
		\frac{(2(a+\ell 2)-1)!!2^{(2n+1)(a+\ell 2)}/ ((a+\ell 2)!)^{2n+1}}{(2\ell -1)!!2^{(2n+1)\ell }/ (\ell !)^{2n+1}} \mod 2^{s+1}. 	
	\end{eqnarray*}
	Since $F_1(x)\equiv G_1(x)\equiv 1\mod 2$, the $F$-isocrystals $\mathscr{E}_{F},\mathscr{E}_G$ are defined over $\mathbb{A}_k^1$. 

	(ii) We set $B^{(0)}(r)=\frac{(-2)^{(2n+1)r}}{(r!)^{2n+1}}$ and $B^{(1)}(r)=\frac{2^{(2n+1)r}(2r-1)!!}{(r!)^{2n+1}}$ and $B^{(i+2)}=B^{(i)}$. 
	Then these sequences satisfy conditions (a,b,c',d,e). For condition (c'), the constants $u(i,1,m)$ are given by  
	\begin{displaymath}
		u(0,1,m)=1, \quad u(1,1,m)=(-1)^m ,\quad u(i+2,1,m)=u(i,1,m). 
	\end{displaymath}
	Since $F_1(x)\equiv G_1(x)\equiv 1\mod 2$, $F(x)/G(x^2)$ extends to a global function of $\mathscr{O}_{\widehat{\mathbb{A}}_R^1}$ by theorem \ref{thm Dwork congruence} and so is $F(x)/G(x)$. 
	Then the assertion follows. 
\end{proof}
\begin{secnumber} \label{proof congruence} \textit{Proof of theorem \ref{thm Dwork congruence}}(i'). 
	We prove assertion (i') by modifying the argument of (\cite{Dw68} theorem 2). 
	Note that condition (c') implies the following congruence relation:
	\begin{equation}
		\frac{B^{(i)}(a+n2+m2^{s+1})}{B^{(i+1)}(n+m2^{s})} \equiv \frac{B^{(i)}(a+n2)}{B^{(i+1)}(n)} \mod 2^{s}.
		\label{condition c weak}
	\end{equation}
	When $n<0$, we set $B^{(i)}(n)=0$. 
	We set $A=B^{(0)}$, $B=B^{(1)}$ and for $a\in \{0,1\}$, $j,N\in \mathbb{Z}$, we set
	\begin{eqnarray*}
	U_{a}(j, N) &=&A(a+2(N-j)) B(j)-B(N-j) A(a+2j), \\ 
	H_{a}(m, s, N) &=&   \sum_{j=m2^s}^{(m+1)2^s-1} U_{a}(j, N).
\end{eqnarray*}
	Then the assertion is equivalent to 
	\begin{equation}
		H_{a}(m,s,N)\equiv 0 \mod 2^{s} B^{(s+1)}(m),\quad \textnormal{for } s\ge 0, m\ge 0, N\ge 0.
		\label{equivalent congruenc}
	\end{equation}
	By condition (b), we have $A(a+2m)/B(m)\in R$ and hence 
	\begin{displaymath}
		U_a(m,N) \equiv 0 \mod B(m).
	\end{displaymath}
	Then equation \eqref{equivalent congruenc} for $s=0$ follows from the fact that $H_a(m,0,N)=U_a(m,N)$. 

	We now prove by induction on $s$. We write the induction hypothesis 
	\begin{displaymath}
		\alpha_s: H_a(m,u,N)\equiv 0 \mod 2^u B^{(u+1)}(m),\quad \textnormal{for } u\in [0,s), m,N\ge 0.
	\end{displaymath}
	We may assume $\alpha_{s}$ for fixed $s\ge 1$. The main step is to show for $0\le t\le s$ that 
	\begin{displaymath}
		\beta_{t,s}: v(s,t,m) H_a(m,s,N+m2^s)\equiv \sum_{j=0}^{2^{s-t}-1} B^{(t+1)}(j+m2^{s-t})H_a(j,t,N)/B^{(t+1)}(j)\mod 2^s B^{(s+1)}(m).
	\end{displaymath}
	where $v(s,t,m)=1$ or $-1$ depending on $s,t,m$. 
	
	We list some elementary facts (cf. \cite{Dw68} 2.5-2.7)
	\begin{eqnarray}
		&\sum_{m=0}^{T} H_{a}(m, s, N)=0 \quad \textnormal{ if } (T+1) 2^{s}>N & \label{2.5} \\
		&H_{a}(m, s, N)= H_{a}(2m, s-1, N) + H_{a}(1+2m, s-1, N) \quad \textnormal{ if } s\ge 1 & \label{2.6} \\
		&B^{(t)}\left(i+m 2^s\right) \equiv 0 \mod B^{(s+t)}(m) \quad \text { if } 0 \leq i \leq 2^s-1, s,t \geq 0. & \label{2.7} 
	\end{eqnarray}

	We first prove $\beta_{0,s}$. We have
	\begin{eqnarray}
	&H_{a}\left(m, s, N+m 2^{s}\right)=\sum_{j=0}^{2^{s}-1} U_{a}\left(j+m 2^{s}, N+m 2^{s}\right), & \nonumber \\
	&\qquad U_{a}\left(j+m 2^{s}, N+m 2^{s}\right)=A(a+2(N-j)) B(j+m 2^{s})-B(N-j) A\left(a+2 j+m 2^{s+1}\right). \label{2.8} &
	\end{eqnarray}
	By \eqref{condition c weak}, we have 
	\begin{displaymath}
		A\left(a+2 j+m 2^{s+1}\right)= A(a+2j)B(j+m2^s)/B(j)+X_j 2^s B(j+m 2^s),
	\end{displaymath}
	where $X_j\in R$. Then the right hand side of \eqref{2.8} is 
	\begin{displaymath}
		B\left(j+m 2^{s}\right)\biggl(U_{a}(j, N) / B(j)-2^{s} X_{j} B(N-j)\biggr).
	\end{displaymath}
	Since $U_{a}(j,N)=H_a(j,0,N)$, we obtain
	\begin{displaymath}
		H_{a}\left(m, s, N+m 2^{s}\right)=
		\sum_{j=0}^{2^{s}-1} B\left(j+m 2^{s}\right)H_{a}(j, 0, N) / B(j)-2^{s} \sum_{j=0}^{2^{s}-1} X_{j} B\left(j+m 2^{s}\right) B(N-j).
	\end{displaymath}
	Since $X_jB(N-j)\in R$, it follows from \eqref{2.7} ($B=B^{(1)}$) that the second sum is congruent to zero modulo $2^s B^{(s+1)}(m)$. This proves $\beta_{0,s}$ with $v(s,0,m)=1$. 

	With $s$ fixed, $s\ge 1$, $t$ fixed, $0\le t \le s-1$, we show that $\beta_{t,s}$ together with $\alpha_s$ imply $\beta_{t+1,s}$. 
	To do this we put $j=\mu+2i$ in the right side of $\beta_{t,s}$ and write it in the form 
	\begin{displaymath}
	\sum_{\mu=0}^{1} \sum_{i=0}^{2^{s-t-1}} B^{(t+1)}\left(\mu+ 2i+m 2^{s-i}\right) H_{a}(\mu+2 i, t, N) / B^{(t+1)}(\mu+2i).
	\end{displaymath}
	By condition (c'), we have, 
	\begin{eqnarray*}
	&&B^{(t+1)}\left(\mu+2i+m2^{s-t}\right) \\
	&&=
	u(t+1,s-t-1,m)\left(B^{(t+1)}(\mu+2 i) B^{(t+2)}\left(i+m2^{s-t-1}\right) / B^{(t+2)}(i)\right)+X_{i, \mu} 2^{s-t} B^{(t+2)}\left(i+m 2^{s-t-1}\right),
	\end{eqnarray*}
	where $X_{i,\mu}\in R$. Thus the general term in the above double sum is
	\begin{displaymath}
		u(t+1,s-t-1,m)\biggl(B^{(t+2)}(i+m2^{s-t-1})H_a(\mu+2i,t,N)/B^{(t+2)}(i)\biggr)+Y_{i,\mu},
	\end{displaymath}
	where the error term:
	\begin{displaymath}
	Y_{i, \mu}=X_{i, \mu} 2^{s-t} B^{(t+2)}\left(i+m 2^{s-t-1}\right) H_{a}(\mu+2 i, t, N) / B^{(t+1)}(\mu+2i).
	\end{displaymath}
	For this error term, since $t<s$, we can apply $\alpha_s$ to conclude that
	\begin{displaymath}
		Y_{i,\mu}\equiv 0 \mod B^{(t+2)}(i+m2^{s-t-1})2^s.
	\end{displaymath}
	Then we can use \eqref{2.7} to conclude that 
	\begin{displaymath}
		Y_{i,\mu}\equiv 0 \mod 2^s B^{(s+1)}(m).
	\end{displaymath}
	After modulo $2^s B^{(s+1)}(m)$, the right side of $\beta_{t,s}$ is equal to 
	\begin{displaymath}
	u(t+1,s-t-1,m)\sum_{\mu=0}^{1} \sum_{i=0}^{2^{s-t-1}-1} B^{(t+2)}\left(i+m2^{s-t-1}\right) H_{a}(\mu+2 i, t, N) / B^{(t+2)}(i). 
	\end{displaymath}
	By reversing the order of summation and using \eqref{2.6}, the above sum is the same as
	\begin{displaymath}
	u(t+1,s-t-1,m)\sum_{i=0}^{2^{s-t-1}-1} B^{(t+2)}\left(i+m2^{s-t-1}\right) H_{a}(i, t+1, N) / B^{(t+2)}(i),
	\end{displaymath}
	which proves $\beta_{t+1,s}$. In particular, we obtain $\beta_{s,s}$, which states
	\begin{equation}
		v(s,s,m) H_a(m,s,N+m2^s)\equiv B^{(s+1)}(m)H_a(0,s,N)/B^{(s+1)}(0)\mod 2^s B^{(s+1)}(m).		
		\label{beta ss}
	\end{equation}
	We now consider the statement (with $s$ fixed before)
	\begin{displaymath}
		\gamma_{N}:H_a(0,s,N)\equiv 0 \mod 2^s.
	\end{displaymath}
	We know that $\gamma_N$ is true for $N<0$. Let $N'$ (if it exists) be the minimal value of $N$ for which $\gamma_{N'}$ fails. For $m\ge 1$, since $B^{(s+1)}(0)$ is a unit, we have by \eqref{beta ss}
	\begin{eqnarray*}
		H_a(m,s,N')\equiv v(s,s,m) B^{(s+1)}(m)H_a(0,s,N'-m2^s)/B^{(s+1)}(0) \equiv 0 \mod 2^s. 
	\end{eqnarray*}
	Applying this to \eqref{2.5}, we obtain that
	\begin{displaymath}
		H_a(0,s,N')\equiv 0 \mod 2^s.
	\end{displaymath}
	Thus $\gamma_N$ is valid for all $N$ and equation \eqref{beta ss} implies $\alpha_{s+1}$. This proves assertion (i'). \hfill \qed 
\end{secnumber}

\end{document}